\documentclass{article}
\usepackage[english]{babel}
\usepackage{graphicx}
\usepackage{slashed}
\usepackage{stmaryrd}
\usepackage{mathtools}
\usepackage{amsmath,amsfonts,amssymb}

\usepackage{ntheorem}
\usepackage{dsfont}
\usepackage{verbatim}
\usepackage{float}
\usepackage{accents}
\usepackage{mathrsfs}
\usepackage{amsmath}
\usepackage{amssymb}
\usepackage{hyperref}
\usepackage{tikz}
\usepackage{pgfplots}
\let\savering\ring
\let\ring\relax
\usepackage{mathabx}
\let\ring\savering
\pgfplotsset{compat=1.14}
\usetikzlibrary{patterns}
\usetikzlibrary{positioning,arrows,arrows.meta}
\usepackage{geometry}
\usepackage{lipsum} 
\usepackage{titlesec}

\titleformat{\subsection}[runin]
       {\normalfont\bfseries}
       {\thesubsection}
       {0.5em}
       {}
       [.]
       
\titleformat{\subsubsection}[runin]
       {\normalfont\bfseries}
       {\thesubsubsection}
       {0.5em}
       {}
       [.]
       
\usepackage{enumitem}
\theoremseparator{.} 
\newtheorem{Th}{Theorem}[section]
\newtheorem{Def}[Th]{Definition}
\newtheorem{Rq}[Th]{Remark}
\newtheorem{Pro}[Th]{Proposition}
\newtheorem{Cor}[Th]{Corollary}
\newtheorem{Lem}[Th]{Lemma}
\theoremstyle{empty}
\newtheorem{refproof}{Proof}

\newcommand{\R}{\mathbb{R}}
\newcommand{\C}{\mathscr{C}}
\newcommand{\T}{\mathbf{T}}

\newcommand{\K}{\widehat{\mathbb{P}}_0}

\newcommand{\dr}{\mathrm{d}}

\newenvironment{proof}{\noindent\textit{Proof.~}}{\hfill$\Box$\bigbreak} 

\title{Global existence and modified scattering for the solutions to the Vlasov-Maxwell system with a small distribution function}

\author{L\'eo Bigorgne\footnote{Univ Rennes, CNRS, IRMAR - UMR 6625, F-35000 Rennes, France.
{\em E-mail address:} {\tt leo.bigorgne@univ-rennes.fr}}}
\date{}
\begin{document}

\maketitle
    
\begin{abstract}
The purpose of this paper is twofold. In the first part, we provide a new proof of the global existence of the solutions to the Vlasov-Maxwell system with a small initial distribution function. Our approach relies on vector field methods, together with the Glassey-Strauss decomposition of the electromagnetic field, and does not require any support restriction on the initial data or smallness assumption on the Maxwell field. Contrary to previous works on Vlasov systems in dimension $3$, we do not modify the linear commutators and avoid then many technical difficulties.

In the second part of this paper, we prove a modified scattering result for these solutions. More precisely, we obtain that the electromagnetic field has a radiation field along future null infinity and approaches, for large time, a smooth solution to the vacuum Maxwell equations. As for the Vlasov-Poisson system, in constrast, the distribution function converges to a new density function $f_\infty$ along \textit{modifications} of the characteristics of the free relativistic transport equation. In order to define these logarithmic corrections, we identify an effective asymptotic Lorentz force. By considering logarithmical modifications of the linear commutators, defined in terms of derivatives of the asymptotic Lorentz force, we finally prove higher order regularity results for $f_\infty$.
\end{abstract}

    \tableofcontents

\section{Introduction}

The Vlasov-Maxwell system, which is used to describe the dynamics of collisionless plasma, can be written as
\begin{alignat}{2}
& \partial_t f+\widehat{v} \cdot \nabla_x f + ( E+\widehat{v} \times B) && \cdot \nabla_v f = 0, \label{VM11}  \\
&  \nabla_x \cdot E = \int_{\R^3_v}f \mathrm{d}v,  && \partial_t E = \nabla_x \times B -\int_{\R^3_v} \widehat{v}f \mathrm{d}v, \label{VM22}  \\
&  \nabla_x \cdot B = 0,   && \partial_t B = - \nabla_x \times E, \label{VM33}
\end{alignat} 
where
\begin{itemize}
\item $f:\R_{+, \, t} \times \R^3_x \times \R^3_v \rightarrow \R_+$ is the density distribution function of the particles. 
\item $\widehat{v}=\frac{v}{v^0}$, with $v^0:=\sqrt{1+|v|^2}$, is the relativistic speed of a particle of momentum $v \in \R^3_v$.
\item $\int_{\R^3_v} f \mathrm{d}v$ and $\int_{\R^3_v} \widehat{v}f \mathrm{d}v$ are respectively the total charge density and the total current density.
\item $E,B:\R_{+, \, t} \times \R^3_x \rightarrow \R^3$ are respectively the electric and the magnetic field.
\end{itemize}
For simplicity, we assume that the plasma is composed by one species of particles of charge $q=1$ and mass $m=1$. Our results can be extended without any additional difficulty to several families of particles of different charges and positive masses\footnote{The case of massless particles requires in fact a different analysis \cite{massless}.}. We refer to \cite{Glassey} for a detailed introduction to these equations.

The initial value problem for the Vlasov-Maxwell equations, together with a regular initial data set $(f_0,E_0,B_0)$ composed by a function $f_0:\R^3_x \times \R^3_v \rightarrow \R_+$ and two fields $E_0,B_0:\R^3_x \rightarrow \R^3$ satisfying the constraint equations $\nabla_x \cdot E_0=\int_v f_0 \mathrm{d}v$ and $\nabla_x \cdot B_0=0$, is well-posed \cite{Woll}. On the other hand, the global existence problem for classical solutions to the Vlasov-Maxwell system is still open\footnote{In contrast, a global in time existence result for weak solutions was proved by DiPerna-Lions \cite{DPLions} and revisited in \cite{Reinweak}} and has only been addressed in some particular cases, such as under certain symmetry assumptions \cite{GSc90,GSc97,Gsc98,LukStrain,Rein2,XWang}. For the general case, since the pioneering work of Glassey-Strauss \cite{GlStrauss}, several continuation criteria have been obtained \cite{GlasseyStrauss0,GlStracrit,KlSta,BGP2,Pallard1,SAI,LukStrain2,Kunze,Pallard2,Patel}.

\subsection{Small data solutions to the Vlasov-Maxwell system} Much more is known for this particular perturbative regime, in which global existence holds and the solutions disperse. In \cite{GSt}, for small compactly supported initial data, Glassey-Strauss proved the optimal decay rate $\int_v f \mathrm{d}v \lesssim t^{-3}$ on the velocity average of the distribution function and obtained estimates for the electromagnetic field and its first order derivatives. Shorly after, in the multi-species case, the smallness assumptions on the individual particle densities is relaxed by \cite{GSc}. Later, Schaeffer \cite{Sc} removed the support restriction on the velocity variable. However, his method lead to a loss on the estimate of $\int_v f \mathrm{d}v$. 

It is only recently that all the compact support assumptions on the initial data were removed in two independent results \cite{dim3,Wang}. Both of these works are based on vector field methods and the latter use also Fourier analysis. These robust approaches allowed to derive sharp pointwise decay estimates on the solutions and their (high order) derivatives. Moreover, in \cite{dim3}, the initial decay hypothesis in $v$ is optimal and improved estimates on certain \textit{null} components of the electromagnetic field are derived. Finally, using the framework of Glassey-Strauss and without any compact support restriction, Wei-Yang derived a global existence result which does not require the initial Maxwell field to be small \cite{WeiYang}. 

In the first part of this article, we provide an alternative but shorter proof of the main results of \cite{dim3,Wang}, without assuming any smallness assumption on the electromagnetic field. Compared to \cite{WeiYang}, we require more regularity on the initial data but our method allows us to control the derivatives of the solutions, up to an arbitrary order $N$. These informations are needed for the second part of the paper.

\subsection{Modified scattering results for the Vlasov-Poisson system} 

Sharp decay estimates for the small data solutions to the Vlasov-Poisson system were first derived by \cite{Bardos} and then, with various improvements, by \cite{HRV,Poisson,Duan,smallSchaeffer} (for the relativistic cases, see \cite{rVP0,rVPWang,rVP}). Modified scattering for these solutions was established in \cite{Choi} and then in \cite{scattPoiss,Panka}, where more informations is obtained on the asymptotic dynamics governing the modification of the linear characteristics. Furthermore, a scattering map has been constructed by \cite{scattmap} and let us finally mention that similar results hold for perturbations of a point charge \cite{PausaWid,PWY}.

In the second part of this paper, we investigate such problems in the context of the Vlasov-Maxwell equations. In particular, as \cite{scattPoiss} for the Vlasov-Poisson system, we prove that
$$ \int_{\R^3_x} f(t,x,v) \dr x \to Q_\infty (v), \qquad \text{as $t \to +\infty$.}$$
The scattering charge $Q_\infty$ is deeply related to the leading order term of the asymptotic expansion of both the charge density $\int_v f \dr v$ and the current density. It allows us to define an asymptotic Lorentz force $v \mapsto \mathrm{Lor} (v)$, from which we deduce the modified scattering statement for $f$ (see Theorem \ref{Th0} and Remark \ref{Rqcharact} for more details). We also prove higher order regularity properties for the limit distribution $f_\infty$, which require a more thorough analysis. To our knowledge, there is no such regularity result for the Vlasov-Poisson system.

\subsection{Vector field methods for relativistic transport equations}

Our analysis of the asymptotic behavior of both the electromagnetic field and the distribution function relies on vector field methods (see Section \ref{subsecvfm} for an overview of the key ideas). These kind of techniques was first developed by Klainerman \cite{Kl85} in order to study solutions to nonlinear wave equations and then adapted by \cite{CK} to the Maxwell equations. It is only recently that the approach has been adapted to relativistic transport equations by Fajman-Joudioux-Smulevici \cite{FJS}, leading in particular to a proof of the stability of Minkowski spacetime for both the massive and massless Einstein-Vlasov system \cite{FJS3,EVmassless} (see also \cite{LindbladTaylor,Taylor} for alternative proofs). Our work \cite{dim3} concerning the small data solutions to the Vlasov-Maxwell system relies on such techniques as well. The method has also been successfully used to derive boundedness and decay estimates for the solutions to the massless Vlasov equation on a fixed Kerr black hole \cite{ABJ,Schwa}. Finally, even if it concerns the non-relativistic setting, let us also mention that such approaches have been applied in the collisional regime \cite{ Chatu1,Chatu2,Chatu3}.

In order to deal with slowly decaying error terms, all the works on the small data solutions to massive relativistic Vlasov systems or the Vlasov-Poisson system \cite{FJS3,dim3,Poisson,Duan}, based on vector field methods, require to dynamically modify certain linear commutators of the Vlasov operator. One of the main novelty of this article consists in proving that the solutions are global without using these modified vector fields, which considerably simplifies the analysis. For this, even though certain quantities grow logarithmically in time, we are able to close the energy estimates by identifying several hierarchies in the commuted equations (see Section \ref{SubsecdiffVlasov} for more details). It is then important to derive the optimal decay rate $t^{-3}$ for $\int_v f \mathrm{d}v$ and its derivatives by a method allowing well-chosen weighted $W^{N,\infty}_{x,v}$ norms of the distribution function to grow slowly in time. We believe that this approach could be applied to other system of equations, in particular for both the Einstein-Vlasov and the Vlasov-Poisson systems.

\subsection{The main result} We present here a short version of our main result, stated in Theorems \ref{Th1}-\ref{Th2} below, where we also describe the behavior of the derivatives of the solutions.
\begin{Th}\label{Th0}
Any solution $(f,E,B)$ to the Vlasov-Maxwell system \eqref{VM11}-\eqref{VM33} arising from a small initial distribution function and smooth as well as sufficiently decaying initial data is global in time. Moreover,
\begin{enumerate}
\item there exists a solution $(E^{\mathrm{vac}}\!,B^{\mathrm{vac}})$ to the vacuum Maxwell equations\footnote{The vacuum Maxwell equations are given by \eqref{VM22}-\eqref{VM33} with $f=0$.} approaching $(E,B)$ as $t+|x| \! \to \infty$,
$$ \forall \, (t,x) \in \R_+\times \R^3, \qquad   \left|E-E^{\mathrm{vac}} \right|(t,x) +\left|B-B^{\mathrm{vac}} \right|(t,x) \leq C_q (1+t+|x|)^{-1-q}, \qquad \frac{1}{2} \leq q <1.$$
\item The Lorentz force has a self-similar asymptotic profile $v \mapsto \mathrm{Lor}(v)$,
$$ \forall \, (t,x,v) \in \R_+ \times \R^3_x \times \R^3_v, \qquad \big| t^2\big( E(t,x+t\widehat{v})+\widehat{v} \times B(t,x+t\widehat{v}) \big)-\mathrm{Lor}(v) \big| \lesssim \langle x \rangle^2 |v^0|^8\frac{\log^n(3+t)}{1+t},$$
where $\langle x \rangle :=(1+|x|^2)^{\frac{1}{2}}$ and, say, $n=70$. We have modified scattering to a new density function $f_{\infty}:\R^3_x \times \R^3_v \rightarrow \R_+$,
$$ \forall \; t \geq 3, \qquad \big\| f \big( t,X_{\C}(t,\cdot,\cdot),\cdot \big)-f_{\infty} \big\|_{L^1_{x,v} \cap L^{\infty}_{x,v}} \lesssim t^{-1} \log^n(t),$$
where the cartesian components $X^k_\C$ of the modified spatial characteristics $X_{\C} \in \R^3_x$ are defined as
$$ X^k_{\C}(t,x,v) := x^k+t\widehat{v}^k-\frac{\log(t)}{v^0} \left( \mathrm{Lor}^k(v)-\widehat{v} \cdot \mathrm{Lor}(v) \, \widehat{v}^k \right), \qquad 1 \leq k \leq 3.$$
\end{enumerate}
\end{Th}
\begin{Rq}
No modification of the spatial characteristics is in fact required in the exterior of the light cone $\{|x| \geq t\}$ in order to prove such a result (see Section \ref{AnnexeC2}). We already observed in \cite{ext} that the small data solutions to the Vlasov-Maxwell system have a better behavior in this region.

Similarly, no correction of the linear characteristics should in principle be necessary in order to prove a scattering statement in higher dimensions. This is consistent with the result of \cite{PankaPoisson} concerning the Vlasov-Poisson system in dimension $d\geq4$ and our study of the asymptotic behavior of the small data solutions of the Vlasov-Maxwell system in high dimensions \cite{dim4}.

The case of massless particles differs from the case of massive particles treated in this paper. Indeed, in view of \cite{massless}, we expect the small data solutions of the massless Vlasov-Maxwell system to satisfy linear scattering in dimension $d=3$.
\end{Rq}
\begin{Rq}\label{Rqcharact}
The behavior of the Lorentz force along the linear trajectories suggests that the characteristics of the Vlasov-Maxwell system satisfy, for $t \gg 1$,
$$ \dot{X} =\widehat{V}, \qquad \dot{V} \approx t^{-2} \, \mathrm{Lor}(V),  \qquad \qquad X(0)=x_0, \qquad V(0)=v_0.$$
Hence, we can presume that $V$ converges to $v$, so that 
$$V(t) \approx v-\frac{1}{t}\mathrm{Lor}(v), \qquad \dot{X}(t) \approx \widehat{v}-\frac{1}{tv^0}\mathrm{Lor}(v)+\frac{\widehat{v} \cdot \mathrm{Lor}(v)}{tv^0} \widehat{v} +O(t^{-2})$$
and we can then expect $X(t) \approx X_{\C}(t,x_0,v)$. 

Moreover, we could in fact decompose $\mathrm{Lor}(v)$ as $E^\infty(v)+\widehat{v} \times B^\infty(v)$ and observe that, as $v \to 0$,
$$ X_{\C}(t,x,v) =x+tv-\log(t) E^\infty(v) +o(v).$$
In other words, for small velocities, the modified characteristics $X_\C$ of the Vlasov-Maxwell system approach the ones constructed in \cite{scattPoiss} for the Vlasov-Poisson system.
\end{Rq} 

\subsection{Structure of the paper}

In Section \ref{Sec2} we introduce the notations and the tools used throughout this article. Then, we state our main results, Theorems \ref{Th1}-\ref{Th2}, and present the key ideas of the proof. In Section \ref{Sec3}, we set up the bootstrap assumptions and discuss their immediate consequences. Section \ref{Sec4} concerns the study of the distribution function. In particular, we prove that weighted $L^\infty_{x,v}$ norms of $f$ and its derivatives grow at most logarithmically and we improve the bootstrap assumption on their velocity average. Then, in Section \ref{Sec5}, we conclude the proof of the global existence of the small data solutions to \eqref{VM11}-\eqref{VM33} by exploiting the Glassey-Strauss decomposition of the electromagnetic field in order to improve the bounds on $(E,B)$ and their derivatives. Next, in Section \ref{Sec6} we refine our estimates by proving that the particle current density and the electromagnetic field have a self-similar asymptotic profile. This allows us to define the modified trajectories along which the distribution function converges. Section \ref{Sec7} is devoted to the scattering results for the electromagnetic field. A crucial part of the proof consists in constructing a scattering map for the vacuum Maxwell equations. In Section \ref{Secenergy}, we relate the conserved total energy of the system to the ones of the scattering states. Finally, Appendix \ref{SecA}-\ref{SecB} contain two useful computations and Appendix \ref{SecAnnexeC} presents alternative expressions for the profile of $F$ and the modified characteristics.

\subsubsection*{Acknowledgements} I would like to thank Jacques Smulevici for suggesting me this problem. This work was conduced within the France 2030 framework porgramme, the Centre Henri Lebesgue ANR-11-LABX-0020-01.

\section{Preliminaries and detailed statement of the main result}\label{Sec2}
\subsection{Basic notations}

In this paper we work on the $1+3$ dimensional Minkowski spacetime $(\R^{1+3},\eta)$. We will use two sets of coordinates, the Cartesian $(x^0=t,x^1,x^2,x^3)$, in which $\eta=\mathrm{diag}(-1,1,1,1)$, and null coordinates $(u,\underline{u},\theta, \varphi)$, where
$$\underline{u}=t+r, \qquad u=t-r, \qquad r:= |x|=\sqrt{|x^1|^2+|x^2|^2+|x^3|^2},$$
and $(\theta,\varphi)\in ]0,\pi[ \times ]0,2 \pi[$ are spherical coordinates on the spheres of constant $(t,r)$. These coordinates are defined globally on $\R^{1+3}$ apart from the usual degeneration of spherical coordinates and at $r=0$. Sometimes, for a tensor field $T$ defined on $\R_+ \times \R^3_x$, it will be convenient to write
$$ T(u,\underline{u},\omega):=T\left( \frac{\underline{u}+u}{2},\frac{\underline{u}-u}{2} \omega \right), \qquad \underline{u} \geq 0, \quad |u| \leq \underline{u}, \quad \omega \in \mathbb{S}^2 .$$
We will work with the null frame $(L,\underline{L},e_\theta,e_\varphi)$, where $L=2\partial_u$, $\underline{L}=2\partial_{\underline{u}}$ are null derivatives and $(e_\theta,e_\varphi)$ is the standard orthonormal basis on the spheres. More precisely,
$$L=\partial_t+\partial_r , \qquad \underline{L}=\partial_t-\partial_r, \qquad e_\theta=\frac{1}{r}\partial_{\theta}, \qquad e_\varphi=\frac{1}{r \sin \theta} \partial_{\varphi}.$$
The Einstein summation convention will often be used, for instance $v^{\mu}\partial_{x^{\mu}} f=\sum_{\mu = 0}^3v^{\mu}\partial_{x^{\mu}}f$. The latin indices goes from $1$ to $3$ and the greek indices from $0$ to $3$. We will raise and lower indices using the Minkowski metric $\eta$, so that $x^i=x_i$ and $x^0=-x_0$. 

The four-momentum vector $\mathbf{v}=(v^{\mu})_{0 \leq \mu \leq 3}$ is parameterized by $v=(v^i)_{1 \leq i \leq 3} \in \R^3_v$ and $v^0=\sqrt{1+|v|^2}$ since the mass of the particles is equal to $1$. Let $(v^L,v^{\underline{L}},v^{e_1},v^{e_2})$ be the null components of the momentum vector and $\slashed{v}=(v^{e_\theta},v^{e_\varphi})$ its angular part, so that
$$\mathbf{v}=v^L L+ v^{\underline{L}} \underline{L}+v^{e_\theta}e_\theta+v^{e_\varphi}e_\varphi, \qquad v^L=\frac{v^0+\frac{x_i}{r}v^i}{2} , \qquad v^{\underline{L}}=\frac{v^0-\frac{x_i}{r}v^i}{2}, \qquad |\slashed{v}|^2=|v^{e_\theta}|^2+|v^{e_\varphi}|^2.$$
The relativistic speed $\widehat{v} \in \R^3$ is given by $\widehat{v}^i=\frac{v^i}{v^0}$ and, for convenience, we define the quantities
$$ \widehat{v}^0 := \frac{v^0}{v^0}=1, \qquad  \widehat{v}^L := \frac{v^L}{v^0}, \qquad  \widehat{v}^{\underline{L}} := \frac{v^{\underline{L}}}{v^0}, \qquad \widehat{\slashed{v}}:= \frac{\slashed{v}}{v^0} , \qquad \widehat{v}^{e_A} := \frac{v^{e_A}}{v^0}, \quad A \in \{ \theta,\varphi \}.$$
Sometimes, we will write $(|v^0|^pg)(w)$ to denote $|w^0|^pg(w)$, where $w \in \R^3_v$ and $g:\R^3_v \rightarrow \R$.

In order to capture the good properties of certain geometric quantities associated to the solutions in the good null directions $(L,e_\theta,e_\varphi)$, we introduce the Faraday tensor $F_{\mu \nu}$, which is a $2$-form, and the four-current density $J(f)_{\mu}$,
\begin{equation}\label{eq:defF} F = \begin{pmatrix}
0 & E^1 & E^2 & E^3\\
-E^1 & 0 & -B^3 & B^2 \\
- E^2 & B^3 & 0& -B^1 \\
- E^3 &-B^2 & B^1 & 0
\end{pmatrix}, \qquad \quad J(f) := \begin{pmatrix}
\, -\int_{\R^3_v}f \mathrm{d}v \, \\[2.4pt] \int_{\R^3_v}\frac{v_{1}}{v^0}f \mathrm{d}v \\[2.3pt] \int_{\R^3_v}\frac{v_{2}}{v^0}f \mathrm{d}v \\[2.3pt] \int_{\R^3_v}\frac{v_{3}}{v^0}f \mathrm{d}v
\end{pmatrix}.
\end{equation}
The cartesian components of $F$ are then either equal to $0$ or to a component of $\pm(E,B)$. We will in fact be more interested in its null decomposition $(\alpha(F),\underline{\alpha}(F),\rho(F),\sigma(F))$ defined, for $A \in \{\theta, \varphi\}$, as
\begin{equation}\label{defnullcompoF}
 \alpha(F)_{e_A}:=F_{e_A L}, \qquad \underline{\alpha}(F)_{e_A}:=F_{e_A \underline{L}}, \qquad \rho(F):=\frac{1}{2}F_{\underline{L} L}, \qquad \sigma(F):=F_{e_\theta e_\varphi}.
 \end{equation}
In particular, $\rho(F)=E \cdot \partial_r$ and $-\sigma(F)=B \cdot \partial_r$ are the radial components of the electric field and the magnetic field. Moreover, the $1$-forms $\alpha(F)$ and $\underline{\alpha}(F)$ are tangential to the $2$-spheres and we will use the pointwise norms
\begin{align*}
 |\alpha(F)|^2 &:= |\alpha(F)_{e_\theta}|^2+|\alpha(F)_{e_\varphi}|^2, \qquad \qquad |\underline{\alpha}(F)|^2 := |\underline{\alpha}(F)_{e_\theta}|^2+|\underline{\alpha}(F)_{e_\varphi}|^2, \\
|F|^2 &:= \sum_{0 \leq \mu < \nu \leq 3} |F_{\mu \nu}|^2 = \frac{1}{2}|\alpha(F)|^2+\frac{1}{2}|\underline{\alpha}(F)|^2+|\rho (F)|^2+|\sigma(F)|^2.
\end{align*}
The Vlasov equation \eqref{VM11} can be rewritten as
\begin{equation}\label{Vlasov1}
\T_F(f)=0, \qquad \text{where} \qquad \T_F : f \mapsto \partial_t f+\widehat{v} \cdot \nabla_x f+\widehat{v}^{\mu}{F_{\mu}}^j \partial_{v^j}f,
\end{equation}
and the Maxwell equations \eqref{VM22}-\eqref{VM33} take a concise form. The Gauss-Ampère law and the Gauss-Farady law\footnote{Note that $\nabla^{\mu} {}^*\! F_{\mu \nu}=0$ is equivalent to $\nabla_{[\lambda}F_{\mu \nu]}:=\nabla_\lambda F_{\mu \nu}+\nabla_\mu F_{ \nu \lambda}+\nabla_\nu F_{  \lambda \mu}=0$.}
\begin{equation}\label{Maxwell23}
\nabla^{\mu} F_{\mu \nu}=J(f)_{\nu}, \qquad \qquad \nabla^\mu {}^* \!F_{\mu \nu}=0,
\end{equation}
where ${}^* \! F_{\mu \nu} = \frac{1}{2} F^{\lambda \sigma} \varepsilon_{ \lambda \sigma \mu \nu}$ is the Hodge dual of $F$ and $\varepsilon$ the Levi-Civita symbol. Here $\nabla$ stands for the covariant derivative (or Levi-Civita connection), so that \eqref{Maxwell23} holds in any coordinate system.

The operators $\nabla_x$ and $\nabla_v$ will denote the standard gradients in $x$ and $v$ respectively. For instance,
$$ \nabla_x f= \big(\partial_{x^1} f, \, \partial_{x^2} f, \, \partial_{x^3} f \big), \qquad  \qquad \nabla_v f = \big( \partial_{v^1} f, \, \partial_{v^2} f,  \, \partial_{v^3}f \big).$$
Given a $2$-form $G$ and $0 \leq \lambda \leq 3$, we will denote by $\nabla_{\partial_{x^\lambda}} G$ the covariant derivative of $G$ according to $\partial_{x^\lambda}$, where $\partial_{x^0}=\partial_t$. For any multi-index $\kappa \in \{0, 1,2, 3\}^p$, we define $\nabla_{t,x}^\kappa G:=\nabla_{\partial_{x^{\kappa_1}}} \dots \nabla_{\partial_{x^{\kappa_p}}} G$. In cartesian coordinates, we then have
$$ \nabla_{t,x}^\kappa (G)_{\mu \nu} = \partial^\kappa_{t,x} \big( G_{\mu \nu} \big), \qquad 0 \leq \mu , \, \nu \leq 3. $$
Finally, for $x \in \R^3$ we will use the Japanese brackets $\langle x \rangle := (1+|x|^2)^{\frac{1}{2}}$ and the notation $D_1 \lesssim D_2$ will stand for the statement that $\exists \, C>0$ a positive constant independent of the solutions such as $ D_1 \leq C D_2$.

\subsection{Backward light cones and future null infinity}\label{subsec2} The scattering state for a smooth electromagnetic field $F$, which in our case is also called radiation field, will be a tensor field depending on the variables $(u,\omega)\in \R \times \mathbb{S}^2$. It will be obtained as the limit, when $\underline{u} \to +\infty$, of $rF(u,\underline{u},\omega)$. For this reason, we introduce the backward light cones $\underline{C}_{\underline{u}}$ and give their induced volume form $\dr \mu_{ \underline{C}_{\underline{u}}}$ in accordance with the choice of the null vector field $\underline{L}$ as their normal vector. Let, for any $\underline{u} \geq 0$, 
$$ \underline{C}_{\underline{u}}:= \{(t,x)  \in \mathbb{R}_+ \times \mathbb{R}^3 \, | \; t+|x|=\underline{u} \}, \qquad  \dr \mu_{\underline{C}_{\underline{u}}} =\frac{1}{2}r^2 \dr u \dr \mu_{\mathbb{S}^2},$$
where $\dr \mu_{\mathbb{S}^2}=\sin(\theta) \dr \theta \dr \varphi$ is the volume form on $\mathbb{S}^2$. 

Even if we will not need of this formalism, we mention that the radiation field is in fact defined on a part of the conformal boundary of the Minkowski space, called future null infinity $\mathcal{I}^+$ and corresponding to the future end points of the null geodesics $t-|x|=u$. It can be viewed as $\underline{C}_{+\infty}$. More precisely, $$(t,r,\omega) \mapsto (T(t,r)=\tan^{-1}(t+r)+\tan^{-1}(t-r), \qquad R(t,r)=\tan^{-1}(t+r)-\tan^{-1}(t-r), \omega) \in \R \times \mathbb{S}^3$$ is a conformal diffeomorphism between Minkowski spacetime and the interior of the triangle $0 \leq R \leq \pi$, $|T|=\pi-R$ of the space $\R \times \mathbb{S}^3$, equipped with the metric $-\dr T^2+\dr R^2+\sin^2( R) \dr \mu_{\mathbb{S}^2}$. Then $$ \mathcal{I}^+ := \{ (T,R,\omega) \in \R \times \mathbb{S}^3 \; | \; 0 < R < \pi, \quad T=\pi-R \}.$$
Past null infinity $\mathcal{I}^-$ is defined similarly as $\{ 0<R < \pi, \; T=R-\pi \}$ and can be viewed as $t-|x|=-\infty$.

\begin{tikzpicture}
\fill[color=gray!35] (3,0)--(7,0)--(7,4)--(0,4)--(0,3)--(3,0);
\draw [-{Straight Barb[angle'=60,scale=3.5]}] (0,-0.3)--(0,4);
\draw [-{Straight Barb[angle'=60,scale=3.5]}] (-0.2,0)--(7,0) ;
\draw (3,0.2)--(3,-0.2);
\draw (-0.18,3)--(0.2,3);
\draw (0,-0.5) node{$r=0$};
\draw (-0.7,0) node{$t=0$};
\draw (3,-0.5) node{$\underline{u}$};
\draw (-0.5,3) node{$\underline{u}$};
\draw (-0.5,3.8) node{$t$};
\draw (6.8,-0.5) node{$r$}; 
\draw[densely dashed] (3,0)--(0,3) node[scale=1.2,below, midway] {$\underline{C}_{\underline{u}}$};
\draw (4.8,2.5) node[ color=black!100, scale=1.2] {$\{t+|x| \geq \underline{u} \}$};

    \node (II)   at (11,1.8)   {};

\path  
  (II) +(90:2.5)  coordinate    (IItop)
       +(-90:2.5) coordinate    (IIbot)
       +(0:2.5)   coordinate    (IIright)
       ;
\draw 
	
      (IIbot) --
          node[midway, above , sloped]    {\footnotesize{$r=0$}}   
          (IItop) --
          node[midway, above, sloped] {$\mathcal{I}^+$,  \hspace{0.8mm} \footnotesize{ $\underline{u}=+\infty$}} 
          (IIright) -- 
          node[midway, below, sloped] {$\mathcal{I}^-$, \hspace{0.8mm} \footnotesize{$ u=-\infty$}}
      (IIbot) -- cycle;
\draw (10.9,-0.8) node[scale=0.8]{$i^-$};
\draw (13.7,1.8) node[scale=0.8]{$i^0$};
\draw (10.9,4.4) node[scale=0.8]{$i^+$}; 
\draw (10.9,4.6) node[scale=0.8]{}; 

\draw[densely dashed] (11.7,0)--node[scale=0.6,midway, above, sloped]{$\underline{u}=\mathrm{cst}$}(11,0.7)--node[scale=0.6,midway, above, sloped]{$u=\mathrm{cst}$}(12.8,2.5);    
      
      \node[align=center,font=\bfseries, yshift=-2em] (title) 
    at (current bounding box.south)
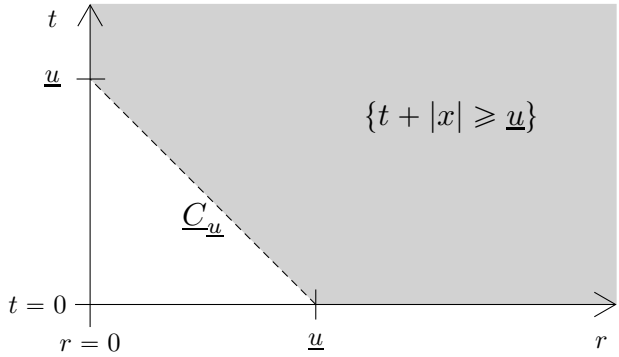
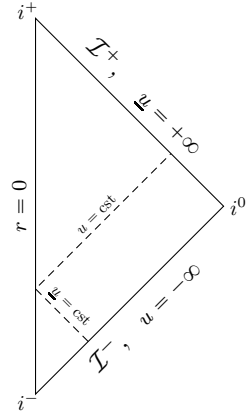
    {The set $\underline{C}_{\underline{u}}$ and the Penrose diagram of the Minkowski space.};
\end{tikzpicture}

\subsection{Charged electromagnetic field}\label{Subsecpurecharge}

For our global existence result, it will be sufficient to assume that the electromagnetic field verifies $|F|(0,\cdot) \lesssim r^{-2}$ whereas our scattering result will require a slightly stronger initial decay hypothesis. However, if the plasma is not neutral, one cannot expect $F$ to decay faster than $r^{-2}$. Indeed, if $(f,F)$ is a sufficiently regular solution to \eqref{Vlasov1}-\eqref{Maxwell23} on $[0,T[$, we obtain from Gauss's law that the total charge
$$  Q_F(t):=\lim_{r \rightarrow + \infty}  \int_{\mathbb{S}^2 } \rho(F)(t,r\omega) r^2 \dr \mu_{\mathbb{S}^2}= \int_{x \in \R^3} \int_{v \in \R^3} f(t,x,v) dv dx, \qquad t \in [0,T[,$$
is a conserved quantity and that $|F| =o(r^{-2})$ implies $Q_F=0$. In order to avoid such a restrictive assumption, we introduce the pure charge part $\overline{F}$ of $F$,
\begin{equation}\label{defFoverlin}
\overline{F}(t,x) := \frac{Q_F}{4 \pi |x|^2} \frac{x_i}{|x|} dt \wedge dx^i , \qquad \rho(\overline{F})(t,x)= \frac{Q_F}{4 \pi |x|^2}, \qquad \alpha(\overline{F})=\underline{\alpha}(\overline{F})=\sigma(\overline{F})=0,
\end{equation}
which corresponds to the electromagnetic field generated by a point charge $Q_F$ at $x=0$. One can verify that $Q_{\overline{F}}=Q_F$, so that $F-\overline{F}$ is chargeless and it will then be consistent to assume that $F$ has an asymptotic expansion of the form $F=\overline{F}+O(r^{-2-\delta})$, $\delta >0$. In fact, $E=E^{\mathrm{df}}+E^{\mathrm{cf}}$ and $B=B^{\mathrm{df}}+B^{\mathrm{cf}}$ can be decomposed into their divergence-free and curl-free components. Then, $B^{\mathrm{cf}}=0$ and $E^{\mathrm{cf},i}=\overline{F}_{0i}+O(r^{-3})$ if $J(f)_0$ is sufficiently regular, so that the stronger initial decay assumption required for the scattering result concerns the divergence-free components of $E$ and $B$.

\subsection{Commutation vector fields}\label{subsecvfm}

We will derive estimates on both the electromagnetic field and the distribution function using vector field methods. These kind of approaches are usually based on
\begin{itemize}
\item a set of vector fields, which commute with the linear operator of the equation studied.
\item Energy inequalities, applied in order to prove boundedness for $L^2$ or $L^1$ norms of the solutions and their derivatives. For instance, see \cite[Section~$4.1$]{dim3}.
\item Weighted Sobolev embeddings, such as \cite[Theorem~$6$]{FJS}, used to obtain decay estimates on the fields.
\end{itemize}
In this paper, in order to simplify the analysis, we will prove $L^\infty$ estimates and then obtain pointwise decay estimates on the solutions in a different way (see Section \ref{SubsecIngred} for more details). We now elaborate on the commutators for the Maxwell equations and the ones for the relativistic transport equation.
\begin{Def}
Let $\mathbb{K}$ be the set composed by the vector fields
$$ \partial_t, \qquad \partial_{x^i}, \qquad \Omega_{0i}:=t\partial_{x^i}+x^i \partial_t, \qquad \Omega_{jk}:=x^j\partial_{x^k}-x^k\partial_{x^j}, \qquad S:=t\partial_t+ x^\ell \partial_{x^\ell}=t\partial_t+r\partial_r,$$
where $1 \leq i \leq 3$ and $1\leq j < k \leq 3$. The translations $\partial_{x^{\mu}}$, the Lorentz boosts $\Omega_{0i}$ and the rotations $\Omega_{jk}$ are Killing vector fields, so that they generate isometries of the Minkowski space. The scaling vector field $S$ is merely conformal Killing.
\end{Def}
We will use this set for differentiating the electromagnetic field geometrically. More precisely, for a $2$-form $F$ and a vector field $Z=Z^{\mu}\partial_{x^{\mu}}$, the Lie derivative $\mathcal{L}_Z(F)$ of $F$ with respect to $Z$ is given, in coordinates, by
$$ \mathcal{L}_Z(F)_{\mu \nu} = Z(F_{\mu \nu})+\partial_{\mu}(Z^{\lambda})F_{\lambda \nu}+\partial_{\nu}(Z^{\lambda}) F_{\mu \lambda}.$$
Furthermore, if $F$ is a smooth solution to the vacuum Maxwell equations $\nabla^{\mu} F_{\mu \nu}=\nabla^{\mu} {}^* \! F_{\mu \nu}=0$ and $Z \in \mathbb{K}$, then $\mathcal{L}_Z(F)$ is also a solution to the vacuum Maxwell equations, that is $\nabla^{\mu} \mathcal{L}_Z(F)_{\mu \nu}=\nabla^{\mu} {}^* \! \mathcal{L}_Z(F)_{\mu \nu}=0$.
\begin{Def}
Let $\K$ be the set composed by
$$ \partial_t, \qquad \partial_{x^i}, \qquad \widehat{\Omega}_{0i}:=t\partial_{x^i}+x^i \partial_t+v^0 \partial_{v^i}, \qquad \widehat{\Omega}_{jk}:=x^j\partial_{x^k}-x^k\partial_{x^j}+v^j\partial_{v^k}-v^k\partial_{v^j}, \qquad S=t\partial_t+r\partial_r,$$
where $1 \leq i \leq 3$ and $1\leq j < k \leq 3$. In fact, $\widehat{\partial}_{x^\mu}=\partial_{x^\mu}$, $\widehat{\Omega}_{0i}$ and $\widehat{\Omega}_{jk}$ are obtained as the complete lift, a classical operation in differential geometry\footnote{We refer to \cite[Section~$2G$]{FJS} for more details about the relations between the Vlasov operator on a Lorentzian manifold and the complete lift of its Killing vector fields.}, of the Killing fields $\partial_{x^\mu}$, $\Omega_{0i}$ and $\Omega_{jk}$.
\end{Def}
These vector fields have good commutation properties with the linear transport operator $\T_0=\partial_t+\widehat{v}\cdot \nabla_x$. Indeed, $[\T_0,S]=\T_0$ and $[v^0\T_0,\widehat{Z}]=0$ for all $\widehat{Z} \in \K \setminus \{S \}$.

In order to consider higher order derivatives, we introduce an ordering on $\mathbb{K}=\{Z^i \, | \, 1 \leq i \leq 11\}$ and on $\K=\{\widehat{Z}^i\,  | \, 1 \leq i \leq 11\}$. It will be convenient to assume that $Z^{11}=\widehat{Z}^{11}=S$ and $\widehat{Z^i}=\widehat{Z}^i$ for any $1 \leq i \leq 10$. Moreover, for a multi-index $\beta \in \llbracket 1,11 \rrbracket^p$ of length $|\beta|=p$, we denote by $\mathcal{L}_{Z^{\beta}}$ the Lie derivative $\mathcal{L}_{Z^{\beta_1}} \dots \mathcal{L}_{Z^{\beta_p}}$ of order $|\beta|$. Similarly, we define $\widehat{Z}^{\beta}$ as $\widehat{Z}^{\beta_1} \dots \widehat{Z}^{\beta_p}$. Note the equivalence between the pointwise norms
\begin{equation}\label{equinorm}
  \; \sum_{|\gamma| \leq N} \left|\mathcal{L}_{Z^{\gamma}}(F)\right| \lesssim \sum_{|\beta| \leq N} \sum_{0 \leq \mu,\nu \leq 3} \left| Z^{\beta}(F_{\mu \nu}) \right| \lesssim \sum_{|\gamma| \leq N} \left|\mathcal{L}_{Z^{\gamma}}(F)\right|.
\end{equation} 
Since $\mathcal{L}_{\partial_{x^\mu}}(F)$ and $\partial_{x^\mu}f$ have a better behavior than the other derivatives, it will be crucial, in order to identify certain hierarchies in the commuted equations, to count the number of homogeneous vector fields composing $Z^\beta$ or $\widehat{Z}^\beta$. We denote by $\beta_H$ (respectively $\beta_T$) the number of homogeneous vector fields $\Omega_{0i}$, $\Omega_{jk}$ and $S$ (respectively translations $\partial_{x^\mu}$) composing $Z^\beta$. Remark $\beta_H+\beta_T=|\beta|$ and that $\widehat{Z}^\beta$ is also composed by $\beta_H$ homogenous vector fields and $\beta_T$ translations. If $Z^\beta = \Omega_{01}\partial_t S$, we have $\beta_H=2$ and $\beta_T=1$. 

The following geometric commutation formula, proved in \cite[Lemma~$2.8$]{massless}, will be fundamental for us.
\begin{Lem}\label{LemCom}
Let $G$ be a $2$-form and $g : [0,T[ \times \R^3_x \times \R^3_v \rightarrow \R$ be a function, both of class $C^1$, such that 
$$\nabla^{\mu} G_{\mu \nu} = J(g)_\nu, \qquad \nabla^{\mu} {}^* \! G_{\mu \nu}=0.$$
Let further $Z \in \mathbb{K} \setminus \{ S \}$ be a Killing vector field and $\widehat{Z} \in \K \setminus \{ S \}$ be its complete lift. Then,
\begin{align*}
&\nabla^{\mu } \mathcal{L}_Z(G)_{\mu \nu} = J(\widehat{Z}g)_{\nu},  \qquad \qquad \qquad \qquad \nabla^{\mu} {}^* \! \mathcal{L}_Z(G)_{\mu \nu}=0 , \\
& \nabla^{\mu } \mathcal{L}_S(G)_{\mu \nu} = J(Sg)_{\nu}+3J(g)_\nu, \qquad \qquad \nabla^{\mu} {}^* \! \mathcal{L}_S(G)_{\mu \nu}=0, \\
& \widehat{Z} \left( v^\mu {G_\mu}^j \partial_{v^j} g \right)= v^\mu {\mathcal{L}_Z(G)_\mu}^j \partial_{v^j} g +v^\mu {G_\mu}^j \partial_{v^j} \widehat{Z}g, \\
& S \left( v^\mu {G_\mu}^j \partial_{v^j} g \right)=v^{\mu} {\mathcal{L}_S(G)_{\mu}}^j \partial_{v^j} g-2v^\mu {G_\mu}^j \partial_{v^j} g +v^\mu {G_\mu}^j \partial_{v^j}S g.
\end{align*}
\end{Lem}
Iterating the above, we obtain that the structure of the Vlasov-Maxwell equations \eqref{Vlasov1}-\eqref{Maxwell23} are preserved by commutation.
\begin{Pro}\label{Com}
Let $(f,F)$ be a sufficiently regular solution to the Vlasov-Maxwell system. For any multi-index $\beta$, there exists $C^{\beta}_{\gamma, \kappa}, \, C^\beta_\xi \in \mathbb{Z}$ such that
\begin{align}\label{eq:comVla} 
\T_F \big( \widehat{Z}^{\beta} f \big) &= \sum_{\substack{ |\gamma|+|\kappa| \leq |\beta| \\ |\kappa| \leq |\beta|-1}} C^{\beta}_{\gamma,\kappa} \; \widehat{v}^\mu {\mathcal{L}_{Z^{\gamma}} (F)_{\mu}}^j \partial_{v^j} \widehat{Z}^{\kappa} f , \\
\nabla^{\mu} \mathcal{L}_{Z^\beta}(F)_{\mu \nu} &= \sum_{|\xi| \leq |\beta|} C_\xi^\beta J(\widehat{Z}^\xi f), \qquad \qquad \nabla^\mu {}^* \mathcal{L}_{Z^\beta}(F)_{\mu \nu} = 0.
\end{align}
Moreover, the multi-indices $|\gamma|+|\kappa| \leq |\beta|$ satisfy $\gamma_H +\kappa_H \leq \beta_H$ and the equality $\kappa_H = \beta_H$ implies $\gamma_T \geq 1$.
\end{Pro}
\begin{proof}
For the condition on the multi-indices $|\gamma|+|\kappa| \leq |\beta|$, remark from Lemma \ref{LemCom} that $\gamma_H+\kappa_H \leq \beta_H$ and $\gamma_T+\kappa_T=\beta_T$. Hence, if $\kappa_H = \beta_H$, we necessarily have $\kappa_T<\beta_H$ since $|\kappa| < |\beta|$. This implies $\gamma_T \geq 1$.
\end{proof}

\subsection{Weights preserved along the linear flow}\label{subsecweight}

The set $\mathbf{k}_1$ of weight functions composed by
\begin{equation}\label{defzzzz}
 z_{0i}:=t\widehat{v}^i-x^i, \qquad z_{jk}:=x^j\widehat{v}^k-x^k\widehat{v}^j, \qquad \qquad 1 \leq i \leq 3, \quad 1 \leq j < k \leq 3,
 \end{equation}
are conserved along any timelike straight line $t \mapsto (t,x+t\widehat{v})$. They are obtained as $|v^0|^{-1}\eta(\mathbf{v},K)$, where $K$ is a Killing vector field\footnote{On any smooth Lorentzian manifold $(\mathbf{Y},g)$, if $\gamma$ is a timelike geodesic and $K$ a Killing vector field, then $g(\dot{\gamma},K)=\mathrm{constant}$.} and they are then solution to the relativistic transport equation, $\forall \; z \in \mathbf{k}_1$, $\T_0(z)=0$. As a consequence, if $\T_0(g)=0$ then the same goes for $zg$, so that certain weighted norms of $g$ are conserved. In our nonlinear setting these norms will grow logarithmically in time and will then provide useful decay properties on the Vlasov field. For convenience, we will rather work with 
\begin{equation}\label{defz}
 \mathbf{z} := \Big(1+\sum_{z \in \mathbf{k}_1} z^2 \Big)^{\frac{1}{2}}, \qquad \qquad \T_0(\mathbf{z})=\widehat{v}^{\mu} \partial_{x^{\mu}} (\mathbf{z}) =0.
 \end{equation}
In particular, as $z_{0i} \in \mathbf{k}_1$, one has
\begin{equation}\label{eq:propz}
 1 \leq \mathbf{z}, \qquad \qquad \forall \; (t,x,v) \in \R_+ \times \R^3_x \times \R^3_v, \qquad \langle x \rangle \leq \mathbf{z}(t,x+t\widehat{v},v),
 \end{equation}
which will allow us to obtain space decay for $f(t,x+t\widehat{v},v)$, the particle density evaluated along the linear characteristics. Note also the following properties, which will be particularly useful for us in order to exploit the null structure of the system.
\begin{Lem}\label{nullcompo}
The four-momentum vector $\mathbf{v}$ has good null components, $v^{\underline{L}}$ and $\slashed{v}$. More precisely, 
$$ \forall \; (t,x,v) \in \R_+ \times \R^3_x \times \R^3_v, \qquad 0<\widehat{v}^{\underline{L}} \lesssim \frac{1+|t-|x||}{1+t+|x|}+ \frac{ \mathbf{z}}{1+t+|x|}, \qquad \qquad \big|\widehat{\slashed{v}}\big| \lesssim \frac{\mathbf{z}}{1+t+|x|}.$$
In certain circumstances, $v^{\underline{L}}$ will be the best component for exploiting decay in $t-r$. We will then use
 $$|v^0|^{-2}+|\widehat{\slashed{v}}|^2 \leq 4 \widehat{v}^{\underline{L}}.$$ 
\end{Lem}
\begin{proof}
The first two inequalities are proved in \cite[Lemma~$2.4$]{dim3}. The last one ensues from
\begin{equation}\label{eq:techniv}
4v^0v^{\underline{L}}\geq 4v^L v^{\underline{L}}=|v^0|^2-\big|\frac{x^i}{r}v_i \big|^2=1+|v|^2-|v \cdot \partial_r|^2=1+|v \cdot e_{\theta}|^2+|v \cdot e_{\varphi}|^2=1+|\slashed{v}|^2.
\end{equation}
\end{proof}
Since the particles are massive and then travel at a speed strictly lower than $1$, the speed of light, Vlasov fields enjoy much better decay properties along null ray than along timelike geodesics $t \mapsto x+t\widehat{v}$. After a long time, much of the particles should be located in the interior of the light cone. We will capture this property with the following inequality.
\begin{Lem}\label{gainv}
By losing powers of $v^0$ and $\mathbf{z}$, one can gain decay near the light cone $t=|x|$,
$$ \forall \; (t,x,v) \in \R_+ \times \R^3_x \times \R^3_v, \qquad 1 \lesssim \frac{1+|t-|x||}{1+t+|x|}|v^0|^2+ \frac{|v^0|^2 \, \mathbf{z}}{1+t+|x|}. $$
Moreover, in the exterior of the light cone, for $|x| \geq t$, one has $1 \lesssim (1+t+|x|)^{-1}|v^0|^2\mathbf{z}$.
\end{Lem}
\begin{proof}
For the first inequality, note that \eqref{eq:techniv} gives $1 \leq 4 |v^0|^2 \widehat{v}^{\underline{L}}$ and apply Lemma \ref{nullcompo}. For the second one, we refer to \cite[Remark~$2.5$]{dim3}.
\end{proof}
Recall from \cite[Lemma~$3.2$]{dim3} that $\mathbf{z}$ enjoys good commutation properties with the the vector fields of $\widehat{\mathbb{P}}_0$.
\begin{Lem}\label{zpreserv}
For any $a \in \R$ and $\widehat{Z} \in \widehat{\mathbb{P}}_0$, $|\widehat{Z}(\mathbf{z}^a)| \lesssim |a| \mathbf{z}^a$.
\end{Lem}
Finally, motivated by the fact that any regular solution to the linear relativistic transport equation $\T_0 (h)=0$ is constant along the timelike straight lines, $h(t,x+\widehat{v}t,v)=h(0,x,v)$, it will sometimes be useful to work with $g(t,x,v):=f(t,x+t\widehat{v},v)$, in particular during the study of the asymptotic properties of $\int_v f \dr v$ and its derivatives. The following result suggests that $g$ enjoys strong space decay and that its $v$ derivatives behave better than the ones of the distribution function $f$.
\begin{Lem}\label{gweightvderiv}
Let $f:[0,T[ \times \R^3_x \times \R^3_v \rightarrow \R$ be a sufficiently regular function and $g(t,x,v):= f(t,x+t\widehat{v},v)$. Then the following properties hold,
$$ \langle x \rangle^a  |g|(t,x,v) \leq |\mathbf{z}^a f|(t,x+t\widehat{v},v), \qquad v^0|\nabla_v g|(t,x,v) \leq \sum_{\widehat{Z} \in \widehat{\mathbb{P}}_0} \big|\mathbf{z}\widehat{Z} f \big|(t,x+t\widehat{v},v).$$ 
\end{Lem}
\begin{proof}
The first property follows from $\mathbf{z}^2\geq 1+|z_{01}|^2+|z_{02}|^2+|z_{03}|^3$ and $|z_{0i}|(t,x+t\widehat{v},v)=|x^i|$. For the second one, we have, using the Einstein summation convention,
$$
 v^0\partial_{v^j}g(t,x,v) = (v^0 \partial_{v^j} f) (t,x+t\widehat{v},v)+t\partial_{x^j}f(t,x+t\widehat{v},v)- t \widehat{v}_j \widehat{v}^i  \partial_{x^i} f (t,x+t\widehat{v},v).
$$
The result then ensues from $v^0 \partial_{v^j}=\widehat{\Omega}_{0j}-t\partial_{x^j}-x^j \partial_t$ and
\begin{equation}\label{reutiliser} x^j \partial_t +t\widehat{v}^j \widehat{v}^i \partial_{x^i} = (x^j-t \widehat{v}^j) \partial_t +\widehat{v}^j t \partial_t +\widehat{v}^j (t \widehat{v}^i-x^i) \partial_{x^i} +\widehat{v}^j x^i \partial_{x^i}  = -z_{0j} \partial_t+\widehat{v}^j S +\sum_{1 \leq i \leq 3}\widehat{v}^j z_{0i} \partial_{x^i}  .
\end{equation}
\end{proof}

\subsection{Inverse function of the relativistic speed} In order to perform the change of variables $y=x-\widehat{v}t$ for integrals on the domain $\R^3_v$, it will be useful to determine certain properties of the function $v \mapsto \widehat{v}$.
\begin{Lem}\label{cdv}
We define, on the domain $\{ y \in \R^3 \; | \; |y|<1 \}$, the operator $\widecheck{\; \; }$ as 
$$y \mapsto \widecheck{y} = \frac{y}{\sqrt{1-|y|^2}}, \qquad \text{so that} \qquad \forall \; |y| <1, \quad v \in \R^3_v, \qquad \widehat{\widecheck{y}}=y, \quad \widecheck{\widehat{v}}=v .$$
Note also that $v^0=(1-|\widehat{v}|^2)^{\frac{-1}{2}}$. Moreover, for any $(t,x) \in \R_+ \times \R^3$, the Jacobian determinant of the transformation $v \mapsto x-\widehat{v}t$ is equal to $-\frac{t^3}{|v^0|^5}$.
\end{Lem}
\begin{proof}
The fact that $\widecheck{\; \;}$ is the reciprocal function of $\widehat{\; \;}$ can be obtained by direct computations. Let $V$ be the column vector such that its transpose is $V^T=\big( \frac{v^1}{v^0} , \frac{v^2}{v^0} , \frac{v^3}{v^0} \big)$. Then the Jacobian determinant of the transformation $v \mapsto x-\widehat{v}t$ is equal to
$$-\frac{t^3}{|v^0|^3}  \det \left( \mathrm{\mathbf{I}}_3-VV^T  \right) = -\frac{t^3}{|v^0|^3}  \det \left( \mathrm{diag} \left(1,1,1-\frac{|v|^2}{1+|v|^2} \right) \right) =-\frac{t^3}{|v^0|^5}.$$
\end{proof}
Let us also mention the inequality $2(1-|\widehat{v}|) \geq (1-|\widehat{v}|)(1+|\widehat{v}|)=|v^0|^{-2}$, which will be used several times throughout this paper.
\subsection{Complete version of the main result} We are now ready to give a full and detailed version of Theorem \ref{Th0}. Recall the alternative geometric form \eqref{Vlasov1}-\eqref{Maxwell23} of the Vlasov-Maxwell equations \eqref{VM11}-\eqref{VM33}.

\begin{Th}\label{Th1}
Let $N \geq 3$ and $(f_0,F_0)$ be an initial data set of class $C^N$ for the Vlasov-Maxwell system. Consider further $\Lambda  \geq \epsilon >0$, two constants $(N_v,N_x) \in \R_+^2$ and assume that
$$ \sum_{|\gamma| \leq N+1} \, \sup_{x \in \R^3} \, \langle x \rangle^{2+|\gamma|} |\nabla_x^{\gamma} F_0| (x) \leq \Lambda, \qquad  \sum_{|\beta|+|\kappa| \leq N} \, \sup_{(x,v) \in \R^6}  \langle v \rangle^{N_v +|\kappa|} \langle x \rangle^{N_x+|\beta|} \left| \partial_v^{\kappa} \partial_x^\beta f_0\right|(x,v) \leq \epsilon.$$
If $N_v \geq 15$ and $N_x > 7$, there exists $D >0$ and $\epsilon_0>0$, depending only on $(N,N_v,N_x)$, such that, if $\overline{\epsilon} := \epsilon \, e^{D \Lambda} \leq \epsilon_0$, then, the unique solution $(f,F)$ to \eqref{VM11}-\eqref{VM33} arising from these data is global in time. Moreover,
\begin{itemize}
\item the following pointwise estimates hold for the distribution function,
\begin{alignat*}{2}
 \forall \, (t,x,v) \in \R_+ \times \R_v^3 \times \R^3_x, \quad & \forall \, |\beta| \leq N, \qquad && |v^0|^{N_v-3} \left| \mathbf{z}^{N_x-2} \widehat{Z}^\beta f \right| (t,x,v) \lesssim \overline{\epsilon} \, \log^{3N_x+3N}(3+t), \\
 & \forall \, |\kappa| \leq N, \qquad &&  |v^0|^{N_v-3} \left|\partial_{t,x}^\kappa f \right|(t,x, v)  \lesssim \overline{\epsilon}.
 \end{alignat*}
\item The electromagnetic field and its derivatives $\mathcal{L}_{Z^{\gamma}}(F)$, up to order $|\gamma| \leq N-1$, decay as
$$ \forall \, (t,x) \in  \R_+ \times \R^3, \qquad  | \mathcal{L}_{Z^{\gamma}}F|(t,x)  \lesssim \Lambda (1+t+|x|)^{-1}(1+|t-|x||)^{-1}.$$
If $|\gamma| \leq N-2$, the good null components enjoy stronger decay properties near the light cone,
$$ \forall \, (t,x,v) \in \R_+ \times \R^3, \qquad    |\alpha ( \mathcal{L}_{Z^{\gamma}}F)|(t,x)\!+\!|\rho ( \mathcal{L}_{Z^{\gamma}}F)|(t,x)\!+\!|\sigma ( \mathcal{L}_{Z^{\gamma}}F)|(t,x)  \lesssim \Lambda \frac{\log(3+t)}{(1+t+|x|)^2}. $$
 \end{itemize}
\end{Th}
Let us formulate two remarks.
\begin{enumerate}
\item More estimates, such as $\int_v f \dr v \lesssim t^{-3}$, are derived during the proof of Theorem \ref{Th1}.
\item With our method, contrary to our previous work \cite{dim3}, we cannot reach the optimal assumption $N_v=3$. We list in Remark \ref{Rqlossvz} below the precise parts of the proof where the control of higher spatial and momentum moments of $f$ are required.
\end{enumerate}
We now state our scattering result. For this, recall from \eqref{defFoverlin} the definition of the pure charge part $\overline{F}$ of $F$.
\begin{Th}\label{Th2}
Let $0 < \delta \leq 1$ and $(f,F)$ be a smooth solution to the Vlasov-Maxwell system arising from initial data satisfying the assumptions of Theorem \ref{Th1}. Suppose further that the initial electromagnetic field has the asymptotic expansion
\begin{equation}\label{assumasymptoexp}
 \sum_{|\gamma| \leq N+1} \, \sup_{|x| \geq 1} \, \langle x \rangle^{2+\delta+|\gamma|} \big| \nabla_{t,x}^\gamma \big(  F-\overline{F} \, \big) \big|(0,x) \leq \Lambda .
 \end{equation}
Then, with $n:=7(N_x+N)$, we have the following properties.
\begin{itemize}
\item The spatial average of $ f$ converges to a function $Q_{\infty} \in L^1(\R^3_v) \cap L^{\infty}(\R^3_v)$ of class $C^{N-1}$,
 $$\forall \; t \in \R_+, \qquad \bigg\| |v^0|^5 \bigg( \int_{\R^3_x}  f(t,x,v) \dr x - Q_{\infty}(v) \bigg) \bigg\|_{L^1_v\cap L^{\infty}_v} \lesssim \overline{\epsilon} \, \frac{\log^n(3+t)}{1+t}.$$
\item The four-current density $J(\widehat{Z}^\beta f)_{\mu}=\int_v \frac{v_{\mu}}{v^0}\widehat{Z}^\beta f \dr v$ has the following self-similar asymptotic profile. For any $|\beta| \leq N-1$ and $0 \leq \mu \leq 3$,
 $$  \forall \; t \in \R_+^*, \qquad \sup_{|x| < t} \bigg|t^3 \int_{\R^3_v} \frac{v^{\mu}}{v^0} \widehat{Z}^\beta f(t,x,v) \dr v - \frac{x^{\mu}}{t} \left( |v^0|^5Q_{\infty}^\beta \right) \left( \widecheck{\frac{x}{t}} \right)  \bigg| \lesssim \overline{\epsilon} \, \frac{\log^n(3+t)}{t},  \qquad x^0=t,$$
where $Q_\infty^\beta$ can be computed in terms of $\partial_v^\kappa Q_\infty$, $|\kappa| \leq |\beta|$. Moreover, $J(\widehat{Z}^\beta f)$ decays much faster in the exterior of the light cone.
\item The electromagnetic field and their derivatives up to order $|\gamma| \leq N-1$ have a self-similar asymptotic profile $v \mapsto \mathcal{L}_{Z^\gamma}(F)^\infty(v)$,
$$ \forall \, (t,x,v)\in \R_+\times \R^3_x \times \R^3_v, \qquad \left| t^2\mathcal{L}_{Z^\gamma}(F)(t,x+\widehat{v}t) - \mathcal{L}_{Z^\gamma}(F)^\infty(v) \right| \lesssim \Lambda \, \langle x \rangle^2|v^0|^8 \frac{\log^n(3+t)}{(1+t)^\delta}.$$
$F^\infty$ is of class $C^{N-1}$ and the components of $\mathcal{L}_{Z^\gamma}(F)^\infty$ can be computed in terms of $\partial_v^\kappa F^\infty _{\mu \nu}$, $|\kappa| \leq |\gamma|$.
\item We have modified scattering to a state $f_{\infty} \in L^1_{x,v} \cap L^\infty_{x,v}$ of class $C^{N-2}$. For any $|\kappa|+|\beta| \leq N-2$,
 $$  \forall \, t \geq 3, \qquad \left\| |v^0|^{N_v-10+|\xi|}\langle x \rangle^{N_x-4-|\xi|} \Big(\partial_v^\xi \partial_{x}^{\kappa} f(t,X_{\C}(t,x,v),v)-\partial_v^\xi \partial_{x}^{\kappa} f_{\infty}(x,v) \Big) \right\|_{ L^{\infty}_{x,v}} \lesssim \overline{\epsilon} \, \frac{\log^n(t)}{t^\delta},$$
 where the corrections of the linear spatial characteristics are defined as
 \begin{equation}\label{defXCTH}
  X^j_{\C}(t,x,v):= x^j+t\widehat{v}^j-\frac{\log(t)}{v^0} \, \widehat{v}^\mu \Big( F^{\infty,j}_{\mu}(v)+\widehat{v}^j F^{\infty}_{\mu 0}(v) \Big), \qquad 1 \leq j \leq 3  . 
  \end{equation}
 \item The modified complete lifts, of the Lorentz boosts $\widehat{\Omega}_{0k}$ and the rotations $\widehat{\Omega}_{jk}$, and the modified scaling,
 \begin{align*}
  \widehat{\Omega}_{\lambda k}^{\mathrm{mod}} & := \widehat{\Omega}_{\lambda k}- \frac{\log(t)}{v^0} \, \widehat{v}^\mu \Big( \mathcal{L}_{\Omega_{\lambda k}}(F)^{\infty , j}_{\mu }(v) +\widehat{v}^j \mathcal{L}_{\Omega_{\lambda k}}(F)^{\infty }_{\mu 0}(v) \Big) \partial_{x^j}, \qquad 0 \leq \lambda < k \leq 3, \\
 S^{\mathrm{mod}}  & := S+ \frac{\log(t)}{v^0} \, \widehat{v}^\mu \Big( F^{\infty,j}_{\mu}(v)+\widehat{v}^j F^{\infty}_{\mu 0}(v) \Big)  \partial_{x^j}, 
  \end{align*}
 verify the improved estimates $\| \widehat{\Omega}^{\mathrm{mod}}_{\lambda k} f (t,\cdot, \cdot) \|_{L^\infty_{x,v}} \lesssim \overline{\epsilon}$ and $\| S^{\mathrm{mod}} f (t,\cdot, \cdot) \|_{L^\infty_{x,v}} \lesssim \overline{\epsilon}$. 
\item For any $|\gamma| \leq N-3$, there exists a scattering state $\underline{\alpha}^{\mathcal{I}^+}_{\gamma}(u,\omega)$ on $\mathcal{I}^+$ such that,
 $$\forall \; \underline{u} \geq 3, \qquad \sup_{|u| \leq \underline{u}, \; \omega \in \mathbb{S}^2}  \left|r \underline{\alpha }( \mathcal{L}_{Z^{\gamma}}F)(u,\underline{u},\omega)-\underline{\alpha}^{\mathcal{I}^+}_{\gamma}(u,\omega)\right| \lesssim \Lambda \frac{\log(\underline{u})}{\underline{u}}.$$
 Moreover, $\underline{\alpha}^{\mathcal{I}^+}$ is of class $C^{N-3}$ and $\underline{\alpha}^{\mathcal{I}^+}_{\gamma}$ can be expressed in terms of the derivatives of $\underline{\alpha}^{\mathcal{I}^+}$.
 \item The conserved energy of the system can be related to the ones of the scattering states. For all $t \in \R_+$,
 $$\int_{\R^3_x} \int_{\R^3_v} v^0 f(t,x,v) \, \dr v \dr x+\frac{1}{2} \int_{\R^3_x} |F|^2(t,x) \, \dr x= \int_{\R^3_x} \int_{\R^3_v} v^0 f_\infty(x,v) \, \dr v \dr x+\frac{1}{4} \int_{\R^u} \int_{\mathbb{S}^2_\omega} \big|\underline{\alpha}^{\mathcal{I}^+}\big|^2(u,\omega) \, \dr \mu_{\mathbb{S}^2} \dr u.$$
\item If $N \geq 10$, there exists a solution $F^{\mathrm{vac}}$ of class $C^{N-5}$ to the vacuum Maxwell equations \eqref{Maxvac} such that, for any $\frac{1}{2} \leq q < 1$ and $|\gamma| \leq N-10$,
$$ \forall \; (t,x) \in \R_+ \times \R^3, \qquad |\mathcal{L}_{Z^\gamma}(F)-\mathcal{L}_{Z^\gamma}(F)^{\mathrm{vac}}|(t,x) \leq \Lambda C_q (1+t+|x|)^{-1-q}, \qquad \qquad C_q >0.$$
\end{itemize}
\end{Th}
\begin{Rq}
As suggested by the scattering result, we could improve the logarithmic powers in the $L^\infty_{x,v}$ estimates for $f$ stated in Theorem \ref{Th1}. We could then prove that Theorem \ref{Th2} holds for $n=3N_x+3N$. However, such tinies improvement would require a relatively long and technical proof.
\end{Rq}
\begin{Rq}
We emphasize two main differences with previous works on Vlasov systems in dimension $3$ based on vector field methods \cite{FJS3,Poisson,dim3,Duan}.
\begin{enumerate}
\item The logarithmic correction of the linear commutators $\widehat{\Omega}_{\lambda \nu}$ and $S$ can be geometrically interpreted in terms of the asymptotic dynamic of the Lorentz force $\widehat{v}^\mu {F_{\mu k}}$ and its derivatives (see also Remark \ref{RqLor}).
\item Our approach does not require to modify the linear commutators in order to prove the global existence of the solutions, so that we avoid many technical difficulties. In these previous works, the analysis of the Vlasov field relied on propagating $L^1_{x,v}$ bounds. The source term of the wave equations (or the Poisson equation) were estimated through weighted-Sobolev embeddings as $t^3|Z^\beta \int_v  f \dr v| \leq t^3\int_v|\widehat{Z}^\beta f| \dr v \lesssim  \mathbb{E}(t)$, where $\mathbb{E}(t)$ is a certain $L^1_{x,v}$ norm. However, we know from Theorems \ref{Th1}-\ref{Th2} that, in general, $\|\widehat{Z}f\|_{L^1_{x,v}} \gtrsim \log(t)$ if $\widehat{Z}\neq \partial_{t,x}$. As a consequence, the optimal decay $t^{-3}$ cannot be obtained in such a way without modifying the linear commutators.
\end{enumerate}
\end{Rq}
\begin{Rq}
The profile $F^\infty$ of $F$ can be explicitly expressed in terms of the limit of the spatial average $Q_\infty$ (see Remark \ref{Rqexplicompu} and Annexe \ref{AnnexeC1}). Moreover, the Maxwell field admits the decomposition $F=F^T+F^2$, where
$$ \lim_{t \to + \infty} t^2 F(t,x+t\widehat{v}) = \lim_{t \to + \infty} t^2F^T(t,x+t\widehat{v}) =F^\infty(v), \qquad \qquad \lim_{\underline{u} \to + \infty} rF^T(u,\underline{u},\omega)=0.$$
In other words, the part of the electromagnetic field which gives rise to $F^\infty$ (respectively $\underline{\alpha}^{\mathcal{I}^+}$) has no impact on $\underline{\alpha}^{\mathcal{I}^+}$ (respectively $F^\infty$).
\end{Rq}

\subsection{Key ingredients of the proof}\label{SubsecIngred}
For the global existence result, our strategy relies on vector field methods and a continuity argument. The proof then essentially consists in improving bootstrap assumptions, which are pointwise decay estimates on the solutions and their derivatives. The scattering statements are then obtained by refining the analysis carried out during of the proof of Theorem \ref{Th1} and by investigating further the asymptotic behavior of the electromagnetic field.

\subsubsection{The large Maxwell field}

The assumptions of Theorems \ref{Th1}-\ref{Th2} imply that, initially, the distribution function $f$ is at most of size $\epsilon \ll 1$ and the electromagnetic field $F$ is at most of size $\Lambda$. The goal of our bootstrap argument is to prove that these properties are preserved over time. Our proof allows for $\Lambda$ to be large for the following reasons.
\begin{itemize}
\item Since the Maxwell equations are \textit{linear}, we can expect $F(t,\cdot)$ and its derivatives to be at most of size $\Lambda+C\epsilon \sim \Lambda$, provided that $\epsilon$ is small enough. Here, the constant $C$ possibly depends on $\Lambda$. Indeed, the data are bounded by $\Lambda$ and we expect the source term $J(f)$ to remains of size $\epsilon$.
\item In contrast, the Vlasov equation is nonlinear and we can expect, at first glance, to bound $\|\partial_{t,x}^\kappa f(t,\cdot)\|_{L^\infty_{x,v}}$ by $\epsilon+D\Lambda\epsilon = C(\Lambda)\epsilon$.
\end{itemize}
In fact, since our argument will rely on Grönwall's inequality, $C(\Lambda)$ will rather be of the form $e^{D\Lambda}$. The difficulty, if $\Lambda$ is large, is related to the logarithmic growth of quantities such as $\|\widehat{\Omega}_{01} f  \|_{L^\infty_{x,v}}$. More precisely, certain error terms are at the threshold of time integrability. Consequently a naive application of Grönwall's inequality would lead to $\|\widehat{\Omega}_{01} f  \|_{L^\infty_{x,v}} \lesssim \epsilon \, (1+t)^{D \Lambda}$. We discuss how to circumvent this obstacle in the next section.

\subsubsection{Estimates for the Vlasov field}\label{SubsecdiffVlasov} In order to control sufficiently well the electromagnetic field and close our estimates, we would like to recover the linear decay for $|\int_v \widehat{Z}^\beta f(t,x,v) \dr v | \lesssim t^{-3}$, with $|\beta| \leq N-1$, and similar quantities. This is done as follows
\begin{itemize}
\item The main step consists in proving that $|v^0|^{N_v}  \mathbf{z}^{N_x}\widehat{Z}^\beta f$ grows slowly, and in fact logarithmically, in time.
\item Then, by performing the standard change of variables $y=x-t\widehat{v}$, we are able to reduce the problem to proving a uniform bound for the spatial averages $|v^0|^5\int_y \widehat{Z}^\beta f(t,y,v) \dr y$. This turns out to be a consequence of the first step as well but our argument requires a loss of regularity, which is why we do not attain the optimal decay $t^{-3}$ for the top order derivatives $|\beta|=N$.
\end{itemize}
Let us illustrate certain difficulties of the first step, which relies on Duhamel formula, by considering the first order derivatives. If $Z \in \mathbb{K} \setminus \{S \}$ is a Killing vector field, then
\begin{equation}\label{eq:trian}
 |\T_F(\widehat{Z}f)|=|\widehat{v}^\mu {\mathcal{L}_Z(F)_{\mu}}^j \partial_{v^j}f | \lesssim \sum_{1 \leq j \leq 3}\frac{t+|x|}{v^0} |\widehat{v}^\mu {\mathcal{L}_Z(F)_{\mu}}^j||\partial_{t,x} f|+\text{better terms.}
 \end{equation}
Since $\mathcal{L}_Z(F)$ is supposed to decay as\footnote{This pointwise decay estimate is consistent with the expected behavior of the source term of the Maxwell equations.} $|\mathcal{L}_Z(F)| \lesssim \Lambda (1+t+|x|)^{-1}(1+|t-|x||)^{-1}$, there are two problems.
\begin{enumerate}
\item The decay rate degenerates near the light cone $t=|x|$.
\item Even far from the light cone, $|T_F(\widehat{Z}f)| \sim \Lambda t^{-1} |\partial_{t,x} f|$ is not integrable in time, preventing us to prove that $\|\widehat{Z}f\|_{L^\infty_{x,v}}$ grows slowly by a direct application of Grönwall's inequality if $\Lambda$ is large. 
\end{enumerate}

We deal with the first issue by taking advantage of the null structure of the Lorentz force, which, roughly speaking, allows us to transform decay in $t-r$ into decay in $t+r$. More precisely, $\widehat{v}^\mu {\mathcal{L}_Z(F)_{\mu}}^j$ can be decomposed as the sum of terms containing either a good null component $\alpha$, $\rho$ or $\sigma$ of $\mathcal{L}_Z(F)$ or one of the good null components of $\widehat{v}$. The first ones enjoy improved decay estimates near the light cone whereas the latters allow us to exploit the decay in $t-r$. We refer to Lemmas \ref{Lorentzforce} and \ref{techLemTF} for more details.

We circumvent the second problem by identifying hierarchies in the commuted equations. More precisely, if $Z= \partial_{x^\mu}$ is a translation, one can use that $|\mathcal{L}_{\partial_{x^\mu}}(F)|\lesssim t^{-1}(1+|t-|x||)^{-2}$ in order to prove that $T_F(\partial_{x^\mu}f)$ is in fact time integrable. Then, one can observe that the system of the commuted Vlasov equations \eqref{eq:trian} is in some sense triangular and expect $\|\widehat{Z}f\|_{L^\infty_{x,v}}$ to grow at most logarithmically. A toy model for the system of the first order commuted equations, once the null structure is well understood, is then
$$
\T_F(g)=\Lambda (1+t)^{-2} g+\Lambda (1+t)^{-3}h, \qquad \quad \T_F(h)=\Lambda (1+t)^{-1} g+\Lambda (1+t)^{-2}h, \qquad \qquad g \geq 0 , \, h \geq 0,
$$
where $g$ is supposed to capture the behavior of $|\partial_{x^\mu}f|$, $0 \leq \mu \leq 3$, and $h$ the one of $|\widehat{Z}f|$, with $\widehat{Z}$ is a homogeneous vector field such as $\widehat{\Omega}_{01}$. The source terms having $h$ as a factor represent the strongly decaying error terms in \eqref{eq:trian}. Using Duhamel formula and applying Grönwall's inequality, we have, for $\mathbb{E}(t):= \|g(t,\cdot,\cdot) \|_{L^\infty_{x,v}}+\|h(t,\cdot,\cdot) \|_{L^\infty_{x,v}}$, 
$$ \mathbb{E}(t) \leq \mathbb{E}(0)+\int_{s=0}^t \frac{\Lambda }{ 1+s} \mathbb{E}(s) \dr s, \qquad \qquad \mathbb{E}(t) \leq \mathbb{E}(0)\,(1+t)^{\Lambda}.$$
As mentioned earlier, without any smallness assumption on $\Lambda$, this estimate is not good enough in order to derive a satisfying decay estimate for $\int_v f \dr v$. The idea then is to exploit that 
$$\T_F \big( \log^{-1}(3+t) \big) \leq 0, \qquad \T_F \big(h \log^{-2}(3+t) \big) \leq \Lambda (1+t)^{-1}\log^{-2}(3+t) g+\Lambda (1+t)^{-2}h\log^{-2}(3+t).$$
By considering the hierarchised norm $\overline{\mathbb{E}}(t) :=\|g(t,\cdot,\cdot) \|_{L^\infty_{x,v}}+\|h(t,\cdot,\cdot) \|_{L^\infty_{x,v}} \log^{-2}(3+t)$, we finally get
$$ \overline{\mathbb{E}}(t) \leq \overline{\mathbb{E}}(0)+\int_{s=0}^t \frac{2\Lambda }{ (1+s)\log^{2}(3+s)}\overline{\mathbb{E}}(s) \dr s, \qquad \qquad \overline{\mathbb{E}}(t) \leq \overline{\mathbb{E}}(0) e^{ 2\Lambda}.$$
More generally, the hierarchies are determined by the number of homogeneous vector fields $\beta_H$ composing $\widehat{Z}^\beta$ and the exponent of the weight $\mathbf{z}$.

A new difficulty arises for the higher order derivatives since we do not have improved estimates at our disposal on the good null components of $\mathcal{L}_{Z^\gamma}(F)$, for $|\gamma| \geq N-1$. This time, we transform decay in $t-r$ into decay in $t+r$ by losing powers of $|v^0|^2\mathbf{z}$ through Lemma \ref{gainv}. For this, it is important to observe that, in the error terms, a such $\mathcal{L}_{Z^\gamma}(F)$ is always multiplied by a low order derivative of $f$. We can then close the estimates by propagating weaker $L^\infty_{x,v}$ norms of $\widehat{Z}^\beta f$ when $|\beta| \geq N-1$.
\begin{Rq}
Let us make some comparisons between the decay properties of the electromagnetic $F$ and the ones of the electric field $E$ associated to a solution to the Vlasov-Poisson system arising from small data. 
\begin{itemize}
\item As $\|E (t,\cdot)\|_{L^\infty_x}\lesssim t^{-2}$ and $|F|(t,x) \lesssim t^{-1} (1+|t-|x||)^{-1}$, the electromagnetic field has a much weaker decay rate near the light cone $t=r$ than $E$.
\item The difference is even more marked for their derivatives since $|\partial_{t,x}^\kappa E|(t,x) \lesssim t^{-2-|\kappa|}$ whereas we merely have $|\mathcal{L}_{\partial_{t,x}^\kappa}F|(t,x) \lesssim t^{-1} (1+|t-|x||)^{-1-|\kappa|}$. Thus, in order to exploit the extra decay provided by these derivatives of $F$, one has to take advantage of the null structure of the system or to lose powers of $|v^0|^2\mathbf{z}$.
\end{itemize} 
\end{Rq}

\subsubsection{Estimates for the electromagnetic field}

We control the cartesian components of $\mathcal{L}_{Z^\gamma}(F)$ using the representation formula for the wave equation since, for instance, $\Box F_{01}=-\int_v \partial_{x^1}f+\widehat{v}^1 \partial_t f \dr v$. However, two difficulties arise for the higher order derivatives.
\begin{enumerate}
\item There is a loss of regularity. We need to control $\int_v \widehat{v}^\mu \partial_{t,x} \widehat{Z}^\gamma f \dr v $ in order to estimate $\mathcal{L}_{Z^\gamma}(F)$.
\item With our method, we do not have the optimal decay rate for $\int_v  \widehat{Z}^\gamma f \dr v$, $|\gamma|=N$. Moreover, any logarithmic loss would prevent us to close our estimates.
\end{enumerate} 
We treat the first problem by using the Glassey-Strauss decomposition of the electromagnetic field \cite{GlStrauss}, presented in detail in Section \ref{subsecGS}. The idea is to express the derivatives $\partial_{x^\mu}$ in terms of derivatives tangential to backward light cones and $\T_0=\partial_t+\widehat{v} \cdot \nabla_x$, which is transverse to light cones. Exploiting then the Vlasov equation $\T_F(f)=0$, we can perform integration by parts and save one derivative. 

We deal with the second issue by estimating $\nabla_{t,x} \mathcal{L}_{Z^\xi}(F)$, for $|\xi|=N-1$, by the Glassey-Strauss decomposition of the derivatives of the electromagnetic field. Roughly speaking, it allows us to control the inhomogeneous part of $\nabla_{t,x} \mathcal{L}_{Z^\xi}(F)$ by $\int_v |v^0|^3|\widehat{Z}^\beta f| \dr v$, where $|\beta| \leq N-1$ (see Proposition \ref{GSdecomoderiv} and Corollary \ref{estikernels} for more details). However, with this process, we get a bad control of the other top order derivatives near the light cone, 
$$|\mathcal{L}_{ZZ^\xi}(F)|(t,x) \lesssim (1+t+|x|)|\nabla_{t,x} \mathcal{L}_{Z^\xi} F|(t,x)+|\mathcal{L}_{Z^\xi} F|(t,x)\lesssim (1+|t-r|)^{-2} \log(3+|t-r|), \quad \; \;  |\xi| = N-1.$$ This forces us to lose a power more of $|v^0|^2 \mathbf{z}$ for the estimates of the top order derivatives of the Vlasov field $f$.

Once we proved that the solutions are global in time, we use null properties of the Maxwell equations \eqref{Maxwell23} to derive the existence of a scattering state for $F$ and its derivatives. We then address the problem of finding a solution $F^{\mathrm{vac}}$ to the vacuum Maxwell equations which approaches $F$ by constructing a scattering map for these equations. For this, we make crucial use of the corresponding result for the homogeneous wave equation \cite{SchlueLindblad}. This is carried out in Section \ref{Sec7}.

\subsubsection{Modified scattering result}\label{subsecmotivatedby}

In the context of the Vlasov-Poisson system, except for the trivial solution, the distribution function does not converge along the linear characteristics \cite{Choi22}. We then do not expect $f(t,x+t\widehat{v},v)$ to converge and the reason is related to the long range effect of the Lorentz force (recall Remark \ref{Rqcharact}). More precisely, isolating the leading order term of the source term of the Maxwell equations,
$$  \sup_{|x| < t} \bigg|t^3 \int_{\R^3_v} \frac{v^{\mu}}{v^0}  f(t,x,v) \dr v - \frac{x^{\mu}}{t} \left( |v^0|^5Q_{\infty} \right) \left( \widecheck{\frac{x}{t}} \right)  \bigg| = O \big(t^{-\frac{\delta}{2}} \big),   \qquad Q_\infty(v) := \lim_{t \to + \infty} \int_{\R^3_x} f(t,x,v) \dr v,$$
where $x^0=t$, we are able to prove $t^2F(t,x+t\widehat{v})= F^\infty(v) +O(t^{-\delta/2} )$. Consequently, the slow decay of the electromagnetic field along timelike trajectories implies  that the right hand side of
$$ \partial_t \big( f(t,x+t\widehat{v},v) \big)\! = \! \frac{t}{v^0} \widehat{v}^\mu \big( {F_{\mu}}^j(t,x+t\widehat{v})+\widehat{v}^j F_{\mu 0}(t,x+t\widehat{v}) \big) \partial_{x^j} f(t,x+t\widehat{v},v)+ O \big(t^{-\frac{\delta}{2}} \big)$$
should not be time integrable, preventing $f(t,x+t\widehat{v},v)$ to converge. Instead, by considering the logarithmic corrections $X_\C$, given in \eqref{defXCTH}, of the timelike straight lines, one can compensate the worst term in the right hand side of the previous identity and prove the modified scattering statement $f(t,X_\C,v) \to f_\infty (x,v)$.

Although the regularity of $f_\infty$ according to $x$ can be obtained in a similar fashion, the regularity in $v$ requires a more thorough analysis. In fact, $ v^0 \partial_{v^i}  ( f(t,X_\C,v) )$ can be expressed as terms such as $\widehat{\Omega}_{0i}f(t,X_\C,v)$ which, contrary to $\partial_{t,x}f(t,X_\C,v)$, does not converge. The reason is related to the weak decay of the error term $[\T_F,\widehat{\Omega}_{0i}] \sim t^{-1}$. As for the characteristics, the idea consists in considering a logarithmic correction of $\widehat{\Omega}_{0i}$, introduced and studied in Section \ref{subsecmodicom}, which has improved commutation properties with the Vlasov operator $\T_F$. As stated in Theorem \ref{Th2}, these corrections are given in terms of first order derivatives of the effective electromagnetic field $F^\infty(v)$.

\subsection{Null properties of electromagnetic fields}

We recall here the classical results which will be used throughout this paper in order to study solutions to the Maxwell equations
\begin{equation}\label{Maxwithsource}
\nabla^{\mu} F_{\mu \nu} = J_{\nu}, \qquad \qquad \nabla^{\mu} {}^* \! F_{\mu \nu}=0,
\end{equation}
where the source term $J=J_{\mu} \mathrm{d}x^{\mu}$ is a sufficiently regular $1$-form. In particular, solutions to the vacuum Maxwell equations will satisfy
\begin{equation}\label{Maxvac}
\nabla^{\mu} F_{\mu \nu} = 0, \qquad \qquad \nabla^{\mu} {}^*  \! F_{\mu \nu}=0.
\end{equation} 
We point out that some of the estimates presented here could be refined in a general setting. For the purpose of performing energy estimates during the construction of the scattering map for \eqref{Maxvac}, we recall the electromagnetic stress–energy tensor.
\begin{Def}\label{Defenergytensor}
Let $G$ be a $2$-form of class $C^1$ such that $\nabla^{\mu} G_{\mu \nu}=J_{\nu}$ and $\nabla^{\mu} {}^* \! G_{\mu \nu}=0$. The energy-momentum tensor $\mathbb{T}[G]_{\mu \nu}$ is defined as 
$$ \mathbb{T}[G]_{\mu \nu}:=G_{\mu \beta} {G_{\nu}}^{\beta}-\frac{1}{4}\eta_{\mu \nu} G_{\xi \lambda} G^{\xi \lambda}.$$
Moreover, we have
$$ \nabla^{\mu} T[G]_{\mu \nu}=G_{\nu \lambda} J^{\lambda}, \qquad T[G]_{L L}=|\alpha(G)|^2, \qquad T[G]_{\underline{L} \underline{L} }=|\underline{\alpha}(G)|^2, \qquad T[G]_{L \underline{L}}=|\rho(G)|^2+|\sigma(G)|^2.$$
\end{Def}
We now present inequalities relying on the relations
\begin{equation}\label{eq:nullderiv}
(t-r)\underline{L}=S-\frac{x^i}{r}\Omega_{0i}, \quad (t+r)L=S+\frac{x^i}{r}\Omega_{0i}, \quad re_{\theta}=-\cos (\varphi) \Omega_{13}-\sin(\varphi)\Omega_{23} ,\quad re_{\varphi}= \Omega_{12}.
\end{equation}
\begin{Lem}\label{improderiv0}
Let $G$ be a sufficiently regular solution to the Maxwell equations \eqref{Maxwithsource} with a smooth source term $J$. Then 
$$ \forall \; |x| \geq \frac{1+t}{2}, \qquad \left(\left| \nabla_{\underline{L}} \alpha( G) \right|+\left| \nabla_{\underline{L}} \rho( G) \right|+\left| \nabla_{\underline{L}} \sigma( G) \right|\right)(t,x) \lesssim |J|(t,x)+ \sum_{|\gamma| \leq 1} \frac{\left| \mathcal{L}_{Z^{\gamma}}(G)\right|(t,x)}{1+t+|x|} $$
and
$$\forall \, (t,x) \in \R_+ \times \R^3, \qquad \left| \nabla_{L} \left(r \underline{\alpha}(G) \right) \right|(t,x) \lesssim r|J|(t,x)+ \sum_{|\gamma| \leq 1}  \left| \rho (\mathcal{L}_{Z^{\gamma}}G)\right|\!(t,x)+\left| \sigma (\mathcal{L}_{Z^{\gamma}}G)\right|\!(t,x)  .$$
\end{Lem}
\begin{Rq}
Compared to $Z \in \mathbb{K}$, $Z \neq \partial_{x^{\mu}}$, the derivatives tangential to the light cone $(L,e_{\theta},e_{\varphi})$ provide an extra decay in $t+r$ whereas $\underline{L}$ merely provides an additional decay in $t-r$. The second estimate then reflects that $\alpha$, $\rho$ and $\sigma$ are the good null components. The last inequality provides an improved control of $\nabla_L(r\underline{\alpha})$ near the light cone and will be useful in order to prove the existence of scattering states.
\end{Rq}
\begin{proof}
Let us denote by $\slashed{\nabla}$ the intrinsic covariant differentiation on the spheres and by $\zeta$ any of the null components $ \alpha,  \, \underline{\alpha}, \, \rho$ or $ \sigma$. Then, according to \cite[Lemma $D.2$]{massless}, we have for all $(t,x) \in \R_+ \times \R^3$,
 $$(1+t+|x|)\left| \nabla_L \zeta( G) \right|\!(t,x)+(1+|x|)\left| \slashed{\nabla} \zeta( G) \right|\!(t,x)+(1+|t-|x||)\left| \nabla_{\underline{L}} \zeta( G) \right|\!(t,x)\lesssim  \sum_{|\gamma| \leq 1} \left| \zeta (\mathcal{L}_{Z^{\gamma}}G)\right|\!(t,x).$$
We now express the Maxwell equations in null coordinates. According to \cite[equations~$(M_1'')$~-~$(M_6'')$]{CK}, we have for any $A \in \{\theta , \varphi \}$,
\begin{alignat*}{2} 
&\nabla_{\underline{L}}  \rho(G)-\frac{2}{r} \rho(G)- \slashed{\nabla}^{e_B} \underline{\alpha}(G)_{e_B}  =  J_{\underline{L}}, \qquad \qquad && \nabla_{\underline{L}}  \alpha(G)_{e_A}-\frac{\alpha(G)_{e_A}}{r}+\slashed{\nabla}_{e_A} \rho (G)-\varepsilon^{AB} \slashed{\nabla}_{e_B} \sigma(G) = J_{e_A}  , \\ 
& \nabla_{\underline{L}}  \sigma(G)-\frac{2}{r} \sigma(G)+ \varepsilon^{AB} \slashed{\nabla}_{e_A} \underline{\alpha}(G)_{e_B} = 0, \qquad \qquad  && \nabla_{L}  \underline{\alpha}(G)_{e_A}+\frac{\underline{\alpha}(G)_{e_A}}{r}-\slashed{\nabla}_{e_A} \rho (G)-\varepsilon^{AB} \slashed{\nabla}_{e_B} \sigma(G) = J_{e_A}. 
\end{alignat*}
This allows us to deduce the first estimate. For the last one, use the same arguments and remark further that $\nabla_{L}e_A=0$ implies
$$\left|\nabla_{L}(r \underline{\alpha})\right|\lesssim \sum_{B \in \{ \theta , \varphi \} } \left|\nabla_{L}(r \underline{\alpha})_{e_B}\right|= \sum_{B \in \{ \theta , \varphi \} } \left|\nabla_{L}(r \underline{\alpha}_{e_B})\right|= \sum_{B \in \{ \theta , \varphi \} } \left|r\nabla_{L}  \underline{\alpha}(G)_{e_B}+\underline{\alpha}(G)_{e_B}\right| .$$
\end{proof}
In the same spirit, we have the following identity which is proved in \cite[Proposition~$3.7$, equation $(18)$]{dim3}.
\begin{Lem}\label{improderiv}
For any sufficiently regular $2$-form $G$ and any null component $\zeta \in \{ \alpha, \underline{\alpha}, \rho, \sigma \}$,
$$\forall \; (t,x) \in \R_+ \times \R^3, \qquad \left|  \zeta( \nabla_{t,x}G) \right|(t,x)  \lesssim  \sum_{|\gamma| \leq 1}\frac{\left| \zeta (\mathcal{L}_{Z^{\gamma}}G)\right|(t,x)}{1+|t-|x||} +\frac{\left| \mathcal{L}_{Z^{\gamma}}(G)\right|(t,x)}{1+t+|x|}.$$
\end{Lem}
We now illustrate how the previous Lemmas can be used in order to obtain improved estimates for the good null components of the electromagnetic field.
\begin{Cor}\label{Corgoodnull}
Consider a $2$-form $G$ of class $C^1$, solution to the Maxwell equations \eqref{Maxwithsource} with a continuous source term $J$. Assume that there exists two constants $C[G] >0$ and $q >0$ such that
\begin{equation}\label{eq:assump}
 \forall \; (t,x) \in \R_+ \times \R^3, \qquad (1+t+|x|) |J|(t,x)+\sum_{|\gamma| \leq 1} |\mathcal{L}_{Z^{\gamma}}(G)|(t,x) \leq \frac{C[G]}{(1+t+|x|)(1+|t-|x||)^q}.
 \end{equation}
Then, for all $(t,x) \in \R_+ \times \R^3$,
$$  \big( |\alpha (G) |+|\rho (G)|+|\sigma(G)| \big)(t,x) \lesssim C[G]  \left\{
    \begin{array}{ll}
     (1+t+|x|)^{-1-q} & \mbox{if $0<q<1$, } \\
        \log(3+t) (1+t+|x|)^{-2} & \mbox{if $q=1$, } \\
          (1+t+|x|)^{-2}(1+|t-|x||)^{-q+1} & \mbox{if $q>1$}.
    \end{array}
\right. $$
Moreover, if $G$ is merely defined on $[0,T[\times \R^3$, $T>0$, we have the weaker estimate for the case $q >1$,
$$ \forall \; (t,x) \in [0,T[ \times \R^3, \qquad \big( |\alpha (G) |+|\rho (G)|+|\sigma(G)| \big)(t,x) \lesssim C[G] (1+t+|x|)^{-2} \qquad \mbox{if $q>1$}. $$
\end{Cor}
\begin{proof}
Note first that the assumptions give $|G|(t,x) \lesssim (1+t+|x|)^{-1-q}$ if $1+t \geq 2|x|$ or $|x| \geq 2(1+t)$. We then fix $(t,r\omega) \in \R_+ \times \R^3$ such that $1+t \leq 2r \leq 4(1+t)$, $\omega \in \mathbb{S}^2$ and we denote by $\zeta$ any of the null components $\alpha$, $\rho$ or $\sigma$. Consider further
$$ \phi (u,\underline{u}) := \zeta (G) \left( \frac{\underline{u}+u}{2}, \frac{\underline{u}-u}{2} \omega \right).$$
By Lemma \ref{improderiv0} and \eqref{eq:assump}, we have
$$\left| \nabla_{\partial_u} \phi \right|(u,\underline{u}) = \frac{1}{2}\left| \nabla_{\underline{L}} \zeta(G) \right|\left( \frac{\underline{u}+u}{2}, \frac{\underline{u}-u}{2} \omega \right) \lesssim \frac{C[G]}{(1+\underline{u})^2(1+|u|)^{q}}.$$
Now, remark that, for $t-r \leq 0$,
\begin{align*}
 |\zeta(G)|(t,r\omega)&=|\phi|(t-r,t+r)= \left|\phi(-t-r,t+r)+\int_{u=-t-r}^{-|t-r|} \nabla_{\partial_u} \phi(u,t+r) \dr u \right| \\
 & \lesssim |\zeta(G)|(0,t\omega+r\omega)+\frac{C[G]}{(1+t+r)^2} \int_{u=-t-r}^{-|t-r|} \frac{\dr u}{(1+|u|)^{q}} .
 \end{align*}
 Similarly, if $t-r \geq 0$, we obtain by integrating between $u=t-r$ and $t+r$,
 $$ |\zeta(G)|(t,r\omega) \lesssim |\zeta(G)|(t+r,0)+\frac{C[G]}{(1+t+r)^2} \int_{u=|t-r|}^{t+r} \frac{\dr u}{(1+|u|)^{q}}.$$
 By \eqref{eq:assump}, $|\zeta(G)|(t+r,0)+|\zeta(G)|(0,t\omega+r\omega)\lesssim C[G] (1+t+r)^{-1-q}$ and the first part of the result then follows from the computations of the integrals in the previous two estimates. For the case $q=1$, note that $\log(1+t+r) \leq 3\log (3+t)$ since $r \leq 2+2t$.
 
 If $G$ is merely defined on $[0,T[ \times \R^3$ and $t <T$, then we cannot apply the previous computations in the case $t \geq r$. Instead, we integrate between $u=0$ and $t-r$ in order to get
 $$ |\zeta(G)|(t,r\omega) \lesssim |\zeta(G)|\left(\frac{t+r}{2},\frac{t+r}{2}\omega \right)+\frac{C[G]}{(1+t+r)^2} \int_{u=0}^{|t-r|} \frac{\dr u}{(1+|u|)^{q}}.$$
 It remains to bound $|\zeta(G)|\left(\frac{t+r}{2},\frac{t+r}{2}\omega\right) $ by the estimate obtained in the region $t \leq r$ and to compute the integral in the three different cases. 
\end{proof}
Finally, we prove pointwise decay estimates for a solution to the homogeneous wave equation. Since the cartesian components $F_{\mu \nu}$ of a solution $F$ to the vacuum Maxwell equations verify $\Box F_{\mu \nu}=0$, the next result will also allow us to estimate such electromagnetic fields.
\begin{Pro}\label{decaylinMax}
Let $\phi$ be a $C^2$ solution to the free wave equation $\Box \, \phi =0$ such that
$$ \mathcal{E}^q[\phi]:= \sup_{x \in \R^3} \, \langle x \rangle^{q} \left| \phi \right|(0,x)+ \sup_{x \in \R^3} \, \langle x \rangle^{q+1} \left|\partial_{t,x} \phi \right|(0,x) <+\infty, \qquad q \geq 2.$$
Then, there holds
\begin{equation*}
\forall \; (t,x) \in \R_+ \times \R^3, \qquad \left| \phi \right|(t,x) \lesssim \frac{\mathcal{E}^q[\phi]}{(1+t+|x|)(1+|t-|x||)^{q-1}}.
\end{equation*}
\end{Pro}
\begin{proof}
By Kirchhoff's formula we have
$$ \phi(t,x) = \frac{1}{4\pi t^2} \int_{|y-x|=t} \phi(0,y)\dr y+\frac{1}{4\pi t} \int_{|y-x|=t}\frac{y-x}{|y-x|}\cdot \nabla_y \phi(0,y)+ \partial_t \phi(0,y) \dr y.$$
We obtain the result by applying\footnote{The case $2<p<3$, not considered by \cite{WeiYang}, can be treated as the case $p=2$ since $\int_b^a \frac{\lambda \dr \lambda}{(1+\lambda)^p} \leq (1+b)^{p-2}\int_b^a \frac{\lambda \dr \lambda}{(1+\lambda)^2}$.} \cite[Lemma~$4.1$]{WeiYang}, which gives that for any $h \in C(\R^3)$ such that $|h|(x) \leq K_0(1+|x|)^{-p}$,
\begin{equation}\label{decayhYang}
 \int_{|y-x|=t}|h|(y) \dr y \leq \left\{ 
	\begin{array}{ll}
        8 \pi K_0 t^2 (1+t+|x|)^{-1}(1+|t-|x||)^{-p+1} & \mbox{if $2\leq p <3$, } \\
        4\pi K_0 t (1+t+|x|)^{-1}(1+|t-|x||)^{-p+2} & \mbox{if $p \geq 3$}.
    \end{array} 
\right. 
\end{equation}
\end{proof}

\section{Strategy of the proof and the bootstrap assumptions}\label{Sec3}

Let $N \geq 3$, $N_v \geq 15$, $N_x > 7$ and consider an initial data set $(f_0,F_0)$ satisfying the hypotheses of Theorem \ref{Th1}. By a standard local well-posedness argument, there exists a unique maximal solution $(f,F)$ to the Vlasov-Maxwell system arising from these data. Let $T_{\mathrm{max}} \in \R_+^* \cup\{+\infty\}$ such that the solution is defined on $[0,T_{\mathrm{max}}[$. By continuity, there exists a largest time $0<T\leq T_{\mathrm{max}}$ and a constant $C_{\mathrm{boot}}>0$, independent of $\epsilon$, such that the following bootstrap assumptions hold. For all $(t,x) \in [0,T[ \times \R^3$,
\begin{flalign}
&\forall \; |\gamma| \leq N-1,&  \left| \mathcal{L}_{Z^{\gamma}}(F) \right|(t,x) & \leq \frac{C_{\mathrm{boot}}\Lambda}{(1+t+|x|)(1+|t-|x||)}, \tag{BA1} &\label{boot1}\\
&\forall \; |\gamma| = N-1,&  \left| \nabla_{t,x}\mathcal{L}_{Z^{\gamma}}(F) \right|(t,x) & \leq  \frac{ C_{\mathrm{boot}}\Lambda \log (3+|t-|x||)}{(1+t+|x|)(1+|t-|x||)^2}, \tag{BA2} &\label{boot2} \\
&\forall \; |\beta| = N-2,&  \left| \int_{\R_v^3} \frac{v^{\mu}}{v^0} \widehat{Z}^{\beta} f(t,x,v) \dr v\right| & \leq  \frac{C_{\mathrm{boot}} \Lambda}{(1+t+|x|)^3}, \qquad \qquad \qquad \qquad \qquad 0 \leq \mu \leq 3. \tag{BA3} &\label{boot3}
\end{flalign}
The goal consists in improving, for $C_{\mathrm{boot}}$ chosen large enough and if $\epsilon$ is small enough, \eqref{boot1}-\eqref{boot3}. We stress that \eqref{boot3} will only be used for the proof of Proposition \ref{Proproofnullcompo}, where we improve the pointwise decay estimates for the good null components of the electromagnetic field. 

\subsection{Immediate consequences of the bootstrap assumptions}\label{subsecpointwisedecay}

We start by improving, near the light cone, the estimates for the good null components of the electromagnetic field and its derivatives up to order $N-2$. 
\begin{Pro}\label{Proproofnullcompo}
For any $|\gamma| \leq N-2$ and all $(t,x) \in [0,T[ \times \R^3$, we have
\begin{align*}
\big(|\alpha (\mathcal{L}_{Z^{\gamma}} F)|+|\rho (\mathcal{L}_{Z^{\gamma}} F)|+|\sigma (\mathcal{L}_{Z^{\gamma}} F)| \big)(t,x) &\lesssim \frac{\Lambda \log(3+t)}{(1+t+|x|)^2(1+|t-|x|)^{\gamma_T}}, \\
|\underline{\alpha}(\mathcal{L}_{Z^{\gamma}} F)|(t,x) & \lesssim \frac{\Lambda}{(1+t+|x|)(1+|t-|x||)^{1+\gamma_T}},
\end{align*}
where we recall that $\gamma_T$ is number of translations $\partial_{x^{\mu}}$ composing $Z^{\gamma}$.
\end{Pro}
\begin{proof}
Consider $|\gamma| \leq N-2$ and recall from Proposition \ref{Com} that $\mathcal{L}_{Z^{\gamma}}F$ is solution to the Maxwell equations \eqref{Maxwithsource} with a source term which is a linear combination of $J(\widehat{Z}^{\beta}f)$, $|\beta| \leq N-2$, which are bounded by the bootstrap assumption \eqref{boot3}. Hence, by applying Corollary \ref{Corgoodnull} and using the bootstrap assumption \eqref{boot1}, we get 
\begin{align*}
\left(|\alpha (\mathcal{L}_{Z^{\gamma}} F)|+|\rho (\mathcal{L}_{Z^{\gamma}} F)|+|\sigma (\mathcal{L}_{Z^{\gamma}} F)| \right)(t,x) &\lesssim \Lambda\log(3+t)(1+t+|x|)^{-2}, \\
|\underline{\alpha}(\mathcal{L}_{Z^{\gamma}} F)|(t,x)\lesssim \left| \mathcal{L}_{Z^{\gamma}}(F) \right|(t,x) & \lesssim \Lambda (1+t+|x|)^{-1}(1+|t-|x||)^{-1}.
\end{align*}
Now, remark that for any $0 \leq \mu \leq 3$ and $Z \in \mathbb{K}$, we have $[Z,\partial_{x^{\mu}}]=0$ or $[Z,\partial_{x^{\mu}}]=\pm \partial_{x^{\lambda}}$ for a certain $0 \leq \lambda \leq 3$. As a consequence, and since $\mathcal{L}_{\partial_{x^{\mu}}}=\nabla_{\partial_{x^\mu}}$, there exists constants $D_{\kappa, \xi}^{\gamma} \in \mathbb{N}$ such that
\begin{equation}\label{eq:improtranla}
 \mathcal{L}_{Z^{\gamma}}(F) = \sum_{|\kappa| = \gamma_T}\sum_{|\xi| \leq |\gamma|-\gamma_T} D_{\kappa, \xi}^{\gamma}\mathcal{L}_{\partial_{t,x}^{\kappa}Z^{\xi}}(F)= \sum_{|\kappa| = \gamma_T}\sum_{|\xi| \leq |\gamma|-\gamma_T} D_{\kappa, \xi}^{\gamma}\nabla_{t,x}^{\kappa}\mathcal{L}_{Z^{\xi}}(F).
 \end{equation}
The result then follows from $\gamma_T$ applications of Lemma \ref{improderiv}.
\end{proof}

In contrast, we have a very bad control of the top order derivatives near the light cone.
\begin{Pro}\label{Prodecaytoporder}
For any $|\gamma| = N$, there holds
$$ \forall \, (t,x) \in [0,T[ \times \R^3, \qquad \left|\mathcal{L}_{Z^{\gamma}} F \right|(t,x) \lesssim  \Lambda \frac{\log(3+|t-|x||)}{(1+|t-|x||)^{2+\gamma_T}}.$$
If $|\gamma| \leq N-1$, we have the better estimate
$$ \forall \, (t,x) \in [0,T[ \times \R^3, \qquad \left|\mathcal{L}_{Z^{\gamma}} F \right|(t,x) \lesssim  \Lambda (1+t+|x|)^{-1} (1+|t-|x||)^{-1-\gamma_T}.$$
\end{Pro}
\begin{proof}
Let $|\gamma|=N$, $(t,x) \in [0,T[ \times \R^3$ and note that $|\mathcal{L}_ZG| \lesssim (1+t+r) |\nabla_{t,x} G|+|G|$ for any $Z \in \mathbb{K}$ and any $2$-form $G$. Consequently, we obtain from the bootstrap assumptions \eqref{boot1}-\eqref{boot2} that
$$ \left|\mathcal{L}_{Z^{\gamma}}F\right|(t,x) \lesssim (1+t+|x|) \frac{\Lambda \log(3+|t-|x||)}{(1+t+|x|)(1+|t-|x||)^2}+ \frac{\Lambda}{(1+t+|x|)(1+|t-|x||)} \lesssim \Lambda \frac{\log(3+|t-|x||)}{(1+|t-|x||)^2}.$$
As previously, when $\gamma_T \geq 1$, the extra decay in $t-r$ is given by \eqref{eq:improtranla} and Lemma \ref{improderiv}. The case $|\gamma| \leq N-1$ is easier and follows from \eqref{boot1}, \eqref{eq:improtranla} and Lemma \ref{improderiv}. 
\end{proof}

\subsection{Structure of the proof} The remainder of the paper is divided as follows.
\begin{enumerate}
\item First, in Section \ref{Sec4}, we prove that for any $|\beta| \leq N$, an $L^{\infty}_{x,v}$ norm of $\widehat{Z}^{\beta}f$, weighted by powers of $v^0$ and $\mathbf{z}$, grows at most logarithmically in time. Next, we control uniformly in time weighted space averages of $\widehat{Z}^{\beta}f$, for any $|\beta| \leq N-1$. This will allow us to prove, in Section \ref{subsecaveragev}, decay estimates for $\int_v \widehat{Z}^{\beta}f \dr v$ and improve \eqref{boot3}.
\item Then, we introduce the Glassey-Strauss decomposition of the electromagnetic field in Section \ref{subsecGS}. It allows us to improve the bootstrap assumptions \eqref{boot1} and \eqref{boot2}, respectively in Sections \ref{subsecba1} and \ref{subsecba2}, thus implying the global existence of the solution $(f,F)$.
\item Finally, refining the estimates carried out during the previous steps, we prove our modified scattering result for the distribution function in Section \ref{Sec6}. The scattering result for the electromagnetic field is treated in Section \ref{Sec7} and will require an additional step, the construction of a scattering map for the vacuum Maxwell equations.
\end{enumerate}
\begin{Rq}\label{Rqlossvz}
If one is interested in relaxing the assumptions on $N_v$ and $N_x$, though it would force us to either modify the proof or obtain weaker rate of convergences, we give here the precise results where losses in $v^0$ and $\mathbf{z}$ occur.
\begin{itemize}
\item Two powers of $\mathbf{z}$ are lost in order to close the $L^\infty_{x,v}$ estimates in Proposition \ref{estiLinfini}. $5+\delta$ powers of $\mathbf{z}$ are required in order to apply Lemma \ref{Lempartieltxave} and prove boundedness for $\int_x f \dr x$ and its derivatives.
\item Three powers of $v^0$ are lost for closing the $L^\infty_{x,v}$ estimates, eight for the pointwise decay estimates (see Lemma \ref{Lemgmoy} and Proposition \ref{Proestimoyvkernet}). Finally, the Glassey-Strauss decomposition of the derivative of the Maxwell field requires to lose four powers of $v^0$, as suggested by Proposition \ref{GSdecomoderiv} and Corollary \ref{estikernels}.
\end{itemize} 
Note that the various applications of Proposition \ref{estimoyv} will not require to control as much moments of $f$ than the results mentioned here.
\end{Rq}

\section{Estimates for the distribution function}\label{Sec4}

\subsection{Control of the Lorentz force} In view of the structure of the error terms for the commuted Vlasov equations, given by Proposition \ref{Com}, it is important to obtain precise estimates of the Lorentz force and its derivatives by exploiting its null structure. 
\begin{Lem}\label{Lorentzforce}
Let $|\gamma| \leq N-2$ and $j \in \llbracket 1, 3 \rrbracket$. For all $(t,x,v) \in [0,T[ \times \R^3_x \times \R^3_v$, we have
$$
\frac{1}{v^0}\left| \widehat{v}^{\mu} {\mathcal{L}_{Z^{\gamma}}(F)_{\mu}}^j \right|(t,x)  \lesssim  \frac{\Lambda \, \log(3+t)}{(1+t+|x|)^2} +\frac{\Lambda \, \widehat{v}^{\underline{L}}}{(1+t+|x|)(1+|t-|x||)}.
$$
If $\gamma_T \geq 1$, then we have the improved estimate
$$
\frac{1}{v^0}\left| \widehat{v}^{\mu} {\mathcal{L}_{Z^{\gamma}}(F)_{\mu}}^j \right| (t,x)  \lesssim \frac{\Lambda}{(1+t+|x|)^{\frac{5}{2}}} +\frac{\Lambda \, \widehat{v}^{\underline{L}}}{(1+t+|x|)(1+|t-|x||)^{2}}.
$$
\end{Lem}
\begin{proof}
Recall the definition of the null components of a $2$-form \eqref{defnullcompoF} and expand $\widehat{v}^{\mu} {F_{\mu}}^j$ according to the null frame $(\underline{L},L,e_\theta,e_{\varphi})$ in order to get
\begin{align}
\nonumber \left| \widehat{v}^{\mu} {F_{\mu}}^j \right| & = \left| \widehat{v}^{L} {F_{L}}^j+\widehat{v}^{\underline{L}} {F_{\underline{L}}}^j+\widehat{v}^{e_\theta} {F_{e_\theta}}^j+\widehat{v}^{e_\varphi} {F_{e_\varphi}}^j \right| \\ 
& \lesssim \widehat{v}^L(|\alpha(F)|+|\rho(F)|)+\widehat{v}^{\underline{L}}(|\rho (F)|+ |\underline{\alpha} (F)|)+| \widehat{\slashed{v}}|(|\sigma (F)|+|\alpha (F)|+|\underline{\alpha} (F)|).  \label{nullexpand} 
\end{align}
Since $\widehat{v}^L, \, \widehat{v}^{\underline{L}}, \, |\widehat{\slashed{v}}| \leq 1$ and $|\widehat{\slashed{v}}|+|v^0|^{-1} \leq 2 \sqrt{\widehat{v}^{\underline{L}}}$ by Lemma \ref{nullcompo}, we obtain
\begin{equation}\label{eq:compu1bis}
  \frac{1}{v^0}|\widehat{v}^{\mu} {F_{\mu}}^j| \lesssim \sqrt{\widehat{v}^{\underline{L}}} \Big( |\alpha(F)|+|\rho(F)|+|\sigma(F)| \Big)+ \widehat{v}^{\underline{L}} |\underline{\alpha} (F)|.
  \end{equation}
Remark that the same applies to $\mathcal{L}_{Z^{\gamma}}(F)$, $|\gamma| \leq N-2$, so that the first estimate ensues from Proposition \ref{Proproofnullcompo}. Assume now that $\gamma_T \geq1$ and apply once again \eqref{eq:compu1bis} to $\mathcal{L}_{Z^{\gamma}}F$ together with Proposition \ref{Proproofnullcompo}. We obtain
\begin{align*}
\frac{1}{v^0}\left| \widehat{v}^{\mu} {\mathcal{L}_{Z^{\gamma}}(F)_{\mu}}^j \right|(t,x) & \lesssim \frac{\Lambda \, \log(3+t)\sqrt{\widehat{v}^{\underline{L}}}}{(1+t+|x|)^2(1+|t-|x||)} +\frac{\Lambda \, \widehat{v}^{\underline{L}}}{(1+t+|x|)(1+|t-|x||)^2} \\
& \lesssim \Lambda\frac{\log^2(3+t)}{(1+t+|x|)^3}+\frac{\Lambda \, \widehat{v}^{\underline{L}}}{(1+t+|x|)(1+|t-|x||)^2},
\end{align*}
which implies the result.
\end{proof}
If $N-1 \leq |\gamma|\leq N$, we do not have improved estimates on the null components of the electromagnetic field. Moreover, if $|\gamma|=N$ and $\gamma_T=0$, we have a very bad control of $\mathcal{L}_{Z^{\gamma}}F$ near the light cone. The idea then is to transform decay in $t-r$ into decay in $t+r$ at the cost of powers of $\mathbf{z}$ and $v^0$.
\begin{Lem}\label{Lorentzforce2}
Consider $|\gamma| \leq N$ and $j \in \llbracket 1, 3 \rrbracket$. Then, for all $(t,x,v) \in \R_+ \times \R^3_x \times \R^3_v$,
$$\frac{1}{v^0}\left| \widehat{v}^{\mu} {\mathcal{L}_{Z^{\gamma}}(F)_{\mu}}^j \right|(t,x) \lesssim \frac{1}{v^0}\left| \mathcal{L}_{Z^{\gamma}} F \right|(t,x)  \lesssim \Lambda\frac{\log(3+t+|x|)}{(1+t+|x|)^2}|v^0|^3 \mathbf{z}^2(t,x,v).$$
and, if $\gamma_T \geq 1$,
$$\frac{1}{v^0}\left| \widehat{v}^{\mu} {\mathcal{L}_{Z^{\gamma}}(F)_{\mu}}^j \right|(t,x) \lesssim \frac{1}{v^0}\left| \mathcal{L}_{Z^{\gamma}} F \right|(t,x)  \lesssim \Lambda\frac{\log(3+t+|x|)}{(1+t+|x|)^3}|v^0|^3 \mathbf{z}^2(t,x,v).$$
\end{Lem}
\begin{proof}
Recall from Lemma \ref{gainv} that $(1+t+r)^2 \lesssim (1+|t-r|)^2|v^0|^4\mathbf{z}^2$. The first estimate then follows from Proposition \ref{Prodecaytoporder} and the second one from \eqref{boot2} together with \eqref{eq:improtranla}. 
\end{proof}
\begin{Rq}\label{Rqbetteresti}
If $|\gamma| \leq N-1$, we have $| \mathcal{L}_{Z^{\gamma}} F |(t,x)  \lesssim \Lambda (1+t+|x|)^{-2} |v^0|^2 \mathbf{z}(t,x,v)$. If $|\gamma| \leq N-2$, by combining Lemmas \ref{nullcompo} and \ref{Lorentzforce}, we could even save a power of $|v^0|^3 \mathbf{z}$ in the first estimate of the Lorentz force and then avoid any loss in $v$.
\end{Rq}

\subsection{Pointwise bounds for $f$ and its derivatives} As explained in Section \ref{SubsecdiffVlasov}, the main difficulties here are related to the weak decay rate of the electromagnetic field. We deal with them by exploiting several hierarchies in the commuted equations and by taking advantage of the null structure of the Lorentz force. Our approach, based on the method of characteristics, will require various applications of the following result.

\begin{Lem}\label{techLemTF}
Let $g: [0,T[ \times \R^3_x \times \R^3_v \rightarrow \R_+$ and $h: [0,T[ \times \R^3_x \times \R^3_v \rightarrow \R_+$ be two nonnegative sufficiently regular functions such that, for all $(t,x,v) \in [0,T[ \times \R^3_x \times \R^3_v$,
$$ |\T_F(g)|(t,x,v) \leq \frac{C_g}{(1+t)\log^2(3+t)}g+ \frac{C_g \, \widehat{v}^{\underline{L}}}{(1+|t-|x||)\log^2(3+|t-|x||)}g+\frac{1}{(1+t)\log^2(3+t)}h,$$
for some constant $C_g >0$. Then,
 $$\forall \, (t,x,v) \in [0,T[ \times \R^3_x \times \R^3_v, \qquad \quad |g|(t,x,v) \leq \Big( \|g(0,\cdot,\cdot)\|_{L^{\infty}_{x,v}}+ 3 \| h \|_{L^\infty_{t,x,v}} \Big) e^{6C_g}.$$
\end{Lem}
\begin{proof}
Fix, for all this proof, $(x,v) \in \R^3_x \times \R^3_v$ and denote by $t \mapsto (X_t,V_t)$ the characteristic of the operator $\T_F=\partial_t+\widehat{v}^i\partial_{x^i}+\widehat{v}^{\mu} {F_{\mu}}^j \partial_{v^j}$ satisfying
$$\forall \; 1 \leq j \leq 3, \qquad  \dot{X}^{j}_t = \widehat{V}^{j}_t, \quad \dot{V}^{j}_t = \widehat{V}^{\mu}_t{F_{\mu}}^j(t,X_t), \qquad X_0=x, \quad V_0=v.$$
According to Duhamel formula, we have
$$ \forall \, t \in [0,T[ , \qquad  g\big(t,X_t,V_t\big)=g(0,x,v)+\int_{s=0}^t \T_F(g)\big(s,X_s,V_s \big) \dr s.$$
We are then lead to introduce the two functions
$$\psi_1(s):= (1+s)^{-1} \log^{-2}(3+s), \qquad \psi_2(s):= \widehat{v}^{\underline{L}}\big(X_s \big) \big(1+\big|s-|X_s| \, \big|\big)^{-1}\log^{-2}\big(3+\big|s-|X_s| \, \big|\big).$$
In view of the expression of $\T_F(g)$, we have for all $t \in [0,T[$, 
$$g\big(t,X_t,V_t \big) \leq \|g(0,\cdot,\cdot)\|_{L^{\infty}_{x,v}}+  \| h \|_{L^\infty_{t,x,v}} \int_{s=0}^{t} \psi_1(s) \dr s+\int_{s=0}^t C_g\big( \psi_1(s)+\psi_2(s)\big) g \big(s,X_s,V_s \big) \dr s .$$
Consequently, Grönwall's inequality and
$$ \int_{s=0}^{+\infty} \psi_1(s) \dr s = \int_{s=0}^{+\infty} \frac{\dr s}{(1+s) \log^2(3+s)} \leq \int_{s=0}^{+\infty} \frac{3\, \dr s}{(3+s) \log^2(3+s)} \leq \frac{3 }{\log(3)} \leq 3$$
yield
$$ \forall \, t \in [0,T[, \qquad \sup_{0 \leq s \leq t}g\big(s,X_s,V_s \big) \leq \Big( \|g(0,\cdot,\cdot) \|_{L^{\infty}_{x,v}}+ 3 \| h \|_{L^\infty_{t,x,v}} \Big) \exp \left( 3C_g +C_g\int_{s=0}^t \psi_2(s) \dr s \right).$$
It remains us to estimate the integral of $\psi_2$. For this, we will perform a change of variables reflecting that the Vlasov operator reads, in the coordinate system $(u,x,v)$, where $u=t-|x|$,
$$ \T_F = \partial_u-\widehat{v}^i \frac{x_i}{|x|}\partial_u+\widehat{v}^i \partial_{x^i}+\widehat{v}^{\mu} {F_{\mu}}^j \partial_{v^j}=2\widehat{v}^{\underline{L}} \partial_u+\widehat{v}^i \partial_{x^i}+\widehat{v}^{\mu} {F_{\mu}}^j \partial_{v^j} .$$
As $\widehat{v}^{\underline{L}} >0$ by Lemma \ref{nullcompo}, we can then parameterize $t \mapsto (X_t,V_t)$ by the variable $u$. Hence, we perform the change of variables $\widetilde{u}(s)=s-|X_s|$, so that $\widetilde{u}'(s)=2\widehat{V}^{\underline{L}}(X_s)>0$ and
$$ \int_{s=0}^t \psi_2(s) \dr s = \int_{u=t-|x|}^{\widetilde{u}(t)} \frac{\dr u}{2(1+|u|) \log^2(3+|u|)} \leq \int_{u \in \R} \frac{\dr u}{2(1+|u|) \log^2(3+|u|)} \leq 3.$$

\end{proof}

We are now able to prove that quantities such as $\mathbf{z}\widehat{Z}^{\beta}f$ are almost uniformly bounded in phase space. We recall that for a multi-index $\beta$, the number of homogeneous vector fields (respectively translations) composing $\widehat{Z}^\beta$ is denoted by $\beta_H$ (respectively $\beta_T$).
\begin{Pro}\label{estiLinfini}
There exists $D>0$, depending only on $(N,N_v,N_x)$, such that the following estimates hold. For all $(t,x,v) \in [0,T[ \times \R^3_x \times \R^3_v$,
\begin{flalign}
&\forall \, 0 \leq q \leq N_x, \; \; |\beta|\leq N-2, &  |v^0|^{N_v}|\mathbf{z}^q \, \widehat{Z}^{\beta} f|(t,x,v) & \lesssim \epsilon \, e^{D \Lambda} \log^{3q+3\beta_H}(3+t) , &\label{eq:bootLinf} \\
&\forall \, 0 \leq q \leq N_x-2, \; \; |\beta|\leq N, &  |v^0|^{N_v-3}|\mathbf{z}^q \, \widehat{Z}^{\beta} f|(t,x,v) & \lesssim  \epsilon \, e^{D \Lambda} \log^{3q+3\beta_H}(3+t). &\label{eq:bootLinf2}
\end{flalign}
Throughout this paper, it will be convenient to work with $\overline{\epsilon} := \epsilon \, e^{(D+1)\Lambda}$.
\end{Pro}
\begin{proof}
For simplicity, we assume here that $N \geq 4$ and we sketch the proof of the case $N=3$ in Remark \ref{Rqlessderiv} below. Note further that, by interpolation, it suffices to deal with the cases $q \in \{0,N_x\}$ for \eqref{eq:bootLinf} and $q \in \{ 0, N_x-2\}$ for \eqref{eq:bootLinf2}. Motivated by the analysis of the toy model carried out in Section \ref{SubsecdiffVlasov}, we introduce the following hierarchised norms in order to deal with non-integrable error terms and still obtain satisfying estimates if the electromagnetic field is large. Consider, for $(N_0,p,q)=(N-2,N_v,N_x)$ or $(N,N_v-3,N_x-2)$,
$$ \mathbb{E}_{N_0}^{p,q}[f](t,x,v) := \sum_{|\beta| \leq N_0} \frac{|v^0|^{p}|  \widehat{Z}^{\beta} f|(t,x,v)}{\log^{3\beta_H}(3+t)}+\frac{|v^0|^{p}|\mathbf{z}^q \, \widehat{Z}^{\beta} f|(t,x,v)}{\log^{3q+3\beta_H}(3+t)}$$
and let us prove that, for all $(t,x,v) \in [0,T[ \times \R^3_x \times \R^3_v$,
\begin{align}\label{eq:aprouverbound}
\T_F \Big(\mathbb{E}_{N-2}^{N_v,N_x}[f] \Big)(t,x,v) & \lesssim \frac{\Lambda \, \mathbb{E}[f]_{N-2}^{N_v,N_x}(t,x,v)}{(1+t)\log^2(3+t)} + \frac{\Lambda \, \widehat{v}^{\underline{L}}(x) \, \mathbb{E}[f]_{N-2}^{N_v,N_x}(t,x,v)}{(1+|t-|x||)\log^2(3+|t-|x||)}   , \\[4pt]
   \T_F \Big(\mathbb{E}_{N}^{N_v-3,N_x-2}[f] \Big)(t,x,v) & \lesssim \frac{\Lambda \,\mathbb{E}_{N}^{N_v-3,N_x-2}[f](t,x,v)}{(1+t)\log^2(3+t)} + \frac{\Lambda \, \widehat{v}^{\underline{L}}(x) \, \mathbb{E}_{N}^{N_v-3,N_x-2}[f]}{(1+|t-|x||)\log^2(3+|t-|x||)}   \nonumber \\[2pt]
   & \quad + \frac{\Lambda \,\mathbb{E}_{N-2}^{N_v,N_x}[f](t,x,v)}{(1+t)\log^2(3+t)} \label{eq:aprouverbound2}
\end{align}
We are able to apply $\T_F$ to these energy norms since $\T_F(|h|)=\T_F(h) \frac{h}{|h|}$ almost everywhere for any $h \in W^{1,1}_{\mathrm{loc}}$. The result would then follow from two applications of Lemma \ref{techLemTF}. Fix now $(t,x,v) \in [0,T[\times \R^3_x \times \R^3_v$ as well as either $|\beta| \leq N-2$, $p=N_v$ and $a \in \{0,N_x \}$ or $|\beta| \leq N$, $p=N_v-3$ and $a \in \{ 0, N_x-2 \}$. Remark then, since $\T_F(\log^{-1}(3+t)) <0$, that
\begin{align}
\nonumber \T_F\left( \frac{|v^0|^p \, \mathbf{z}^a \, |\widehat{Z}^{\beta} f|}{\log^{3a+3\beta_H}(3+t)} \right) &\leq p\, \T_F \big(v^0 \big) \frac{ |v^0|^{p-1}\, \mathbf{z}^a \, |\widehat{Z}^{\beta} f |}{\log^{3a+3\beta_H}(3+t)} +a\,\T_F(\mathbf{z})\frac{|v^0|^p\,\mathbf{z}^{a-1}\, |\widehat{Z}^{\beta} f|}{\log^{3a+3\beta_H}(3+t)} \\
& \quad +\T_F\big(\widehat{Z}^{\beta} f\big)\frac{\widehat{Z}^{\beta} f}{|\widehat{Z}^\beta f|} \frac{|v^0|^p\, \mathbf{z}^a}{{\log^{3a+3\beta_H}(3+t)}} . \label{TFexpression}
 \end{align}
It is important to note that the second term on the right hand side vanishes if $a=0$. We start by dealing with the first two terms on the right hand side since the last one requires a more thorough analysis. As $|\nabla_v v^0| \leq 1$, we obtain, by applying Lemma \ref{Lorentzforce},
\begin{equation}\label{boundTfv0}
 \frac{1}{v^0}\big| \T_F(v^0) \big|(t,x,v) =  \frac{1}{v^0}\left|\widehat{v}^{\mu}{F_{\mu}}^j \partial_{v^j}(v^0) \right|(t,x) \lesssim \frac{\Lambda\log(3+t)}{(1+t+|x|)^2} +\frac{\Lambda \, \widehat{v}^{\underline{L}}}{(1+t+|x|)(1+|t-|x||)} ,
\end{equation}
so that
\begin{equation}
\big|\T_F \big(v^0 \big) \big| \frac{ |v^0|^{p-1}\, |\mathbf{z}^a \widehat{Z}^{\beta} f |(t,x,v)}{\log^{3a+3\beta_H}(3+t)}  \lesssim \left( \frac{\Lambda}{(1+t)^\frac{3}{2}} +\frac{\Lambda \, \widehat{v}^{\underline{L}}}{(1+t)(1+|t-|x||)} \right)\frac{ |v^0|^{p}\, |\mathbf{z}^a \widehat{Z}^{\beta} f |(t,x,v)}{\log^{3a+3\beta_H}(3+t)}.  \color{white} \square \square \color{black} \label{eq:boundv0z}
 \end{equation}
Next, recall from \eqref{defz} the identity $\widehat{v}^{\mu}\partial_{x^\mu}(\mathbf{z})=0$ and remark that $|\nabla_v \mathbf{z}| \lesssim \frac{t+r}{v^0}$. We get, using Lemma \ref{Lorentzforce},
$$ \left| \T_F(\mathbf{z})\right|(t,x,v) \lesssim  \sum_{1 \leq j \leq 3} \frac{t+|x|}{v^0}\left|\widehat{v}^{\mu}{F_{\mu}}^j (t,x) \right| \lesssim \frac{\Lambda\log(3+t)}{1+t+|x|} +\frac{\Lambda \, \widehat{v}^{\underline{L}}}{1+|t-|x||}.$$
Using Young inequality for products, we obtain, if $a \neq 0$,
$$ \frac{\mathbf{z}^{a-1}}{\log^{3a}(3+t)} \leq  \frac{a-1}{a\log^{3}(3+t)} \frac{\mathbf{z}^a}{\log^{3a}(3+t)}+\frac{1}{a \log^{3}(3+t)} .$$
We then deduce that
\begin{align}
&a|\T_F(\mathbf{z})|\frac{|v^0|^p\,\mathbf{z}^{a-1}\, |\widehat{Z}^{\beta} f|}{\log^{3a+3\beta_H}(3+t)} \nonumber \\
& \qquad \quad  \lesssim  \left(\frac{\Lambda}{(1+t)\log^2(3+t)} +\frac{\Lambda \, \widehat{v}^{\underline{L}}}{(1+|t-|x||) \log^3(3+t)} \right) \! \left(  \frac{|v^0|^p \, |\widehat{Z}^{\beta} f|}{\log^{3\beta_H}(3+t)}+ \frac{|v^0|^p \, \mathbf{z}^a \, |\widehat{Z}^{\beta} f|}{\log^{3a+3\beta_H}(3+t)} \right). \label{eq:boundv0z0}
\end{align}
We now focus on the last term in \eqref{TFexpression}. The first step consists in applying the commutation formula of Proposition \ref{Com} and to remark that $v^0\partial_{v^i}=\widehat{\Omega}_{0i}-t\partial_{x^i}-x^i\partial_t$. We can then bound 
$$\big| \T_F \big( \widehat{Z}^\beta f \big)\big| |v^0|^p\, \mathbf{z}^a \log^{-3a-3\beta_H}(3+t)$$ by a linear combination of terms of the following form. The good ones, which are strongly decaying and can then be easily handled,
\begin{equation}\label{eq:ComTF1}
\mathcal{G}^{p,a}_{\gamma,\kappa}:=   \frac{1}{v^0}\left|\widehat{v}^{\mu} {\mathcal{L}_{Z^{\gamma}}(F)_{\mu}}^j \right|\frac{|v^0|^p \, \mathbf{z}^a\big| \widehat{\Omega}_{0j} \widehat{Z}^{\kappa} f \big|}{\log^{3a+3\beta_H}(3+t)} , \quad \qquad |\gamma|+|\kappa| \leq |\beta| , \qquad |\kappa| \leq |\beta|-1,
\end{equation}
and the bad ones,
\begin{equation}\label{eq:ComTF2}
 \mathcal{B}^{p,a}_{\gamma,\kappa}:=   (t+r)\sup_{1 \leq j \leq 3}\frac{1}{v^0}\left|\widehat{v}^{\mu} {\mathcal{L}_{Z^{\gamma}}(F)_{\mu}}^j \right| \frac{|v^0|^p \, \mathbf{z}^a\big| \partial_{t,x} \widehat{Z}^{\kappa} f \big|}{\log^{3a+3\beta_H}(3+t)} ,  \qquad   \left\{
    \begin{array}{ll}
       \gamma_H+\kappa_H \leq \beta_H , \\
        \kappa_H = \beta_H  \Rightarrow \gamma_T \geq 1,
    \end{array}
\right.
\end{equation}
where, again, $|\gamma|+|\kappa| \leq |\beta|$ and $|\kappa| \leq |\beta|-1$. We emphasize that $\widehat{Z}^{\xi}:=\partial_{t,x}\widehat{Z}^{\kappa}$ is composed by the same number of homogeneous vector fields than $\widehat{Z}^{\kappa}$, so that $\xi_H=\kappa_H$. In contrast, $\widehat{Z}^{\zeta}:=\widehat{\Omega}_{0j} \widehat{Z}^{\kappa}$ verifies $\zeta_H=\kappa_H+1$. Moreover, $\widehat{\Omega}_{0j} \widehat{Z}^{\kappa}$ and $\partial_{t,x}\widehat{Z}^{\kappa}$ are of order at most $|\beta|$.

Consider first the case $|\beta| \leq N-2$, so that $p=N_v$ and $a \in \{0, N_x \}$, and fix two multi-indices $|\gamma| \leq |\beta|$, $|\kappa| \leq |\beta|-1$. Then, according to Lemma \ref{Lorentzforce}, we have
\begin{align}
\nonumber \mathcal{G}^{N_v,a}_{\gamma,\kappa} &\lesssim \Lambda\left( \frac{\log(3+t)}{(1+t+|x|)^2} +\frac{\widehat{v}^{\underline{L}}}{(1+t+|x|)(1+|t-|x|)} \right) \frac{|v^0|^{N_v} \, \big| \mathbf{z}^a \widehat{\Omega}_{0j} \widehat{Z}^{\kappa} f \big|(t,x,v)}{\log^{3a+3\beta_H}(3+t)} \\
& \lesssim \left( \frac{\Lambda}{(1+t)^{\frac{3}{2}}} +\frac{\Lambda \, \widehat{v}^{\underline{L}}}{(1+t)^{\frac{1}{2}}(1+|t-|x|)} \right) \frac{|v^0|^{N_v} \, \big| \mathbf{z}^a \widehat{\Omega}_{0j} \widehat{Z}^{\kappa} f \big|(t,x,v)}{\log^{3a+3(\kappa_H+1)}(3+t)} . \label{eq:boundgood1}
\end{align}
We now focus on $\mathcal{B}^{N_v,a}_{\gamma,\kappa}$ and we start by treating the case $\kappa_H = \beta_H$ and $\gamma_T \geq 1$. Applying once again Lemma \ref{Lorentzforce}, we get
\begin{align}
\nonumber \mathcal{B}^{N_v,a}_{\gamma,\kappa} &\lesssim (t+|x|)\left( \frac{\Lambda}{(1+t+|x|)^{\frac{5}{2}}} +\frac{\Lambda \, \widehat{v}^{\underline{L}}}{(1+t+|x|)(1+|t-|x|)^2} \right)\frac{|v^0|^{N_v} \, \big|\mathbf{z}^a \partial_{t,x} \widehat{Z}^{\kappa} f \big|}{\log^{3a+3\beta_H}(3+t)} \\
& \leq \left( \frac{\Lambda}{(1+t)^{\frac{3}{2}}} +\frac{\Lambda \, \widehat{v}^{\underline{L}}}{(1+|t-|x|)^2} \right) \frac{|v^0|^{N_v} \, \big| \mathbf{z}^a \partial_{t,x} \widehat{Z}^{\kappa} f \big|}{\log^{3a+3\kappa_H}(3+t)}. \label{eq:boundbad2}
\end{align}
Otherwise $\kappa_H \leq \beta_H-1$, so necessarily $\beta_H \ge 1$, and 
\begin{align}
\nonumber \mathcal{B}^{N_v,a}_{\gamma,\kappa} &\lesssim (t+|x|)\left( \frac{\Lambda\log(3+t)}{(1+t+|x|)^2} +\frac{\Lambda \, \widehat{v}^{\underline{L}}}{(1+t+|x|)(1+|t-|x|)} \right) \frac{|v^0|^{N_v} \big|\mathbf{z}^a \partial_{t,x} \widehat{Z}^{\kappa} f \big|}{\log^{3a+3\beta_H}(3+t)} \\
& \leq  \left( \frac{\Lambda}{(1+t)\log^2(3+t)} +\frac{\Lambda \, \widehat{v}^{\underline{L}}}{(1+|t-|x|)\log^3(3+t)} \right) \frac{|v^0|^{N_v} \big|\mathbf{z}^a \partial_{t,x} \widehat{Z}^{\kappa} f \big|}{\log^{3a+3\kappa_H}(3+t)}. \label{eq:boundbad3}
\end{align}
We obtain from \eqref{TFexpression}-\eqref{eq:boundv0z0} and \eqref{eq:boundgood1}-\eqref{eq:boundbad3},
 $$ \T_F\left( \frac{|v^0|^{N_v} \, \mathbf{z}^a \, |\widehat{Z}^{\beta} f|}{\log^{3a+3\beta_H}(3+t)} \right) \lesssim \frac{\Lambda \,  \mathbb{E}_{N-2}^{N_v,N_x}[f](t,x,v)}{(1+t)\log^2(3+t)} + \frac{\Lambda \, \widehat{v}^{\underline{L}}(x) \, \mathbb{E}_{N-2}^{N_v,N_x}[f](t,x,v)}{(1+|t-|x||)\log^2(3+t)}+ \frac{\Lambda \, \widehat{v}^{\underline{L}}(x) \,  \mathbb{E}_{N-2}^{N_v,N_x}[f](t,x,v)}{(1+|t-|x||)^2} .$$
As $|t-r| \gtrsim t$ for $r \geq 2t$ and $t \geq |t-r|$ otherwise, we have
\begin{equation}\label{eqkevatal15}
(1+|t-r|)^{-1}\log^{-2}(3+t) \lesssim (1+t)^{-1}\log^{-2}(3+t)+(1+|t-r|)^{-1}\log^{-2}(3+|t-r|)
\end{equation}
and we then deduce that \eqref{eq:aprouverbound} holds. Lemma \ref{techLemTF} then implies \eqref{eq:bootLinf}.

Assume now that $N-1 \leq |\beta|\leq N$, $p=N_v-3$ and $a \in \{ 0, N_x-2 \}$. We fix two multi-indices $\gamma$, $\kappa$ verifying $|\gamma|+|\kappa|\leq |\beta|$, $|\kappa| \leq |\beta|-1$ and we consider two cases.

\textit{Case $1$, $|\gamma|\leq N-2$.} The Lorentz force can still be estimated using Lemma \ref{Lorentzforce}. One can then follow the analysis carried out in \eqref{eq:boundgood1}-\eqref{eqkevatal15} and obtain
\begin{equation}\label{eq:inter1}
\mathcal{G}^{N_v-3,a}_{\gamma,\kappa}, \; \mathcal{B}^{N_v-3,a}_{\gamma,\kappa}  \lesssim \frac{\Lambda \,  \mathbb{E}_{N}^{N_v-3,N_x-2}[f](t,x,v)}{(1+t)\log^2(3+t)} + \frac{\Lambda \, \widehat{v}^{\underline{L}}(x) \, \mathbb{E}_{N}^{N_v-3,N_x-2}[f](t,x,v)}{(1+|t-|x||)\log^2(3+|t-|x||)}.
\end{equation}
where the term $\mathcal{B}^{N_v-3,a}_{\gamma,\kappa}$ is of course merely defined when $\gamma_T$ and $\kappa_H$ satisfy the condition given in \eqref{eq:ComTF2}.

\textit{Case $2$, $N-1 \leq |\gamma|\leq N$.} Then, as $N \geq 4$, we have $|\kappa| \leq 1$ so that we will be able to control the terms \eqref{eq:ComTF1}-\eqref{eq:ComTF2} using \eqref{eq:bootLinf}. In particular, we are allowed to lose two powers of $|v^0|^2 \mathbf{z}$ in the upcoming estimates in order to deal with the weak decay rate of $\mathcal{L}_{Z^\gamma}F$ near the light cone. More precisely, using first Lemma \ref{Lorentzforce2} and then $a +2 \leq N_x$,
\begin{align}
\nonumber \mathcal{G}^{N_v-3,a}_{\gamma,\kappa} &\lesssim \frac{\Lambda \log(3+t+|x|)}{(1+t+|x|)^2} |v^0|^3 \mathbf{z}^2(t,x,v) \frac{|v^0|^{N_v-3}  \big|\mathbf{z}^{a} \widehat{\Omega}_{0j} \widehat{Z}^{\kappa}  f \big|(t,x,v)}{ \log^{3a+3\beta_H}(3+t) } \\
& \lesssim \frac{\Lambda}{(1+t)^{\frac{3}{2}}} \frac{|v^0|^{N_v}  \big|\mathbf{z}^{N_x} \widehat{\Omega}_{0j} \widehat{Z}^{\kappa}  f \big|(t,x,v)}{ \log^{3N_x+3(\kappa_H+1)}(3+t) }. \nonumber
\end{align}
Next, consider the terms \eqref{eq:ComTF2} and assume first that $\gamma_T \geq 1$. In that case, $\mathcal{B}^{N_v-3,a}_{\gamma,\kappa}$ can be easily handled since it is strongly decaying. Indeed, using again Lemma \ref{Lorentzforce2}, we get
\begin{align*}
 \mathcal{B}^{N_v-3,a}_{\gamma,\kappa} &\lesssim  (t+|x|)\frac{\Lambda \log(3+t+|x|)}{(1+t+|x|)^3} |v^0|^3 \mathbf{z}^2(t,x,v) \frac{ |v^0|^{N_v-3} \big|\mathbf{z}^{a} \partial_{t,x} \widehat{Z}^{\kappa}  f \big|(t,x,v)}{\log^{3a+3\beta_H}(3+t)} \\
 & \lesssim  \frac{\Lambda}{(1+t)^{\frac{3}{2}}} \frac{|v^0|^{N_v}  \big|\mathbf{z}^{N_x} \partial_{t,x} \widehat{Z}^{\kappa}  f \big|(t,x,v)}{ \log^{3N_x+3\kappa_H}(3+t) }.
 \end{align*}
Finally, if $\gamma_T=0$, we necessarily have $\gamma_H=|\gamma| \geq N-1 \geq 3$. Since $\beta_H \geq \gamma_H+\kappa_H$, we have $\kappa_H \leq \beta_H-3$, so that $3a+3\beta_H \geq 3(a+2)+3\kappa_H+3$. Thus, Lemma \ref{Lorentzforce2} yields
\begin{align*}
 \mathcal{B}^{N_v-3,a}_{\gamma,\kappa} &\lesssim  (t+|x|)\frac{\Lambda \log(3+t+|x|)}{(1+t+|x|)^2} |v^0|^3 \mathbf{z}^2(t,x,v) \frac{ |v^0|^{N_v-3} \big|\mathbf{z}^{a} \partial_{t,x} \widehat{Z}^{\kappa}  f \big|(t,x,v)}{\log^{3a+3\beta_H}(3+t)} \\
 & \lesssim \frac{\Lambda }{(1+t) \log^2(3+t)}  \frac{ |v^0|^{N_v} \big|\mathbf{z}^{a+2} \partial_{t,x} \widehat{Z}^{\kappa}  f \big|(t,x,v)}{\log^{3(a+2)+3\kappa_H}(3+t)}.
 \end{align*}
 We then deduce that, in this case,
 $$
\mathcal{G}^{N_v-3,a}_{\gamma,\kappa}, \; \mathcal{B}^{N_v-3,a}_{\gamma,\kappa}  \lesssim \frac{\Lambda }{(1+t) \log^2(3+t)}  \mathbb{E}_{N-2}^{N_v,N_x}[f](t,x,v).
$$
The estimate \eqref{eq:aprouverbound2} ensues from \eqref{eq:inter1} and this last inequality. To conclude the proof, it then remains to apply again the previous Lemma \ref{techLemTF}.
\end{proof}
\begin{Rq}\label{Rqlessderiv}
If $N=3$, the proof of Lemma \ref{estiLinfini} requires an additional step. Once the estimate for $\mathbb{E}^{N_v,N_x}_{N-2}[f]$ is proved, we need to control the intermediary norm $\mathbb{E}^{N_v-1,N_x-1}_{N-1}[f]$. For this, compared to the treatment of $\mathbb{E}^{N_v-3,N_x-2}_{N}[f]$ carried out during the proof of Proposition \ref{estiLinfini}, there are two differences.
\begin{itemize}
\item First, we can exploit the much stronger decay estimate satisfied by the derivatives of order $N-1$ of the electromagnetic field than on its top order ones (see Proposition \ref{Prodecaytoporder}). It explains why we can propagate higher moments for the derivatives of order $N-1$ of $f$ than for the top order ones.
\item Moreover, for controlling sufficiently well $\mathcal{B}^{N_v-1,0}_{\gamma,\kappa}$ and $\mathcal{B}^{N_v-1,N_x-1}_{\gamma,\kappa}$ in the case $\beta_H=\kappa_H$, we can prove, through a direct application of Lemma \ref{improderiv0}, that the good null components of $\mathcal{L}_{Z^\gamma}(F)$ still satisfy improved estimates when $|\gamma|=N-1$ and $\gamma_T \geq 1$.
\end{itemize}
Finally, in order to bound uniformly in time $\mathbb{E}^{N_v-3,N_x-2}_{N}[f]$, the analysis of the terms \eqref{eq:ComTF1}-\eqref{eq:ComTF2} is slightly more technical. It is necessary to consider three cases ($|\gamma| \leq N-2$, $|\gamma|=N-1$ as well as $|\gamma|=N$) and to use the estimates on the first two energy norms.
\end{Rq}

\subsection{Uniform boundedness of the spatial averages}

We start by a preparatory result, which will also be useful later in Section \ref{Sec6}. Recall the constant $\overline{\epsilon} := \epsilon \, e^{(D+1)\Lambda}$ introduced in Proposition \ref{estiLinfini}.
\begin{Lem}\label{Lempartieltxave}
For any $|\beta| \leq N-1$, we have
$$
\forall \, (t,v) \in [0,T[ \times \R^3_v, \qquad  |v^0|^{N_v-6}\bigg| \partial_t \int_{\R^3_x}  \widehat{Z}^{\beta} f(t,x,v) \dr x \bigg| \lesssim \overline{\epsilon} \, \frac{\log^{3N_x+3N}(3+t)}{(1+t)^2}.
$$
 \end{Lem}
\begin{proof}
Fix $|\beta| \leq N-1$, $t \in [0,T[$ and $v \in \R^3_v$. Integrating the commutation formula of Proposition \ref{Com} for $\widehat{Z}^{\beta}f$ and performing integration by parts in $x$ give
$$ \partial_t \int_{\R^3_x} \widehat{Z}^{\beta} f(t,x,v) \mathrm{d} x =  -\int_{\R^3_x} \widehat{v}^{\mu} {F_{\mu}}^j \partial_{v^j} \widehat{Z}^{\beta} f(t,x,v) \mathrm{d} x+\sum_{|\gamma|+|\kappa| \leq |\beta|}C^{\beta}_{\gamma , \kappa} \int_{\R^3_x} \widehat{v}^{\mu} {\mathcal{L}_{Z^{\gamma}}(F)_{\mu}}^j \partial_{v^j} \widehat{Z}^{\kappa}f(t,x,v) \mathrm{d} x. $$
Now, we write
$$ v^0\partial_{v^j} = \widehat{\Omega}_{0j} -x^j \partial_t -t\partial_{x^j} = \widehat{\Omega}_{0j} -(x^j-\widehat{v}^jt) \partial_t-v^jS+v^jx^i\partial_{x^i} -t\partial_{x^j}, \qquad  |x^j-\widehat{v}^jt| \leq \mathbf{z},$$
so that, integrating once again by parts,
\begin{align*}
\left| \partial_t \int_{\R^3_x}  \widehat{Z}^{\beta} f(t,x,v) \mathrm{d} x \right| \lesssim  \sum_{\substack{|\gamma|+|\kappa| \leq |\beta|+1 \\ |\gamma| \leq |\beta|}} \; \sup_{1 \leq j \leq 3} \; & \int_{\R^3_x} \frac{1}{v^0}\left|\widehat{v}^{\mu} {\mathcal{L}_{Z^{\gamma}}(F)_{\mu}}^j(t,x) \right| \left| \mathbf{z} \widehat{Z}^{\kappa} f\right|(t,x,v) \mathrm{d} x \\
 & + \int_{\R^3_x} \frac{t+|x|}{v^0} \left|\widehat{v}^{\mu} {\nabla_{t,x}\mathcal{L}_{Z^{\gamma}}(F)_{\mu}}^j(t,x) \right| \left| \widehat{Z}^{\kappa} f\right|(t,x,v) \mathrm{d} x.
 \end{align*}
According to the bootstrap assumptions \eqref{boot1}-\eqref{boot2} and Lemma \ref{gainv}, we have 
$$| {\mathcal{L}_{Z^{\gamma}}(F)_{\mu}}^j(t,x) |\lesssim \Lambda(1+t+|x|)^{-2}|v^0|^2\mathbf{z}, \qquad |\nabla_{t,x} {\mathcal{L}_{Z^{\gamma}}(F)_{\mu}}^j(t,x) |\lesssim \Lambda\log(3+t+|x|)(1+t+|x|)^{-3}|v^0|^4\mathbf{z}^2,$$
 so that, 
\begin{align*}
\left| \partial_t \int_{\R^3_x}  \widehat{Z}^{\beta}f(t,x,v) \mathrm{d} x \right| & \lesssim  \Lambda \sum_{|\kappa| \leq |\beta|+1} \int_{\R^3_x} \frac{\log(3+t+|x|)}{(1+t+|x|)^2}|v^0|^3 \left| \mathbf{z}^2 \widehat{Z}^{\kappa} f \right|(t,x,v)  \mathrm{d} x \\
& \leq \Lambda \sup_{|\kappa| \leq |\beta|+1} \sup_{x \in \R^3} \left(\frac{\log(3+t+|x|)}{(1+t+|x|)^2} |v^0|^3\left| \mathbf{z}^{N_x-2} \widehat{Z}^{\kappa} f \right|(t,x,v) \right)  \int_{\R^3_x}  \frac{\mathrm{d} x}{\mathbf{z}^{N_x-4}(t,x,v)}.
 \end{align*}
Remark then that, in view of \eqref{defz} and $N_x>7$,
\begin{equation*}
\int_{\R^3_x}  \frac{\mathrm{d} x}{\mathbf{z}^{N_x-4}(t,x,v)} \leq \int_{\R^3_x}  \frac{\mathrm{d} x}{(1+|x-\widehat{v}t|)^{N_x-4}}=\int_{y \in \R^3}  \frac{\mathrm{d} y}{(1+|y|)^{N_x-4}} < +\infty.
\end{equation*}
Then, multiply both sides of the inequality by $|v^0|^{N_v-6}$ and bound above the right hand side by applying Proposition \ref{estiLinfini}. It remains to use $\Lambda \epsilon \, e^{D\Lambda} \leq \epsilon \, e^{(D+1)\Lambda} = \overline{\epsilon}$.
\end{proof}
\begin{Rq}
If $|\beta| \le N-3$, by using the estimates of the Lorentz force provided by Lemma \ref{Lorentzforce}, we can even prove $|v^0|^{N_v} |\partial_t \int_x \widehat{Z}^\beta f\dr v| \lesssim \overline{\epsilon} (1+t)^{-2} \log^{-3N_x-3N}(3+t)$.
\end{Rq}
Note now that $|\int_x \widehat{Z}^{\beta} f(0,x,v) \dr x| \leq 2\sup_x |\mathbf{z}^4\widehat{Z}^{\beta} f|(0,x,v)\leq 2\epsilon$. Hence, by integrating in time the inequality of the previous Lemma \ref{Lempartieltxave}, we obtain, for any $|\beta| \leq N-1$,
$$\forall \, (t,v) \in [0,T[ \times \R^3_v , \qquad \quad |v^0|^{N_v-6}\bigg|\int_{\R^3_x}  \widehat{Z}^{\beta} f(t,x,v) \dr x \bigg| \lesssim \epsilon+\overline{\epsilon} \int_{\tau=0}^t \frac{\log^{3N_x+3N}(3+\tau)}{(1+\tau)^2}  \, \dr \tau \lesssim \overline{\epsilon}.$$
It directly implies the following result.
\begin{Cor}\label{Corspatialave}
Let $|\beta| \leq N-1$ and $\psi : \mathbb{S}^2_\omega \times \R^3_v \rightarrow \R$ be a function such that $\|\psi(\cdot,v) \|_{L^{\infty}_{\omega}} \lesssim |v^0|^{N_v-6}$. Then, for any $\omega \in \mathbb{S}^2$,
$$\forall \, (t,v) \in [0,T[ \times \R^3_v, \qquad \qquad \bigg|\psi(\omega,v)\int_{\R^3_x}  \widehat{Z}^{\beta} f(t,x,v) \dr x \bigg| \lesssim \overline{\epsilon} .$$
\end{Cor}
We allowed the function $\psi$ to depend on a parameter $\omega \in \mathbb{S}^2$ in order to prove optimal decay estimates on certain elements of the Glassey-Strauss decomposition of the electromagnetic field, defined as integral kernels.

\subsection{Pointwise decay estimates for velocity averages}\label{subsecaveragev} We prove here that the decay rate of $\int_v \widehat{Z}^{\beta} f \dr v$, for $|\beta| \leq N-1$, coincides with the one of the linear setting. In particular, we improve the bootstrap assumption \eqref{boot3}. The starting point consists in performing the change of variables $y=x-t\widehat{v}$. For this, recall Lemma \ref{cdv} and that $y \mapsto \widecheck{y}$ is the inverse function of $v \mapsto \widehat{v}$.
\begin{Lem}\label{cdvbis}
Let $g:[0,T[ \times \R^3_x \times \R^3_v \rightarrow \R$ be a sufficiently regular function. Then, 
$$\forall \, (t,x) \in [0,T[ \times \R^3, \qquad  t^3\int_{\R^3_v} g(t,x-\widehat{v}t,v) \dr v=\int_{|y-x|<t} (|v^0|^5g)\left(t,y,\widecheck{\frac{x-y}{t}} \right) \dr y.$$
\end{Lem}
This change of variables is motivated by the linear case. Any regular solution to the relativistic transport equation $\partial_t h+\widehat{v} \cdot \nabla_x  h=0$ is constant along the timelike straight lines, $h(t,x+\widehat{v}t,v)=h(0,x,v)$. The previous lemma, applied for $g(t,x,v)=h(0,x,v)$, then leads to $\int_v h(t,x,v)\dr v \lesssim t^{-3}$.

As a first step, we control $\int_v|\widehat{Z}^{\beta} f| \dr v$ for any $|\beta| \leq N$, which has a slightly slower decay rate than in the linear case in the interior of the light cone. These estimates will also be useful on their own.
\begin{Pro}\label{estimoyv}
Let $|\beta| \leq N$ and $0\leq a \leq N_x-6$. Then, the following properties hold.
\begin{itemize}
\item Almost optimal pointwise decay estimate,
$$ \forall \, (t,x) \in [0,T[ \times \R^3_x, \qquad \int_{\R^3_v} |v^0|^{N_v-8} \left|\mathbf{z}^a \widehat{Z}^{\beta} f\right|(t,x,v) \dr v  \lesssim \overline{\epsilon} \, \frac{\log^{3N_x+3N}(3+t)}{(1+t)^3}.$$
\item Improved decay estimates near and in the exterior of the light cone,
\begin{align*}
&\forall \, |x| \leq t <T,& \qquad \int_{\R^3_v} |v^0|^{N_v-8-2a} \left| \widehat{Z}^{\beta} f\right|(t,x,v) \dr v  &\lesssim \overline{\epsilon} \, \log^{3N_x+3N}(3+t) \frac{(1+t-|x|)^a}{(1+t)^{3+a}}, \\
&\forall \, t < \sup( |x|,T),& \qquad \int_{\R^3_v} |v^0|^{N_v-8-2a} \left| \widehat{Z}^{\beta} f\right|(t,x,v) \dr v  &\lesssim \overline{\epsilon} \, \frac{\log^{3N_x+3N}(3+t)}{(1+t+|x|)^{3+a}}.
\end{align*}
\end{itemize}
\end{Pro}
\begin{proof}
Fix $|\beta| \leq N$, $(t,x) \in [0,T[ \times \R^3_x$ and $0 \leq a \leq N_x-6$. If $t \leq 1$, we have by Proposition \ref{estiLinfini},
$$ \int_{\R^3_v} |v^0|^{N_v-7} \left| \mathbf{z}^a \widehat{Z}^{\beta} f \right|(t,x,v) \dr v \lesssim    \sup_{v \in \R^3} |v^0|^{N_v-3}\left|\mathbf{z}^{N_x-6} \widehat{Z}^{\beta} f \right|(t,x,v) \int_{ \R^3_w} \frac{\dr w}{\langle w \rangle^4} \lesssim \overline{\epsilon} .$$
Assume now, unless $T\leq 1$, that $t \geq 1$ and introduce the function $g(t,x,v) :=  |v^0|^{N_v-8}  | \mathbf{z}^{a} \widehat{Z}^{\beta} f |(t,x+t\widehat{v},v)$. Applying the previous Lemma \ref{cdvbis} to $g$, we get
$$ t^3\int_{\R^3_v} |v^0|^{N_v-8} \left|\mathbf{z}^{a} \widehat{Z}^{\beta} f \right|(t,x,v) \dr v \leq \int_{|y-x|<t} \sup_{v \in \R^3}|v^0|^5g(t,y,v) \dr y \leq \sup_{(y,v) \in \R^3 \times \R^3}|v^0|^5\langle y \rangle^4g(t,y,v) \int_{\R^3_y} \frac{\dr y}{\langle y \rangle^4}.$$
Using now Lemma \ref{gweightvderiv} and then Proposition \ref{estiLinfini}, we obtain
$$ t^3\int_{\R^3_v} |v^0|^{N_v-8} \left|\mathbf{z}^{a} \widehat{Z}^{\beta} f \right|(t,x,v) \dr v \leq \sup_{(y,v) \in \R^3 \times \R^3}|v^0|^{N_v-3}\left|\mathbf{z}^{a+4}\widehat{Z}^{\beta} f\right|(t,y,v) \lesssim \overline{\epsilon} \log^{3a+12+3N}(3+t).$$
This concludes the proof of the first estimate, which, together with Lemma \ref{gainv}, imply the second one as well as the last one in the region $t < |x| \leq 2t$. If $|x| \geq 2t$, remark that $\mathbf{z} \gtrsim 1+|x-t\widehat{v}| \gtrsim 1+t+|x|$, so that
$$\int_{\R^3_v} |v^0|^{N_v-7} \left| \widehat{Z}^{\beta} f\right|(t,x,v) \dr v \lesssim (1+t+|x|)^{-N_x+2} \sup_{(y,v) \in \R^3 \times \R^3} |v^0|^{N_v-3}\left|\mathbf{z}^{N_x-2} \widehat{Z}^{\beta} f \right|(t,y,v) \int_{ \R^3_w} \frac{\dr w}{\langle w \rangle^4} .$$
It remains to apply Proposition \ref{estiLinfini}.
\end{proof}
Our goal now is to remove the logarithmic loss of the estimate of $\int_{v} \widehat{Z}^{\beta} f \dr v$ provided by Proposition \ref{estimoyv}. Since our analysis will rely on the following result, we will not be able to deal with top order derivatives. We recall that $N_x-3>4$.
\begin{Lem}\label{Lemgmoy}
Let $g:[0,T[ \times \R^3_x \times \R^3_v \rightarrow \R$ be a sufficiently regular function. Then, for all $|x| <t < T$,
$$ \left|t^3\int_{\R^3_v} g(t,x-\widehat{v}t,v) \dr v-\int_{y \in \R} (|v^0|^5g)\left(t,y,\widecheck{\frac{x}{t}} \right) \dr y \right| \lesssim \frac{1}{t} \sup_{(y,v) \in \R^3 \times \R^3} \langle y \rangle^{N_x-3} \left(|v^0|^7 |g|+|v^0|^8|\nabla_v g|\right)(t,y,v)  .$$
\end{Lem}
\begin{proof}
According to Lemma \ref{cdvbis}, we have
$$ t^3\int_{\R^3_v} g(t,x-\widehat{v}t,v) \dr v-\int_{y \in \R} g\left(t,y,\widecheck{\frac{x}{t}} \right) \dr v = \mathcal{I}_1+\mathcal{I}_2,$$
where
\begin{align*}
\mathcal{I}_1 & := \int_{|x-y|<t} (|v^0|^5g)\left(t,y, \widecheck{\frac{x-y}{t}} \right) \dr y - \int_{|x-y|<t} (|v^0|^5g)\left(t,y, \widecheck{\frac{x}{t}} \right)  \dr y, \\
\mathcal{I}_2 & := - \int_{|x-y| \geq t} (|v^0|^5g)\left(t,y, \widecheck{\frac{x}{t}} \right) \dr y.
\end{align*}
Since, by Lemma \ref{cdv}, we have $|\nabla_y \widecheck{y}| \lesssim \sqrt{1-|y|^2}^{-3}=\langle \widecheck{y} \rangle^3=|v^0|^3(\widecheck{y})$, the mean value theorem gives us
$$ \left| (|v^0|^5g)\left(t,y, \widecheck{\frac{x-y}{t}} \right) -  (|v^0|^5g)\left(t,y, \widecheck{\frac{x}{t}} \right) \right| \lesssim \frac{|y|}{t}\sup_{v \in \R^3} |v^0|^7|g|(t,y,v) +|v^0|^8 |\nabla_v g|(t,y,v) .$$
Consequently, 
\begin{align*}
\left| \mathcal{I}_1 \right| \lesssim \frac{1}{t} \sup_{(y,v) \in \R^3 \times \R^3} \langle y \rangle^{N_x-3}\left(|v^0|^7|g|(t,y,v) +|v^0|^8 |\nabla_v g|(t,y,v) \right) \int_{|x-y|<t} \frac{\dr y}{\langle y \rangle^{N_x-4}}, \qquad N_x-4 >3.
\end{align*}
In order to bound $\mathcal{I}_2$ recall that $|x|<t$ and remark that, for $v=\widecheck{x/t}$ and any $y \in \R$ such that $|y-x| \geq t$,
\begin{align*}
 1 & = |v^0|^2 \left( 1-\frac{|x|^2}{t^2}\right)  \leq |v^0|^2 \frac{|y|(t+|x|)}{t^2} \leq 2 \frac{|y||v^0|^2}{t}.
 \end{align*}
We then finally deduce that
 $$ |\mathcal{I}_2| \leq \frac{2}{t} \int_{|y-x| \geq t} \left(  |v^0|^7 g\right) \left( t,y,\widecheck{\frac{x}{t}} \right) \frac{\langle y \rangle^{N_x-3}}{\langle y \rangle^{N_x-4}} \mathrm{d}y \leq \frac{4}{t} \sup_{(y,v) \in \R^3 \times \R^3} |v^0|^7\langle y \rangle^{N_x-3} |g|(t,y,v).$$
\end{proof}
We are able to prove that the decay of quantities such as $\int_v \widehat{Z}^{\beta}f \dr v$ is optimal. We state a general result since we will later have to deal with integral kernels.
\begin{Pro}\label{Proestimoyvkernet}
Let $|\beta| \leq N-1$ and $\Psi : \mathbb{S}^2_\omega \times \R^3_v \rightarrow \R$ be a sufficiently regular function such that $\|\Psi (\cdot,v) \|_{L^{\infty}_{\omega}}+\|v^0\nabla_v \Psi (\cdot,v) \|_{L^{\infty}_{\omega}} \lesssim |v^0|^{N_v-11}$. Then, for any $\omega \in \mathbb{S}^2$,
$$ \forall \, (t,x) \in [0,T[ \times \R^3_x, \qquad \left|\int_{\R^3_v} \Psi(\omega,v) \widehat{Z}^{\beta} f(t,x,v) \dr v \right| \lesssim \frac{\overline{\epsilon}}{(1+t+|x|)^3}.$$
\end{Pro}
\begin{proof}
Assume first that $|x| \leq t \leq 1$ or $|x| \geq t$. Then, as $|\Psi|(\cdot,v) \lesssim |v^0|^{N_v-9}$, it suffices to use the first or the third estimate of Proposition \ref{estimoyv}, applied for $a=1/2$. Otherwise $t > \max(1,|x|)$ and we introduce, for any $\omega \in \mathbb{S}^2$, $g_\omega(t,x,v)=\Psi(\omega,v)\widehat{Z}^{\beta}f(t,x+t\widehat{v},v)$. Using first Lemma \ref{gweightvderiv} and then Proposition \ref{estiLinfini}, we have
\begin{align}
 \sup_{(y,v) \in \R^3 \times \R^3} &\langle y \rangle^{N_x-3} \left(|v^0|^7 |g_{\omega}|+|v^0|^8|\nabla_v g_{\omega}|\right)(t,y,v)\nonumber \\ & \; \lesssim \sup_{(y,v) \in \R^3 \times  \R^3}|\nabla_v \Psi|(\omega,v)|v^0|^8 \left| \mathbf{z}^{N_x-3} \widehat{Z}^{\beta}f\right|(t,y,v)+\sum_{|\kappa| \leq 1} |\Psi|(\omega,v)|v^0|^7\left| \mathbf{z}^{N_x-2} \widehat{Z}^{\kappa}\widehat{Z}^{\beta}f\right|(t,y,v) \nonumber \\
 & \; \lesssim \sum_{|\xi| \leq N} \, \sup_{(y,v) \in \R^3 \times  \R^3}|v^0|^{N_v-3} \left| \mathbf{z}^{N_x-2} \widehat{Z}^{\xi}f\right|(t,y,v) \lesssim \overline{\epsilon} \, \log^{3N_x+3N}(3+t).\label{eq:estigomega}
 \end{align} 
Now, apply Lemma \ref{Lemgmoy} to $g_\omega$ in order to get
$$\forall \, \omega \in \mathbb{S}^2, \qquad  t^3 \left| \int_{\R^3_v} \Psi(\omega,v) \widehat{Z}^{\beta} f(t,x,v) \dr v \right| \lesssim \left| \int_{\R^3_y} \left( |v^0|^5  g_{\omega}\right)\left(t,y,\widecheck{\frac{x}{t}} \right) \dr y \right|+\overline{\epsilon} \frac{\log^{3N_x+3N}(3+t)}{t},$$
As $t \geq 1$, it remains to bound by $\overline{\epsilon}$ the first term on the right hand side. For this, perform the change of variables $z=y-t\widehat{v}$ and apply Corollary \ref{Corspatialave} with $\psi (\omega,v)=|v^0|^5 \Psi(\omega,v)$.
\end{proof}
The next result is a direct application of the previous Proposition to $\Psi(\omega,v)=v^{\mu}/v^0$, for any $0 \leq \mu \leq 3$.
\begin{Cor}\label{Corboo3}
For any $|\beta| \leq N-1$, the decay of the current density $J(\widehat{Z}^\beta f)$ is optimal. There exists a constant $C>0$ independent of $\epsilon$ such that,
$$ \forall \, (t,x) \in [0,T[ \times \R^3_x, \qquad \left|\int_{\R^3_v} \frac{v^{\mu}}{v^0} \widehat{Z}^{\beta} f(t,x,v) \dr v \right| \leq \frac{C \, \overline{\epsilon}}{(1+t+|x|)^3}, \qquad 0 \leq \mu \leq 3.$$
\end{Cor}
If $\epsilon$ satisfies $C \overline{\epsilon}=C \epsilon \,  e^{(D+1)\Lambda} < C_{\mathrm{boot}} \, \Lambda$, it improves the bootstrap assumption \eqref{boot3}.

\subsection{Improved estimates for derivatives of velocity averages}

In the linear case, derivatives of averages in $v$, such as $\partial_{t,x} \int_v f \dr v$, enjoy stronger decay properties. Our study of the top order derivatives of the electromagnetic field will require the following improved estimates.
\begin{Pro}\label{gainderivdecaymoyv}
Let $|\beta| \leq N-1$, $\mu \in \llbracket 0,3 \rrbracket$ and $\Phi : \mathbb{S}^2 \times \R^3_v \rightarrow \R$ be a sufficiently regular function such that $\|\Phi (\cdot,v) \|_{L^{\infty}_{\omega}}+\|v^0\nabla_v \Phi (\cdot,v) \|_{L^{\infty}_{\omega}} \lesssim |v^0|^{N_v-10}$. Then, for any $\omega \in \mathbb{S}^2$,
$$ \forall \, (t,x) \in [0,T[ \times \R^3_x, \qquad \left|\int_{\R^3_v} \Phi(\omega,v) \partial_{x^\mu}\widehat{Z}^{\beta} f(t,x,v) \dr v \right| \lesssim \overline{\epsilon} \, \frac{\log^{3N_x+3N}(3+t)}{(1+t+|x|)^4}.$$
\end{Pro}
\begin{proof}
Let $(t,x) \in [0,T[ \times \R^3_x$ and remark that, if $|x| \geq t-1$, the result is given by Proposition \ref{estimoyv}, applied for $a=1$. We then consider the case $t-|x| \geq 1$. Using \eqref{eq:nullderiv} together with $t\Omega_{ij}=\frac{x^i}{r}\Omega_{0j}-\frac{x^j}{r}\Omega_{0i}$, one has
$$\mathcal{I}^\beta_{t,x} :=|t-|x|| \left|\int_{\R^3_v} \Phi(\omega,v) \partial_{x^\mu}\widehat{Z}^{\beta} f(t,x,v) \dr v \right| \leq \sum_{Z \in \mathbb{K}} \left|\int_{\R^3_v} \Phi(\omega,v) Z\widehat{Z}^{\beta} f(t,x,v) \dr v \right|.$$
Fix now $Z \in \mathbb{K}$. If $Z$ is a translation $\partial_{x^{\mu}}$ or if $Z=S$, then $Z \in \widehat{\mathbb{P}}_0$. Otherwise, either $Z=\Omega_{ij}$ is a rotation and $Z=\widehat{Z}-v^i\partial_{v^j}+v^j\partial_{v^i}$ or $Z=\Omega_{0k}$ is a Lorentz boost and $Z=\widehat{Z}-v^0\partial_{v^k}$, so that
$$\mathcal{I}^\beta_{t,x} \leq \sum_{\widehat{Z} \in \widehat{\mathbb{P}}_0} \left|\int_{\R^3_v} \Phi(\omega,v) \widehat{Z}\widehat{Z}^{\beta} f(t,x,v) \dr v \right|+\sum_{\lambda =0}^{3} \sum_{k=1}^{ 3}\left|\int_{\R^3_v} \Phi(\omega,v) v^{\lambda}\partial_{v^k}\widehat{Z}^{\beta} f(t,x,v) \dr v \right| .$$
Integration by parts and $|\partial_{v^k}(\Phi(\omega,v) v^{\lambda})| \leq v^0|\nabla_v \Phi|(\omega,v)+|\Phi|(\omega,v) \lesssim |v^0|^{N_v-10}$ yield
$$\left|\int_{\R^3_v} \Phi(\omega,v) \partial_{x^\mu}\widehat{Z}^{\beta} f(t,x,v) \dr v \right| \lesssim \frac{1}{|t-|x||}\sum_{|\kappa| \leq 1} \int_{\R^3_v} |v^0|^{N_v-10} \left| \widehat{Z}^\kappa \widehat{Z}^{\beta} f\right|(t,x,v)  \dr v.$$
As $t-|x| \geq 1$, it remains to apply once again Proposition \ref{estimoyv} for $a=1$.
\end{proof}

\section{Improvement of the bootstrap assumptions on the electromagnetic field}\label{Sec5}
We are now able to prove pointwise decay estimates for the Maxwell field and its derivatives. We improve first \eqref{boot1} whereas the case of the top order derivatives \eqref{boot2} will require a different strategy since we did not recover the linear decay $t^{-3}$ for $\int_v \widehat{Z}^\beta f(t,x,v) \dr v$, $|\beta|=N$. 

\subsection{The Glassey-Strauss decomposition of the electromagnetic field}\label{subsecGS}

We separate $F$ as well as its derivatives $\mathcal{L}_{Z^{\gamma}}(F)$ into three parts according to the Glassey-Strauss decomposition. For this, recall from \eqref{eq:defF} the relation between the electric field $E$, the magnetic field $B$ and the Faraday tensor $F$. We have $E^i=F_{0i}$, $B^1=F_{32}$, $B^2=F_{13}$ and $B^3=F_{21}$. To simplify the statement of the decomposition, we introduce a weight tensor field.
\begin{Def} 
Let $\mathbf{w}_{\mu\nu}(\omega,v)$ be the antisymmetric tensor defined for all $(\omega,v) \in \mathbb{S}^2 \times \R^3_v$ by
\begin{equation*}
\mathbf{w}_{0i}(\omega,v)=-\mathbf{w}_{i0}(\omega,v):= \omega_i+\widehat{v}_i, \qquad \mathbf{w}_{jk}(\omega,v)  := \omega_j\widehat{v}_k-\omega_k\widehat{v}_j, \qquad 1 \leq i,j,k\leq 3,
\end{equation*}
where $\omega_i:=x_i/|x|$ if $x \in \R^3$ satisfies $\omega=x/|x|$. We further define
$$\mathcal{W}_{\mu \nu}(\omega,v):= \frac{\mathbf{w}_{\mu \nu}(\omega,v)}{1+\omega \cdot \widehat{v}}.$$
\end{Def}
\begin{Rq}
Since $\mathbf{w}$ is antisymmetric, $\mathbf{w}_{\mu\mu}=0$ for any $\mu \in \llbracket 0 , 3 \rrbracket$. Note also that $1+\omega \cdot \widehat{v} =2v^L>0$.
\end{Rq}
We can now prove an adaptation of Glassey-Strauss result \cite[Theorem~$3$]{GlStrauss}. The key idea of their proof consists in rewriting the standard derivatives $\partial_{t,x}$ as combinations of derivatives tangential to a backward light cone, which naturally appears in the representation formula for solutions to wave equations, and $\mathbf{T}_0:=\partial_t+\widehat{v}\cdot \nabla_x$, the free relativistic tranport operator which is transverse to light cones. To avoid any confusion with the scaling vector field, we do not keep the notation $S$, used by Glassey-Strauss, in order to denote $\mathbf{T}_0$.
\begin{Pro}\label{GSdecompo}
Let $|\gamma| \leq N-1$. Then, there exists $C^{\gamma}_{\beta}, \, N^{\gamma}_{\xi,\kappa}\in \mathbb{N}$ such that $$4\pi\mathcal{L}_{Z^{\gamma}}(F)=\mathcal{L}_{Z^{\gamma}}(F)^{\mathrm{data}}+\mathcal{L}_{Z^{\gamma}}(F)^T+\mathcal{L}_{Z^{\gamma}}(F)^S,$$ where, for any $0 \leq \mu , \, \nu \leq 3$ and with $\omega=\frac{y-x}{|y-x|}$ in the following integrals,
\begin{itemize}
\item $\mathcal{L}_{Z^{\gamma}}(F)^{\mathrm{data}}_{\mu \nu}$ can be explicitly computed in terms of the initial data. More precisely,
$$ \mathcal{L}_{Z^{\gamma}}(F)^{\mathrm{data}}_{\mu \nu}(t,x) \! = 4\pi \mathcal{L}_{Z^{\gamma}}(F)^{\mathrm{hom}}_{\mu \nu}(t,x) - \! \sum_{|\beta| \leq |\gamma|} \frac{C^{\gamma}_{\beta}}{t}\!\int_{|y-x|=t}\! \int_{\R^3_v}\! \left( \mathcal{W}_{\mu \nu}(\omega,v)\!-\delta_\mu^0 \widehat{v}^\nu\!+\!\delta^0_{\nu}\widehat{v}^\mu\right)\! \widehat{Z}^{\beta}f(0,y,v) \dr v \dr y $$
and $\mathcal{L}_{Z^{\gamma}}(F)^{\mathrm{hom}}_{\mu \nu}$ is the unique solution to the homogeneous wave equation $\Box \mathcal{L}_{Z^{\gamma}}(F)^{\mathrm{hom}}_{\mu \nu} = 0$ which initially verifies $\mathcal{L}_{Z^{\gamma}}(F)^{\mathrm{hom}}_{\mu \nu}(0,\cdot)=\mathcal{L}_{Z^{\gamma}}(F)_{\mu \nu}(0,\cdot)$ and $\partial_t \mathcal{L}_{Z^{\gamma}}(F)^{\mathrm{hom}}_{\mu \nu}(0,\cdot)=\partial_t \mathcal{L}_{Z^{\gamma}}(F)_{\mu \nu}(0,\cdot)$.
\item The $2$-form $\mathcal{L}_{Z^{\gamma}}(F)^T$ is given by 
$$ \mathcal{L}_{Z^{\gamma}}(F)^T_{\mu \nu}(t,x) := -\sum_{|\beta| \leq |\gamma|} C^{\gamma}_{\beta} \, \big[ \widehat{Z}^{\beta} f \big]^T_{\mu \nu}(t,x),$$
where the field $[\widehat{Z}^\beta f ]^T$ generated by $\widehat{Z}^\beta f$ is  
$$\big[ \widehat{Z}^{\beta} f \big]^T_{\mu \nu}(t,x):=\int_{|y-x| \leq t} \int_{\R^3_v} \frac{\mathcal{W}_{\mu \nu}(\omega,v)}{|v^0|^2(1+\omega \cdot \widehat{v})}\widehat{Z}^{\beta} f(t-|y-x|,y,v) \frac{\mathrm{d}v \mathrm{d}y}{|y-x|^2}.$$
\item The $2$-form $\mathcal{L}_{Z^{\gamma}}(F)^S$ is defined by
$$  \mathcal{L}_{Z^{\gamma}}(F)^S_{\mu \nu}(t,x) := \sum_{|\xi|+|\kappa| \leq |\gamma|}N^{\gamma}_{\xi, \kappa} \int_{|y-x| \leq t} \int_{\R^3_v} \big(\widehat{Z}^{\kappa}f \, \widehat{v}^{\lambda}{\mathcal{L}_{Z^{\xi}}(F)_{\lambda}}^j \big)(t-|y-x|,y,v) \partial_{v^j}  \mathcal{W}_{\mu\nu}(\omega, v) \frac{\mathrm{d}v \mathrm{d}y}{|y-x|}.$$
\end{itemize}
\end{Pro}
\begin{proof}
Fix $|\gamma| \leq N-1$ and apply Proposition \ref{Com} in order to rewrite the Maxwell equations satisfied by $\mathcal{L}_{Z^{\gamma}}(F)$ as
\begin{equation}\label{deffgamma}
 \nabla^{\mu} \mathcal{L}_{Z^{\gamma}}(F)_{\mu \nu} = \int_{\R^3_v} \frac{v_{\nu}}{v^0} f_{\gamma}(t,x,v) \dr v, \quad \nabla^{\mu} {}^* \! \mathcal{L}_{Z^{\gamma}}(F)_{\mu \nu}=0, \qquad   \nu \in \llbracket 0 , 3 \rrbracket, \quad f_{\gamma}:=\sum_{|\beta|\leq |\gamma|} C_{\beta}^{\gamma} \, \widehat{Z}^{\beta} f,
 \end{equation}
with $C^{\gamma}_{\beta} \in \mathbb{N}$. Introduce further the electric $E_{\gamma}$ and  magnetic $B_{\gamma}$ parts of $\mathcal{L}_{Z^{\gamma}}(F)$,
\begin{equation}\label{defEB}
E^i_{\gamma}:=\mathcal{L}_{Z^{\gamma}}(F)_{0i}, \quad i \in \llbracket 1,3 \rrbracket, \qquad  B_{\gamma}^1=\mathcal{L}_{Z^{\gamma}}(F)_{32}, \qquad B_{\gamma}^2=\mathcal{L}_{Z^{\gamma}}(F)_{13}, \qquad B_{\gamma}^3=\mathcal{L}_{Z^{\gamma}}(F)_{21}.
\end{equation} In such a way, our framework exactly corresponds to the one of Glassey-Strauss. More precisely, one can compute the source terms of the wave equations satisfied by the components of $E_{\gamma}$ and $B_{\gamma}$. For any $0 \leq \mu, \nu \leq 3$, we have
\begin{flalign}\label{eq:waveF}
&\quad \Box \mathcal{L}_{Z^{\gamma}}(F)_{\mu \nu}=\int_{\R^3_v} \widehat{v}_{\mu}\partial_{x^\nu}f_{\gamma}-\widehat{v}_{\nu}\partial_{x^{\mu}} f_{\gamma}\dr v, \qquad \text {so, for instance,} \qquad \Box E^i_{\gamma} = -\int_{\R^3_v} \partial_{x^i}f_{\gamma}+\widehat{v}_i\partial_t f_{\gamma}\dr v.&
\end{flalign}
Applying \cite[Theorem~$3$]{GlStrauss} to $(f_{\gamma},E_{\gamma}, B_{\gamma})$ provides us, for any $0 \leq \mu , \, \nu \leq 3$, 
$$4\pi\mathcal{L}_{Z^{\gamma}}(F)_{\mu \nu}=\mathcal{L}_{Z^{\gamma}}(F)^{\mathrm{data}}_{\mu \nu}+\mathcal{L}_{Z^{\gamma}}(F)^T_{\mu \nu}-\int_{|y-x| \leq t} \int_{\R^3_v}  \mathcal{W}_{\mu\nu}(\omega, v) \left( \mathbf{T}_0 f_{\gamma} \right)(t-|y-x|,y,v) \frac{\mathrm{d}v \mathrm{d}y}{|y-x|},$$
where we recall that $\mathbf{T}_0=\widehat{v}^{\lambda}\partial_{x^{\lambda}}$. Note that the constants $C^{\gamma}_{\beta}$ in the definitions of $\mathcal{L}_{Z^{\gamma}}(F)^{\mathrm{data}}$, $\mathcal{L}_{Z^{\gamma}}(F)^T$ and $f_{\gamma}$ are the same. Applying the commutation formula of Proposition \ref{Com} for any $|\beta| \leq |\gamma|$ yields
\begin{equation}\label{eqT0}
 \T_0 f_{\gamma} =-\sum_{|\beta| \leq |\gamma|} C^{\gamma}_{\beta} \, \widehat{v}^{\mu} {F_{\mu}}^j\partial_{v^j} \widehat{Z}^{\beta}f+ C^{\gamma}_{\beta}\sum_{|\xi|+|\kappa| \leq |\beta|} C^{\beta}_{\xi, \kappa} \widehat{v}^{\mu} {\mathcal{L}_{Z^{\xi}}(F)_{\mu}}^j\partial_{v^j} \widehat{Z}^{\kappa}f,
\end{equation}
with $C^{\beta}_{\xi, \kappa} \in \mathbb{N}$. It remains to integrate by parts in $v$ and to recall that $\nabla_{v^j}\cdot\widehat{v}^{\mu}{\mathcal{L}_{Z^{\xi}}(F)_{\mu}}^j={\mathcal{L}_{Z^{\xi}}(F)_{j}}^j=0$.
\end{proof}
It will then be important to estimate the kernels introduced in the previous Proposition. 
\begin{Lem}\label{LemtechestiW}
For all $(\omega,v) \in \mathbb{S}^2 \times \R^3_v$, we have $|\omega+\widehat{v}|^2, \, |\omega \wedge \widehat{v}|^2 \leq 2(1+\omega \cdot \widehat{v})$ and $(1+\omega \cdot \widehat{v})^{-1}\leq 2|v^0|^2$.
\end{Lem}
\begin{proof}
For the second inequality, simply note that $$2|v^0|^2(1+\omega \cdot \widehat{v}) \geq 2|v^0|^2( 1-|\widehat{v}|)=2v^0(v^0-|v|) \geq (v^0+|v|)(v^0-|v|)=|v^0|^2-|v|^2=1.$$
For the first ones, since $|\omega|=1$ and $|\widehat{v}| \leq 1$,
$$  |\omega+\widehat{v}|^2=|\omega|^2+|\widehat{v}|^2+2\omega \cdot \widehat{v} \leq 2(1+\omega \cdot \widehat{v}), \quad |\omega \wedge \widehat{v}|^2=|\omega|^2|\widehat{v}|^2-|\omega \cdot \widehat{v}|^2\leq (1+\omega \cdot \widehat{v})(1-\omega \cdot \widehat{v})\leq 2(1+\omega \cdot \widehat{v}).$$
\end{proof}
\begin{Cor}\label{estiW}
For any $0 \leq \mu, \, \nu \leq 3$ and all $(\omega,v) \in \mathbb{S}^2 \times \R^3_v$, there holds
$$ \left|\mathcal{W}_{\mu \nu}\right|(\omega,v) \leq 2v^0, \qquad \quad \frac{\left|\mathcal{W}_{\mu \nu}\right|(\omega,v)}{|v^0|^2(1+\omega \cdot \widehat{v})} \leq 4 v^0, \qquad \quad \left|\nabla_v \mathcal{W}_{\mu \nu}\right|(\omega,v) \leq 6v^0.$$
We have similar bounds for their first order derivatives,
$$ \left|\nabla_v \mathcal{W}_{\mu \nu}\right|(\omega,v) \lesssim v^0, \qquad \quad \left| \nabla_v \left(\frac{\mathcal{W}_{\mu \nu}(\omega,v)}{|v^0|^2(1+\omega \cdot \widehat{v})} \right) \right| \lesssim v^0, \qquad \quad \left|\nabla_v \nabla_v \mathcal{W}_{\mu \nu}\right|(\omega,v) \lesssim v^0.$$
\end{Cor}
\begin{proof}
The first two inequalities are a direct consequence of the previous lemma. The other ones ensue from straightforward computations carried out in Lemma \ref{Lemapp}.
\end{proof}
\begin{Rq}
These bounds are sharp. Let us focus for instance on the first one, $|\mathcal{W}_{\mu \nu}|(\omega,v) \leq 2v^0$. For this, consider, for any $v \in \R^3_v $, the function $\phi_v : \omega \mapsto 1+\omega \cdot \widehat{v}$ defined on $\mathbb{S}^2$. Then,
$$ \min_{\omega \in \mathbb{S}^2} \, \phi_v(\omega)=\frac{v^0-|v|}{v^0}=\frac{1}{v^0(v^0+|v|)} \leq \frac{1}{|v^0|^2} , \qquad \qquad \max_{\omega \in \mathbb{S}^2} \, \phi_v(\omega) = \frac{v^0+|v|}{v^0} \geq 1.$$
By continuity, there exists $\omega_v \in \mathbb{S}^2$ such that $1+\omega_v \cdot \widehat{v}=|v^0|^{-2}$. Then, using $|\omega+\widehat{v}|^2= 2(1+\omega \cdot \widehat{v}) -|v^0|^{-2}$, we have
$$ \sum_{1 \leq i \leq 3} \left|\mathcal{W}_{0i}\right|^2(\omega_v,v)=\frac{|\omega_v+\widehat{v}|^2}{|1+\omega_v \cdot \widehat{v}|^2}=\frac{1}{1+\omega_v \cdot \widehat{v}}\left( 2-\frac{1}{|v^0|^2(1+\omega_v \cdot \widehat{v})} \right)=v^0.$$
\end{Rq}
In order to improve the bootstrap assumption \eqref{boot2}, we will need to use the Glassey-Strauss decomposition of the spatial derivatives of the electromagnetic field. A similar result holds for the time derivative but we will estimate it by exploiting the Maxwell equations. For instance, one can check that \eqref{VM22}-\eqref{VM33} imply $|\nabla_{\partial_t} F| \lesssim \sum_{1 \leq k \leq 3}|\nabla_{\partial_{x^k}} F|+|J(f)|$. We lighten the notations by denoting the Lorentz force as
\begin{equation}\label{defLorforce}
K^j:=\widehat{v}^{\mu} {F_\mu}^{j}, \qquad \qquad K_\xi^j := \widehat{v}^{\mu} {\mathcal{L}_{Z^{\xi}}(F)_\mu}^{j}, \qquad 1 \leq j \leq 3, \quad  1\leq |\xi| \leq N. 
\end{equation}
\begin{Pro}\label{GSdecomoderiv}
Let $|\gamma|=N-1$ and $1 \leq k \leq 3$. Then, $\nabla_{\partial_{x^k}} \mathcal{L}_{Z^{\gamma}}(F)$ can be written as
$$ 4\pi \nabla_{\partial_{x^k}} \mathcal{L}_{Z^{\gamma}}(F) = A^{\mathrm{data}}_{\gamma,k}+A^{\mathrm{ver}}_{\gamma,k}+A^{TT}_{\gamma,k}+A^{TS}_{\gamma,k}+A^{SS}_{\gamma,k},$$
where the five $2$-forms verify the following properties. We fix $0 \leq \mu, \nu \leq 3$ and we use again the notation $\omega=\frac{y-x}{|y-x|}$ in the integrals written below. Moreover, we give the definition of the kernels at the end of the statement\footnote{We point out that we are only interested in the qualitative properties of these kernels.}.
\begin{itemize}
\item $A^{\mathrm{data}}_{\gamma,k}$ can be explicitly computed in terms of the initial data,
\begin{align*}
 A^{\mathrm{data}}_{\gamma,k, \, \mu \nu}(t,x) &= 4 \pi \partial_{x^k}\mathcal{L}_{Z^{\gamma}}(F)^{\mathrm{data}}_{\mu \nu}(t,x) - \sum_{|\beta| \leq N-1} C^{\gamma}_{\beta}\frac{1}{t^2}\int_{|y-x|=t} \int_{\R^3_v} \mathcal{D}_{\mu \nu}^k(\omega,v) \widehat{Z}^{\beta}f(0,y,v) \dr v \dr y \\
 & \quad -\sum_{|\beta| \leq N-1}C^{\gamma}_{\beta}\frac{1}{t}\int_{|y-x|=t} \int_{\R^3_v} \mathcal{C}_{\mu \nu}^k(\omega,v) \T_0\widehat{Z}^{\beta}f(0,y,v) \dr v \dr y .
 \end{align*}
\item $A^{\mathrm{ver}}_{\gamma,k}$ is the vertex term,
$$A^{\mathrm{ver}}_{\gamma,k, \, \mu \nu}(t,x) = \sum_{|\beta| \leq N-1} C^{\gamma}_{\beta} \int_{\sigma \in \mathbb{S}^2} \int_{\R^3_v} \mathcal{D}_{\mu \nu}^k(\sigma,v) \widehat{Z}^{\beta} f(t,x,v) \dr v \dr \mu_{\mathbb{S}^2}.$$
\item $A^{TT}_{\gamma,k}$ is the most singular term,
$$ A^{TT}_{\gamma,k, \, \mu\nu}(t,x) := \sum_{|\beta| \leq N-1}C^{\gamma}_{\beta} \int_{|y-x|\leq t } \int_{\R^3_v} \mathcal{A}_{\mu \nu}^k(\omega,v) \widehat{Z}^{\beta} f(t-|y-x|,y,v)  \frac{\dr v \dr y}{|y-x|^3}$$
and the crucial identity $ \int_{|\sigma|=1} \mathcal{A}_{\mu \nu}^k(\sigma,\widehat{v})\dr \mu_{\mathbb{S}^2}=0$ holds for all $v \in \R^3_v$.
\item $A^{T,S}_{\gamma,k}$ is given by
$$ A^{T,S}_{\gamma,k, \, \mu\nu}(t,x) := \sum_{|\xi|+|\kappa| \leq N-1}N^{\gamma}_{\xi, \kappa} \int_{|y-x| \leq t} \int_{\R^3_v} \nabla_v \mathcal{B}_{\mu \nu}^k(\omega,v) \cdot \big(\widehat{Z}^{\kappa}f \, K_{\xi} \big)(t-|y-x|,y,v) \frac{\mathrm{d}v \mathrm{d}y}{|y-x|^2}.$$
\item $A^{SS}_{\gamma,k}$ is the sum of the four following quantities, where $N^{\gamma}_{\xi ,\zeta, \kappa} \in \mathbb{N}$,
\begin{align*}
A^{SS,I}_{\gamma,k, \, \mu \nu}&:=\sum_{|\xi|+|\zeta|+|\kappa| \leq N-1} N^{\gamma}_{\xi ,\zeta, \kappa} \, \int_{|y-x| \leq t} \int_{\R^3_v} \left[\nabla_v \big( \nabla_v\mathcal{C}_{\mu\nu}^k(\omega,\cdot) \cdot K_{\xi} \big) \cdot K_{\zeta}\widehat{Z}^{\kappa}f\right](t-|y-x|,y,v) \frac{\mathrm{d}v \mathrm{d}y}{|y-x|}, \\
A^{SS,II}_{\gamma,k, \, \mu \nu}&:=\sum_{|\xi|+|\kappa| \leq N-1} N^{\gamma}_{\xi , \kappa} \, \int_{|y-x| \leq t} \int_{\R^3_v} \nabla_v  \mathcal{C}_{\mu\nu}^k(\omega,v) \cdot (\T_0(K_\xi) \widehat{Z}^{\kappa}f)(t-|y-x|,y,v) \frac{\mathrm{d}v \mathrm{d}y}{|y-x|}, \\
A^{SS,III}_{\gamma,k, \, \mu \nu}&:=\sum_{|\xi|+|\kappa| \leq N-1} N^{\gamma}_{\xi , \kappa} \, \int_{|y-x| \leq t} \int_{\R^3_v} \mathcal{C}_{j,\mu \nu}^{k,n} (\omega,v) (\partial_{x^n}(K_{\xi}^j)\widehat{Z}^{\kappa}f)(t-|y-x|,y,v) \frac{\mathrm{d}v \mathrm{d}y}{|y-x|}, \\
A^{SS,IV}_{\gamma,k, \, \mu \nu}&:=\sum_{|\xi|+|\kappa| \leq N-1} N^{\gamma}_{\xi , \kappa} \, \int_{|y-x| \leq t} \int_{\R^3_v} \mathcal{C}_{j,\mu \nu}^{k,n} (\omega,v) (K_{\xi}^j \partial_{x^n}\widehat{Z}^{\kappa}f)(t-|y-x|,y,v) \frac{\mathrm{d}v \mathrm{d}y}{|y-x|}.
\end{align*}
\item The kernels are smooth functions of $(\omega,v)\in \mathbb{S}^2 \times \R^3_v$ given by
\begin{align*}
\mathcal{A}_{\mu\nu}^k(\omega,v)&:=-3\frac{\mathbf{w}_{\mu \nu}(\omega,v)\omega_k}{|v^0|^4(1+\omega \cdot \widehat{v})^4}-3\frac{\mathbf{w}_{\mu \nu}(\omega,v)\widehat{v}_k}{|v^0|^2(1+\omega \cdot \widehat{v})^3}+\frac{\delta_{k\mu}\widehat{v}_{\nu}-\delta_{k\nu}\widehat{v}_{\mu}}{|v^0|^2(1+\omega \cdot \widehat{v})^2}, \\
\mathcal{B}^k_{\mu \nu} (\omega,v) &:= 3\frac{\mathbf{w}_{\mu \nu}(\omega,v)\omega_k}{|v^0|^2(1+\omega \cdot \widehat{v})^3}-2\frac{\mathbf{w}_{\mu \nu}(\omega,v)\widehat{v}_k}{(1+\omega \cdot \widehat{v})^2}-\frac{\delta_{k\mu}\widehat{v}_{\nu}-\delta_{k\nu}\widehat{v}_{\mu}}{1+\omega \cdot \widehat{v}}, \\
\mathcal{C}^k_{\mu \nu}(\omega,v) &:= \frac{\omega_k \mathbf{w}_{\mu \nu}(\omega,v)}{(1+\omega \cdot \widehat{v})^2}, \qquad \qquad \mathcal{C}_{j,\mu \nu}^{k,n} (\omega,v) := \mathcal{C}^k_{\mu \nu}(\omega,v) \frac{\delta_j^n-\widehat{v}_j\widehat{v}^n}{v^0}, \\
\mathcal{D}^k_{\mu \nu}(\omega,v) & := \frac{\omega_k \mathbf{w}_{\mu \nu}(\omega,v)}{|v^0|^2(1+\omega \cdot \widehat{v})^3}.
\end{align*}
\end{itemize}
\end{Pro}
\begin{proof}
Let $k \in \llbracket 1,3 \rrbracket$, $|\gamma|=N-1$ and recall from \eqref{deffgamma} the definition of $f_{\gamma}$ and that $\mathcal{L}_{Z^{\gamma}}(F)$ solves the Maxwell equations with source term $J(f_{\gamma})$. Recall further the electric and magnetic parts $(E_{\gamma},B_{\gamma})$ of $\mathcal{L}_{Z^{\gamma}}(F)$, introduced in \eqref{defEB}. In the same spirit as in the proof of Proposition \ref{GSdecompo}, we apply\footnote{See also the original version of the result, \cite[Theorem~$4$]{GlStrauss}.} \cite[Theorem~$5.4.1$]{Glassey} to $(f_{\gamma},E_{\gamma},B_{\gamma})$. This yields
$$ \nabla_{\partial_{x^k}} \mathcal{L}_{Z^{\gamma}}(F)_{\mu \nu} = A^{\mathrm{data}}_{\gamma,k, \, \mu \nu}+A^{\mathrm{ver}}_{\gamma,k, \, \mu \nu}+A^{TT}_{\gamma,k, \, \mu \nu}+\widetilde{A}^{\; T,S}_{\gamma,k, \, \mu \nu}+\widetilde{A}^{\,SS}_{\gamma,k, \, \mu \nu},$$
where
\begin{align*}
\widetilde{A}^{\; T,S}_{\gamma,k, \,\mu \nu}&:=\int_{|y-x| \leq t} \int_{\R^3_v} \mathcal{B}^k_{\mu \nu}(\omega,v) \big( \T_0f_\gamma \big)(t-|y-x|,y,v) \frac{\mathrm{d}v \mathrm{d}y}{|y-x|^2}, \\
\widetilde{A}^{\, SS}_{\gamma,k, \, \mu \nu} &:=-\int_{|y-x| \leq t} \int_{\R^3_v} \mathcal{C}^k_{\mu \nu}(\omega,v)  \big( \T_0\T_0f_\gamma \big)(t-|y-x|,y,v) \frac{\mathrm{d}v \mathrm{d}y}{|y-x|},
\end{align*}
as well as $ \int_{|\sigma|=1} \mathcal{A}^k_{\mu \nu}(\sigma,\widehat{v})\dr \mu_{\mathbb{S}^2}=0 $. One can then prove that $\widetilde{A}^{\; T,S}_{\gamma,k, \,\mu \nu}=A^{T,S}_{\gamma,k, \, \mu \nu}$ by rewriting $\T_0 f_{\gamma}$ using the (commuted) Vlasov equation. More precisely, we use \eqref{eqT0} and we then integrate by parts in $v$. It remains to deal with $\widetilde{A}^{\, SS}_{\gamma,k, \, \mu \nu}$ and we recall for this that $\nabla_v \cdot K_{\xi}=\nabla_{v^j}\cdot\widehat{v}^{\mu}{\mathcal{L}_{Z^{\xi}}(F)_{\mu}}^j=0$. Hence, using again \eqref{eqT0}, we get that there exists $N^{\gamma}_{\xi,\kappa} \in \mathbb{N}$ such that
$$ \T_0 \T_0(f_{\gamma})= \sum_{|\xi|+|\kappa| \leq |\gamma|}  N^{\gamma}_{\xi,\kappa} \T_0\partial_{v^j} \left( K_\xi^j \widehat{Z}^{\kappa}f\right).$$
Now, we write $\T_0 \partial_{v^j}=\partial_{v^j} \T_0-\partial_{v^j}(\widehat{v}^n)\partial_{x^n}$ and we apply the commutation formula of Proposition \ref{Com} to $\widehat{Z}^{\kappa}f$. We get
$$ \T_0\partial_{v^j} \left( \widehat{v}^{\lambda}{\mathcal{L}_{Z^{\xi}}(F)_{\lambda}}^j \widehat{Z}^{\kappa}f\right)=\partial_{v^j} \! \left(\T_0(K_\xi^j) \widehat{Z}^{\kappa}f\right)+\partial_{v^j} \! \left(K_\xi^j \T_0( \widehat{Z}^{\kappa}f)\right)-\frac{\delta_j^n-\widehat{v}_j\widehat{v}^n}{v^0}\left(\partial_{x^n}(K_{\xi}^j)\widehat{Z}^{\kappa}f+K_\xi^j \partial_{x^n} \widehat{Z}^{\kappa}f \right), $$
so that, by integration by parts in $v$ for the quantities related to the two first terms on the right hand side of the previous equality,
$$\widetilde{A}^{\, SS}_{\gamma,k, \,\mu \nu}=A^{SS,II}_{\gamma,k, \, \mu \nu}\!+A^{SS,III}_{\gamma,k, \, \mu \nu}\!+A^{SS,IV}_{\gamma,k, \,\mu \nu}\!+\!\sum_{|\xi|+|\kappa| \leq |\gamma|} \! N^{\gamma}_{\xi,\kappa}\!\int_{|y-x| \leq t} \int_{\R^3_v}\! \nabla_{v^j}\! \left( \mathcal{C}_{\mu \nu}^k(\omega,v)\right)\! (K_\xi^j \T_0( \widehat{Z}^{\kappa}f))(\tau_y,y,v) \frac{\mathrm{d}v \mathrm{d}y}{|y-x|},$$
where $\tau_y: =t-|y-x|$. Finally, we deal with the last term by applying first the commutation relation of Proposition \ref{Com}, giving that $\T_0(\widehat{Z}^{\kappa}f)=-K\cdot \nabla_v\widehat{Z}^{\kappa}f+C^{\kappa}_{\zeta,\beta} K_{\zeta}\cdot \nabla_v\widehat{Z}^{\beta}f$, and then by integrating by parts in $v$.
\end{proof}
These Kernels and their derivatives can be estimated by a direct application of Lemmas \ref{LemtechestiW} and \ref{Lemapp}.
\begin{Cor}\label{estikernels}
For any $1\leq k, \,  j , \, n \leq 3$ and for all $v\in \R^3_v$, we have
$$ \left(\left| \mathcal{A}^k\right|+\left| \nabla_v \mathcal{A}^k\right|+\left|\nabla_v \mathcal{B}^k\right|+\left| \mathcal{C}^{k,n}_j\right|+\left|\nabla_v \mathcal{C}^{k,n}_j\right|+\left|\nabla_v  \mathcal{C}^k\right|+\left|\nabla_v \nabla_v \mathcal{C}^k\right|+\left| \mathcal{D}^k\right|+\left|\nabla_v \mathcal{D}^k\right|\right)(\cdot,v) \lesssim |v^0|^3.$$
\end{Cor}

\subsection{Three integral bounds}

The estimate of most of the terms listed in Propositions \ref{GSdecompo} and \ref{GSdecomoderiv} will in fact be reduced to the analysis of three different integrals. We will deal with all of them by applying a particular case of \cite[Lemma~$6.5.2$]{Glassey}.
\begin{Lem}\label{LemGlassey}
Let $p \in \R$ and $g:\R_+^2 \rightarrow \R_+$ be a continuous function. Then, for all $(t,x) \in [0,T[ \times \R^3\setminus \{0\}$,
$$ \int_{|y-x| \leq t} g(t-|y-x|,|y|) \frac{\dr y}{|y-x|^p} = \frac{2 \pi}{|x|} \int_{\tau=0}^t \int_{\lambda=||x|-t+\tau|}^{|x|+t-\tau} g(\tau,\lambda) \lambda  \, \dr \lambda \frac{\dr \tau}{(t-\tau)^{p-1}}.$$
\end{Lem}
The following result will be useful for controlling $\mathcal{L}_{Z^{\gamma}}(F)^S$ and $A^{SS}_{\gamma,k}$.
\begin{Lem}\label{LemJ}
For any $b \geq 4$ and for all $(t,x) \in \R_+ \times \R$, there holds
$$ \mathbf{Y}^{p=1}_{b,1}(t,x):= \int_{|y-x| \leq t } \frac{1}{(1+t-|y-x|+|y|)^{b}(1+|t-|y-x|-|y||)} \frac{\dr y}{|y-x|}\lesssim \frac{\log(3+|t-|x||)}{(1+t+|x|)(1+|t-|x||)^{b-2}}.$$
\end{Lem}
\begin{proof}
Note first that, on the domain of integration, 
\begin{equation}\label{eq:gaintminusr}
t-|y-x|+|y| \geq t-|y|-|x|+|y|=t-|x|, \qquad t-|y-x|+|y| \geq |y|\geq |x|-|y-x| \geq |x|-t,
\end{equation}
so that $\mathbf{Y}^{p=1}_{b,1}(t,x) \leq (1+|t-|x||)^{-b+4} \, \mathbf{Y}^{p=4}_{4,1}(t,x)$ and it suffices to treat the case $b=4$. By continuity, we can assume further that $x \neq 0$. According to Lemma \ref{LemGlassey},
$$
  \mathbf{Y}^{p=1}_{4,1}(t,x) \leq   \frac{2 \pi}{|x|} \int_{\tau=0}^t \int_{\lambda=||x|-t+\tau|}^{|x|+t-\tau}  \frac{\dr \lambda \dr \tau }{(1+\tau+\lambda)^3(1+|\tau-\lambda|)}       .
$$
We perform the change of variables $\underline{u}=\tau+\lambda$ and $u=\tau-\lambda$. Then, on the domain of integration $||x|-t| \leq \underline{u}\leq t+|x|$ and $ u \leq ||x|-t|$. Moreover, $u\geq -\underline{u}$ since $2\tau \geq 0$. Consequently,
$$
 \mathbf{Y}^{p=1}_{4,1}(t,x)   \leq \frac{ \pi}{|x|} \int_{\underline{u}=||x|-t|}^{t+|x|} \int_{u= -\underline{u}}^{||x|-t|} \frac{\dr u}{1+|u|} \frac{\dr \underline{u}}{(1+\underline{u})^3}   \leq \frac{2\pi}{|x|} \int_{\underline{u}=||x|-t|}^{t+|x|}  \frac{\log(3+\underline{u})}{(1+\underline{u})^3} \dr \underline{u}.
 $$ 
Now, remark that
\begin{align*}
\mathbf{Y}^{p=1}_{4,1}(t,x)  & \lesssim \frac{2 \pi \log(3+|t-|x||)}{(1+|t-|x||)|x|} \int_{\underline{u}=||x|-t|}^{t+|x|}  \frac{\dr \underline{u}}{(1+\underline{u})^2}  \\
 &= \frac{2 \pi \log(3+|t-|x||)}{(1+t+|x|)(1+|t-|x||)^2}\frac{t+|x|-|t-|x||}{|x|}
 \end{align*}
and it remains to note that the last factor on the right hand side is bounded by $2\min(t,|x|)/|x|\leq 2$.
\end{proof}
We will apply the next lemma in order to deal with $\mathcal{L}_{Z^{\gamma}}(F)^T$ and $A^{T,S}_{\gamma,k}$.
\begin{Lem}\label{LemKb}
Let, for any $b \geq 3$ and all $(t,x) \in \R_+ \times \R^3$,
$$ \mathbf{Y}^{p=2}_b(t,x):= \int_{|y-x| \leq t } (1+t-|y-x|+|y|)^{-b} \frac{\dr y}{|y-x|^2}  .$$
Then, the following range of estimates holds. For any $0 < \delta \leq 1$,
$$ \mathbf{Y}^{p=2}_b(t,x) \lesssim   \delta^{-1} (1+t+|x|)^{-2+\delta}(1+|t-|x||)^{-b-\delta+3}, \quad \; \;  \mathbf{Y}^{p=2}_b(t,x) \lesssim (1+t+|x|)^{-2}(1+|t-|x||)^{-b+3} \log(1+t).$$
\end{Lem}
\begin{proof}
In view of \eqref{eq:gaintminusr}, we have $\mathbf{Y}^{p=2}_b(t,x) \leq (1+|t-|x||)^{-b+3} \mathbf{Y}^{p=2}_3(t,x)$ and it suffices to treat the case $b=3$. Then remark that
$$  \mathbf{Y}^{p=2}_3(t,x)= \mathbf{K}_{[0,\frac{t}{2}]}+\mathbf{K}_{[\frac{t}{2},t]}, \qquad \qquad \mathbf{K}_I:= \int_{|y-x| \in I } (1+t-|y-x|+|y|)^{-3} \frac{\dr y}{|y-x|^2}.$$
On the domain of integration of $\mathbf{K}_{[0,\frac{t}{2}]}$, we have $t-|y-x|+|y| \gtrsim t+|x|$. Indeed, $t-|y-x| \geq t/2$ and $|y| \geq |x|-t$ (which controls $|x|/2$ if $|x| \geq 2t$). Consequently,
\begin{equation}\label{eq:tobeusedjustbelow}
 \mathbf{K}_{[0,\frac{t}{2}]} \lesssim (1+t+|x|)^{-3} \int_{r=0}^{\frac{t}{2}} \dr r \leq \frac{1}{2}(1+t+|x|)^{-2}.
 \end{equation}
Applying Lemma \ref{LemGlassey}, we have
$$ \mathbf{K}_{[\frac{t}{2},t]} \leq \frac{2 \pi}{|x|}\int_{\tau=0}^{\frac{t}{2}} \int_{\lambda=||x|-t+\tau|}^{|x|+t-\tau} \frac{\dr \lambda \dr \tau}{(1+\tau+\lambda)^{2}(t-\tau)}.$$
Now, observe that for all $0 \leq \tau \leq t/2$,
\begin{align}
 \frac{1}{|x|(t-\tau)}\int_{\lambda=||x|-t+\tau|}^{|x|+t-\tau} \frac{\dr \lambda }{(1+\tau+\lambda)^{2}}&=\frac{2\min(|x|,t-\tau)}{|x|(t-\tau)(1+t+|x|)(1+\tau+||x|-t+\tau|)} \nonumber \\
 &\leq \frac{8}{\max(|x|,t)(1+t+|x|)(1+\tau+|t-|x||)}. \label{eq:tobeusebelow2}
\end{align}
Let $0 \leq \delta \leq 1$ and write $(1+\tau+|t-|x||) \geq (1+\tau)^{1-\delta} (1+|t-|x||)^{\delta}$. It remains to integrate in $\tau$ in order to derive the expected range of estimates for $\mathbf{K}_{[\frac{t}{2},t]}$.
\end{proof}
Finally, a part of our analysis of $A^{TT}_{\gamma,k}$ will rely on the following estimate.
\begin{Lem}\label{LemKbis}
For all $(t,x) \in [1,+\infty[ \times \R^3$, we have
$$ \mathbf{Y}^{p=3}_3(t,x):= \int_{1\leq |y-x| \leq t } (1+t-|y-x|+|y|)^{-3} \frac{\dr y}{|y-x|^3} \lesssim   \frac{\log(t) }{(1+t+|x|)^3} .$$
\end{Lem}
\begin{proof}
The inequality can be easily proved if $t \leq 2$ so we assume $t \geq 2$. Start by writing
$$  \mathbf{Y}^{p=3}_3(t,x)= \overline{\mathbf{K}}_{[1,\frac{t}{2}]}+\overline{\mathbf{K}}_{[\frac{t}{2},t]}, \qquad \qquad \overline{\mathbf{K}}_I:= \int_{|y-x| \in I } (1+t-|y-x|+|y|)^{-3} \frac{\dr y}{|y-x|^3}.$$
Following \eqref{eq:tobeusedjustbelow}, we have
$$ \overline{\mathbf{K}}_{[1,\frac{t}{2}]} \lesssim (1+t+|x|)^{-3} \int_{r=1}^{\frac{t}{2}} \frac{\dr r}{r} \leq \log \left( \frac{t}{2} \right)(1+t+|x|)^{-3}.$$
Next, we apply Lemma \ref{LemGlassey} to get
$$ \overline{\mathbf{K}}_{[\frac{t}{2},t]} \leq \frac{2 \pi}{|x|}\int_{\tau=0}^{\frac{t}{2}} \int_{\lambda=||x|-t+\tau|}^{|x|+t-\tau} \frac{\dr \lambda \dr \tau}{(1+\tau+\lambda)^{2}(t-\tau)^2}.$$
If $2t \geq |x|$, we use \eqref{eq:tobeusebelow2} and $t-\tau \geq t/2$ in order to derive $\overline{\mathbf{K}}_{[\frac{t}{2},t]} \lesssim t^{-2}(1+t+|x|)^{-1}\log(1+t/2)$, which implies the result. Otherwise, $2t \leq |x|$ and we have, for all $0 \leq \tau \leq t/2$,
\begin{align}
 \frac{1}{|x|(t-\tau)^2}\int_{\lambda=||x|-t+\tau|}^{|x|+t-\tau} \frac{\dr \lambda \dr \tau}{(1+\tau+\lambda)^{2}}&=\frac{2\min(|x|,t-\tau)}{|x|(t-\tau)^2(1+t+|x|)(1+\tau+||x|-t+\tau|)} \nonumber \\
 &\leq \frac{2}{|x|(1+t+|x|)(1+|x|-t)(t-\tau)}. \label{eq:tobeusebelow3}
\end{align}
We get, as $2 \leq 2t \leq |x|$,
\begin{align*}
\mathbf{K}_{[\frac{t}{2},t]} \leq 4 \pi \log(2) |x|^{-1}(1+|x|/2)^{-1}(1+t+|x|)^{-1} \lesssim (1+t+|x|)^{-3}.
\end{align*}
\end{proof}

\subsection{The derivatives of order up to $N-1$}\label{subsecba1}
In this subsection, we prove pointwise decay estimates for each of the elements of the decomposition of $\mathcal{L}_{Z^{\gamma}}(F)$ provided by Proposition \ref{GSdecompo}. We start by dealing with $\mathcal{L}_{Z^{\gamma}}(F)^{\mathrm{data}}$, which is defined on $\R_+ \times \R^3$.
\begin{Pro}\label{estidata}
There exists $C_{\mathrm{data}}>0$, depending only on $N$, such that
$$\forall \, |\gamma| \leq N-1, \quad  \forall \, (t,x) \in \R_+  \times \R^3, \qquad \quad \left| \mathcal{L}_{Z^{\gamma}}(F)^{\mathrm{data}} \right|(t,x) \leq \Lambda \, C_{\mathrm{data}} (1+t+|x|)^{-1}(1+|t-|x||)^{-1}.$$
\end{Pro}
\begin{proof}
In view of the assumptions on the initial data (see Theorem \ref{Th1}) and applying Corollary \ref{estiW} in order to estimate $\mathcal{W}_{\mu \nu}$, we have for any $|\beta| \leq N-1$, $\omega \in \mathbb{S}^2$ and $0 \leq \mu , \nu \leq 3$, 
\begin{align}
\forall \, y \in \R^3,  \qquad \qquad  \left| \mathcal{L}_{Z^{\beta}}(F) \right|(0,y)+\langle y \rangle \left| \nabla_{t,x}\mathcal{L}_{Z^{\beta}}(F) \right|(0,y) & \lesssim \sum_{|\kappa| \leq |\beta|+1} \langle y \rangle^{|\kappa|} \left| \nabla_{t,x}^{\kappa}F \right|(0,y) \lesssim \Lambda \, \langle y \rangle^{-2}, \nonumber \\
\left| \int_{\R^3_v} \left(\mathcal{W}_{\mu \nu}(\omega,\widehat{v})-\delta^0_\mu \widehat{v}^\nu+\delta^0_\nu \widehat{v}^\mu \right) \widehat{Z}^{\beta}f(0,y,v)\dr v \right| & \leq  3\int_{\R^3_v} |v^0|^{N_v} \left|\widehat{Z}^{\beta}f(0,y,v) \right| \frac{\dr v}{\langle v \rangle^{N_v-1}} \lesssim \epsilon \, \langle y \rangle^{-N_x}. \nonumber 
\end{align}
The estimates, at $t=0$, for the time derivatives of the solutions are obtained by using that the equations \eqref{VM11}-\eqref{VM33} are initially verified. Using \eqref{decayhYang} for $p=N_x \geq 3$, we then deduce that 
\begin{equation}\label{eq:partief}
\forall \, (t,x) \in \R_+ \times \R^3 , \qquad |\mathcal{L}_{Z^{\gamma}}(F)^{\mathrm{data}}-\mathcal{L}_{Z^{\gamma}}(F)^{\mathrm{hom}}|(t,x) \lesssim \epsilon(1+t+|x|)^{-1}(1+|t-|x||)^{-N_x+1}
\end{equation} 
and it remains to use $\epsilon \leq \Lambda$ and to apply Proposition \ref{decaylinMax} to $\mathcal{L}_{Z^{\gamma}}(F)^{\mathrm{hom}}_{\mu \nu}$.
\end{proof}
Next, we consider $\mathcal{L}_{Z^{\gamma}}(F)^S$, which is strongly decaying far from the light cone.
\begin{Pro}\label{estiS}
For any $|\gamma| \leq N-1$, there holds
$$ \forall \, (t,x) \in [0,T[ \times \R^3, \qquad \left| \mathcal{L}_{Z^{\gamma}}(F)^S \right|(t,x) \lesssim \overline{\epsilon} \, \Lambda \frac{\log(3+|t-|x||)}{(1+t+|x|)(1+|t-|x||)^2}.$$
\end{Pro}
\begin{proof}
Fix $0 \leq \mu,\nu\leq 3$ and recall from Proposition \ref{GSdecompo} the definition of $\mathcal{L}_{Z^{\gamma}}(F)^S$. We have, with $\omega=\frac{y-x}{|y-x|}$, 
\begin{align*}
\left|\mathcal{L}_{Z^{\gamma}}(F)^S_{\mu \nu}\right|&(t,x) \lesssim \\
&  \sum_{|\xi|+|\kappa| \leq |\gamma|} \int_{|y-x| \leq t}\left|{\mathcal{L}_{Z^{\xi}}(F)_{\lambda}}^j\right|(t-|y-x|,y) \left| \int_{\R^3_v} \widehat{v}^{\lambda} \partial_{v^j}  \mathcal{W}_{\mu\nu}(\omega, v) \widehat{Z}^{\kappa}f  (t-|y-x|,y,v)\mathrm{d}v \right| \frac{ \mathrm{d}y}{|y-x|}.
\end{align*}
Fix now $|\xi| + |\kappa| \leq N-1$, $j \in \llbracket 1,3 \rrbracket$ and $\lambda \in \llbracket 0 , 3 \rrbracket$. In view of Corollary \ref{estiW}, $\Psi(\omega,v):=\widehat{v}^{\lambda} \partial_{v^j}  \mathcal{W}_{\mu\nu}(\omega, v)$ satisfies $|\Psi|(\cdot,v)+|v^0 \nabla_v \Psi|(\cdot,v) \lesssim |v^0|^2 \leq |v^0|^{N_v-11}$. As $|\kappa| \leq N-1$, Proposition \ref{Proestimoyvkernet} then gives us
$$\forall \, (\sigma,\tau,y) \in \mathbb{S}^2 \times [0,T[ \times \R^3, \qquad \left| \int_{\R^3_v} \widehat{v}^{\lambda} \partial_{v^j}  \mathcal{W}_{\mu\nu}(\sigma, v) \widehat{Z}^{\kappa}f  (\tau,y,v)\mathrm{d}v \right| \lesssim \frac{\overline{\epsilon}}{(1+\tau+|y|)^{3}}.$$
Applying this last inequality for $(\sigma,\tau)=(\omega,t-|y-x|)$ and estimating the electromagnetic field using the bootstrap assumption \eqref{boot1}, we get
$$
 \left|\mathcal{L}_{Z^{\gamma}}(F)^S\right|(t,x)  \lesssim   \int_{|y-x|\leq t} \frac{ \overline{\epsilon} \, \Lambda}{(1+t-|y-x|+|y|)^4(1+|t-|y-x|-|y||)} \frac{\dr y}{|y-x|}   =  \overline{\epsilon}\, \Lambda  \mathbf{Y}^{p=1}_{4,1}(t,x). 
$$
The result then ensues from Lemma \ref{LemJ}.
\end{proof}
We finally deal with $\mathcal{L}_{Z^{\gamma}}(F)^T$, which actually enjoys stronger decay properties than $ \mathcal{L}_{Z^{\gamma}}(F)^S$ for $t\sim |x|$ (see Remark \ref{RqT} below).
\begin{Pro}\label{decayFT}
For any $|\gamma| \leq N-1$ and all $(t,x) \in [0,T[ \times \R^3$, we have
$$ \left| \mathcal{L}_{Z^{\gamma}}(F)^T \right|(t,x) \lesssim 
        \overline{\epsilon}  (1+t+|x|)^{-\frac{7}{4}}(1+|t-|x||)^{-\frac{1}{4}} .$$
\end{Pro}
\begin{proof}
In view of the definition of $\mathcal{L}_{Z^{\gamma}}(F)^T$, introduced in Proposition \ref{GSdecompo}, we have
$$  \left| \mathcal{L}_{Z^{\gamma}}(F)^T \right|(t,x) \lesssim \! \sum_{0 \leq \mu, \nu \leq 3} \, \sum_{|\beta| \leq |\gamma|} \int_{|y-x| \leq t } \left| \int_{\R^3_v} \frac{\mathcal{W}_{\mu\nu}(\omega,v)}{|v^0|^2(1+\omega \cdot \widehat{v})} \widehat{Z}^{\beta} f (t-|y-x|,y,v) \dr v \right| \! \frac{\dr y}{|y-x|^2}, \quad \omega=\frac{y-x}{|y-x|}.$$
Fix $0 \leq \mu, \nu \leq 3$, $|\beta| \leq |\gamma|$ and recall from Corollary \ref{estiW} that $\Psi(\sigma,v):=\frac{\mathcal{W}_{\mu\nu}(\sigma,v)}{|v^0|^2(1+\omega \cdot \widehat{v})}$ satisfies $|\Psi|(\cdot,v)+|\nabla_v \Psi|(\cdot,v) \lesssim v^0$. We then obtain from Proposition \ref{Proestimoyvkernet} that
$$ \forall \, \sigma \in \mathbb{S}^2, \; \forall \, (\tau,z) \in [0,T[\times \R^3, \qquad \left| \int_{\R^3_v} \frac{\mathcal{W}_{\mu\nu}(\sigma,v)}{|v^0|^2(1+\sigma \cdot \widehat{v})} \widehat{Z}^{\beta} f (\tau,z,v) \dr v \right| \lesssim \frac{\overline{\epsilon}}{(1+\tau+|z|)^3}.$$
Applying this estimate for $\sigma=\omega$, $\tau=t-|y-x|$ and $z=y$, we get from Lemma \ref{LemKb} that
$$  \left| \mathcal{L}_{Z^{\gamma}}(F)^T \right|(t,x) \lesssim  \overline{\epsilon} \, \mathbf{Y}^{p=2}_3(t,x) \lesssim \overline{\epsilon} (1+t+|x|)^{-\frac{7}{4}}(1+|t-|x||)^{-\frac{1}{4}} .$$
\end{proof}
\begin{Rq}\label{RqT}
In fact, Lemma \ref{LemKb} also provides $\left| \mathcal{L}_{Z^{\gamma}}(F)^T \right|(t,x) \lesssim \overline{\epsilon} (1+t+|x|)^{-2}\log(1+t)$. Moreover, the estimate could be significantly improved in the exterior of the light cone, where $|x|\geq t$.
\end{Rq}
If the constant $C_\mathrm{boot}$ is chosen such that $C_\mathrm{boot} \geq 2C_\mathrm{data}$ and if $\epsilon$ is small enough, Propositions \ref{estidata}, \ref{estiS} and \ref{decayFT} allow us to improve the bootstrap assumption \eqref{boot1}.

\subsection{The top order derivatives}\label{subsecba2}

In this section, we estimate all the terms listed in Proposition \ref{GSdecomoderiv} in order to improve the bootstrap assumption \eqref{boot2}. We start by dealing with the ones depending explicitly on the data.

\begin{Pro}\label{estidataA}
There exists a constant $\overline{C}_{\mathrm{data}}$, depending only on $N$, such that, for any $k \in \llbracket 1, 3 \rrbracket$ and $|\gamma| = N-1$,
$$ \forall \, (t,x) \in [0,T[ \times \R^3, \qquad \left|A^{\mathrm{data}}_{\gamma,k}\right|(t,x) \leq \Lambda \, \overline{C}_{\mathrm{data}} (1+t+|x|)^{-1}(1+|t-|x||)^{-2}.$$
\end{Pro}
\begin{proof}
Recall from Propositions \ref{GSdecompo}, \ref{GSdecomoderiv} the expression of $A^{\mathrm{data}}_{\gamma,k}$ and from Corollaries \ref{estiW}, \ref{estikernels} the bounds on the kernels. Hence, for $(t,x) \in [0,T[ \times \R^3$,
$$\left|A^{\mathrm{data}}_{\gamma,k}\right|(t,x) \lesssim \left|\nabla_{\partial_{x^k}} \mathcal{L}_{Z^{\gamma}}(F)^{\mathrm{hom}}\right|(t,x) +\sum_{|\beta| \leq |\gamma|+1}\min(t^{-1},t^{-2}) \int_{|y-x|=t} \int_{\R^3_v} |v^0|^3 \left| \widehat{Z}^{\beta}f\right|(0,y,v) \dr v \dr y.$$
As $[\partial_{x^\mu},Z]=0$ or $[\partial_{x^\mu},Z]=\pm \partial_{x^\nu}$ for any $Z \in \mathbb{K}$, by the equivalence of the pointwise norms \eqref{equinorm} and in view of the smallness assumptions on the initial data, there holds
\begin{align*}
 \left|\nabla_{\partial_{x^k}} \mathcal{L}_{Z^{\gamma}}(F)^{\mathrm{hom}} \right|(0,y) =\left|\nabla_{\partial_{x^k}} \mathcal{L}_{Z^{\gamma}}(F) \right|(0,y)  \lesssim \sum_{1 \leq |\kappa| \leq N} \langle y \rangle^{|\kappa|-1} \left| \nabla_{t,x}^{\kappa}F \right|(0,y) \lesssim  \Lambda \langle y \rangle^{-3}, \\
 \left|\nabla_{t,x}\nabla_{\partial_{x^k}} \mathcal{L}_{Z^{\gamma}}(F)^{\mathrm{hom}} \right|(0,y) = \left|\nabla_{t,x} \nabla_{\partial_{x^k}} \mathcal{L}_{Z^{\gamma}}(F) \right|(0,y)  \lesssim \sum_{2 \leq |\kappa| \leq N+1} \langle y \rangle^{|\kappa|-2} \left| \nabla_{t,x}^{\kappa}F \right|(0,y) \lesssim  \Lambda \langle y \rangle^{-4}
\end{align*}
As $\nabla_{\partial_{x^k}} \mathcal{L}_{Z^{\gamma}}(F)^{\mathrm{hom}}_{\mu \nu}$ is solution to the homogeneous wave equation, Proposition \ref{decaylinMax} gives
$$ \left|\nabla_{\partial_{x^k}} \mathcal{L}_{Z^{\gamma}}(F)^{\mathrm{hom}}\right|(t,x) \lesssim  \Lambda (1+t+|x|)^{-1}(1+|t-|x||)^{-2}.$$
Since $|v^0|^{-N_v+3} \in L^1(\R^3_v)$, we have for any $|\beta| \leq N$,
$$\int_{\R^3_v} |v^0|^3 \left| \widehat{Z}^{\beta}f\right|(0,y,v) \dr v \lesssim \langle y \rangle^{-N_x} \sup_{|\kappa|+|\xi| \leq N}\sup_{(x,v) \in \R^6} |v^0|^{N_v+|\xi|}\langle x \rangle^{N_x+|\kappa|} \left| \partial_{v}^{\kappa} \partial_x^{\xi} f\right|(0,x,v) \lesssim \epsilon \, \langle y \rangle^{-N_x}.$$
Consequently, as $N_x \geq 5$, we have
$$\left|A^{\mathrm{data}}_{\gamma,k}\right|(t,x) \lesssim  \Lambda (1+t+|x|)^{-1}(1+|t-|x||)^{-2} +\epsilon \min(t^{-1},t^{-2}) \mathcal{Q}_{t,x}, \qquad \mathcal{Q}_{t,x}:= \int_{|y-x|=t} \frac{ \dr \mu_{\mathbb{S}^2}}{\langle y \rangle^{5}}.$$
As $\epsilon \leq \Lambda$, it remains to prove $\min(t^{-1},t^{-2}) \mathcal{Q}_{t,x} \lesssim (1+t+|x|)^{-1}(1+|t-|x||)^{-2}$ and, for this, we consider two cases.
\begin{itemize}
\item If $t\leq 1$, then $|y|\geq |x|-1$ on the domain of integration and $\mathcal{Q}_{t,x} \lesssim 4\pi t^2\langle x \rangle^{-5}$. It remains to remark that $\langle x \rangle \geq 1+t+|x| \geq 1+|t-|x||$ and $t^{-1} \leq t^{-2}$ in the region considered.
\item Otherwise $t \geq 1$ and we have $\mathcal{Q}_{t,x} \lesssim t(1+t+|x|)^{-1}(1+|t-|x||)^{-2}$ according to the estimate \eqref{decayhYang}. The result ensues from $t^{-2} \leq t^{-1}$ in the domain treated here.
\end{itemize}
\end{proof}
Next, we consider the vertex term.
\begin{Pro}\label{Aw}
Let $k \in \llbracket 1, 3 \rrbracket$ and $|\gamma| = N-1$. We have 
$$\forall \, (t,x) \in [0,T[ \times \R^3, \qquad \left|A^{\mathrm{ver}}_{\gamma,k}\right|(t,x) \lesssim \overline{\epsilon} \, (1+t+|x|)^{-3}.$$
\end{Pro}
\begin{proof}
Fix $0 \leq \mu,\nu \leq 3$, $(t,x) \in [0,T[\times \R^3$ and recall $N_v \geq 15$, so that Corollary \ref{estikernels} implies $| \mathcal{D}^k_{\mu\nu}|(\omega,v)+|v^0\nabla_v \mathcal{D}^k_{\mu\nu}|(\omega,v) \lesssim |v^0|^{N_v-11}$. Proposition \ref{Proestimoyvkernet}, applied for $\Psi=\mathcal{D}^k_{\mu \nu}$ and to any $|\beta| \leq N-1$, then yields
$$ |A^{\mathrm{ver}}_{\gamma,k, \, \mu\nu}|(t,x) \lesssim \sum_{|\beta| \leq |\gamma|} \int_{\sigma \in \mathbb{S}^2} \left|\int_{\R^3_v} \mathcal{D}_{\mu \nu}^k(\sigma,v) \widehat{Z}^{\beta} f(t,x,v) \dr v \right| \dr \mu_{\mathbb{S}^2} \lesssim \frac{\overline{\epsilon}}{ (1+t+|x|)^{3}}  \int_{\sigma \in \mathbb{S}^2}  \dr \mu_{\mathbb{S}^2}=\frac{4 \pi \overline{\epsilon}}{ (1+t+|x|)^{3}} .$$
\end{proof}
We now estimate $A^{T,S}_{\gamma,k}$. Note that the next result could be easily improved but it is more than enough for the purpose of improving the bootstrap assumption \eqref{boot2}.
\begin{Pro}\label{ATS}
For any $k \in \llbracket 1,3 \rrbracket$ and $|\gamma|=N-1$, there holds
$$\forall \, (t,x) \in [0,T[ \times \R^3, \qquad \left|A^{T,S}_{\gamma,k}\right|(t,x) \lesssim \overline{\epsilon} \, \Lambda (1+t+|x|)^{-\frac{3}{2}}(1+|t-|x||)^{-2}.$$
\end{Pro}
\begin{proof}
Let $0 \leq \mu,\nu \leq 3$, $(t,x) \in [0,T[\times \R^3$ and recall that $K_{\xi}^j:=\widehat{v}^{\lambda} {\mathcal{L}_{Z^{\xi}}(F)_{\lambda}}^j$. Consequently, $|A^{T,S}_{\gamma,k,\mu\nu}|(t,x)$ is bounded by a linear combination of terms of the form
$$ \mathcal{Q}^{\xi,\kappa}_{t,x}:= \int_{|y-x| \leq t} \! \left| {\mathcal{L}_{Z^{\xi}}(F)_{\lambda}}^j\right|(t-|y-x|,y)  \int_{\R^3_v} \big|\partial_{v^j} \mathcal{B}_{\mu \nu}^k(\omega,v)  \big| \left|\widehat{Z}^{\kappa}f(t-|y-x|,y,v) \right| \mathrm{d}v\frac{  \mathrm{d}y}{|y-x|^2},$$
with $|\xi|+|\kappa| \leq N-1$ and where we recall that $\omega=(y-x)/|y-x|$. Since $|\partial_{v^j} \mathcal{B}^k_{\mu \nu}(\omega,v)| \lesssim |v^0|^3$ by Corollary \ref{estikernels} and $N_v\geq 13$, Proposition \ref{estimoyv}, applied for $a=1$, provides
$$ \int_{\R^3_v} \big|\partial_{v^j} \mathcal{B}^k_{\mu \nu}(\omega,v)  \big| \left|\widehat{Z}^{\kappa}f(t-|y-x|,y,v) \right| \mathrm{d}v \lesssim \frac{\overline{\epsilon}\, (1+|t-|y-x|-|y||)}{(1+t-|y-x|+|y|)^{3+\frac{1}{2}}}.$$
Moreover, as $|\xi| \leq N-1$, the bootstrap assumption \eqref{boot1} gives 
$$\left| {\mathcal{L}_{Z^{\xi}}(F)_{\lambda}}^j\right|(\tau,y)\lesssim \Lambda (1+t-|y-x|+|y|)^{-1}(1+|t-|y-x|-|y||)^{-1}.$$ Consequently, the last two estimates yield
$$ \mathcal{Q}^{\xi,\kappa}_{t,x} \lesssim \overline{\epsilon} \, \Lambda \int_{|y-x| \leq t} (1+t-|y-x|)^{-4-\frac{1}{2}} \frac{  \mathrm{d}y}{|y-x|^2} = \overline{\epsilon} \, \Lambda \mathbf{Y}_{4+1/2}^{p=2}(t,x)$$
and the result follows from Lemma \ref{LemKb}.
\end{proof}
We pursue with the analysis of $A^{SS}_{\gamma,k}$. As for the previous term, the estimate could be improved.
\begin{Pro}\label{ASS}
We have, for any $k \in \llbracket 1,3 \rrbracket$ and $|\gamma|=N-1$,
$$\forall \, (t,x) \in [0,T[ \times \R^3, \qquad \left|A^{SS}_{\gamma,k}\right|(t,x) \lesssim \overline{\epsilon} \, \Lambda \, \langle \Lambda \rangle \, (1+t+|x|)^{-1}(1+|t-|x||)^{-\frac{5}{2}}.$$
\end{Pro}
\begin{proof}
We fix $(t,x) \in [0,T[ \times \R^3$ and we recall that $K_{\xi}^j:=\widehat{v}^{\lambda} {\mathcal{L}_{Z^{\xi}}(F)_{\lambda}}^j$. Recall further from Proposition \ref{GSdecomoderiv} that $A^{SS}_{\gamma,k}$ can be decomposed as the sum of four terms. Bounding the kernel in $A^{SS,I}_{\gamma,k}$ by Corollary \ref{estikernels} and estimating the derivatives of the electromagnetic field using \eqref{boot1}, we have
$$ \left|A^{SS,I}_{\gamma,k}\right|\!(t,x) \lesssim \! \sum_{|\kappa| \leq N-1} \int_{|y-x|\leq t} \frac{ \Lambda^2}{(1+\tau+|y|)^2(1+|\tau-|y||)^2} \int_{\R^3_v}\! |v^0|^3 \! \left| \widehat{Z}^{\kappa} f \right|(\tau,y,v) \dr v \frac{\dr y}{|y-x|}, \quad \tau:=t-|y-x|.$$
For the next two terms, recall that $\T_0=\widehat{v}^{\lambda} \partial_{x^{\lambda}}$ and the expression of $K_{\xi}$. Recall further from Corollary \ref{estikernels} that the integral kernels are bounded by $|v^0|^3$. Consequently, we can bound $|A^{SS,II}_{\gamma,k}|(t,x)+|A^{SS,III}_{\gamma,k}|(t,x)$ by a linear combination of terms of the form
$$ \mathcal{R}^{\xi,\kappa}_{t,x}:= \int_{|y-x| \leq t} \! \left| \nabla_{t,x}\mathcal{L}_{Z^{\xi}}(F)\right|(t-|y-x|,y)  \int_{\R^3_v}|v^0|^3 \left| \widehat{Z}^{\kappa}f(t-|y-x|,y,v) \right| \mathrm{d}v\frac{  \mathrm{d}y}{|y-x|},$$ 
where $|\xi|+|\kappa| \leq N-1$. We estimate the electromagnetic field through \eqref{boot2} if $|\xi|=N-1$ or by Proposition \ref{Prodecaytoporder} if $|\xi|\leq N-2$. This leads to the bound
$$ \left|A^{SS,II}_{\gamma,k}\right|\!(t,x)+\left|A^{SS,III}_{\gamma,k}\right|\!(t,x) \lesssim \! \sum_{|\kappa| \leq N-1} \int_{|y-x|\leq t} \frac{\Lambda \log(3+|\tau-|y||)}{(1+\tau+|y|)(1+|\tau-|y||)^2} \int_{\R^3_v} |v^0|^3 \left| \widehat{Z}^{\kappa} f \right|(\tau,y,v) \dr v \frac{\dr y}{|y-x|},$$
where $\tau=t-|y-x|$. Controlling the velocity average through the improved estimates of Proposition \ref{estimoyv} yields, as $N_v \geq 13$,
\begin{align*}
&\left|A^{SS,I}_{\gamma,k}\right|\!(t,x)+ \left|A^{SS,II}_{\gamma,k}\right|\!(t,x)+\left|A^{SS,III}_{\gamma,k}\right|\!(t,x) \\ &\qquad \qquad \quad  \lesssim \overline{\epsilon} \, \Lambda \, \langle \Lambda \rangle \int_{|y-x|\leq t} \frac{\log^{3N_x+3N+1}(3+t-|y-x|+|y|)}{(1+t-|y-x|+|y|)^5(1+|t-|y-x|-|y||)}  \frac{\dr y}{|y-x|}  \lesssim \overline{\epsilon} \, \Lambda \, \langle \Lambda \rangle \mathbf{Y}^{p=1}_{4+3/4,1}(t,x).
\end{align*}
Finally, we can bound similarly $|A^{SS,IV}_{\gamma,k}|(t,x)$ by a linear combination of terms of the form
$$ \overline{\mathcal{R}}^{\xi,\kappa}_{t,x}:= \int_{|y-x| \leq t} \! \left| \mathcal{L}_{Z^{\xi}}(F)\right|(t-|y-x|,y) \left| \int_{\R^3_v} \mathcal{V}(\omega,v) \partial_{x^n} \widehat{Z}^{\kappa}f(t-|y-x|,y,v) \mathrm{d}v\right|\!\frac{  \mathrm{d}y}{|y-x|},$$ 
where $|\xi|+|\kappa| \leq N-1$, $1 \leq n \leq 3$ and $\mathcal{V}(\omega,v)$ is of the form $\mathcal{C}^{k,n}_{j,\mu \nu}(\omega,v) \widehat{v}^{\lambda}$. We get from Corollary \ref{estikernels} that $|\mathcal{V}(\omega,v)|+|v^0\nabla_v \mathcal{V}(\omega,v)| \lesssim |v^0|^4$, so that Proposition \ref{gainderivdecaymoyv} gives
$$ \forall \, (\tau,y,\sigma)\in [0,T[\times \R^3 \times \mathbb{S}^2, \qquad \Big| \int_{\R^3_v} \mathcal{V}(\sigma,v)  \partial_{x^n}\widehat{Z}^{\kappa}f(\tau,y,v) \mathrm{d}v\Big| \lesssim \overline{\epsilon} \, \frac{\log^{3N_x+3N}(3+\tau)}{(1+\tau+|y|)^4}.$$
Applying it to $\sigma=\omega$ and $\tau=t-|y-x|$ and estimating the electromagnetic field using \eqref{boot1}, we get
$$\overline{\mathcal{R}}^{\xi,\kappa}_{t,x}\lesssim \overline{\epsilon} \, \Lambda \int_{|y-x| \leq t} \frac{\log^{3N_x+3N}(3+t-|y-x|)}{(1+t-|y-x|+|y|)^5(1+|t-|y-x|-|y||)}\frac{  \mathrm{d}y}{|y-x|}\lesssim \overline{\epsilon} \, \Lambda \mathbf{Y}^{p=1}_{4+3/4,1}(t,x).$$
Consequently, $|A^{SS}_{\gamma,k}|(t,x) \lesssim \overline{\epsilon} \, \Lambda \, \langle \Lambda \rangle \, \mathbf{Y}^{p=1}_{4+3/4,1}(t,x)$, so that the result ensues from Lemma \ref{LemJ}.
\end{proof}
Finally, we deal with the most problematic term, the one with an integral kernel carrying the non-integrable weight $|y-x|^{-3}$.
\begin{Pro}\label{ATT}
Let $k \in \llbracket 1,3 \rrbracket$ and $|\gamma|=N-1$. Then,
$$\forall \, (t,x) \in [0,T[ \times \R^3, \qquad \left|A^{TT}_{\gamma,k}\right|(t,x) \lesssim \overline{\epsilon} \, \frac{\log(3+t)}{(1+t+|x|)^3}.$$
\end{Pro}
\begin{proof}
Let $0 \leq \mu,\nu \leq 3$, $|\beta| \leq N-1$ and
$$ G^{\beta}_{\sigma}(\tau,y):= \int_{\R^3} \mathcal{A}_{\mu \nu}^k (\sigma,v) \widehat{Z}^{\beta}f(\tau,y,v) \dr v, \qquad  (\sigma, \tau , y ) \in \mathbb{S}^2\times [0,T[ \times \R^3.$$
Recall from Corollary \ref{estikernels} the bound on the kernel $\mathcal{A}^k_{\mu \nu}$ and apply Proposition \ref{Proestimoyvkernet} for $\Psi= \mathcal{A}^k_{\mu \nu}$. We obtain
$$ \forall \, (\sigma, \tau , y ) \in \mathbb{S}^2\times [0,T[ \times \R^3, \qquad \qquad \left|G^{\beta}_{\sigma}\right|(\tau,y) \lesssim \overline{\epsilon} \, (1+\tau+|y|)^{-3},$$
which, applied for $(\sigma,\tau)=(\omega,t-|y-x|)$, yields
$$ \left|A^{TT}_{\gamma,k,\mu \nu} \right|(t,x) \lesssim \overline{\epsilon} \, \mathbf{Y}^{p=3}_3(t,x)+\sum_{|\beta| \leq N-1} \mathcal{U}^\beta_{t,x}, \qquad \mathcal{U}^\beta_{t,x}:= \left|\int_{|y-x|\leq 1 } \int_{\R^3_v} \mathcal{A}^k_{\mu \nu}(\omega,v) \widehat{Z}^{\beta} f(t-|y-x|,y,v)   \frac{\dr v \dr y}{|y-x|^3} \right|.$$
Fix $|\beta| \leq N-1$ and recall from Proposition \ref{GSdecomoderiv} that the average of $\sigma \mapsto \mathcal{A}^k_{\mu \nu}(\sigma,\cdot)$ on $\mathbb{S}^2$ vanishes. Hence,
\begin{align*}
\mathcal{U}^\beta_{t,x}&=\left| \int_{|y-x|\leq 1 } \int_{\R^3_v} \mathcal{A}^k_{\mu \nu}(\omega,v) \left(\widehat{Z}^{\beta} f(t-|y-x|,y,v) -\widehat{Z}^{\beta}f(t-|y-x|,x,v) \right) \frac{\dr v \dr y}{|y-x|^3}\right| \\
& \leq \int_{|y-x|\leq 1 } \left|G^{\beta}_{\omega}(t-|y-x|,y)-G^{\beta}_{\omega}(t-|y-x|,x) \right| \frac{ \dr y}{|y-x|^3}.
\end{align*}
For any $(\sigma,\tau)\in \mathbb{S}^2 \times [0,T[$, we apply the mean value theorem to $s \mapsto  G^{\beta}_{\sigma}(\tau,x+s(y-x))$ on the interval $[0,1]$. Then, there exists $x_{\sigma,\tau}$ in the segment $[x,y]\subset \R^3$ such that
$$ G^{\beta}_{\sigma}(\tau,y)-G^{\beta}_{\sigma}(\tau,x) =\omega \cdot \nabla_{x} G^{\beta}_{\sigma}(\tau,x_{\sigma,\tau}) |y-x|, \qquad \omega=\frac{y-x}{|y-x|}.$$
Apply now Proposition \ref{gainderivdecaymoyv} for $\Phi=\mathcal{A}_{\mu \nu}$ in order to get, for any $1 \leq i \leq 3$,
$$ \forall \, (\sigma, \tau , z ) \in \mathbb{S}^2\times [0,T[ \times \R^3,  \qquad \left|\partial_{x^i}G^{\beta}_{\sigma}\right|(\tau,z)= \bigg| \int_{\R^3} \mathcal{A}^k_{\mu \nu} (\sigma,v) \partial_{x^i}\widehat{Z}^{\beta}f(\tau,z,v) \dr v \bigg|\lesssim \overline{\epsilon} \, \frac{\log^{3N_x+3N}(3+\tau)}{(1+\tau+|z|)^{4}}.$$
Applying the last two identities for $\sigma=\omega$, $\tau=t-|y-x|$ and $z=x_{\sigma,\tau}$ yields
$$ \mathcal{U}^\beta_{t,x} \lesssim  \int_{|y-x|\leq 1 } \frac{\overline{\epsilon}}{(1+t-|y-x|+|x_{\omega,t-|y-x|}|)^3} \frac{ \dr y}{|y-x|^2}.$$
As $|y-x| \leq 1$ and $x_{\omega,t-|y-x|} \in [x,y]$, we have $1+t-|y-x| \geq \frac{1}{2}(1+t)$ and $|x_{\omega,t-|y-x|}| \geq |x|-1$, so that
$$ \left|A^{TT}_{\gamma,k,\mu \nu} \right|(t,x) \lesssim \overline{\epsilon} \, \mathbf{Y}^{p=3}_3(t,x)+\overline{\epsilon} \, (1+t+|x|)^{-3}.$$
We conclude the proof by applying Lemma \ref{LemKbis}.
\end{proof}
As in the previous subsection, if $C_\mathrm{boot}$ is chosen such that $C_\mathrm{boot}\geq 2\overline{C}_\mathrm{data}$ and if $\epsilon$ is small enough, we can improve the bootstrap assumption \eqref{boot2} for the spatial derivatives $\nabla_{\partial_{x^k}} \mathcal{L}_{Z^\gamma}(F)$, with $1 \leq k \leq 3$ and $|\gamma|=N-1$, by applying Propositions \ref{estidataA}-\ref{ATT}. The time derivative can then be controlled using 
$$\big|\nabla_{\partial_t} \mathcal{L}_{Z^\gamma}(F) \big| \lesssim \sum_{1 \leq k \leq 3} \big|\nabla_{\partial_{x^k}} \mathcal{L}_{Z^\gamma}(F)\big|+\sum_{|\beta| \leq |\gamma|} |J(\widehat{Z}^\beta f)|,$$ which follows from the commuted Maxwell equations (see Proposition \ref{Com}). We stress, for the smallness condition on $\epsilon$, that $\overline{\epsilon} \, \langle \Lambda \rangle^2 \leq 2 \epsilon \, e^{(D+3)\Lambda}$.

\section{Modified scattering for the distribution function}\label{Sec6}
In this section, we determine the precise asymptotic behavior of the particle density and its derivatives under the additional assumption \eqref{assumasymptoexp} on the initial electromagnetic field. In particular, we determine the self-similar profile of the currect density $J(f)$ as well as the one of the Maxwell field $F$ and we define modified trajectories along which $f$ converges to a new smooth density function.
\subsection{Convergence of the spatial averages}
Since the solution $(f,F)$ is global in time, all the statements of Sections \ref{Sec3}-\ref{Sec5} hold true for $T=+\infty$. We can then deduce that $\int_x \widehat{Z}^\beta f \dr x$ converges to a function defined on $\R^3_v$.
\begin{Pro}\label{ProQinf}
Let $|\beta| \leq N-1$. There exists a continuous function $ Q_{\infty}^{\beta} \in L^1_v \cap L^{\infty}_v $ such that
$$\forall \, t \in \R_+, \qquad \bigg\|   |v^0|^{N_v-6}\bigg(Q_{\infty}^\beta- \int_{\R^3_x} \widehat{Z}^{\beta}f(t,x,\cdot) \dr x \bigg) \bigg\|_{L^{\infty}(\R^3_v)} \lesssim \overline{\epsilon} \, \frac{\log^{3N_x+3N}(3+t)}{1+t} .$$
\end{Pro}
\begin{Rq}
This estimate directly implies that $|v^0|^{N_v-10}\int_{\R^3_x} \widehat{Z}^{\beta}f(t,x,\cdot) \dr x \to |v^0|^{N_v-10}Q^\beta_{\infty}$ in $L^1(\R^3_v)$, as $t \to +\infty$, with the same rate for convergence. 
\end{Rq}
\begin{proof}
Let $v\in \R^3_v$ and apply Lemma \ref{Lempartieltxave} in order to get, for all $0 \leq t \leq s$,
$$|v^0|^{N_v-6}\left|\int_{\R^3_x} \widehat{Z}^{\beta}f(s,x,v) \dr x-\int_{\R^3_x} \widehat{Z}^{\beta}f(t,x,v) \dr x\right| \lesssim \overline{\epsilon} \int_{\tau=t}^s \frac{\log^{3N_x+3N}(3+\tau)}{(1+\tau)^2}\dr \tau \leq \overline{\epsilon} \, \frac{\log^{3N_x+3N}(3+t)}{1+t}.$$
Consequently, there exists $Q_{\infty}^{\beta} \in  L^{\infty}_v $ such that $\int_{\R^3_x} \widehat{Z}^{\beta}f(s,x,v) \dr x \to Q_{\infty}^{\beta}$ in $L^{\infty}_v$ as $s \to +\infty$. Moreover, letting $s \to +\infty$ in the previous estimate provides the rate of convergence stated in the Proposition. It implies $|v^0|^{N_v-6}Q^{\beta}_{\infty} \in L^{\infty}_v$ and then, as $N_v >9$, $Q_{\infty}^{\beta} \in  L^1_v$.
\end{proof}
It turns out that these functions are differentiable for $|\beta| \leq N-2$ and that $\partial_{v^i} Q^{\beta}_{\infty}$ can be related to other such functions $Q^\kappa_\infty$. For this reason, if $\widehat{Z}^\kappa = \widehat{\Omega}_{0i} \widehat{Z}^\beta$, we will use $Q^{\widehat{\Omega}_{0i} \beta}_\infty$ in order to denote $Q^\kappa_\infty$.
\begin{Pro}\label{derivVQ}
For any $|\beta| \leq N-2$, $Q^\beta_\infty \in C^{N-1-|\beta|}(\R^3_v)$ and its derivatives can be obtained by iterating the relations
\begin{equation}\label{lienderiv} v^0 \partial_{v^i} Q^{\beta}_\infty = Q^{\widehat{\Omega}_{0i} \beta}_\infty-\widehat{v}^iQ^\beta_\infty, \qquad 1 \leq i \leq 3.
\end{equation}
\end{Pro}
\begin{proof}
Let $(t,v) \in \R_+ \times \R^3_v$ and remark that
$$ v^0 \partial_{v^i}\int_{\R^3_x} \widehat{Z}^\beta f(t,x,v) \dr x = \int_{\R^3_x} \widehat{\Omega}_{0i} \widehat{Z}^\beta f(t,x,v) \dr x-t\int_{\R^3_x}\partial_{x^i} \widehat{Z}^\beta f(t,x,v) \dr x-\int_{\R^3_x}x^i \partial_t\widehat{Z}^\beta f(t,x,v) \dr x.$$
Writing $\partial_t=-\widehat{v} \cdot \nabla_x-\widehat{v}^\mu {F_{\mu}}^j \partial_{v^j}+\T_F$, we get by performing integration by parts,
$$ v^0 \partial_{v^i}\int_{\R^3_x} \widehat{Z}^\beta f(t,x,v) \dr x = \int_{\R^3_x} \widehat{\Omega}_{0i} \widehat{Z}^\beta f(t,x,v) -\widehat{v}^i \widehat{Z}^\beta f(t,x,v) \dr x+\int_{\R^3_x}x^i\left(\widehat{v}^\mu {F_{\mu}}^j \partial_{v^j}-\T_F \right)\left( \widehat{Z}^\beta f \right)(t,x,v) \dr x .$$
According to the previous Proposition \ref{ProQinf}, the first term on the right hand side converges to $Q^{\widehat{\Omega}_{0i} \beta}_\infty-\widehat{v}^iQ^\beta_\infty$, as $t \to +\infty$ and in $L^\infty(\R^3_v)$. Following the proof of Lemma \ref{Lempartieltxave} and then using Proposition \ref{estiLinfini}, one can prove
\begin{align*}
\left| \int_{\R^3_x}x^i\left(\widehat{v}^\mu {F_{\mu}}^j \partial_{v^j}-\T_F \right)\left( \widehat{Z}^\beta f \right)(t,x,v) \dr x \right| & \lesssim \Lambda \frac{\log(3+t)}{1+t}\sup_{|\kappa| \leq |\beta|+1} \sup_{x \in \R^3}  \left||v^0|^3 \mathbf{z}^{N_x-2} \widehat{Z}^{\kappa} f \right|(t,x,v) \\
& \lesssim \overline{\epsilon} \, \frac{\log^{3N_x+3N}(3+t)}{1+t} .
 \end{align*}
We then deduce \eqref{lienderiv} and, by a direct induction, $Q^\beta_\infty \in C^{N-1-|\beta|}(\R^3_v)$.
\end{proof}
Let us mention that any $Q^\beta_\infty$ can be written as a combination of $Q_\infty$ and $Q^\kappa_\infty$, where $\widehat{Z}^\kappa$ is only composed by complete lifts of Lorentz boosts $\widehat{\Omega}_{0i}$.
\begin{Pro}\label{ProformQ}
Let $|\beta| \leq N-1$. Then,
\begin{itemize}
\item if $\beta_T \geq 1$, which means that $\widehat{Z}^\beta$ is composed by at least one translation, we have $Q^\beta_\infty=0$.
\item Otherwise there exists $n+|\kappa| \leq |\beta|$ such that $\widehat{Z}^\beta=S^n \widehat{Z}^\kappa$ and $Q^\beta_\infty=(-3)^n Q^\kappa_\infty$.
\item Moreover, if $\widehat{Z}^\beta=\widehat{\Omega}_{jk} \widehat{Z}^\kappa$, $1 \leq j< k \leq 3$, then $Q^\beta_\infty=\widehat{v}^j Q^{ \widehat{\Omega}_{0k}\kappa}-\widehat{v}^k Q^{\widehat{\Omega}_{0j}\kappa}$.
\end{itemize}
\end{Pro}
\begin{proof} Assume first that $\beta_T \geq 1$. Since $[\widehat{Z},\partial_{x^\mu}]=0$ or $\pm \partial_{x^\nu}$ for any $0 \leq \mu \leq 3$ and $\widehat{Z} \in \widehat{\mathbb{P}}_0$, it suffices to consider the case $\widehat{Z}^\beta=\partial_{x^\mu} \widehat{Z}^\xi$. Then, by either applying Lemma \ref{Lempartieltxave} or by performing integration by parts,
\begin{equation*}
 \left|\int_{\R^3_x} \partial_t \widehat{Z}^{\xi} f(t,x,v) \dr x \right| \lesssim \overline{\epsilon} (1+t)^{-\frac{3}{2}} \to 0, \qquad \int_{\R^3_x} \partial_{x^i} \widehat{Z}^\xi f(t,x,v) \dr x=0, \qquad 1 \leq i \leq 3.
 \end{equation*}
Otherwise $\beta_T=0$ and since $S$ commute with $\widehat{\Omega}_{jk}$ and $\widehat{\Omega}_{0i}$, there exists $n+|\kappa| \leq |\beta|$ such that $\widehat{Z}^\beta=S^n \widehat{Z}^\kappa$. The result follows from an easy induction and the following properties, which hold for any $|\xi | \leq N-2$,
\begin{equation*}
\left|\int_{\R^3_x}t \partial_t \widehat{Z}^{\xi} f(t,x,v) \dr x \right| \lesssim \overline{\epsilon} (1+t)^{-\frac{1}{2}} \to 0, \qquad \int_{\R^3_x} x_i\partial_{x^i} \widehat{Z}^\xi f(t,x,v) \dr x=-\int_{\R^3_x}\widehat{Z}^\xi f(t,x,v) \dr x, \qquad 1 \leq i \leq 3.
\end{equation*}
 Finally, if $\widehat{Z}^\beta=\widehat{\Omega}_{jk} \widehat{Z}^\kappa$, note that by integration by parts, $\int_x \widehat{Z}^\beta f \dr x= \widehat{v}^j\int_x v^0\partial_{v^k} \widehat{Z}^\kappa f \dr x-\widehat{v}^k \int_x v^0 \partial_{v^j}  \widehat{Z}^\kappa f \dr x$ and it remains to apply the previous Proposition \ref{derivVQ}.
\end{proof}
We are now able to establish the precise behavior of $J(f)$ in the interior of the light cone. In other words, we improve Corollary \ref{Corboo3}. No such result holds for the exterior region since the decay can be arbitrary fast (we refer for this to the third estimate of Proposition \ref{estimoyv}). Recall the notation $x^0=t$.
\begin{Pro}\label{ProprifilJ}
For any $|\beta| \leq N-1$, the components of the electric current density $J(\widehat{Z}^{\beta}f)$, that is $J^{\mu}(\widehat{Z}^{\beta}f)=\int_{\R^3_v} \frac{v^\mu}{v^0} \widehat{Z}^{\beta}f \mathrm{d} v $, satisfy
$$\forall \; |x| <t, \qquad \bigg| t^3J^{\mu}(\widehat{Z}^{\beta}f)(t,x) - \frac{x^{\mu}}{t}\left( |v^0|^5 Q^{\beta}_{\infty}\right) \left( \widecheck{\frac{x}{t}} \right) \bigg| \lesssim \overline{\epsilon} \, \frac{\log^{3N_x+3N}(3+t)}{t}, \qquad \mu \in \llbracket 0,3 \rrbracket. $$
\end{Pro}
\begin{proof}
Let $|\beta| \leq N-1$, $0 \leq \mu \leq 3$ and $|x| < t$. Apply Lemma \ref{Lemgmoy} and the estimate \eqref{eq:estigomega} to $g(t,x,v):=\widehat{v}^\mu \widehat{Z}^{\beta}f(t,x+t\widehat{v},v)$. Since the spatial average of $|v^0|^5g$ is equal to the one of $\widehat{v}^{\mu}|v^0|^5\widehat{Z}^{\beta}f$, we get
\begin{equation}\label{eq:machintrucbidule}
\bigg| t^3\int_{\R^3_v} \frac{v^{\mu}}{v^0} \widehat{Z}^{\beta}f(t,x,v) \dr v - \int_{\R^3_y} \left( \frac{v^{\mu}}{v^0} |v^0|^5\widehat{Z}^{\beta} f \right) \left(t,y, \widecheck{\frac{x}{t}} \right) \dr y \bigg| \lesssim \overline{\epsilon} \, \frac{\log^{3N_x+3N}(3+t)}{t}.
\end{equation}
As $N_v-6 \geq 5$, we obtain from Proposition \ref{ProQinf} that
$$\forall \, v \in \R^3_v, \qquad \bigg|   \frac{v^{\mu}}{v^0} |v^0|^5Q_{\infty}^\beta(v)- \frac{v^{\mu}}{v^0} |v^0|^5\int_{\R^3_y} \widehat{Z}^{\beta}f(t,y,v) \dr y \bigg| \lesssim \overline{\epsilon} \, \frac{\log^{3N_x+3N}(3+t)}{1+t}.$$
The result follows from \eqref{eq:machintrucbidule} and the last estimate, applied for $v=\frac{\widecheck{x}}{t}$.
\end{proof}
\subsection{Self-similar asymptotic profile of the electromagnetic field}\label{Subsecmachin}
To identify the profile of $F$, we will see that $Q_\infty$ generates an effective electromagnetic field. For this, we study $F^T$ since it is the element of the Glassey-Strauss decomposition of $F$ with the slower decay rate along timelike geodesics $t\mapsto (t,x+t\widehat{v})$. If the plasma is not neutral, $Q_F \neq 0$, we will also have to improve the estimate for $F^{\mathrm{data}}$.
\subsubsection{Behavior of $\mathcal{L}_{Z^\gamma}(F)^T$ along timelike straight lines} It will be convenient to lighten the notations by denoting the kernel in the integral defining $F^T$, which was bounded in Corollary \ref{estiW}, as
\begin{equation}\label{eq:Wt}
 \mathcal{W}^T(\omega,v):=\frac{\mathcal{W}(\omega,v)}{|v^0|^2(1+\omega \cdot \widehat{v})}, \qquad \left|\mathcal{W}^T\right|(\cdot,v)+\left|\nabla_v \mathcal{W}^T\right|(\cdot,v) \lesssim v^0.
\end{equation}
\begin{Def}\label{Defexplici}
Let, for any $|\beta| \leq N-1$, $[\widehat{Z}^{\beta}f]^{\infty}(v)$ be the $2$-form defined as
$$\forall \, v \in \R^3_v, \qquad \left[\widehat{Z}^{\beta}f \right]^{\infty}(v) := \int_{\substack{|z| \leq 1 \\  |z+\widehat{v}| < 1-|z|  }} \mathcal{W}^T \! \left(\frac{z}{|z|},\frac{\widecheck{z+\widehat{v}}}{1-|z|} \right) \left( |v^0|^5Q^{\beta}_{\infty} \right)\left(\frac{\widecheck{z+\widehat{v}}}{1-|z|} \right) \frac{\dr z}{|z|^2(1-|z|)^3}.$$
\end{Def}
\begin{Rq}
We recall our convention $(|v^0|^5Q^\beta_\infty)(w):=|w^0|^5Q^\beta_\infty(w)$, for any $w \in \R^3_v$. 
\end{Rq}
\begin{Rq}\label{welldef}
It is crucial to observe that the domain of integration is included in $\{0 \leq |z| \leq \frac{1+|\widehat{v}|}{2}\}$. Indeed, if $|z| \geq \frac{1+|\widehat{v}|}{2}$, we have
$$ |z+\widehat{v}| \geq |z|-1+1-|\widehat{v}| \geq \frac{1-|\widehat{v}|}{2} \geq 1-|z|.$$
Consequently, 
\begin{equation*}
|z| \leq 1, \, |z+\widehat{v}| <1-|z| \quad \Rightarrow \quad  \frac{1}{4|v^0|^2} \leq \frac{1-|\widehat{v}|}{2} \leq  1-|z| \leq 1.
 \end{equation*}
\end{Rq}
In order to transform decay in $|t-r|$ into decay in $t$ along timelike trajectories, we will use the next property.
\begin{Lem}\label{Lemxvt}
Let $(x,v) \in \R^3_x \times \R^3_v$. Then,
$$ \forall \, 1 \leq t \leq 4 \langle x \rangle|v^0|^2, \quad 1 \leq 4\frac{\langle x \rangle |v^0|^2}{t}, \qquad \qquad \forall \, t \geq  4 \langle x \rangle|v^0|^2, \qquad t-|x+t\widehat{v}| \geq \frac{t}{4|v^0|^2}.$$
\end{Lem}
\begin{proof} It suffices to observe that
\begin{equation*}
 \forall \, t \geq 4|x||v^0|^2, \qquad t \geq \frac{2|x|}{1-|\widehat{v}|}, \qquad \text{so that} \qquad  t-|x+t\widehat{v}| \geq t-\frac{1-|\widehat{v}|}{2}t-|\widehat{v}|t = t-\frac{1+|\widehat{v}|}{2}t \geq \frac{t}{4|v^0|^2}.
\end{equation*}
\end{proof}
We have the following convergence result.
\begin{Pro}\label{Proinducedfield}
Let $|\beta| \leq N-1$ and $(x,v) \in \R^3_x \times \R^3_v$. For all $t \geq 1$, there holds
$$ \left| t^2\big[\widehat{Z}^\beta f\big]^T(t,x+\widehat{v}t) - \big[\widehat{Z}^\beta f\big]^{\infty}(v) \right| \lesssim \overline{\epsilon} \, \langle x \rangle^2|v^0|^8 \frac{\log^{3N_x+3N+1}(3+t)}{t}.$$
\end{Pro}
\begin{proof}
Fix $|\beta| \leq N-1$, $(t,x,v) \in [1,+\infty[ \times \R^3_x \times \R^3_v$ and recall from Proposition \ref{GSdecompo} the definition of $[\widehat{Z}^\beta f]^T$. Next, we split the domain of integration of $[\widehat{Z}^\beta f]^T$ into two parts,
\begin{align*}
t^2\big[\widehat{Z}^\beta f\big]^T(t,x+\widehat{v}t)&= t^2 \int_{\substack{|y-x-t\widehat{v}|\leq t \\ |y-x| \geq t-|y-x-t\widehat{v}|}} \int_{\R^3_w} \mathcal{W}^T\!\left(\frac{y-x}{|y-x|},w\right) \widehat{Z}^{\beta}f\big(t-|y-x-t\widehat{v}|,y,w\big) \frac{\mathrm{d}w \mathrm{d}y}{|y-x-t\widehat{v}|^2}+\mathcal{J}, \\
\mathcal{J}&:=\int_{\substack{|z| \leq 1 \\ |z+\widehat{v}| < 1-|z|} } \int_{\R^3_w} \mathcal{W}^T\!\left(\frac{z}{|z|},w\right) \widehat{Z}^{\beta}f\big(t(1-|z|),x+tz+t\widehat{v},w\big) \mathrm{d}w \frac{ t^3\mathrm{d}z}{|z|^2},
\end{align*}
where we performed the change of variables $z=(y-x-t\widehat{v})/t$ in order to obtain the second integral $\mathcal{J}$. As we shall see below, this splitting is useful in order to identify and isolate the asymptotic profile. We start by controlling the first term. For this, note that \eqref{eq:Wt}, $N_v \geq 10$ and the last two estimates of Proposition \ref{estimoyv}, applied for $a=1$, yields, for all $(\omega,\tau,y) \in \mathbb{S}^2\times \R_+\times\R^3$,
$$ \bigg|\int_{\R^3_w} \mathcal{W}^T\!\left(\omega,w\right) Z^{\beta}f\big(\tau,y,w\big)\dr w\bigg| \lesssim \int_{\R^3_w} w^0 \left| Z^{\beta}f \right|\big(\tau,y,w\big) \dr w \lesssim \overline{\epsilon} \, \log^{3N_x+3N}(3+\tau) \frac{1+\max(\tau-|y|,0)}{(1+\tau+|y|)^4}.$$
Remark now that $|y-x| \geq t-|y-x-t\widehat{v}|$ implies
$$t-|y-x-t\widehat{v}|-|y| \leq t-|y-x-t\widehat{v}|-|y-x|+ |x|\leq |x|.$$
Hence, applying first the previous estimate for $\tau=t-|y-x-t\widehat{v}|$ and then \eqref{eq:gaintminusr}, we get
\begin{align*}
 \left|t^2\big[\widehat{Z}^\beta f\big]^T(t,x+\widehat{v}t) - \mathcal{J} \right| & \lesssim \overline{\epsilon} \, (1+|x|) t^2 \int_{\substack{|y-x-t\widehat{v}|\leq t \\ |y-x| \geq t-|y-x-t\widehat{v}|}}  \frac{\log^{3N_x+3N}(3+t-|y-x-t\widehat{v}|)}{(1+t-|y-x-t\widehat{v}|+|y|)^4}\frac{ \mathrm{d}y}{|y-x-t\widehat{v}|^2} \\
 & \lesssim \overline{\epsilon} \, \langle x \rangle \, \frac{\log^{3N_x+3N}(3+|t-|x+t\widehat{v}||)}{1+|t-|x+t\widehat{v}||} t^2 \, \mathbf{Y}^{p=2}_{3}(t,x+t\widehat{v}) .
\end{align*}
According to Lemma \ref{LemKb}, $t^2 \mathbf{Y}^{p=2}_3(t,x+t\widehat{v})\lesssim \log(1+t)$. By applying Lemma \ref{Lemxvt}, we then deduce
\begin{align*}
 \left|t^2\big[\widehat{Z}^\beta f\big]^T(t,x+\widehat{v}t) - \mathcal{J} \right| & \lesssim  \overline{\epsilon} \, \langle x \rangle \log(1+t) \left( \frac{\langle x \rangle|v^0|^2}{1+t}+|v^0|^2\frac{\log^{3N_x+3N}(3+t)}{1+t}\right) ,
\end{align*}
so that it remains us to compare $\mathcal{J}$ with $[\widehat{Z}^{\beta}f]^\infty(v)$. As in Section \ref{subsecaveragev}, it is convenient to change the reference frame and work with $g^\beta(\tau,y,w):=\widehat{Z}^{\beta}f(\tau,y+\tau\widehat{w},w)$. In view of Lemma \ref{cdv}, the change of variables $y=x+tz+\widehat{v}t-\widehat{w}t(1-|z|)$, for $z$ fixed, leads to
$$ \mathcal{J}=\int_{\substack{|z| \leq 1 \\  |z+\widehat{v}| < 1-|z|  }} \int_{\left|x-y+tz+\widehat{v}t\right| < t(1-|z|)}\mathcal{W}^T\!\left(\frac{z}{|z|},w\right)\, \big(|v^0|^5g^\beta \big)\big(t(1-|z|),y,w \big) \frac{\mathrm{d} y \mathrm{d}z}{|z|^2(1-|z|)^3},$$
where we used $w$ to denote the following function of $(y,z)$,
$$ w=\frac{\widecheck{x-y+tz+t\widehat{v}}}{t(1-|z|)} \Leftrightarrow \widehat{w}=\frac{x-y+tz+t\widehat{v}}{t(1-|z|)}.$$
By the triangular inequality, we have $ |\mathcal{J}-[\widehat{Z}^{\beta}f]^\infty| \leq \mathcal{J}_1+\mathcal{J}_2+\mathcal{J}_3$, where

\begin{align*}
\mathcal{J}_1 & := \int_{\substack{|z| \leq 1 \\  |z+\widehat{v}| < 1-|z|  }} \int_{\left|x-y+tz+t\widehat{v}\right| < t(1-|z|) }\left|\Delta_1^\beta \right| \frac{\mathrm{d} y \mathrm{d}z}{|z|^2(1-|z|)^3}, \\
\Delta_1^\beta & := \mathcal{W}^T\!\left(\frac{z}{|z|},w\right)\big(|v^0|^5g^\beta \big) \big(t(1-|z|),y,w \big)- \mathcal{W}^T\!\left( \frac{z}{|z|}, \frac{\widecheck{z+\widehat{v}}}{1-|z|} \right) \big(|v^0|^5g^\beta \big) \!\left(t(1-|z|),y,\frac{\widecheck{z+\widehat{v}}}{1-|z|} \right), \\
\mathcal{J}_2 & := \left|\int_{\substack{|z| \leq 1 \\  |z+\widehat{v}| < 1-|z|  }} \int_{\left|x-y+tz+t\widehat{v}\right| \geq t(1-|z|) }  \mathcal{W}^T\!\left( \frac{z}{|z|}, \frac{\widecheck{z+\widehat{v}}}{1-|z|} \right) \!\big(|v^0|^5g^\beta \big)\! \left(t(1-|z|),y,\frac{\widecheck{z+\widehat{v}}}{1-|z|} \right) \frac{\mathrm{d} y \mathrm{d}z}{|z|^2(1-|z|)^3} \right|\!, \\
\mathcal{J}_3 & := \int_{\substack{|z| \leq 1 \\  |z+\widehat{v}| < 1-|z|  }}\left| \Delta_3^{\beta} \right|\, \frac{ \mathrm{d}z}{|z|^2(1-|z|)^3} ,\\
\Delta^\beta_3 & :=  \mathcal{W}^T \! \left( \frac{z}{|z|}, \frac{\widecheck{z+\widehat{v}}}{1-|z|} \right) \left[ \int_{\R^3_y } \big(|v^0|^5\widehat{Z}^{\beta} f \big)\left(t(1-|z|),y,\frac{\widecheck{z+\widehat{v}}}{1-|z|} \right)\mathrm{d}y -  \left( |v^0|^5Q^{\beta}_{\infty} \right)\left(\frac{\widecheck{z+\widehat{v}}}{1-|z|} \right)  \right],
\end{align*}
where, for $\Delta_3^\beta$, we used that the spatial average of $g^\beta$ is equal to the one of $\widehat{Z}^{\beta}f$. In view of Remark \ref{welldef}, we will be able to transform time decay for the integrands of $\mathcal{J}_i$ into decay in $t$, at the cost of powers of $v^0$. In particular, Remark \ref{welldef} and $N_x>7$ imply the following inequality that we will use several times,
\begin{equation}\label{eq:easybound}
 \int_{\substack{|z| \leq 1 \\  |z+\widehat{v}| < 1-|z|  }} \int_{\R^3_y} \frac{\mathrm{d} y}{\langle y \rangle^{N_x-4}} \frac{\mathrm{d} z}{|z|^2(1-|z|)^n} \lesssim  \int_{\substack{|z| \leq 1 \\  |z+\widehat{v}| < 1-|z|  }} \frac{\mathrm{d} z}{|z|^2(1-|z|)^n} \leq 2^{2n+2} \pi |v^0|^{2n}, \qquad n \in \mathbb{N} .
\end{equation}
We start by dealing with $\mathcal{J}_1$. Since $|\nabla_V \widecheck{V}| \lesssim (1-|V|^2)^{-3/2} = |\widecheck{V}^0|^3$ for all $|V|<1$ by Lemma \ref{cdv} and in view of the bounds \eqref{eq:Wt} on $\mathcal{W}^T$, the mean value theorem yields
\begin{align*}
\left| \Delta_1^\beta \right|  &\leq  \frac{|x-y|}{t(1-|z|)}\sup_{V \in \R^3}  |V^0|^{9}\left(|g^\beta|+|\nabla_v g^{\beta}| \right)(t(1-|z|),y,V) \\
& \leq \frac{1+|x|}{t(1-|z|)\langle y \rangle^{N_x-4}}\sup_{(X,V) \in \R^6}  |V^0|^{9} \langle X \rangle^{N_x-3}|\left(|g^\beta|+|\nabla_v g^{\beta}| \right)(t(1-|z|),X,V) .
 \end{align*}
By applying Lemma \ref{gweightvderiv} and then the estimates of Proposition \ref{estiLinfini}, we obtain
$$\left| \Delta_1^\beta \right|  \leq  \frac{\langle x \rangle}{t(1-|z|)\langle y \rangle^{N_x-4}} \sum_{|\kappa| \leq N}\sup_{(X,V) \in \R^6}  |V^0|^{9} \left| \mathbf{z}^{N_x-2} \widehat{Z}^{\kappa} f \right|(t(1-|z|),X,V) \lesssim  \frac{\overline{\epsilon} \, \langle x \rangle \, \log^{3N_x+3N}(3+t)}{t(1-|z|)\langle y \rangle^{N_x-4}}   ,$$
where we used $N_v \geq 12$ and $|\beta|+1 \leq N$. We then deduce from \eqref{eq:easybound} that
$$ \mathcal{J}_1 \lesssim \overline{\epsilon} \, \langle x \rangle|v^0|^8\frac{\log^{3N_x+3N}(3+t)}{t} .$$
Next, we control $\Delta_3^\beta$ using $|\mathcal{W}^T|(\cdot,V) \lesssim V^0$, $N_v \geq 12$ and Proposition \ref{ProQinf}. This allows us to bound $\mathcal{J}_3$ through \eqref{eq:easybound},  
$$\Delta_3^{\beta} \lesssim \overline{\epsilon} \, \frac{\log^{3N_x+3N}(3+t)}{(1+t)(1-|z|)}, \qquad \qquad \mathcal{J}_3 \lesssim  \overline{\epsilon} \, |v^0|^8 \frac{\log^{3N_x+3N}(3+t)}{1+t}.$$
Finally, remark that on the domain of integration of $\mathcal{J}_2$, we have, for $\widehat{w}= \frac{z+\widehat{v}}{1-|z|}$,
$$ 1 = |w^0|^2 \left( 1- \frac{|z+\widehat{v}|^2}{(1-|z|)^2} \right)=|w^0|^2\frac{(1-|z|+|z+\widehat{v}|)(1-|z|-|z+\widehat{v}|)}{(1-|z|)^2} \leq |w^0|^2\frac{2|x-y|}{(1-|z|)t}.$$
Since $|\mathcal{W}^T|(\cdot,w)\lesssim w^0$, we get
$$\mathcal{J}_2  := \frac{\langle x \rangle}{t} \sup_{\tau \leq t}\sup_{(y,w)\in \R^6} |w^0|^8 \langle y \rangle^{N_x-3} \left|g^\beta \right|(\tau,y,w) \int_{\substack{|z| \leq 1 \\  |z+\widehat{v}| < 1-|z|  }} \int_{\R^3_y } \frac{\dr y}{\langle y \rangle^{N_x-4}} \frac{ \mathrm{d}z}{|z|^2(1-|z|)^4} .$$
Using once again Lemma \ref{gweightvderiv} together with Proposition \ref{estiLinfini}, we get, in view of \eqref{eq:easybound},
$$ \mathcal{J}_2  \lesssim \overline{\epsilon} \,  \langle x \rangle \, t^{-1} \log^{3N_x+3N}(3+t) |v^0|^8.$$
\end{proof}
This directly provides us the asymptotic profile of $\mathcal{L}_{Z^\gamma}(F)^T=-\sum_{|\beta| \leq |\gamma|} C_\beta^\gamma [\widehat{Z}^\beta f]^T$.
\begin{Cor}\label{Corinduced}
Let $|\gamma| \leq N-1$ and $\mathcal{L}_{Z^\gamma}(F)^{\infty} :=-\sum_{|\beta| \leq |\gamma|} C_\beta^\gamma [\widehat{Z}^\beta f]^\infty$. Then,
$$ \forall \, (t,x,v)\in [1,+\infty[ \times \R^3_x \times \R^3_v, \qquad \left| t^2\mathcal{L}_{Z^\gamma}(F)^T(t,x+\widehat{v}t) - \mathcal{L}_{Z^\gamma}(F)^{\infty}(v) \right| \lesssim \overline{\epsilon} \, \langle x \rangle^2|v^0|^8 \frac{\log^{3N_x+3N+1}(3+t)}{t}.$$
Moreover, if $Z^\gamma$ contains a translation $\partial_{x^\mu}$ or the scaling vector field $S$, then $\mathcal{L}_{Z^\gamma}(F)^{\infty}=0$.
\end{Cor}
\begin{proof}
We only have to focus on the second part of the statement. Recall from the proof of Proposition \ref{ProformQ} that we can reduce the analysis to the cases $Z^\gamma=\partial_{x^\lambda} Z^\kappa$, if $\gamma_T \geq 1$, and $Z^\gamma=S Z^\kappa$ otherwise. Recall further from the commutation formula of Lemma \ref{LemCom} and Proposition \ref{Com}, that
$$ \nabla^{\mu} \mathcal{L}_{\partial_{x^\lambda} Z^\kappa}(F)_{\mu \nu} = \sum_{|\xi| \leq |\kappa|} C^\kappa_\xi J(\partial_{x^\lambda} \widehat{Z}^\xi f )_{\nu}, \qquad \nabla^{\mu} \mathcal{L}_{S Z^\kappa}(F)_{\mu \nu} = \sum_{|\xi| \leq |\kappa|} C^\kappa_\xi J(S \widehat{Z}^\xi f )_{\nu}+3C^\kappa_\xi J(\widehat{Z}^\xi f)_{\nu}.$$
It remains to recall from Proposition \ref{ProformQ} that $Q^{\partial_{x^\lambda} \xi}_\infty=0$ and $Q^{S\xi}_\infty=-3Q^\xi$, so that $\mathcal{L}_{Z^\gamma}(F)^{\infty}=0$.
\end{proof}

\subsubsection{Behavior of $\mathcal{L}_{Z^\gamma}(F)^{\mathrm{data}}$ along timelike straight lines}

Recall from Proposition \ref{GSdecompo} and \eqref{eq:partief} that $F^{\mathrm{data}}$ is the sum of $F^{\mathrm{hom}}$, which verifies $\Box F^{\mathrm{hom}}_{\mu \nu}=0$, and a term which is strongly decaying in the interior of the light cone. If $Q_F \neq 0$, $F$ decay initially as $r^{-2}$ and one cannot expect to prove strong decay estimates for $ F^{\mathrm{hom}}$ through Proposition \ref{decaylinMax}. For this reason, we need to analyse in detail the homogeneous part $F^{\mathrm{hom}}$. It turns out that it decays faster in the interior of the light cone and then along timelike straight lines, so that it will not contribute to the asymptotic Lorentz force. 

In order to improve the naive estimate of Proposition \ref{estidata}, one can remark that the leading order term $\overline{F}(0,x)=\frac{Q_Fx_i}{4\pi|x|^3}  \mathrm{d} t \wedge \dr x^i$ of the asymptotic expansion of $F^{\mathrm{hom}}(0,\cdot)$ corresponds to the static electromagnetic field generated by a point charge $Q_F$ located at $x=0$. It is derived from the potential $A=Q(4\pi  r)^{-1} \dr t$ which satisfies the Lorenz gauge, and then $\Box A_\mu =0$, on $\R\times (\R^3\setminus \{ 0 \})$. To deal with our evolution problem and the singularity of the Newton potential, we introduce 
$$\widetilde{A}(t,x):= \chi(|x|-t) A(t,x)=\frac{Q_F}{4 \pi |x|} \chi (|x|-t) \dr t, \qquad  \chi \in C^\infty (\R,[0,1]), \quad \chi \vert_{]-\infty,1/2]}=0, \quad \chi \vert_{[1,+\infty[}=1.$$ 
Then, $\widetilde{A}$ is smooth on $\R_+ \times \R^3$ and $\Box \widetilde{A}_\mu =0$. It motivates the introduction of
\begin{align*} 
\widetilde{F}(t,x) := \mathrm{d} \widetilde{A}(t,x)&=\frac{Q_Fx_i}{4\pi|x|^3} \chi(|x|-t) \mathrm{d} t \wedge \dr x^i-\frac{Q_Fx_i}{4 \pi |x|^2} \chi' (|x|-t) \mathrm{d} t \wedge \dr x^i \\
& = \chi (|x|-t) \overline{F}(t,x) -\frac{Q_Fx_i}{4 \pi |x|^2} \chi' (|x|-t) \mathrm{d} t \wedge \dr x^i,
\end{align*}
which, in view of $[\Box,\partial_{x^\lambda}]=0$ and $\Box \widetilde{A}_\lambda =0$, verifies $\Box \widetilde{F}_{\mu \nu}=0$. Since,
\begin{itemize}
\item for any $\Gamma \in \mathbb{K} \setminus \{ S \}$, $[\Box \,, \Gamma]=0$ and $[\Box \, , S]=2 \, \Box$,
\item for any $Z=Z^\lambda \partial_{x^\lambda} \in \mathbb{K}$ and any $2$-form $H$, we have $\mathcal{L}_Z(H)_{\mu \nu}= Z(H_{\mu \nu})+\partial_{x^\mu}(Z^\lambda)H_{\lambda \nu}+\partial_{x^\nu}(Z^\lambda)H_{\mu \lambda }$,
\end{itemize}
we further have $\Box \mathcal{L}_{Z^\gamma}(\widetilde{F})_{\mu \nu}=0$ for any $|\gamma| \leq N-1$. The key idea will then be to consider $\mathcal{L}_{Z^\gamma}(F)^{\mathrm{hom}}-\mathcal{L}_{Z^\gamma}(\widetilde{F})$. More precisely, the following estimates hold.

\begin{Pro}\label{Cordataconvcharge0}
For any $|\gamma | \leq N-1$, we have
$$ \forall \, (t,x) \in \R_+ \times \R^3, \qquad \left|\mathcal{L}_{Z^\gamma}(F)^{\mathrm{data}}(t,x)- \mathcal{L}_{Z^\gamma}\big( \widetilde{F} \big) (t,x) \right| \lesssim \Lambda (1+t+|x|)^{-1} (1+|t-|x||)^{-1-\delta}.$$
\end{Pro}
\begin{Rq}
We will not need it in this article, but we have 
$$\big|\mathcal{L}_{Z^\gamma}\big( \widetilde{F} \big)-\chi (|x|-t) \mathcal{L}_{Z^\gamma}\big( \overline{F} \big) (t,x) \big| \lesssim Q_F (1+t)^{-1} \mathds{1}_{0 \leq |x|-t \leq 1} .$$
Moreover, $\mathcal{L}_{\partial_t}(\overline{F})= \mathcal{L}_{\Omega_{jk}}(\overline{F})= \mathcal{L}_{S}(\overline{F})=0$ for all $1 \leq j < k \leq 3$. We refer to \cite[Section~$5$]{dim3} for more informations concerning $\overline{F}$.
\end{Rq}
This result implies that the leading order term of $\mathcal{L}_{Z^\gamma}(F)^{\mathrm{data}}(t,x)$ is supported in the exterior of the light cone. Before proving it, let us investigate its direct consequence for the behavior of $F^{\mathrm{data}}$ along timelike trajectories.
\begin{Pro}\label{Cordataconvcharge}
For any $|\gamma | \leq N-1$, we have
$$ \forall \, (t,x,v) \in [1,+\infty[ \times \R^3_x \times \R^3_v, \qquad \left|t^2 \mathcal{L}_{Z^\gamma}(F)^{\mathrm{data}}(t,x+t\widehat{v}) \right| \lesssim \Lambda \, \langle x \rangle^2 |v^0|^4  t^{-\delta}.$$
\end{Pro}
\begin{proof}
Let $(t,x,v) \in [1,+\infty[ \times \R^3_x \times \R^3_v$. If $t \leq 4 \langle x \rangle|v^0|^2$, it suffices to apply Proposition \ref{estidata}, providing 
$$|\mathcal{L}_{Z^\gamma}(F)^{\mathrm{data}}(t,x+t\widehat{v})| \lesssim \Lambda \, t^{-1} \leq 16\Lambda \, \langle x \rangle^2 |v^0|^4 \, t^{-3}.$$
Otherwise, according to Lemma \ref{Lemxvt}, we have $t-|x+t\widehat{v}| \geq \frac{t}{4|v^0|^2}$, so that $\chi^{(n)} (|x+t\widehat{v}|-t)=0$ for all $n \in \mathbb{N}$. Consequently, we get from Proposition \ref{Cordataconvcharge0} that
$$ |\mathcal{L}_{Z^\gamma}(F)^{\mathrm{data}}(t,x+t\widehat{v})| \lesssim \Lambda \, t^{-1} (1+|t-|x+t\widehat{v}||)^{-1 -\delta} \leq 16\Lambda \,  |v^0|^4 \, t^{-3}.$$
\end{proof}
The first step of the proof of Proposition \ref{Cordataconvcharge0} consists in controlling the initial data for $\mathcal{L}_{Z^\gamma}(F)^{\mathrm{hom}}$.
\begin{Lem}\label{Lemdata}
The assumption \eqref{assumasymptoexp} on the initial electromagnetic field $F(0,\cdot)$ implies
\begin{equation}\label{eq:data}
 \forall \, |\gamma| \leq N-1, \qquad \sup_{|\kappa| \leq 1} \, \sup_{|x| \geq 1} \langle x \rangle^{2+\delta+|\kappa|} \left| \nabla_{t,x}^\kappa \mathcal{L}_{Z^\gamma}(F)^{\mathrm{hom}}-\nabla_{t,x}^\kappa\mathcal{L}_{Z^\gamma}(\overline{F})\right|(0,x) \lesssim \Lambda.
\end{equation}
Note that $\nabla_{t,x}^\kappa\mathcal{L}_{Z^\gamma}(\overline{F})(0,x)=\nabla_{t,x}^\kappa\mathcal{L}_{Z^\gamma}(\widetilde{F})(0,x)$ for all $|x| \geq 1$ since $\chi =1$ on $[1,+\infty[$.
\end{Lem}
\begin{proof}
As $\overline{F}$ is defined on $\R\times \R^3 \setminus \{ 0 \}$, $\mathcal{L}_{Z^\gamma}(\overline{F})$ is well-defined for $|x| \geq 1$. We point out that $\nabla_t \mathcal{L}_{Z^\gamma}(\overline{F}) (0,\cdot)$ does not necessarily vanish (consider for instance the case $Z^\gamma= \Omega_{01}$). Moreover, $\mathcal{L}_{Z^\gamma}(F)^{\mathrm{hom}}(0,\cdot)=\mathcal{L}_{Z^\gamma}(F)(0,\cdot)$ by definition. Hence, the left hand side of \eqref{eq:data} is bounded by
\begin{align}
\nonumber \sup_{|\kappa| \leq 1} \, \sup_{|x| \geq 1} \, \langle x\rangle^{2+\delta+|\kappa|} \left| \nabla_{t,x}^\kappa \mathcal{L}_{Z^\gamma}(F-\overline{F})\right| \! (0,x) & \lesssim \sup_{|\xi| \leq |\gamma|+1}  \, \sup_{|x| \geq 1} \, \langle x \rangle^{2+\delta+|\xi|} \, \big| \nabla_{t,x}^\xi (F-\overline{F})\big|  (0,x) \color{white} \\
 & \leq \Lambda +\sup_{|\beta| \leq |\gamma|}  \, \sup_{|x| \geq 1} \, \langle x \rangle^{2+\delta+n+|\beta|} \, \big|\nabla_{\partial_t} \nabla_{t,x}^\beta F\big| (0,x), \label{eq:toboundunandaoutri}
 \end{align}
where, in the last step, we used the assumption \eqref{assumasymptoexp} and that $\overline{F}$ is independent of $t$. Now, remark that if $n \geq 1$, the Maxwell equations implies
$$\partial_t (\partial_t^{n-1} \partial_x^\beta B)= -\partial_t^{n-1} \partial_x^\beta ( \nabla_x \times E), \qquad  \partial_t (\partial_t^{n-1} \partial_x^\beta E)= \partial_t^{n-1} \partial_x^\beta ( \nabla_x \times B)-\int_{\R^3_v} \widehat{v} \partial_t^{n-1} \partial_x^\beta f \dr v. $$
Let $\overline{E}$ and $\overline{B}$ be the electric and magnetic field associated to $\overline{F}$ according to \eqref{eq:defF}, so that $\overline{E}^i=\frac{x^i}{4 \pi r^3}Q_F$ and $\overline{B}=0$. As $\nabla_x \times \overline{E}= \nabla_x \times \overline{B} = 0$, we can bound \eqref{eq:toboundunandaoutri} by $\Lambda$ by performing an induction and using \eqref{assumasymptoexp} as well as the initial assumptions on $f$.
\end{proof}

We are now able to prove Proposition \ref{Cordataconvcharge0} and conclude this subsection. As $\epsilon \leq \Lambda$, \eqref{eq:partief} implies 
$$ \forall \, (t,x) \in \R_+ \times \R^3, \qquad \quad | \mathcal{L}_{Z^\gamma}(F)^{\mathrm{data}}-  \mathcal{L}_{Z^\gamma}(F)^{\mathrm{hom}} |(t,x) \lesssim  \Lambda (1+t+|x|)^{-1}(1+|t-|x||)^{-1-\delta}.$$
Finally, $\Box \mathcal{L}_{Z^\gamma}(F)^{\mathrm{hom}}_{\mu \nu}-\Box \mathcal{L}_{Z^\gamma}(\widetilde{F})_{\mu \nu}=0$, the decay assumptions on the initial data given by Lemma \ref{Lemdata} and Proposition \ref{decaylinMax} yield
$$ \forall \, (t,x) \in \R_+ \times \R^3, \qquad \quad | \mathcal{L}_{Z^\gamma}(F)^{\mathrm{hom}}-  \mathcal{L}_{Z^\gamma}(\widetilde{F}) |(t,x) \lesssim  \Lambda (1+t+|x|)^{-1}(1+|t-|x||)^{-1-\delta}.$$

\subsubsection{Self-similar asymptotic profile of $\mathcal{L}_{Z^\gamma}(F)$}
We are now able to study the full Maxwell field.
\begin{Cor}\label{Corinduced2}
For any $|\gamma| \leq N-1$, there exists a $2$-form $ \mathcal{L}_{Z^\gamma}(F)^\infty$, independent of $t$, such that
$$ \forall \, (t,x,v)\in [1,\infty[ \times \R^3_x \times \R^3_v, \qquad \left| t^2\mathcal{L}_{Z^\gamma}(F)(t,x+\widehat{v}t) - \mathcal{L}_{Z^\gamma}(F)^\infty (v) \right| \lesssim \Lambda \, \langle x \rangle^2|v^0|^8 \frac{\log^{3N_x+3N+1}(3+t)}{t^\delta}.$$
Moreover, for any $\eta>0$, there exists $C_\eta >0$ such that,
$$ \forall \, (t,x) \in [1,+\infty[ \times \R^3_x, \; \frac{|x|}{t} \leq 1-\eta, \qquad \left| t^2\mathcal{L}_{Z^\gamma}(F)(t,x) - \mathcal{L}_{Z^\gamma}(F)^\infty \left( \frac{\widecheck{x}}{t} \right) \right| \lesssim \Lambda C_\eta \frac{\log^{3N_x+3N+1}(3+t)}{t^\delta}.$$ 
\end{Cor}
\begin{Rq}\label{Rqexplicompu}
For the most important case, $|\gamma|=0$, we have $4\pi F^\infty= -[f]^\infty$, where $[f]^\infty$ is explicitly written in Definition \ref{Defexplici}.
\end{Rq}
\begin{proof}
Fix $|\gamma| \leq N-1$ and $(t,x,v)\in [1,\infty[ \times \R^3_x \times \R^3_v$. Applying Proposition \ref{estiS} and Lemma \ref{Lemxvt}, we have
$$ t^2 \left| \mathcal{L}_{Z^\gamma}(F)^S\right|(t,x+\widehat{v}t) \lesssim \overline{\epsilon} \,  \Lambda \frac{t \log(3+|t-|x-t\widehat{v}||)}{(1+|t-|x-t\widehat{v}||)^2} \lesssim \Lambda \left(\frac{\langle x \rangle^2|v^0|^4}{t}+|v^0|^4 \frac{\log (3+t)}{t} \right).$$
We then get the first part of the statement using the Glassey-Strauss decomposition given by Proposition \ref{GSdecompo}, Corollary \ref{Corinduced}, where $\mathcal{L}_{Z^\gamma}(F)^\infty$ is introduced, and Proposition \ref{Cordataconvcharge}. For the second part, it suffices to apply the first estimate, with a slight abuse of notation, for $x=0$ and $\widehat{v}=x/t$.
\end{proof}
We deduce from the previous result a uniform bound on $\mathcal{L}_{Z^\gamma}(F)^\infty$. Moreover, it turns out that this quantity vanishes in certain cases, providing improved estimates for $\mathcal{L}_{Z^{\gamma}}(F)$.
\begin{Pro}\label{estiFinfty}
For any $|\gamma| \leq N-1$, we have $| \mathcal{L}_{Z^\gamma}(F)^\infty|(v) \lesssim \overline{\epsilon} \sqrt{v^0}$. Moreover, if $|\gamma| \geq 1$ and $Z^\gamma$ contains a translation $\partial_{x^\mu}$ or the scaling vector field $S$, then $\mathcal{L}_{Z^\gamma}(F)^\infty=0$.
\end{Pro}
\begin{proof}
According to Propositon \ref{decayFT}, $t^2 |\mathcal{L}_{Z^\gamma}(F)^T|(t,t\widehat{v}) \lesssim \overline{\epsilon} (1-|\widehat{v}|)^{-1/4} \leq 2 \, \overline{\epsilon} \sqrt{v^0} $. All the properties then ensue from Corollary \ref{Corinduced}. 
\end{proof}
Finally, we investigate the regularity of $\mathcal{L}_{Z^\gamma}(F)^\infty$. 
\begin{Pro}\label{Proderivinfini}
For any $|\gamma| \leq N-2$ and $0 \leq \mu , \, \nu \leq 3$, $\mathcal{L}_{Z^\gamma}(F)^\infty_{\mu \nu}$ is of class $C^{N-1-|\gamma|}$. Moreover, for any $1 \leq k \leq 3$, we have
$$v^0\partial_{v^k}\mathcal{L}_{Z^\gamma}(F)^\infty_{\mu \nu}= \mathcal{L}_{\Omega_{0k}Z^\gamma}(F)^\infty_{\mu \nu}+2\widehat{v}^k\mathcal{L}_{Z^\gamma}(F)^\infty_{\mu \nu}-\delta_\mu^0  \mathcal{L}_{Z^\gamma}(F)_{k \nu}^\infty-\delta_\mu^k \mathcal{L}_{Z^\gamma}(F)_{0 \nu}^\infty-\delta_\nu^0  \mathcal{L}_{Z^\gamma}(F)^\infty_{\mu k}-\delta_\nu^k \mathcal{L}_{Z^\gamma}(F)_{\mu 0}^\infty.$$
The angular derivatives satisfy
$$(v^j \partial_{v^k}-v^k \partial_{v^j})\mathcal{L}_{Z^\gamma}(F)^\infty_{\mu \nu}=\mathcal{L}_{\Omega_{jk}Z^\gamma}(F)^\infty_{\mu \nu}-\delta_\mu^j  \mathcal{L}_{Z^\gamma}(F)_{k \nu}^\infty+ \delta_\mu^k \mathcal{L}_{Z^\gamma}(F)_{j \nu}^\infty-\delta_\nu^j  \mathcal{L}_{Z^\gamma}(F)_{\mu k}^\infty+ \delta_\nu^k \mathcal{L}_{Z^\gamma}(F)_{\mu j}^\infty.$$
\end{Pro}
\begin{proof}
In order to lighten the notations, we introduce $X:=x+t\widehat{v}$. Then, we compute
\begin{align*}
 v^0 \partial_{v^k}\!\left(\mathcal{L}_{Z^\gamma}(F)_{\mu \nu}(t,X)\right) & = t\left(\delta_k^i-\widehat{v}^k \widehat{v}^i \right) \partial_{x^i} \left( \mathcal{L}_{Z^\gamma}(F)_{\mu \nu} \right)(t,X)  \\
 &=  \left(\Omega_{0k} \mathcal{L}_{Z^\gamma}(F)_{\mu \nu} \right)\!(t,X)-X^k\partial_t\left(\mathcal{L}_{Z^\gamma}(F)_{\mu \nu} \right)\!(t,X)+\widehat{v}^k(x^i-X^i)\partial_{x^i} \left( \mathcal{L}_{Z^\gamma}(F)_{\mu \nu} \right)\!(t,X) \\
 &= \left(\Omega_{0k} \mathcal{L}_{Z^\gamma}(F)_{\mu \nu} \right)(t,X)-\widehat{v}^k \left( S\mathcal{L}_{Z^\gamma}(F)_{\mu \nu} \right)(t,X)\\
 & \quad -x^k\partial_t\left(\mathcal{L}_{Z^\gamma}(F)_{\mu \nu} \right)(t,X)+\widehat{v}^kx^i\partial_{x^i} \left( \mathcal{L}_{Z^\gamma}(F)_{\mu \nu} \right)(t,X).
 \end{align*}
 One can already notice that the last two terms enjoy strong decay properties. More precisely, since Lemma \ref{Lemxvt} implies $ 1+|t-|X|| \gtrsim \frac{1+t}{\langle x \rangle |v^0|^2}$, we have from Proposition \ref{Prodecaytoporder}
 $$t^2\left| -x^k\partial_t\left(\mathcal{L}_{Z^\gamma}(F)_{\mu \nu} \right)(t,X)+\widehat{v}^kx^i\partial_{x^i} \left( \mathcal{L}_{Z^\gamma}(F)_{\mu \nu} \right)(t,X) \right| \lesssim  \frac{\Lambda \langle x \rangle^3 |v^0|^4}{1+t}.$$
The result then follows from
\begin{align}
\mathcal{L}_{SZ^\gamma}(F)_{\mu \nu} &= S \left(  \mathcal{L}_{Z^\gamma}(F)_{\mu \nu}\right)+2\mathcal{L}_{Z^\gamma}(F)_{\mu \nu}, \qquad \qquad \mathcal{L}_{SZ^\gamma}(F)^{\infty}=0, \label{eq:scal} \\
\mathcal{L}_{\Omega_{0k}Z^\gamma}(F)_{\mu \nu} &= \Omega_{0k} \left(  \mathcal{L}_{Z^\gamma}(F)_{\mu \nu} \right)+\delta_\mu^0  \mathcal{L}_{Z^\gamma}(F)_{k \nu}+ \delta_\mu^k \mathcal{L}_{Z^\gamma}(F)_{0 \nu}+\delta_\nu^0  \mathcal{L}_{Z^\gamma}(F)_{\mu k}+ \delta_\nu^k \mathcal{L}_{Z^\gamma}(F)_{\mu 0} \nonumber
\end{align}
and Corollary \ref{Corinduced2}, which give us
$$  \left| t^2v^0 \partial_{v^k}\mathcal{L}_{Z^\gamma}(F)_{\mu \nu}(t,x+t\widehat{v}) - v^0\partial_{v^k}\mathcal{L}_{Z^\gamma}(F)^\infty_{\mu \nu}(v) \right| \lesssim \Lambda \langle x \rangle^3|v^0|^8 \frac{\log^{1+3N_x+3N}(3+t)}{(1+t)^\delta},$$
where $v^0\partial_{v^k}\mathcal{L}_{Z^\gamma}(F)^\infty_{\mu \nu}(v)$ is given in the statement of the Proposition. To get the expression of the angular derivatives, notice that 
\begin{align*}
& (v^j \partial_{v^k}-v^k \partial_{v^j})\left(\mathcal{L}_{Z^\gamma}(F)_{\mu \nu}(t,X)\right)  =  \left(\Omega_{jk} \mathcal{L}_{Z^\gamma}(F)_{\mu \nu} \right)(t,X) -(x^j\partial_{x^k}-x^k\partial_{x^j} )\left( \mathcal{L}_{Z^\gamma}(F)_{\mu \nu} \right)(t,X) , \\
& \mathcal{L}_{\Omega_{jk}Z^\gamma}(F)_{\mu \nu} = \Omega_{jk} \left(  \mathcal{L}_{Z^\gamma}(F)_{\mu \nu} \right)+\delta_\mu^j  \mathcal{L}_{Z^\gamma}(F)_{k \nu}- \delta_\mu^k \mathcal{L}_{Z^\gamma}(F)_{j \nu}+\delta_\nu^j  \mathcal{L}_{Z^\gamma}(F)_{\mu k}- \delta_\nu^k \mathcal{L}_{Z^\gamma}(F)_{\mu j}
 \end{align*}
 and apply the same arguments. The $C^{N-1-|\gamma|}$ regularity is obtained by an induction.
\end{proof}
For later use, we prove that the structure of the asymptotic Lorentz force is preserved by differentiation.
\begin{Cor}\label{CorLorentzforceasympt}
Let $0 \leq \nu \leq 3$ and define
$$ \Delta_{Z^\gamma ,\nu}(t,x,v) := t^2 \frac{\widehat{v}^{\mu}}{v^0}\mathcal{L}_{Z^{\gamma}}(F)_{\mu\nu}(t,x)-\frac{\widehat{v}^\mu}{v^0}\mathcal{L}_{Z^\gamma}(F)^\infty_{\mu \nu}(v), \qquad |\gamma| \leq N-1.$$
For any $|\gamma| \leq N-2$, there holds
\begin{align*}
S \left(\Delta_{Z^\gamma ,\nu} \right) & = \Delta_{SZ^\gamma ,\nu} , \\
\widehat{\Omega}_{jk}\left( \Delta_{Z^\gamma , \nu} \right) & = \Delta_{\Omega_{jk}Z^\gamma , \nu}-\delta_{\nu}^j \Delta_{Z^\gamma , k}+\delta_{\nu}^k \Delta_{Z^\gamma , j}, \qquad 1 \leq j < k \leq 3, \\
\widehat{\Omega}_{0i}\left( \Delta_{Z^\gamma , \nu} \right) & = \Delta_{\Omega_{0i}Z^\gamma , \nu}-\delta_{\nu}^0 \Delta_{Z^\gamma , i}-\delta_{\nu}^i \Delta_{Z^\gamma , 0}+2\frac{t}{v^0}(x^i-t\widehat{v}^i)\widehat{v}^\mu \mathcal{L}_{Z^\gamma}(F)_{\mu \nu}(t,x), \qquad 1 \leq i \leq 3.
\end{align*}
\end{Cor}
\begin{proof}
The first identity follows from $S(t^2)=2t^2$ and \eqref{eq:scal}. For the other ones, start by noticing that, according to Proposition \ref{Proderivinfini} and for $1 \leq i \leq 3$,
\begin{align}
 \widehat{\Omega}_{0i}\left( \frac{\widehat{v}^{\mu}}{v^0}\mathcal{L}_{Z^\gamma}(F)^\infty_{\mu \nu}(v) \right)&= v^0 \partial_{v^i}\left( \frac{\widehat{v}^{\mu}}{v^0}\mathcal{L}_{Z^\gamma}(F)^\infty_{\mu \nu}(v) \right) \nonumber \\
 &= \frac{\widehat{v}^{\mu}}{v^0}\mathcal{L}_{\Omega_{0i}Z^\gamma}(F)^\infty_{\mu \nu}(v)-\delta_\nu^0 \frac{\widehat{v}^{\mu}}{v^0} \mathcal{L}_{Z^\gamma}(F)^\infty_{\mu i}-\delta_\nu^i \frac{\widehat{v}^{\mu}}{v^0} \mathcal{L}_{Z^\gamma}(F)_{\mu 0}^\infty. \label{eq:boosti}
 \end{align}
Similarly, for $1 \leq j < k \leq 3$,
\begin{equation}\label{eq:rotation}
 \widehat{\Omega}_{jk}\left( \frac{\widehat{v}^{\mu}}{v^0}\mathcal{L}_{Z^\gamma}(F)^\infty_{\mu \nu}(v) \right) = \frac{\widehat{v}^{\mu}}{v^0}\mathcal{L}_{\Omega_{jk}Z^\gamma}(F)^\infty_{\mu \nu}(v)-\delta_\nu^j  \frac{\widehat{v}^{\mu}}{v^0} \mathcal{L}_{Z^\gamma}(F)_{\mu k}^\infty+ \delta_\nu^k \frac{\widehat{v}^{\mu}}{v^0} \mathcal{L}_{Z^\gamma}(F)_{\mu j}^\infty.
\end{equation}
Recall that we denote by $\mathbf{v}$ the $4$-vector $(v^\mu)_{0 \leq \mu \leq 4}$, so that
\begin{align}
\widehat{Z}  \left( t^2 \frac{\widehat{v}^{\mu}}{v^0}\mathcal{L}_{Z^{\gamma}}(F)_{\mu\nu} \right) &=\widehat{Z} \left(  \frac{t^2}{|v^0|^2} \right) v^\mu \mathcal{L}_{Z^{\gamma}}(F)_{\mu\nu}+\frac{t^2}{|v^0|^2} \mathcal{L}_{ZZ^{\gamma}}(F)(\mathbf{v},\partial_{x^\nu})+\frac{t^2}{|v^0|^2} \mathcal{L}_{Z^{\gamma}}(F)(\mathbf{v},[Z,\partial_{x^\nu}])  \label{nonulleici} \\
 & \quad +\frac{t^2}{|v^0|^2} \mathcal{L}_{Z^{\gamma}}(F)([Z,\mathbf{v}],\partial_{x^\nu}) + \frac{t^2}{|v^0|^2}\widehat{Z}\left( v^\mu \right)  \mathcal{L}_{Z^{\gamma}}(F)_{\mu \nu} . \label{secompensejustela}
\end{align}
\begin{itemize}
\item If $Z= \Omega_{0i}$, we have $[Z,\mathbf{v}]=-v^i \partial_t-v^0 \partial_{x^i}$ and $\widehat{Z}(v^\mu)=\delta_\mu^0v^i+\delta_\mu^i v^0$, so that the sum of two terms in \eqref{secompensejustela} vanishes. It remains to remark that $[Z,\partial_{x^\nu}]=-\delta_{\nu}^i\partial_t-\delta_\nu^0\partial_i$, $\widehat{Z}(t^2/|v^0|^2)=2t(x^i-t\widehat{v}^i)/|v^0|^2$ and to combine \eqref{eq:boosti} with \eqref{nonulleici}.
\item If $Z=\Omega_{jk}$, there holds $[Z,\mathbf{v}]=-v^j \partial_{x^k}+v^k \partial_{x^j}$ and $\widehat{Z}(v^\mu)=\delta_\mu^k v^j-\delta_\mu^j v^k$, so that the sum of the two terms in \eqref{secompensejustela} vanishes once again. The result then ensues from $\widehat{Z}(t^2/|v^0|^2)=0$, $[Z,\partial_{x^\nu}]=-\delta_{\nu}^j\partial_{x^k}+\delta_\nu^k\partial_{x^j}$, \eqref{eq:rotation} and \eqref{nonulleici}.
\end{itemize}
\end{proof}

\subsection{Convergence of the distribution function along modified characteristics}

Motivated by the discussion in Section \ref{subsecmotivatedby} and by Corollary \ref{Corinduced2}, we modify the linear spatial characteristics $t \mapsto x+t\widehat{v}$ as follows. 
\begin{Def}\label{DeXc}
For $(x,v) \in \R^3_x \times \R^3_v$, let $ X_\C(\cdot,x,v):t \mapsto x+t\widehat{v}+\C(t,v)$ be the trajectory\footnote{Recall that $F^\infty$ is a $2$-form, so that $\widehat{v}^\mu\widehat{v}^\nu F^\infty_{\mu \nu } =0$.}
\begin{align}
 X^{i}_\C(t,x,v) & := x^i+t\widehat{v}^i-\log(t) \widehat{v}^{\mu}F^{\infty,j}_{\mu}(v)\frac{ \delta^i_{j}-\widehat{v}_j \widehat{v}^i}{v^0}  \label{defXC}\\ \nonumber 
 & = x^i+t\widehat{v}^i-\frac{\log(t)}{v^0}\Big( \widehat{v}^{\mu}F^{\infty}_{\mu i}(v)+\widehat{v}^i \widehat{v}^\mu F^{\infty}_{\mu 0} (v) \Big) , \qquad t \in \R_+^*, \quad i \in \llbracket 1,3 \rrbracket .
\end{align}
For simplicity, we will often write $X_\C$ instead of $X_\C(t,x,v)$. By Proposition \ref{estiFinfty}, the components $\C^i$ of the correction term $\C$ verify
\begin{equation}\label{estCor}
\forall \, t >0, \qquad \left| \C^i \right|(t,v)  \lesssim \overline{\epsilon} |v^0|^{-\frac{1}{2}} \log(t), \qquad i \in \llbracket 1,3 \rrbracket.
\end{equation}
\end{Def}
We now bound the time derivative of a function evaluated along the modified characteristics.
\begin{Pro}\label{VMh}
Let $\mathbf{f}:\R_+ \times \R^3_x\times \R^3_v \rightarrow \R$ be a sufficiently regular function and introduce $h(t,x,v):=\mathbf{f}(t,X_\C(t,x,v),v)$. Then, for all $(t,x,v) \in [1,+\infty[ \times \R^3_x \times \R^3_v$,
\begin{align*}
 |\partial_t h| (t,x,v) &\leq \left| \T_F(\mathbf{f}) \right|(t,X_\mathscr{C},v)+\Lambda \frac{\log^{3+3N_x+3N}(3+t)}{(1+t)^{1+\delta}} \sum_{\widehat{Z} \in \widehat{\mathbb{P}}_0}  \left||v^0|^7 \mathbf{z}^2 \widehat{Z}\mathbf{f}\right|(t,X_\mathscr{C},v).
 \end{align*}
\end{Pro}
\begin{proof}
We have, for all $(t,x,v) \in [1,+\infty[ \times \R^3_x \times \R^3_v$,
\begin{align}
\nonumber \partial_t h(t,x,v) &= \left( \partial_t \mathbf{f}+\widehat{v}^i \partial_{x^i} \mathbf{f} \right)(t,X_\mathscr{C},v)+\partial_t \mathscr{C}^i (t,v) \partial_{x^i} \mathbf{f} (t,X_\mathscr{C},v), \\
 & = \T_F(\mathbf{f})(t,X_\mathscr{C})- \widehat{v}^{\mu} {F_{\mu}}^j(t,X_\mathscr{C}) \partial_{v^j} \mathbf{f}(t,X_\mathscr{C},v)+\partial_t \mathscr{C}^i (t,v) \partial_{x^i} \mathbf{f} (t,X_\mathscr{C},v). \label{eq:areprendre}
\end{align}
Recall from \eqref{reutiliser} the relation
\begin{equation}\label{derivv:eq}
v^0\partial_{v^j}=-t\left(\partial_{x^j}-\widehat{v}^j \widehat{v}^i \partial_{x^i}\right) +\widehat{\Omega}_{0j}+z_{0j} \partial_t-\widehat{v}^j S -\sum_{1 \leq i \leq 3}\widehat{v}^j z_{0i} \partial_{x^i}, \qquad 1 \leq j \leq 3,
\end{equation}
in order to rewrite $\partial_{v^j}\mathbf{f}(t,X_\C,v)$. As $v^0\partial_t\C^i(t,v)=-\frac{1}{t}\widehat{v}^{\mu}F^{\infty,j}_{\mu}(v)( \delta^i_{j}-\widehat{v}_j \widehat{v}^i )$, we get
\begin{align*}
 |\partial_t h| (t,x,v) &\leq \left| \T_F(\mathbf{f}) \right|(t,X_\mathscr{C},v)+\sum_{1 \leq j \leq 3} \, \sum_{\widehat{Z} \in \widehat{\mathbb{P}}_0}\left| \widehat{v}^{\mu} {F_{\mu}}^j\right|(t,X_\mathscr{C}) \left| \frac{\mathbf{z}}{v^0} \widehat{Z}\mathbf{f}\right|(t,X_\mathscr{C},v)   \\ 
 & \quad +\frac{1}{tv^0}\left|t^2 F(t,X_\mathscr{C})-F^{\infty}(v) \right| \left|\partial_{t,x} \mathbf{f} \right| (t,X_\mathscr{C},v). 
 \end{align*}
 We deal with the second term on the right hand side of the previous inequality by controlling the Lorentz force through Remark \ref{Rqbetteresti}, so that $\left| \widehat{v}^{\mu} {F_{\mu}}^j\right|(t,X_\C)\lesssim \Lambda (1+t)^{-2} |v^0|^2 \mathbf{z}(t,X_\C,v)$. Next, by Corollary \ref{Corinduced2} and the mean value theorem,
 \begin{align*}
  \left|t^2 F(t,X_\mathscr{C})-F^{\infty}(v) \right| & \leq \left|t^2 F(t,x+t\widehat{v})-F^{\infty}(v) \right|+t^2\left|F(t,X_\mathscr{C})-F(t,x+t\widehat{v}) \right| \\
  & \lesssim \Lambda \, \langle x \rangle^2 |v^0|^8 \frac{\log^{3N_x+3N+1}(3+t)}{(1+t)^\delta}+t^2|\C(t,v)| \sup_{|y-X_\C| \leq |\C|(t,v)}|\nabla_{t,x} F|(t,y) .
  \end{align*}
In view of the estimate of $\nabla_{t,x} F$ given by Lemma \ref{Lorentzforce2} and the bound \eqref{estCor} on $\C$, we have
$$ t^2|\C(t,v)|\sup_{|y-x| \leq |\C|(t,v)}|\nabla_{t,x} F|(t,y)  \lesssim \frac{\Lambda}{\sqrt{v^0}}  t^2 \log (3+t)  \frac{\log(3+t)}{(1+t)^3} |v^0|^4  \sup_{|y-X_\C| \leq |C|(t,v)} \mathbf{z}^2(t,y,v).$$
Since $|\nabla_x \mathbf{z}| \lesssim 1$, the mean value theorem yields
\begin{equation}\label{eq:boundzX}
\mathbf{z}(t,x+t\widehat{v},v) \leq \sup_{|y-X_\C| \leq |\C|(t,v)} \mathbf{z}(t,y,v) \leq   \mathbf{z}(t,X_\C,v)+\frac{\overline{\epsilon}}{\sqrt{v^0}}  \log(3+t) \lesssim   \log(3+t) \, \mathbf{z}(t,X_\C,v) .
\end{equation}
Consequently, as $\langle x \rangle \leq \mathbf{z}(t,x+t\widehat{v},v)$, we have
$$ \left|t^2 F(t,X_\mathscr{C})-F^{\infty}(v) \right| \lesssim \Lambda (1+t)^{-\delta} \log^{3N_x+3N+3}(3+t) |v^0|^8 \mathbf{z}^2(t,X_\C,v) .$$
We then deduce the result from the previous estimates.
\end{proof}
By applying this result to $f$, we obtain that there exists $f_\infty \in L^\infty_{x,v}$ such that $f(t,X_\C,v) \to f_\infty(x,v)$ as $t \to 0$ (see Proposition \ref{mainresultPro} for more details). Applying it again to $\partial_x^\kappa f$ we could easily deduce that $f_\infty$ is smooth with respect to the spatial variables. However, obtaining the regularity in the velocity variables requires a more thorough analysis. Indeed, $\partial_{v^i}(f(t,X_\C,v))$ is deeply related to $\widehat{\Omega}_{0i}f(t,X_\C,v)$ which does not converge.

\subsection{Modified commutators}\label{subsecmodicom}

Let $\widehat{Z} \in \widehat{\mathbb{P}}_0 \setminus \{\partial_t, \partial_{x^1} , \partial_{x^2}, \partial_{x^3} \}$ be a homogeneous vector field. Contrary to the case of the translations, the error term $[\T_F,\widehat{Z}](f)$ does not decay sufficiently fast in order to prove a convergence result for $\widehat{Z}f$, even along the modified characteristics. Indeed, recall from Lemma \ref{LemCom} that
$$ \T_F(\widehat{Z}f)=- \widehat{v}^\mu {\mathcal{L}_Z(F)_{\mu}}^j \partial_{v^j} f+\delta_{\widehat{Z}}^S \, \widehat{v}^\mu {F_\mu}^{j} \partial_{v^j} f$$
and let us identify the terms with the slowest decay rate. Rewriting $\partial_{v^j}$ by using \eqref{derivv:eq} and estimating the electromagnetic field through Remark \ref{Rqbetteresti}, we have
\begin{equation}\label{badterms:eq}
\left| \T_F(\widehat{Z}f) -\frac{t}{v^0}\left(\widehat{v}^\mu {\mathcal{L}_Z(F)_{\mu}}^j-\delta_{\widehat{Z}}^S \, \widehat{v}^\mu {F_\mu}^{j}\right)\left(\delta_j^i-\widehat{v}_j\widehat{v}^i  \right)\partial_{x^i}f \right| \lesssim \Lambda (1+t)^{-2} \sum_{\widehat{\Gamma} \in \widehat{\mathbb{P}}_0} v^0\left|\mathbf{z}^2 \widehat{\Gamma} f \right|.
\end{equation}
In view of Proposition \ref{estiLinfini}, the right hand side is bounded by $\overline{\epsilon} \, (1+t)^{-2} \log^9(3+t)$ and then belongs to $L^1_t L^\infty_{x,v}$. On the other hand, if $\mathcal{L}_Z(F)^\infty$ and $F^\infty$ does not vanish, the decay rate of $t| \mathcal{L}_ZF|+t|F| \lesssim t^{-1}$ along timelike trajectories is at the threshold of time integrability. For this reason, we modify the linear commutator $\widehat{Z}$ in a similar way than we modify the spatial characteristics. More precisely, motivated by Corollary \ref{Corinduced2} and \eqref{badterms:eq}, we introduce the following vector fields.
\begin{Def}\label{Defmodcom}
For any $\widehat{Z} \in \widehat{\mathbb{P}}_0 \setminus \{\partial_t, \partial_{x^1} , \partial_{x^2}, \partial_{x^3},S \}$, we define $\widehat{Z}^{\mathrm{mod}}$ and $S^{\mathrm{mod}}$ as
$$ \widehat{Z}^{\mathrm{mod}} := \widehat{Z}-\log(t)\widehat{v}^{\mu}\mathcal{L}_Z(F)_{\mu}^{\infty, \, j}(v)\frac{\delta_j^i-\widehat{v}_j\widehat{v}^i  }{v^0}\partial_{x^i}, \qquad S^{\mathrm{mod}}:=S+ \log(t)\widehat{v}^{\mu}F_{\mu}^{\infty, \, j}(v)\frac{\delta_j^i-\widehat{v}_j\widehat{v}^i  }{v^0}\partial_{x^i}.$$
We further define the correction coefficients $\C_S^i(t,v)=-\C^i(t,v)$ and
$$ \C_{\widehat{Z}}^i(t,v)= -\log(t)\widehat{v}^{\mu}\mathcal{L}_Z(F)_{\mu}^{\infty, \, j}(v)\frac{\delta_j^i-\widehat{v}_j\widehat{v}^i  }{v^0} = -\frac{\log(t)}{v^0}\Big( \widehat{v}^{\mu} \mathcal{L}_Z(F)_{\mu}^{\infty, \, i}(v)+\widehat{v}^i \widehat{v}^\mu \mathcal{L}_Z(F)_{\mu 0}^{\infty}(v) \Big),$$
so that $S^{\mathrm{mod}}=S+\C^i_S(t,v)\partial_{x^i}$ and $\widehat{Z}^{\mathrm{mod}}=\widehat{Z}+\C_{\widehat{Z}}^i(t,v) \partial_{x^i}$.
\end{Def}
\begin{Rq}\label{Rqscaling}
Recall that $t|\mathcal{L}_S(F)| \lesssim (1+t)^{-1-\delta}$ in domains of the form $\{ t \geq (1- \delta )r \}$ since $\mathcal{L}_S(F)^\infty=0$. This is why we do not need to compensate the term related to $\mathcal{L}_S(F)$ in \eqref{badterms:eq}.
\end{Rq}
We have the improved commutation relations.
\begin{Pro}\label{improvedCom0}
Let $Z \in \mathbb{K} $ be a rotational vector field $\Omega_{jk}$ or a Lorentz boost $\Omega_{0i}$. Then, for $t >0$,
\begin{align*}
[\T_F, \widehat{Z}^{\mathrm{mod}}] & = \frac{1}{t}\left(t^2 \widehat{v}^\mu \left( {\mathcal{L}_Z(F)_{\mu}}^j- \mathcal{L}_Z(F)_{\mu}^{\infty,j}  \right) \frac{\delta_j^i-\widehat{v}_j \widehat{v}^i}{v^0} \right) \partial_{x^i}   \\
& \quad  - \frac{\widehat{v}^\mu}{v^0} {\mathcal{L}_Z(F)_{\mu}}^j\left(\widehat{\Omega}_{0j}+z_{0j} \partial_t-\widehat{v}^j S -\sum_{1 \leq i \leq 3}\widehat{v}^j z_{0i} \partial_{x^i}\right)-\C_{\widehat{Z}}^i \, \widehat{v}^\mu {\mathcal{L}_{\partial_{x^i}}(F)_{\mu}}^j\partial_{v^j}+\widehat{v}^\mu {F_{\mu}}^j \partial_{v^j} \C^i_{\widehat{Z}} \, \partial_{x^i}.
\end{align*}
For the scaling vector field, we have
\begin{multline*}
[\T_F, S^{\mathrm{mod}}]  = -\frac{1}{t}\left( t^2\widehat{v}^\mu \left( {F_{\mu}}^j- F_{\mu}^{\infty,j} \right) \frac{\delta_j^i-\widehat{v}_j \widehat{v}^i}{v^0} \right) \partial_{x^i} +\frac{1}{t}\left(t^2 \widehat{v}^\mu \left( {\mathcal{L}_S(F)_{\mu}}^j- \mathcal{L}_S(F)_{\mu}^{\infty,j}  \right) \frac{\delta_j^i-\widehat{v}_j \widehat{v}^i}{v^0} \right) \partial_{x^i} \\  + \frac{\widehat{v}^\mu}{v^0} \big({F_{\mu}}^j- {\mathcal{L}_S(F)_{\mu}}^j\big) \!\Big(\widehat{\Omega}_{0j}+z_{0j} \partial_t-\widehat{v}^j S -\sum_{1 \leq i \leq 3}\widehat{v}^j z_{0i} \partial_{x^i}\Big) -\C^i_S \, \widehat{v}^\mu {\mathcal{L}_{\partial_{x^i}}(F)_{\mu}}^j\partial_{v^j}+\widehat{v}^\mu {F_{\mu}}^j \partial_{v^j} \C^i_S \, \partial_{x^i}.
\end{multline*}
\end{Pro}
\begin{proof}
Consider first the case $Z=\Omega_{jk}$ or $Z=\Omega_{0i}$. In view of the commutation relation of Lemma \ref{LemCom},
\begin{align*}
[\T_F, \widehat{Z}^{\mathrm{mod}}]&= \T_F \left( \C^i_{\widehat{Z}} \right)\partial_{x^i}+[\T_F,\widehat{Z}]+\C^i_{\widehat{Z}} \, [\T_F,\partial_{x^i}]=\T_F \! \left( \C^i_{\widehat{Z}} \right)\partial_{x^i}-\widehat{v}^\mu {\mathcal{L}_Z(F)_\mu}^j \partial_{v^j}-\C_{\widehat{Z}}^i \, \widehat{v}^\mu {\mathcal{L}_{\partial_{x^i}}(F)_{\mu}}^j\partial_{v^j}.
\end{align*}
It then suffices to use \eqref{derivv:eq} in order to rewrite $\partial_{v^j}$ in the second term and to compute
$$ \T_F \left( \C^i_{\widehat{Z}} \right)=-\frac{1}{t}\widehat{v}^{\mu}\mathcal{L}_Z(F)_{\mu}^{\infty,j}(v)\frac{\delta_j^i-\widehat{v}_j\widehat{v}^i  }{v^0}+\widehat{v}^\mu {F_{\mu}}^j \partial_{v^j} \C^i_{\widehat{Z}}.$$
The case of the scaling $S$ can be treated similarly since $\mathcal{L}_S(F)^\infty=0$ according to Proposition \ref{estiFinfty}.
\end{proof}
Apart from the term involving $\mathcal{L}_S(F)$, already discussed in Remark \ref{Rqscaling}, it is clear that any of the error terms decays almost as $t^{-1-\delta}$ for, say, $|x| <t/2$. At this point, we could then prove that $f_\infty$ is $C^1$ in $v$. However, since we would like to show $f_\infty \in C^{N-2}(\R^3_x \times \R^3_v)$, we need to state higher order commutator formula for the modified vector fields. For this purpose, we introduce the set 
$$\widehat{\mathbb{P}}_0^{\mathrm{mod}} := \left\{ \partial_t, \, \partial_{x^i}, \, \widehat{\Omega}_{0i}^{\mathrm{mod}}, \, \widehat{\Omega}_{jk}^{\mathrm{mod}}, \, S^{\mathrm{mod}} \, \big| \; \quad 1 \leq i \leq 3, \, 1 \leq j < k \leq 3 \right\}, $$
and we consider an ordering on it, so that $\widehat{\mathbb{P}}_0^{\mathrm{mod}}=\{ \widehat{Z}^{\mathrm{mod},i} \, | \; 1 \leq i \leq 11 \}$. Given a multi-index $\beta \in \llbracket 1 , 11 \rrbracket^p$, we will then denote $\widehat{Z}^{\mathrm{mod},\beta_1} \dots \widehat{Z}^{\mathrm{mod},\beta_p}$ by $\widehat{Z}^{\mathrm{mod},\beta}$. We will further denote by $\beta_H$ (respectively $\beta_T$) the number of modified vector fields (respectively translations) composing $\widehat{Z}^{\mathrm{mod},\beta}$, so that $|\beta|=\beta_H+\beta_T$. Furthermore, we will use the schematic notation $P_{p,q}(\C)$ in order to denote any quantity of the form
$$ \prod_{1 \leq k \leq p}   \widehat{Z}^{\xi_k} \! \left( \C^{i_k}_{\widehat{Z}^k}\right), \quad (p,q) \in \mathbb{N}^2, \quad 1 \leq i_k \leq 3, \quad \widehat{Z}^k \! \in \widehat{\mathbb{P}}_0, \quad \sum_{1 \leq k \leq p} |\xi_k| =q, \quad q_T:= \sum_{1 \leq k \leq q}\xi_{k,T}, \quad q_H:=q-q_T, $$
where $q_T \geq 1$ when at least one translation $\partial_{x^\mu}$ is applied to at least one of the correction coefficients. By convention, we set $P_{0,0}(\C)=1$ for $p=q=0$. We recall from \eqref{defzzzz} the weights $z_{\lambda k} \in \mathbf{k}_1$, $0 \leq \lambda < k \leq 3$, which commute with the linear transport operator $\T_0$.
\begin{Pro}\label{improvedCom}
Let $\widehat{Z}^{\mathrm{mod},\beta} \in \widehat{\mathbb{P}}_0^{|\beta|}$. Then, $[\T_F,\widehat{Z}^{\mathrm{mod},\beta}]$ can be written as a linear combination of the following type of terms,
\begin{alignat}{2}
& \frac{1}{v^0t} R\left(\frac{1}{t},\widehat{v}, z\right)P_{p,q}(\C)\Big(t^2 \widehat{v}^\mu  {\mathcal{L}_{Z^\gamma}(F)_{\mu \nu}}- \widehat{v}^\mu \mathcal{L}_{Z^\gamma}(F)_{\mu \nu}^{\infty}(v)  \Big)  \widehat{Z}^{\kappa}, \qquad \qquad &&  \tag{T-1} \label{T1} \\
&  \frac{1}{v^0}R\left(\frac{1}{t},\widehat{v}, z\right) P_{p,q}(\C)\mathcal{L}_{Z^\gamma}(F)_{\lambda \nu} \, \widehat{Z}^\kappa, \qquad \qquad &&   \label{T2} \tag{T-2} \\
& \frac{x^\alpha}{v^0} R\left(\frac{1}{t},\widehat{v}, z\right)  P_{p,q}(\C) \mathcal{L}_{Z^\gamma}(F)_{\lambda \nu} \, \widehat{Z}^\kappa, \qquad  \qquad && \text{$q_T+\gamma_T \geq 1$}, \label{T3} \tag{T-3}
\end{alignat}
where $R$ is a polynomial in $1/t$, $\widehat{v}=(\widehat{v}^i)_{1 \leq i \leq 3} $ and $z=(z_{\mu k})_{0 \leq \mu < k \leq 3}$, of degree $\deg_z R$ in $z$, and 
$$q_H+\deg_{z}R \leq \beta_H, \qquad p \leq \beta_H,  \qquad q+|\gamma|+|\kappa| \leq |\beta|+1, \qquad q, \, |\gamma|, \, |\kappa| \leq |\beta|, \qquad 0 \leq \alpha, \lambda,\nu \leq3.$$
\end{Pro}
\begin{Rq}
In fact, we could prove that, as for the first order commutation formula, most of the error terms satisfy a form of null condition. Since this property is not crucial for our purpose, we chose to demonstrate a result requiring a much simpler analysis.
\end{Rq}
\begin{proof}
Note first that the result holds for any $|\beta|=1$. One can see it by applying either Lemma \ref{LemCom}, for the translation, or Proposition \ref{improvedCom0} and by rewriting all the $v$ derivatives as $v^0 \partial_{v^j}=\widehat{\Omega}_{0j}-t\partial_{x^j}-x^j \partial_t$. Let $n \geq 1$ such that the Proposition holds for any $|\beta| =n$ and consider a multi-index $|\beta_0|=n+1$. Consider further $|\beta|=n$ as well as $\widehat{Z}^{\mathrm{mod}} \in \widehat{\mathbb{P}}_0^{\mathrm{mod}}$ such that $\widehat{Z}^{\mathrm{mod},\beta_0}=\widehat{Z}^{\mathrm{mod}} \widehat{Z}^{\mathrm{mod},\beta}$ and remark
\begin{equation}\label{Comformulaind}
[\T_F,\widehat{Z}^{\mathrm{mod},\beta_0}]= [\T_F,\widehat{Z}^{\mathrm{mod}}]\widehat{Z}^{\mathrm{mod},\beta}+\widehat{Z}^{\mathrm{mod}} [\T_F,\widehat{Z}^{\mathrm{mod},\beta}].
\end{equation}
We can deal with the first term on the right hand side by applying the result for first order operators and by noticing that $\widehat{Z}^\xi \widehat{Z}^{\mathrm{mod},\beta}$, for $|\xi| \leq 1$, can be written as a linear combination of terms of the form
\begin{equation}\label{eq:rewritemachin}
 P_{p,q}(\C)\widehat{Z}^\zeta, \qquad p \leq \beta_H, \quad q_H \leq \beta_H+\xi_H-1, \quad  q \leq |\beta|+|\xi|-1,  \quad q+|\zeta| \leq |\beta|+|\xi|.
 \end{equation}
For the second term, we apply the induction hypothesis, so that $[\T_F,\widehat{Z}^{\mathrm{mod},\beta}]$ can be written as a linear combination of terms of the form \eqref{T1}-\eqref{T3}. In order to deal with them, we will use the following properties.
\begin{itemize}
\item $\partial_t (t)=1$, $\widehat{\Omega}_{0j}^{\mathrm{mod}}(t)=x^j=-z_{0j}-t\widehat{v}^j$, $S^{\mathrm{mod}}(t)=t$ and $\widehat{Z}^{\mathrm{mod}}(t)=0$ otherwise.
\item If $\widehat{Z}^{\mathrm{mod}}=\partial_{x^\mu}$, then $\widehat{Z}^{\mathrm{mod}}(x^k ) =\delta_{\mu}^k$. Otherwise, there exists $0 \leq \lambda \leq 3$ such that $\widehat{Z}^{\mathrm{mod}}(x^k)=\pm x^\lambda+\C^k_{\widehat{Z}}$.
\item $\widehat{\Omega}^{\mathrm{mod}}_{0j}(v^0)=v^j$ for any $1 \leq j \leq 3$ and $\widehat{Z}^{\mathrm{mod}}(v^0)=0$ ortherwise.
\item There exists four polynomials $R_0$, ..., $R_3$ such that
$$ \widehat{Z}^{\mathrm{mod}}\! \left(  R \! \left( 1/t,\widehat{v},z \right) \right) =R_0 \! \left( 1/t,\widehat{v},z \right)+\C_{\widehat{Z}}^i\, R_i \! \left( 1/t,\widehat{v},z \right) \!, \quad \deg_z \! R_0 \leq \deg_z \! R+1, \; \; \deg_z \! R_{i}\leq\deg_z \! R,$$
where we set $\C^i_{\partial_{x^\mu}}:=0$. Moreover, if $\widehat{Z}^{\mathrm{mod}} \neq \widehat{\Omega}_{0j}^{\mathrm{mod}}$, then $\deg_z \! R_0 \leq \deg_z \! R$. This can be obtained by the first property and \cite[Lemma~$3.2$]{dim3}, giving
$$\forall \, \widehat{\Gamma} \in \widehat{\mathbb{P}}_0, \; \forall \, 1 \leq i \leq 3, \; \forall \, z \in \mathbf{k}_1, \qquad \widehat{\Gamma}(v^0 z) \in \{0 \} \cup \mathbf{k}_1, \qquad   \partial_{x^i}(z) \in \{0,1,\widehat{v}^{k} \, | \; 1 \leq k  \leq 3 \} .$$
\item If $\widehat{Z}^{\mathrm{mod}}=\partial_{x^\mu}$, we schematically have $\widehat{Z}^{\mathrm{mod}}\left( P_{p,q}(\C) \right)= P^0_{p,q^0}(\C)$, with $q^0=q+1$ and $q^0_H=q_H$. Otherwise, $\widehat{Z}^{\mathrm{mod}}\left( P_{p,q}(\C) \right)= P^1_{p,q^1}(\C)+P^2_{p+1,q^2}(\C)$, with $q^1=q^2=q+1$, $q^1_H=q_H+1$ and $q^2_H=q_H$.
\item $\widehat{Z}^{\mathrm{mod}} \widehat{Z}^\kappa = \widehat{Z}\widehat{Z}^\kappa+\C^i_{\widehat{Z}} \partial_{x^i} \widehat{Z}^\kappa$ and $\widehat{Z}^{\mathrm{mod}} \mathcal{L}_{Z^\gamma}(F)_{\lambda \nu}$ can be written as a linear combination of
$$ \mathcal{L}_{ZZ^\gamma}(F)_{\lambda \nu}, \qquad \C^i_{\widehat{Z}} \mathcal{L}_{\partial_{x^i}Z^\gamma}(F)_{\lambda \nu}, \qquad \mathcal{L}_{Z^\gamma}(F)_{\mu \xi}, \qquad 0 \leq \mu, \xi \leq 3.$$
\end{itemize}
Hence, we obtain by applying $\widehat{Z}^{\mathrm{mod}}$ to any quantity of the form \eqref{T1}, \eqref{T2} or \eqref{T3} (corresponding to $|\beta|=n$), a combination of terms of the form \eqref{T1}-\eqref{T3} (corresponding to $|\beta_0|=n+1$) as well as
$$ \mathcal{T}[\widehat{Z}^{\mathrm{mod}}]= \frac{1}{t} R\left(\frac{1}{t},\widehat{v}, z\right)P_{p,q}(\C)\widehat{Z}^{\mathrm{mod}}\left(t^2 \frac{\widehat{v}^\mu}{v^0}  {\mathcal{L}_{Z^\gamma}(F)_{\mu \nu}}- \frac{\widehat{v}^\mu}{v^0} \mathcal{L}_{Z^\gamma}(F)_{\mu \nu}^{\infty}(v)  \right)  \widehat{Z}^{\kappa} ,  $$
where $0 \leq \nu \leq 3$, $q+|\gamma|+|\kappa| \leq |\beta|+1$, $\max(q ,  |\gamma|, |\kappa|) \leq |\beta|$, $p \leq \beta_H$ and $q_H+\deg_z R \leq \beta_H$. Assume first that $\widehat{Z}^{\mathrm{mod}}$ is a translation $\partial_{x^\lambda}$. Then,
$$ \mathcal{T}[\partial_{x^\lambda}] = \frac{ 2\delta_{\lambda}^0}{v^0} R\left(1/t,\widehat{v}, z\right)P_{p,q}(\C) \widehat{v}^\mu  {\mathcal{L}_{Z^\gamma}(F)_{\mu \nu}}  \widehat{Z}^{\kappa}+\frac{t}{v^0} R\left(1/t,\widehat{v}, z\right)P_{p,q}(\C)\widehat{v}^\mu  {\mathcal{L}_{\partial_{x^\lambda}Z^\gamma}(F)_{\mu \nu}}  \widehat{Z}^{\kappa}$$
is the sum of a term of type \eqref{T2} and a term of type \eqref{T3}. Otherwise, $\widehat{Z}^{\mathrm{mod}}=\widehat{Z}+\C^i_{\widehat{Z}} \partial_{x^i}$ and, following the previous computations, we have
$$\mathcal{T}[\widehat{Z}^{\mathrm{mod}}]=\mathcal{T}[\widehat{Z}]+\C_{\widehat{Z}}^i \mathcal{T}[\partial_{x^i}]=\mathcal{T}[\widehat{Z}]+\frac{t}{v^0} R\left(1/t,\widehat{v}, z\right)P_{p,q}(\C) \C^i_{\widehat{Z}} \, \widehat{v}^\mu  {\mathcal{L}_{\partial_{x^i}Z^\gamma}(F)_{\mu \nu}}  \widehat{Z}^{\kappa},$$
where the last three terms are of type \eqref{T3}. According to Corollary \ref{CorLorentzforceasympt}, $\mathcal{T}[\widehat{Z}]$ is a combination of terms of type \eqref{T1} and, in the case $\widehat{Z}=\widehat{\Omega}_{0j}$, \eqref{T2}. 
\end{proof}
We now control these error terms and then prove a uniform boundedness statement for $\widehat{Z}^{\mathrm{mod},\beta}f$. Because of regularity issues on the coefficients $\C_{\widehat{Z}}$, which are of class $C^{N-2}$, we are not able to deal with the multi-indices $|\beta| \geq N-1$.
\begin{Pro}\label{ProestCommod}
Let $|\beta| \leq N-2$. For all $(t,x,v) \in [3,+\infty[ \times \R^3_x \times \R^3_v$, there holds
$$  \left| \T_F\left( \widehat{Z}^{\mathrm{mod},\beta} f \right)\right|(t,x,v) \lesssim \Lambda \frac{\log^{3N_x+4N}(t)}{t^{1+\delta}}\sum_{|\kappa| \leq |\beta|}  |v^0|^7 \left| \mathbf{z}^{2+\beta_H} \widehat{Z}^\kappa f \right|(t,x,v).$$
Moreover, 
$$ |v^0|^{N_v-7} \left|\mathbf{z}^{N_x-2-\beta_H} \, \T_F\left( \widehat{Z}^{\mathrm{mod},\beta} f \right)\right|(t,x,v) \lesssim \overline{\epsilon} \, \frac{\log^{6N_x+7N}(t)}{t^{1+\delta}}.$$
\end{Pro}
\begin{proof}
Fix $(t,x,v) \in [3,+\infty[ \times \R^3_x \times \R^3_v$ and let us prove first the following property. Consider $P_{p,q}(\C)$ and $R(1/t,\widehat{v},z)$ a polynomial such that $p \leq \beta_H$, $q \leq |\beta|$ and $q_H+\deg_{z} R \leq \beta_H$. Then,
\begin{align}
\big| R(1/t,\widehat{v},z) \big|\,  \big| P_{p,q} (\C)\big|(t,x,v)  & \lesssim  \frac{\log^{N-2}(t)}{t^{q_T}} \mathbf{z}^{\beta_H}(t,x,v). \label{boundPpq}
\end{align}
For this, remark first that for $|\xi| \leq N-2$, $i \in \llbracket 1 , 3 \rrbracket$ and $\widehat{Z}\in \widehat{\mathbb{P}}_0 \setminus \{ \partial_t, \partial_{x^1}, \partial_{x^2}, \partial_{x^3} ,S\}$,
$$ \left| \widehat{Z}^{\xi} \left( \C^i_{\widehat{Z}} \right) \right|(t,x,v) \leq \sum_{\substack{|\gamma|+|\kappa| \leq |\xi| \\ \gamma_T=\xi_T}}  \mathcal{I}_{\gamma , \kappa}, \qquad \mathcal{I}_{\gamma , \kappa}:= \sum_{0 \leq \nu \leq 3} \left| \widehat{Z}^{\gamma} \log (t) \right|\left| \widehat{Z}^\kappa \left( \frac{\widehat{v}^\mu}{v^0} \mathcal{L}_Z(F)^\infty_{\mu \nu} \right) \right|(v).$$
Note that the case $\widehat{Z}=S$ leads to a similar estimate. 
\begin{itemize}
\item We have $| \widehat{Z}^{\gamma} \log (t) | \lesssim t^{-\xi_T}\mathbf{z}^{\xi_H}(t,x,v) \log(t)$. Indeed, $ |\widehat{Z}^{\gamma} \log (t) |\leq |t^{-\gamma_T} P_{\gamma_H}(x/t) \log(t)|$, where $P_{\gamma_H}$ is a polynomial of degree at most $\gamma_H \leq \xi_H$, and $\gamma_T=\xi_T$. Finally, recall that $|x| /t \leq |x-t\widehat{v}|/t+1 \leq 2 \mathbf{z}(t,x,v)$.
\item To deal with the last factor in $\mathcal{I}_{\gamma, \kappa}$, note first that $|\kappa|+1 \leq N-1$ and that this quantity vanishes if $\kappa$ is composed by at least a translation or the scaling vector field $S$ according to Proposition \ref{estiFinfty}. Then, using first the relations \eqref{eq:boosti}-\eqref{eq:rotation} and then Propositions \ref{estiFinfty}, we get
\begin{equation}\label{eq:calculZmodi}
 \left| \widehat{Z}^\kappa \left( \frac{\widehat{v}^\mu}{v^0} \mathcal{L}_Z(F)^\infty_{\mu \nu} \right) \right|(v) \lesssim \sum_{|\zeta| \leq |\kappa|+1} \left| \frac{\widehat{v}^\mu}{v^0} \mathcal{L}_{Z^\zeta}(F)^\infty_{\mu \nu} \right|(v) \lesssim 
     \overline{\epsilon} .
 \end{equation}
\end{itemize}
We then deduce that
$$ |R(1/t,\widehat{v},z)|\left|  P_{p,q} (\C)\right|(t,x,v) \lesssim \mathbf{z}^{\deg_z\! R}(t,x,v)t^{-q_T} \mathbf{z}^{q_H}(t,x,v) \log^p(t) \, \overline{\epsilon}^{p}  ,$$
which implies \eqref{boundPpq}.

Apply Proposition \ref{improvedCom} in order to reduce the analysis to the treatment of terms of type \eqref{T1}, \eqref{T2} and \eqref{T3}. By Corollary \ref{Corinduced2} and \eqref{boundPpq}, we can bound any term of type \eqref{T1} by
$$ \Lambda \frac{ |v^0|^{8} \log^{3N_x+4N}(t)}{v^0t^{1+\delta}} \langle x-t\widehat{v} \rangle^2  \left|\mathbf{z}^{\beta_H} \widehat{Z}^\kappa f \right|(t,x,v)\lesssim \Lambda \frac{\log^{3N_x+4N}(t)}{t^{1+\delta} }|v^0|^7  \left|\mathbf{z}^{2+\beta_H} \widehat{Z}^\kappa f\right|(t,x,v), $$
since $\langle x-t\widehat{v} \rangle \leq \mathbf{z}(t,x,v)$ and where $|\kappa| \leq N-2$. We deal with the ones of type \eqref{T2} by using \eqref{boot1}, \eqref{boundPpq} and Lemma \ref{gainv}. There are bounded above by
$$ \frac{ \Lambda \log^{N-2}(t) }{(t+|x|)(1+|t-|x||)v^0} \frac{(1+|t-|x||)|v^0|^2\mathbf{z}}{t+|x|} \left| \mathbf{z}^{\beta_H} \widehat{Z}^\kappa f\right| (t,x,v) \lesssim \Lambda \frac{\log^{N-2}(t)}{t^2}v^0\left| \mathbf{z}^{1+\beta_H} \widehat{Z}^\kappa f\right|(t,x,v).$$
Finally, let $\mathcal{T}_3$ be a term of type \eqref{T3}. Using first \eqref{boundPpq} together with Proposition \ref{Prodecaytoporder} and then Lemma \ref{gainv}, 
$$ \mathcal{T}_3 \lesssim  \frac{\Lambda\log^{N-2}(t)}{v^0t^{q_T}(1+|t-|x||)^{1+\gamma_T}}\left| \mathbf{z}^{\beta_H } \widehat{Z}^\kappa f\right|(t,x,v)\lesssim  \Lambda \frac{\log^{N}(t)}{t^2}|v^0|^3\left| \mathbf{z}^{2+\beta_H } \widehat{Z}^\kappa f\right|(t,x,v).$$
We deduce from that the first estimate of the statement, which, through an application of Proposition \ref{estiLinfini}, implies the second one.
\end{proof}

\begin{Cor}\label{Proimproderiv}
Let $|\beta| \leq N-2$. If $\beta_H \leq N_x-2$, there exists $\overline{D} >0$ such that
\begin{equation}\label{lastboot}
  \forall \, t \geq 3, \qquad \big\| |v^0|^{N_v-7}  \widehat{Z}^{\mathrm{mod},\beta} f(t,\cdot , \cdot) \big\|_{L^\infty_{x,v}} \lesssim \epsilon \, e^{\overline{D} \Lambda} .
 \end{equation}
\end{Cor}
\begin{proof}
Note first that we can obtain, by a much simpler analysis than in the proof of Proposition \ref{estiLinfini}, that $\| |v^0|^{N_v} \mathbf{z}^{N_x} \widehat{Z}^{\beta} f(3,\cdot , \cdot) \|_{L^\infty_{x,v}} \lesssim  \overline{\epsilon} $ for all $|\beta| \leq N$. Consequently, using \eqref{eq:rewritemachin} and \eqref{eq:calculZmodi}, we get
$$\forall \, |\beta| \leq N-1, \qquad \big\| |v^0|^{N_v} \mathbf{z}^{N_x} \widehat{Z}^{\mathrm{mod},\beta} f(3,\cdot , \cdot) \big\|_{L^\infty_{x,v}} \lesssim  \sum_{|\kappa| \leq |\beta|} \big\| |v^0|^{N_v} \mathbf{z}^{N_x} \widehat{Z}^{\kappa} f(3,\cdot , \cdot) \big\|_{L^\infty_{x,v}} \lesssim  \overline{\epsilon} .$$ 
Hence, it suffices to prove, according to Lemma \ref{techLemTF}, that
$$ \left| \T_F \left(|v^0|^{N_v-7}  \widehat{Z}^{\mathrm{mod},\beta} f \right)\right|(t,x,v) \lesssim \left(\frac{\Lambda \, |v^0|^{N_v-7} \big| \widehat{Z}^{\mathrm{mod},\beta} f \big|}{(1+t)^{\frac{3}{2}}}+ \frac{\Lambda \, \widehat{v}^{\underline{L}} |v^0|^{N_v-7} \big| \widehat{Z}^{\mathrm{mod},\beta} f \big|}{(1+|t-|x||)^2} \right)+\frac{\overline{\epsilon}}{(1+t)\log^2(3+t)}.$$
for all $(t,x,v) \in [3,T[ \times \R^3_x \times \R^3_v$ and any $|\beta| \leq N-2$. For this, we bound $\T_F(v^0)$ using \eqref{boundTfv0} and we apply the previous Proposition \ref{ProestCommod} in order to control $\T_F(\widehat{Z}^{\mathrm{mod},\beta} f)$.
\end{proof}

\subsection{Regularity of the asymptotic state}

In order to prove that $f_\infty$ is differentiable with respect to $v$, we will need to compute the first order $v$ derivatives of the correction terms in the modified spatial characteristics and to bound their higher order derivatives.
\begin{Lem}\label{dvC}
Let $(i,k) \in \llbracket 1 , 3 \rrbracket^2$. Then, for all $(t,x,v) \in [3,+\infty[\times \R_x^3 \times \R^3_v$,
$$ v^0 \partial_{v^k} \C^i (t,v)= \C_{\Omega_{0k}}^i (t,v)-\widehat{v}^i \C^k (t,v).$$
More generally, for any multi-index $|\kappa| \leq N-1$, 
$$|v^0|^{|\kappa|} |\partial^\kappa_v \C^i| (t,v) \lesssim   \overline{\epsilon} \, |v^0|^{-\frac{1}{2}} \log(t) . $$
\end{Lem}
\begin{proof}
According to \eqref{eq:boosti}, we have for any $\nu \in \llbracket 0,3 \rrbracket$,
$$v^0 \partial_{v^k}\left( \frac{\widehat{v}^\mu}{v^0} F^\infty_{\mu \nu} \right) =\frac{\widehat{v}^\mu}{v^0}\mathcal{L}_{\Omega_{0k}}(F)^\infty_{\mu \nu}-\delta_\nu^0 \frac{\widehat{v}^\mu}{v^0} F^\infty_{\mu k}-\delta_\nu^k \frac{\widehat{v}^\mu}{v^0} F_{\mu 0}^\infty. $$
This implies in particular that
\begin{align*}
v^0 \partial_{v^k}\!\left( \frac{\widehat{v}^i\widehat{v}^\mu}{v^0}F^\infty_{\mu 0}+\frac{\widehat{v}^\mu}{v^0}F^\infty_{\mu i} \right) &= \frac{\widehat{v}^i\widehat{v}^\mu}{v^0}\mathcal{L}_{\Omega_{0k}}(F)^\infty_{\mu 0}+\frac{\widehat{v}^\mu}{v^0}\mathcal{L}_{\Omega_{0k}}(F)^\infty_{\mu i}  -\widehat{v}^i\left(\frac{\widehat{v}^k\widehat{v}^\mu}{v^0}F^\infty_{\mu 0}+\frac{\widehat{v}^\mu}{v^0}F^\infty_{\mu k}\right)\!. 
\end{align*}
In view of the definition of the correction coefficients (see Definitions \ref{DeXc} and \ref{Defmodcom}), we deduce from this last equality the first part of the statement. The second part follows from a direct induction as well as Propositions \ref{estiFinfty}-\ref{Proderivinfini}.
\end{proof}
\begin{Rq}\label{RqLor}
Similarly, we could prove using \eqref{eq:rotation} that $ \Omega_{jk}^v \C^i (t,v)= \C_{\Omega_{jk}}^i (t,v)-\delta^i_j \C^k (t,v)+\delta^i_k \C^j(t,v)$, where $\Omega_{jk}^v:=v^j\partial_{v^k}-v^k \partial_{v^j}$. Consequently, the following quantities, related to the asymptotic Lorentz force, 
$$\Gamma(v):=  \frac{\widehat{v}^\mu}{v^0} \big(F^\infty_{\mu i}(v)+\widehat{v}_i F^\infty_{\mu 0}(v) \big) \dr v^i, \qquad  \Gamma_{Z}(v):= \frac{\widehat{v}^\mu}{v^0} \big(\mathcal{L}_{Z}(F)^\infty_{\mu i}(v)+\widehat{v}_i \mathcal{L}_{Z}(F)^\infty_{\mu 0}(v) \big) \dr v^i,$$
satisfies $\mathcal{L}_{v^0 \partial_{v^k}} (\Gamma)=\Gamma_{\Omega_{0k}}$ and $\mathcal{L}_{\Omega_{jk}^v}(\Gamma)= \Gamma_{\Omega_{jk}}$.
\end{Rq}
We now perform a computation, which holds for any sufficiently regular function $\mathbf{f}$. In particular, we will apply it to $\mathbf{f}=\partial_{t,x}^\kappa f$. We have
\begin{align*}
v^0 \partial_{v^k} \left( \mathbf{f}(t,X_\C,v) \right)&=t\partial_{x^k}\mathbf{f}(t,X_\C,v)-t\widehat{v}^k \widehat{v}^i \partial_{x^i} \mathbf{f}(t,X_\C,v)+v^0 \partial_{v^k} \mathbf{f}(t,X_\C,v)+v^0 \partial_{v^k} \C^i(t,v) \partial_{x^i} \mathbf{f}(t,X_\C,v) .
\end{align*}
Then, we use \eqref{derivv:eq} in order to rewrite the third term on the right hand side. We get
\begin{equation*}
 v^0 \partial_{v^k}\left( \mathbf{f}(t,X_\C,v)\right)=\!\left( \widehat{\Omega}_{0k}  \mathbf{f} + z_{0k} \partial_{t} \mathbf{f} -\widehat{v}^k S \mathbf{f}-\widehat{v}^k \sum_{1 \leq i \leq 3}z_{0i} \partial_{x^i} \mathbf{f} \right)\!(t,X_\C,v)+v^0\partial_{v^k} \C^i(t,v) \partial_{x^i} \mathbf{f}(t,X_\C,v) . 
\end{equation*}
Hence, as $z_{0i}(t,X_\C,v)=-x^i-\C^i(t,v)$, 
\begin{align*}
v^0 \partial_{v^k} (\mathbf{f}(t,X_\C,v))& =\left( \widehat{\Omega}_{0k}  \mathbf{f}\right)(t,X_\C,v)-x^k(\partial_t \mathbf{f}) (t,X_\C,v)-\frac{\C^k(t,v)}{t} (S\mathbf{f})(t,X_\C,v) \\
& \quad +\frac{\C^k(t,v)}{t}X^i_\C \partial_{x^i}\mathbf{f}(t,X_\C,v)  -\widehat{v}^k (S\mathbf{f})(t,X_\C,v)+\widehat{v}^k \C^i(t,v) \partial_{x^i}\mathbf{f}(t,X_\C,v) \\
& \quad +\widehat{v}^k x^i \partial_{x^i}\mathbf{f}(t,X_\C,v)  +v^0 \partial_{v^k} \C^i(t,v) \partial_{x^i} \mathbf{f}(t,X_\C,v) .&
\end{align*}
Now, according to Lemma \ref{dvC},
$$\widehat{\Omega}_{0k}+ v^0 \partial_{v^k} \C^i(t,v) \partial_{x^i}=\widehat{\Omega}_{0k}+\C_{\Omega_{0k}}^i(t,v) \partial_{x^i} -\C^k(t,v) \widehat{v}^i \partial_{x^i}=\widehat{\Omega}_{0k}^{\mathrm{mod}}-\C^k(t,v) \widehat{v}^i \partial_{x^i},$$
and, in view of the relations $S^{\mathrm{mod}}=S-\C^i(t,v)\partial_{x^i}$ and $X_\C^i=x^i+t\widehat{v}^i+\C^i(t,v)$,
\begin{align} \nonumber
v^0 \partial_{v^k} (\mathbf{f}(t,X_\C,v))&= \left(\widehat{\Omega}_{0k}^{\mathrm{mod}}  \mathbf{f}\right)(t,X_\C,v)-\left( \widehat{v}^k+\frac{\C^k(t,v)}{t}\right) \left( S^{\mathrm{mod}}\mathbf{f}\right) (t,X_\C,v)\\
& \quad -x^k \left( \partial_t \mathbf{f} \right) (t,X_\C,v)+\left(\widehat{v}^k+\frac{\C^k(t,v)}{t}\right) x^i \partial_{x^i}\mathbf{f}(t,X_\C,v) . \label{eq:forscmap}
\end{align}
Iterating this process to the functions $\mathbf{f}=\partial_{t,x}^\kappa f$ yields the following result.
\begin{Pro}\label{derivhexpression}
Let $|\kappa|+|\xi| \leq N-2$. Then, there exists functions $P_{\beta}^{\kappa,\xi}$ such that,
$$\forall (t,x,v) \in [3,+\infty[ \times \R^3_x \times \R^3_v, \qquad |v^0|^{|\xi|} \partial^\xi_v \Big( (\partial_{t,x}^\kappa f)(t,X_\C,v) \Big) = \sum_{|\beta| \leq |\kappa|+|\xi|} P_{\beta }^{\kappa,\xi}(t,x,v) \widehat{Z}^{\mathrm{mod},\beta} f(t,X_\C,v)$$
and $P_{\beta}^{\kappa,\xi}(t,x,v)$ is a linear combination of terms of the form $P(x,\widehat{v})M(\C)$, where $P$ is a polynomial and 
$$M(\C)=\prod_{ k =1}^d \frac{1}{t}|v^0|^{|\xi_k|} \partial_v^{\xi_k} \C^{i_k} (t,v), \qquad d+\sum_{1 \leq k \leq d} |\xi_k|  \leq |\xi|, \quad |\beta|+\sum_{1 \leq k \leq d} |\xi_k| \leq |\xi| , \quad \deg_x(P)+\beta_H \leq |\xi|.$$
The value $d=0$ is allowed, in which case we set $M(\C)=1$.
\end{Pro}
In order to prove, through Proposition \ref{VMh}, that the functions considered in the previous statement converge, as $t \to +\infty$, we will be lead to estimate these polynomials and their time derivative.
\begin{Lem}\label{LemestiPoly}
Let $|\kappa|+|\xi| \leq N-2$ and $|\beta| \leq |\kappa|+|\xi|$. Then, for all $(t,x,v) \in [3,+\infty[ \times \R^3_x \times \R^3_v$,
\begin{align*}
\left| P_{\beta}^{\kappa,\xi} \right|(t,x,v)  \lesssim \langle x \rangle^{|\xi|-\beta_H} , \qquad \qquad \left| \partial_t  P_{\beta}^{\kappa,\xi} \right|(t,x,v)  \lesssim \overline{\epsilon} \, \langle x \rangle^{|\xi|-\beta_H} \frac{\log(t)}{t^2}.
\end{align*}
\end{Lem}
\begin{proof}
It suffices to bound terms of the form $P(x,\widehat{v}) M(\C)$ satisfying the conditions given in Proposition \ref{derivhexpression}. The first factor satisfies $|P(x,\widehat{v})| \lesssim \langle x \rangle^{\deg_x \! P}\leq\langle x \rangle^{|\xi|-\beta_H}$ and does not depend on $t$. In view of Lemma \ref{dvC}, we have $| M(\C)| \lesssim \overline{\epsilon}^{\, d} \log^d(t)t^{-d}$, which implies the first estimate. The second one can be obtained similarly. Either $|\partial_t M(\C)|=0$ or $d \geq 1$ and $|\partial_t M(\C)| \lesssim  \overline{\epsilon}^{\, d} \log^d(t)t^{-d-1} $ by Lemma \ref{dvC}.
\end{proof}
We are now able to prove the main result of this paper. For this, let us introduce
$$ h:(t,x,v) \mapsto f(t,x+t\widehat{v}+\C(t,v),v), \qquad h^{\xi,\kappa}:=|v^0|^{|\xi|} \partial_v^\xi \partial_x^\kappa h(t,x,v)=  |v^0|^{|\xi|} \partial_v^\xi \Big(\partial_x^\kappa f (t,X_\C (t,x,v),v) \Big).$$
\begin{Pro}\label{mainresultPro}
There exists a function $f_\infty \in C^{N-2}( \R^3_x \times \R^3_v,\R_+)$ such that, for any $|\kappa|+|\xi| \leq N-2$, 
$$ \forall \, t \geq 3, \qquad \left\||v^0|^{N_v-10+|\xi|}\langle x \rangle^{N_x-4-|\xi|} \Big( \partial_v^{\xi}\partial_{x}^{\kappa} h(t,\cdot , \cdot)- \partial_v^\xi \partial_x^{\kappa} f_\infty \Big) \right\|_{L^{\infty}_{x,v}} \lesssim \overline{\epsilon} \, \frac{\log^{7(N_x+N)}(t)}{t^{\delta}}. $$
In particular, as $N_v >13$ and if $N_x >7+|\xi|$, we have $\partial_v^\xi\partial_x^\kappa f_\infty \in L^1_{x,v}$. 
\end{Pro}
\begin{proof}
Fix $t \geq 3$ and $(x,v) \in  \R^3_x \times \R^3_v$. Applying the previous Proposition \ref{derivhexpression} and Lemma \ref{LemestiPoly}, we get
$$ |\partial_t  h^{\xi,\kappa}|(t,x,v) \lesssim  \sum_{\substack{|\beta| \leq N-2 \\ \beta_H \leq |\xi|}} \langle x \rangle^{|\xi|-\beta_H} \left| \partial_t \widehat{Z}^{\mathrm{mod},\beta} f \right|\! (t,X_\C,v)+\overline{\epsilon} \, \frac{\log(t)}{t^2} \langle x \rangle^{|\xi|-\beta_H}  \left| \widehat{Z}^{\mathrm{mod},\beta} f \right|\! (t,X_\C,v).$$
Next, we recall from \eqref{eq:boundzX} the inequality $\langle x \rangle \lesssim \log(t) \, \mathbf{z}(t,X_\C,v)$ and remark, using the same arguments, that $\mathbf{z}(t,X_\C,v) \lesssim  \log(t) \, \langle x \rangle$ holds as well. Bounding $\partial_t \widehat{Z}^{\mathrm{mod},\beta}f$ by Proposition \ref{VMh}, we then get
\begin{align*}
|v^0|^{N_v-10}\langle x \rangle^{N_x-4-|\xi|} \left|\partial_t h^{\kappa , \xi} \right|(t,x,v) \lesssim \sum_{\substack{|\beta| \leq N-2 \\ \beta_H \leq |\xi|}} \log^{N_x}(t) \, |v^0|^{N_v-10}\left|\mathbf{z}^{N_x-4-\beta_H} \T_F(\widehat{Z}^{\mathrm{mod},\beta} f) \right|(t,X_\mathscr{C},v) &\\
+ \Lambda \frac{\log^{4N_x+3N}(t)}{t^{1+\delta}} \sum_{|\gamma| \leq 1}  |v^0|^{N_v-3}\left| \mathbf{z}^{N_x-2-\beta_H} \widehat{Z}^\gamma \widehat{Z}^{\mathrm{mod},\beta} f\right|(t,X_\mathscr{C},v)&.
\end{align*}
We control the first term on the right hand side by Proposition \ref{ProestCommod} and we claim that the second one is bounded by
$$ \Lambda \frac{\log^{4N_x+4N}(t)}{t^{1+\delta}}\sum_{|\kappa| \leq N-1} |v^0|^{N_v-3} \left| \mathbf{z}^{N_x-2} \widehat{Z}^\kappa  f\right|(t,X_\mathscr{C},v).$$
Indeed, we rewrite the modified vector fields using \eqref{eq:rewritemachin} and we control $P_{p,q}(\C)$ by \eqref{boundPpq}. We then deduce from Proposition \ref{estiLinfini} that
$$ |v^0|^{N_v-10}\langle x \rangle^{N_x-4-|\xi|} \left|\partial_t h^{\kappa , \xi} \right|(t,x,v) \lesssim  \overline{\epsilon} \, \frac{\log^{7N_x+7N}(t)}{t^{1+\delta}} .$$
We obtain from that
\begin{equation}\label{eq:onreutilisedirecteme} \forall \, 3 \leq t \leq \tau , \qquad \Big||v^0|^{N_v-10}\langle x \rangle^{N_x-4-|\xi|}\left( h^{\kappa,\xi}(\tau,x,v)-h^{\kappa , \xi}(t,x,v) \right) \Big| \lesssim \overline{\epsilon} \, \frac{\log^{7(N_x+N)}(t)}{t^\delta}.
\end{equation}
Consequently, there exists $f^{\kappa ,\xi}_\infty \in L^{\infty}_{x,v}$ such that $h^{\kappa , \xi}(t,\cdot,\cdot) \to f^{\kappa , \xi}_{\infty}$ as $t \to +\infty$, uniformly on any compact subset of $\R^3_x \times \R^3_v$. By uniqueness of the limit in $\mathcal{D}'(\R^3_x \times \R^3_v)$ and by continuity of the distributional partial derivatives, we get $f^{\kappa , \xi}_\infty=|v^0|^{|\xi|} \partial_v^\xi \partial_x^\kappa f_\infty$. Letting $\tau \to +\infty$ in \eqref{eq:onreutilisedirecteme} yields the stated rate of convergence and concludes the proof.
\end{proof}
\begin{Rq}\label{Rqfinfty}
We can improve the result for $f_\infty$. Propositions \ref{estiLinfini} and \ref{VMh} give
$$ \forall \, t \geq 3, \qquad \left\||v^0|^{N_v-7}\langle x \rangle^{N_x-2} \left(  f(t,X_\C (t,\cdot,\cdot) , \cdot)- f_\infty \right) \right\|_{L^{\infty}_{x,v}} \lesssim \overline{\epsilon}  \, \frac{\log^{12+3N_x+3N}(t)}{t^{\delta}}. $$
Moreover, we could prove that $f_\infty$ is of class $C^{N-1}$ according to the spatial variables $x$.
\end{Rq}
\begin{Rq}
We could prove that $\partial_v^\xi(\partial_t^n \partial_x^\kappa f(t,X_\C,v)) \to \partial_v^\xi (-\widehat{v} \cdot \nabla_x)^n \partial_x^\kappa f_\infty$. The idea consists in rewriting the time derivatives using $\partial_t=-\widehat{v}\cdot \nabla_x+\T_F-\widehat{v}^\mu {F_\mu}^j\partial_{v^j}$.
\end{Rq}

\section{Scattering result for the electromagnetic field}\label{Sec7}

In this section, we start by defining the scattering state of a sufficiently regular Maxwell field. Then, we construct a scattering map for the vacuum Maxwell equations. Finally, we apply these results together with the estimates derived in Section \ref{subsecpointwisedecay} in order to prove that the electromagnetic field $F$ scatters, in the sense that it is approached by a solution to the homogeneous Maxwell equations.

Since the asymptotic states will be functions of the variables $(u,\theta , \varphi)$, defined on future null infinity $\mathcal{I}^+$ introduced in Section \ref{subsec2}, it will be convenient to work in null coordinates. For a function $\psi (t,x)$, in order to simplify the presentation, we will write $\psi (u,\underline{u},\omega)$ to denote $\psi(\frac{\underline{u}+u}{2},\frac{\underline{u}-u}{2}\omega)$, where $(u,\underline{u},\omega)$ are the null coordinates such that $x=r\omega$, $\underline{u}=t+r$ and $u=t-r$.

The scattering state of a smooth electromagnetic field $G$ will give the leading order term in the asymptotic expansion of $rG$, as $\underline{u} \to +\infty$. This motivates the introduction of the following terminology.
\begin{Def}
Let $\phi : \R_+ \times \R^3 \rightarrow \R$ be a function such that the limit
$$  \Phi(u,\omega):= \lim_{r \to +\infty}  r \phi (u+r,r\omega)= \lim_{\underline{u} \to + \infty} (r  \phi) (u,\underline{u},\omega), \qquad \qquad \Phi(u,\omega) <+\infty$$
exists and is finite for all $(u,\omega) \in \R_u \times \mathbb{S}^2$. Then, we say that the function $\Phi$, defined on $\R_u \times \mathbb{S}^2$, is the radiation field $ \mathscr{R}(\phi)$ of $\phi$ along future null infinity $\mathcal{I}^+$.
\end{Def}
\begin{Def}
Similarly, consider $\beta$, a $1$-form on $\R_+\times \R^3$ tangential to the $2$-spheres\footnote{More generally, we could consider tensor fields tangential to the cones $\underline{C}_{\underline{u}}$.} such that $\beta_{e_\theta}$ and $\beta_{e_{\varphi}}$ have a radiation field $\beta_{e_\theta}^{\mathcal{I}^+}$ and $\beta_{e_\varphi}^{\mathcal{I}^+}$. Then, $\beta^{\mathcal{I}^+}$, defined on $\R_u \times \mathbb{S }^2$ as the $1$-form $\beta^{\mathcal{I}^+}_{e_\theta}\dr \theta+\beta^{\mathcal{I}^+}_{e_\varphi} \dr \varphi$ tangential to the $2$-spheres, is called the radiation field of $\beta$ along $\mathcal{I}^+$.

If $\beta^{\mathcal{I}^+}$ is of class $C^1$, we define
$$\nabla_{\partial_u} (\beta):= \partial_u(\beta^{\mathcal{I}^+}_{e_\theta})\dr \theta+\partial_u(\beta^{\mathcal{I}^+}_{e_\varphi}) \dr \varphi, \qquad \slashed{\nabla}_{e_{\theta}} (\beta)(u,\cdot,\cdot) : = \slashed{\nabla}_{e_\theta}(\beta (u,\cdot,\cdot)), \qquad \slashed{\nabla}_{e_{\varphi}} (\beta)(u,\cdot,\cdot) : = \slashed{\nabla}_{e_\varphi}(\beta (u,\cdot,\cdot)),$$
where $\slashed{\nabla}$ denotes the covariant derivative on $\mathbb{S}^2$.
\end{Def}
We already know from Corollary \ref{Corgoodnull} that, given a sufficiently decaying electromagnetic field $G$, the radiation field of the good null components $\alpha (G)$, $\rho(G)$ and $\sigma (G)$ exist and vanish. Concerning the component $\underline{\alpha} (G)$, we have the following result.
\begin{Pro}\label{Corgoodnull2}
Let $G$ be a $C^1$ solution to the Maxwell equations \eqref{Maxwithsource} with a continuous source term $J$. Assume that there exists three constants $C[G] >0$, $p \in \mathbb{N}$ and $q >0$ such that
\begin{equation}\label{eq:assump0}
 \forall \; (t,x) \in \R_+ \times \R^3, \qquad r|J|(t,x)+\sum_{|\gamma| \leq 1} |\rho(\mathcal{L}_{Z^{\gamma}}G)|(t,x)+|\sigma(\mathcal{L}_{Z^{\gamma}}G)|(t,x) \leq \frac{C[G] \log^p(3+t+|x|)}{(1+t+|x|)^{1+q}}.
 \end{equation}
Then, $\underline{\alpha} (G)$ has a radiation field along $\mathcal{I}^+$. For any $B \in \{\theta , \varphi \}$ and for all $(u,\omega) \in \R_u \times \mathbb{S}^2$, the limit 
$$\underline{\alpha}_{e_B}^{\mathcal{I}^+}(u,\omega) :=\lim_{r \to +\infty} r\underline{\alpha}(G)_{e_B}(r+u,r\omega) = \lim_{\underline{u} \to +\infty} r\underline{\alpha}(G)_{e_B}(u,\underline{u},\omega)  $$
exists and is finite. Moreover,
$$ \forall \; (t,x) \in \R_+ \times \R^3, \qquad \left| r\underline{\alpha}(G)_{e_B}(t,x)-\underline{\alpha}^{\mathcal{I}^+}_{e_B}\left(t-|x|,\frac{x}{|x|} \right) \right| \lesssim C[G]\frac{ \log^p(3+t+|x|)}{(1+t+|x|)^q}. $$
Consequently, $\underline{\alpha}^{\mathcal{I}^+}$ is a continuous tensor field, defined on $\R_u\times \mathbb{S}^2$ and tangential to the $2$-spheres.
\end{Pro}
\begin{proof}
The last inequality of Lemma \ref{improderiv0}, together with \eqref{eq:assump0}, provides
\begin{equation}\label{eq:nablaLralaphabarre}
 \forall \; (t,x) \in \R_+ \times \R^3, \qquad \quad  \left| \nabla_L \left( r \underline{\alpha}(G) \right)\right|(t,x)  \lesssim \log^p(3+t+|x|)(1+t+|x|)^{-1-q}.
\end{equation}
Using the null coordinates $\underline{u}=t+r$ and $u=t-r$, where $x=r\omega$, we get, as $L=2\partial_{\underline{u}}$ and $\nabla_L e_B=0$,
$$ \forall \; 0 \leq \underline{u} \leq \underline{z}, \qquad \left|r\underline{\alpha}(F)(u,\underline{z},\omega)-r\underline{\alpha}(F)(u,\underline{u},\omega)\right| \lesssim \int_{s=\underline{u}}^{\underline{z}} \frac{\log^p(3+s)\dr s}{(1+s)^{1+q}} \lesssim \frac{\log^p(3+\underline{u})}{(1+\underline{u})^{q}},$$
implying the existence of $\underline{\alpha}^{\mathcal{I}^+}_{e_B}$, for any $B\in \{\theta, \varphi\}$, and the rate of convergence given in the statement.
\end{proof}
If the electromagnetic field is sufficiently regular, we can relate the radiation fields of the derivatives of $G$ to the ones of $\underline{\alpha}^{\mathcal{I}^+}$. For this, we will use the bounded functions $\omega_i:=x^i/|x|$ and $\omega_i^A:=\langle \partial_{x^i},e_A \rangle$, where $1 \leq i \leq 3$ and $A \in \{\theta , \varphi \}$, which depend only on $\omega \in \mathbb{S}^2$ and which are given explicitly in Appendix \ref{SecB}. 
\begin{Pro}\label{derivscatMax}
Suppose that $G$ verifies, in addition to the hypotheses of the previous Proposition \ref{Corgoodnull2}, the inequality $|rG|(t,x)\leq C[G]$. Then, for any $Z \in \mathbb{K}$, 
$$\exists \, \underline{\alpha}^{\mathcal{I}^+}_Z \in \mathcal{D}'(\R_u \times \mathbb{S}^2),  \qquad \qquad r \underline{\alpha}(\mathcal{L}_Z G)(\cdot,\underline{u},\cdot)   \xrightharpoonup[\underline{u} \to + \infty]{} \underline{\alpha}^{\mathcal{I}^+}_Z \qquad \text{in $\mathcal{D}'(\R_u \times \mathbb{S}^2)$}.$$
Moreover, for any $1 \leq i \leq 3$ and $1 \leq j < k \leq 3$,
\begin{alignat*}{2}
&  \underline{\alpha}^{\mathcal{I}^+}_{\partial_t}= \nabla_u \underline{\alpha}^{\mathcal{I}^+}, \qquad &&\underline{\alpha}^{\mathcal{I}^+}_{\partial_{x^i}}=-\omega_i \nabla_u \underline{\alpha}^{\mathcal{I}^+}, \qquad \underline{\alpha}^{\mathcal{I}^+}_{S}=u\nabla_u \underline{\alpha}^{\mathcal{I}^+}+\underline{\alpha}^{\mathcal{I}^+} \\
&\underline{\alpha}^{\mathcal{I}^+}_{\Omega_{jk}} = \mathcal{L}_{\Omega_{jk}} \big( \underline{\alpha}^{\mathcal{I}^+} \big), \qquad && \underline{\alpha}^{\mathcal{I}^+}_{\Omega_{0i}}=- \omega_i u\nabla_u  \underline{\alpha}^{\mathcal{I}^+}-2\omega_i \underline{\alpha}^{\mathcal{I}^+}+ \omega^{e_A}_i \slashed{\nabla}_{e_A}  \underline{\alpha}^{\mathcal{I}^+} .
\end{alignat*}
\end{Pro}
This result is proved in Appendix \ref{SecB}.

\subsection{Scattering map for the vacuum Maxwell equations}\label{subsecscatMax}

Before starting the construction of the forward map for the homogeneous Maxwell equations, we introduce two functional spaces adapted to our problem. The first one contains the initial electromagnetic fields which are in $L^2$ and the second one contains the scattering states which belong to $L^2$. For a smooth solution $F$ to \eqref{Maxvac}, this state will be the radiation field of $\underline{\alpha}(F)$. Note that the electromagnetic fields considered in this subsection \ref{subsecscatMax} will be denoted by $F$. Since, we will only consider solutions to the homogeneous Maxwell equations here, there is no risk of confusion with the electromagnetic field of the plasma considered in the remainder of the article.
\begin{Def}
Let $\mathcal{E}_{\{t=0 \}}$ be the set containing all the $2$-form on $\R^{1+3}$ which does not depend on $t$ and which are in $L^2(\R^3)$. Equipped with the norm
$$\|F_0\|^2_{\mathcal{E}_{\{t=0 \}}} := \int_{\R^3_x} \left( |\alpha(F_0)|^2+|\underline{\alpha}(F_0)|^2+2|\rho(F_0)|^2+2|\sigma(F_0)|^2 \right)(x) \dr x,$$
$\mathcal{E}_{\{t=0 \}}$ is a Hilbert space.

We define $\mathcal{E}_{\mathcal{I}^+}$ as the set of the $1$-form on $\R_u \times \mathbb{S}^2$, which are tangential to the $2$-spheres and in $L^2$. For
$$ \| \underline{\alpha}^{\mathcal{I}^+} \|^2_{\mathcal{I}^+} := \int_{\R_u} \int_{\mathbb{S}^2_{\omega}} |\underline{\alpha}^{\mathcal{I}^+}|^2(u,\omega) \dr \mu_{\mathbb{S}^2} \dr u ,$$
$(\mathcal{E}_{\mathcal{I}^+}, \| \cdot \|_{\mathcal{I}^+})$ is a Hilbert space.
\end{Def}
We now state the two main results of this section.
\begin{Th}\label{Thscat}
The linear map
\begin{array}[t]{lrcl}
&\mathscr{F}^+ : & \mathcal{E}_{\{t=0\}} \cap C_c^{\infty}& \longrightarrow  \mathcal{E}_{\mathcal{I}^+} \\
 &   & F_0 & \longmapsto  \lim_{\underline{u} \to +\infty} r\underline{\alpha}(F)(u,\underline{u},\omega), 
\end{array}

\noindent where $F$ is the unique solution to the vacuum Maxwell equations \eqref{Maxvac} such that $F(0,\cdot)=F_0$, is well-defined and preserves the norm $\|F_0\|_{\mathcal{E}_{\{t=0 \}}} =\| \mathscr{F}^+(F_0) \|_{\mathcal{I}^+}$. 

Moreover, this forward map can be uniquely extended in a bijective isometry $\mathscr{F}^+ : \mathcal{E}_{\{t=0\}} \rightarrow \mathcal{E}_{\mathcal{I}^+}$.
\end{Th}
\begin{Rq}
When $F_0 \notin C_c^{\infty}$ but is still sufficiently regular, then $\mathscr{F}^+(F_0)$ is also given by the formula written in Theorem \ref{Thscat}. Otherwise, $\mathscr{F}^+(F_0)$ can still be interpreted, in a weak sense, as the radiation field of $\underline{\alpha}(F)$, with $F$ the solution to \eqref{Maxvac} arising from the data $F_0$ (see Lemma \ref{estiweighednorm} below).
\end{Rq}
The proof will in particular rely on the following result, which is also important in itself. It provides precise estimates for solutions arising from the preimage by $\mathscr{F}^+$ of smooth elements of $\mathcal{E}_{\mathcal{I}^+}$.
\begin{Pro}\label{ProscatMax}
Let $0<a<1/2$, $N \in \mathbb{N}$ and $\underline{\alpha}^{\mathcal{I}^+} \in \mathcal{E}_{\mathcal{I}^+}$ be a sufficiently regular scattering state. Then, the unique solution $F$ to the vacuum Maxwell equations \eqref{Maxvac} satisfying $\mathscr{F}^+(F)=\underline{\alpha}^{\mathcal{I}^+}$ verifies, for any $0\leq q-\frac{1}{2} < a$,
$$  \sum_{|\gamma| \leq N } \| \langle t-r \rangle^{q-\frac{1}{2}} \left| \mathcal{L}_{Z^{\gamma}} F \right|(t,\cdot) \|^2_{L^2_x} \lesssim C[\underline{\alpha}^{\mathcal{I}^+}]:=\sum_{n_1+n_2+n_3 \leq N+3}\int_{\R_u} \! \int_{\mathbb{S}^2_{\omega}} \langle u \rangle^{2a+2n_1} \! \left|\nabla^{n_1}_u \slashed{\nabla}^{n_2}_{e_\theta} \slashed{\nabla}^{n_3}_{e_\varphi} \underline{\alpha}^{\mathcal{I}^+}\right|^2 \! (u,\omega) \dr \mu_{\mathbb{S}^2} \dr u$$
for all $t \in \R_+$. In particular, if $N \geq 4$, we have for any $|\gamma| \leq N-3$ and $|\xi| \leq N-4$,
\begin{align*}
\forall \; (t,x) \in \R_+ \times \R^3, \qquad \quad \left( |\alpha(\mathcal{L}_{Z^{\gamma}}F)|+|\rho(\mathcal{L}_{Z^{\gamma}}F)|+|\sigma(\mathcal{L}_{Z^{\gamma}}F)|\right)(t,x) & \leq \frac{C}{(1+t+|x|)^{1+q}} ,\\
\left|r\underline{\alpha}(\mathcal{L}_{Z^{\xi}}F)(t,x) - \mathscr{F}^+(\mathcal{L}_{Z^{\xi}}F(0,\cdot))  \left( t-|x|,\frac{x}{|x|} \right) \right| & \leq \frac{C}{(1+t+|x|)^q},
\end{align*}
where the constant $C$ depends only on $C[\underline{\alpha}^{\mathcal{I}^+}]$ and $q$.
\end{Pro}

We start by proving that $\mathscr{F}^+$ is well-defined for sufficiently regular electromagnetic field, including those arising from smooth compactly supported data.
\begin{Lem}\label{estiweighednorm}
The linear map $\mathscr{F}^+$ introduced in Theorem \ref{Thscat} is well-defined and extends in an injective isometry from $\mathcal{E}_{\{t=0\}}$ to $\mathcal{E}_{\mathcal{I}^+}$. Moreover, if $F$ is a solution to the free Maxwell equations \eqref{Maxvac} such that
\begin{equation}\label{hypo}
 C_F:= \sum_{|\gamma| \leq 4} \left\|  \mathcal{L}_{Z^{\gamma}}F(0,\cdot) \right\|_{\{t=0\}} < +\infty,
 \end{equation}
then, $\underline{\alpha}(F)$ has a continuous radiation field $\mathscr{F}^+(F(0,\cdot))$ and
\begin{align}
\forall \; (t,x) \in \R_+ \times \R^3, \quad \qquad \qquad \left( |\alpha(F)|+|\rho(F)|+|\sigma(F)|\right)(t,x) & \lesssim C_F(1+t+|x|)^{-\frac{3}{2}} ,\label{esti1good}\\
\left|r\underline{\alpha}(F)(t,x) - \mathscr{F}^+(F(0,\cdot))  \left( t-|x|,\frac{x}{|x|} \right) \right| & \lesssim C_F(1+t+|x|)^{-\frac{1}{2}}. \label{esti2badnull}
\end{align}
This implies that the radiation fields of $\alpha(F)$, $\rho(F)$ and $\sigma(F)$ vanish.

Finally, if $F$ is a mildly regular solution to \eqref{Maxvac} such that $F(0,\cdot) \in \mathcal{E}_{\{t=0\}}$, then $r\underline{\alpha}(F)$ converges to $\mathscr{F}^+(F(0,\cdot))$, as $\underline{u} \to +\infty$, in the space of distributions $\mathcal{D}'(\R_u \times  \mathbb{S}^2)$.
\end{Lem}
\begin{proof}
Recall from Definition \ref{Defenergytensor} the energy momentum tensor $\mathbb{T}[F]_{\mu \nu}$, its principal null components and that $\nabla^{\mu} \mathbb{T}[F]_{\mu 0}=0$. For any $t >0$, the divergence theorem, applied to $\mathbb{T}[F]_{\mu 0}$ in the domain $\{(s,x) \in \R^{1+3} \; | \; 0 \leq s \leq t \}$, gives
$$ \| F(0, \cdot) \|_{\{ t=0 \}} = 4\int_{\R^3_x} \mathbb{T}[F]_{00}(0,x) \dr x= 4\int_{\R^3_x} \mathbb{T}[F]_{00}(t,x) \dr x =2 \sum_{0 \leq \mu , \nu \leq 3}\int_{\R^3_x} |F_{\mu \nu}|^2(t,x) \dr x = 2\|F(t,\cdot) \|_{L^2_x}.$$
This also applies to $\mathcal{L}_{Z^{\gamma}}(F)$, for any $|\gamma| \leq 4$, since it is a solution to the free Maxwell equations \eqref{Maxvac} as well. In view of the equivalence of the pointwise norms \eqref{equinorm}, the standard Klainerman-Sobolev inequality (see for instance Theorem $1.3$ of \cite[Chapter~$II$]{Sogge}) yields, for any $|\gamma| \leq 2$,
\begin{equation}\label{decaytoremember}
\forall \; (t,x) \in \R_+\times \R^3_x, \qquad |\mathcal{L}_{Z^{\gamma}}F|(t,x)  \lesssim \sum_{|\beta| \leq 2+|\gamma|} \;\sum_{0 \leq \mu, \nu \leq 3} |Z^{\beta}(F_{\mu \nu})|(t,x) \lesssim \frac{C_F}{(1+t+|x|)(1+|t-|x||)^{\frac{1}{2}}}.
\end{equation}
Applying Corollary \ref{Corgoodnull} to $\mathcal{L}_{Z^{\xi}}F$, for any $|\xi| \leq 1$ and $q=1/2$, gives
$$ \forall \; |\xi| \leq 1, \; \; \forall \; (t,x) \in \R_+ \times \R^3, \qquad \qquad \left( |\alpha(\mathcal{L}_{Z^{\xi}}F)|+|\rho(\mathcal{L}_{Z^{\xi}}F)|+|\sigma(\mathcal{L}_{Z^{\xi}}F)|\right)(t,x)  \leq C_F(1+t+|x|)^{-\frac{3}{2}}.$$
The existence of the radiation field $\underline{\alpha}^{\mathcal{I}^+}$ of $\underline{\alpha}(F)$ and the rate of convergence given in the statement then follows from Corollary \ref{Corgoodnull2}. Since the convergence is uniform in $(u,\omega)$, $\underline{\alpha}^{\mathcal{I}^+}$ is continuous on $\R_u \times \mathbb{S}^2$. 

Before defining $\mathscr{F}^+$, we need to bound the $L^2$ norm of the radiation field. For this, we prove conservation laws which hold for any mildly regular solution $G$ to the free Maxwell equations \eqref{Maxvac}.

Fix $\underline{u} \geq 0$ and apply the divergence theorem to $\mathbb{T}[G]_{\mu 0}$, in the domain $\{ t+|x| \leq \underline{u} \}$, in order to get
\begin{flalign}\label{energyeq0}
&  \int_{\underline{C}_{\underline{u}}}\! \mathbb{T}[G]_{\underline{L}0} \dr \mu_{\underline{C}_{\underline{u}}}=\!\int_{|x| \leq \underline{u}}\! \mathbb{T}[G]_{00}(0,x) \dr x =\! \frac{1}{4}\! \int_{|x| \leq \underline{u}}\! \left( |\alpha(G)|^2+|\underline{\alpha}(G)|^2+2|\rho(F)|^2+2|\sigma(G)|^2 \right)\!(0,x) \dr x,&
  \end{flalign}
  where
  \begin{equation}\label{energyeq1} 
 \int_{\underline{C}_{\underline{u}}}\mathbb{T}[G]_{\underline{L}0} \dr \mu_{\underline{C}_{\underline{u}}} =\frac{1}{4} \int_{|u| \leq \underline{u}} \int_{\mathbb{S}^2_{\omega}} \left( |\underline{\alpha}(G)|^2+|\rho(G)|^2+|\sigma(G)|^2 \right)(u,\underline{u},\omega) r^2 \dr \mu_{\mathbb{S}^2} \dr u.
 \end{equation}
 
Assume now that $F_{\mu \nu}(0,\cdot) \in C_c^{\infty}(\R^3_x)$ for all $0 \leq \mu,\nu \leq 3$ and let us apply the previous equality to $F$. On the one hand, the right hand side of \eqref{energyeq0} converges to $1/4\|F(0,\cdot)\|^2_{\{t=0\}}$ as $\underline{u} \to +\infty$. On the other hand, we know from Huygens-Fresnel principle that there exists $U>0$ such that $F(t,x)=0$ for all $|t-|x||=|u|\geq U$. This implies that the domain of integration of the integrals in \eqref{energyeq1} is in fact included in $\{|u| \leq U \}$ for all $\underline{u} \geq 0$. The triangular inequality in $L^2$ together with the estimates \eqref{esti1good}-\eqref{esti2badnull} then lead to
 $$\int_{\underline{C}_{\underline{u}}}\mathbb{T}[F]_{\underline{L}0} \dr \mu_{\underline{C}_{\underline{u}}} \xrightarrow[ \underline{u} \to +\infty ]{} \frac{1}{4}\int_{|u| \leq U} \int_{\mathbb{S}^2_{\omega}} |\mathscr{F}^+(F(0,\cdot))|^2 \dr \mu_{\mathbb{S}^2} \dr u = \frac{1}{4}\| \underline{\alpha}^{\mathcal{I}^+} \|^2_{\mathcal{I}^+}  .$$
We can then define $\mathscr{F}^+ : \mathcal{E}_{\{t=0\}} \cap C_c^{\infty} \rightarrow \mathcal{E}_{\mathcal{I}^+}$, with $ \mathscr{F}^+(F(0,\cdot)):=\underline{\alpha}^{\mathcal{I}^+}$, and extends it to an injective isometry from $\mathcal{E}_{\{t=0\}}$ to $\mathcal{E}_{\mathcal{I}^+}$.

Consider now a, say, $C^1$ solution $F$ to \eqref{Maxvac} such that $F(0,\cdot) \in \mathcal{E}_{\{t=0\}}$. Fix $\psi \in C^{\infty}_c(\R_u \times \mathbb{S}^2)$ and $R>0$ satisfying $\mathrm{supp}(\psi) \subset [-R,R] \times \mathbb{S}^2$. Let further $(F_n)_{n \geq 0}$ be a sequence of smooth solutions to the vacuum Maxwell equations such that $F_n(0,\cdot)$ is compactly supported for any $n \in \mathbb{N}$ and $F_n(0,\cdot) \to F(0,\cdot)$ in $\mathcal{E}_{\{t=0\}}$. Fix $A \in \{\theta, \varphi \}$ and start by observing that
$$ \left| \left(r\underline{\alpha}(F)_{e_A}-\mathscr{F}^+(F(0,\cdot))_{e_A} \right) \psi \right| \! \lesssim \! \left(\left| r\underline{\alpha}(F)- r\underline{\alpha}(F_n)\right| \! + \! \left| r\underline{\alpha}(F_n)\!- \! \mathscr{F}^+(F_n(0,\cdot))\right|\! + \! \left| \mathscr{F}^+((F_n-F)(0,\cdot)) \right| \right) \! \mathds{1}_{|u| \leq R}.$$
Then, in order to prove $r\underline{\alpha}_{e_A} \rightarrow \mathscr{F}^+(F(0,\cdot))_{e_A}$ in $\mathcal{D}'(\R_u \times \mathbb{S}^2)$, as $\underline{u} \to +\infty$, it suffices to prove that the integral on $\R_u \times \mathbb{S}^2$ of each of the three terms on the right hand side converges to $0$ as $\underline{u} \to + \infty$. For this, consider $\epsilon >0$ and start by noticing that the energy equality \eqref{energyeq0}-\eqref{energyeq1}, applied to $F-F_n$, gives
$$\forall \; n \geq 0, \quad \forall \underline{u} \geq 0, \qquad  \int_{\R_u} \int_{\mathbb{S}^2_{\omega}}  |r\underline{\alpha}(F)-r\underline{\alpha}(F_n)|^2(u,\underline{u},\omega)  \dr \mu_{\mathbb{S}^2} \dr u \leq \| F(0,\cdot)-F_n(0,\cdot)\|^2_{\{t=0\}}.$$
According to \eqref{esti2badnull}, applied to $F_n$, there exists a contant $C_n$, such that
$$\forall \; n \in \mathbb{N}, \; \; \forall \; \underline{u} \geq 0, \qquad \int_{|u| \leq R} \int_{\mathbb{S}^2_{\omega}}  |r\underline{\alpha}(F_n)(u,\underline{u},\omega)-\mathscr{F}^+(F_n(0,\cdot))(u,\omega)|  \dr \mu_{\mathbb{S}^2} \dr u \leq \frac{C_n}{(1+\underline{u})^{\frac{1}{2}}}.$$
Moreover, since $\mathscr{F}^+$ is an isometry, we have $\left\| \mathscr{F}^+(F_n(0,\cdot))-\mathscr{F}^+(F(0,\cdot)) \right\|_{\mathcal{I}^+}=\| F(0,\cdot)-F_n(0,\cdot)\|_{\{t=0\}}$. The last four estimates, together with the Cauchy-Schwarz inequality in $L^2([-R,R]\times \mathbb{S}^2)$, yields
\begin{align*}
\left|\int_{\R_u}\int_{\mathbb{S}^2_{\omega}} \left(r\underline{\alpha}(F)_{e_A}(u,\underline{u},\omega)-\mathscr{F}^+(F)_{e_A}(u,\omega) \right) \psi (u,\omega) \dr \mu_{\mathbb{S}^2} \dr u \right|\lesssim & \| F(0,\cdot)-F_n(0,\cdot)\|_{\{t=0\}}+\frac{C_n}{(1+\underline{u})^{\frac{1}{2}}} ,
\end{align*}
for all $n \in \mathbb{N}$ and $\underline{u} \geq 0$. For a sufficiently large $n$ and $\underline{U}$, which depends on $n$, we can bound the right hand side by $\epsilon$ for all $\underline{u} \geq \underline{U}$. This concludes the proof of the last part of the lemma. 

It remains to show that for any $F(0,\cdot)$ satisfying \eqref{hypo}, then $\mathscr{F}^+(F(0,\cdot))=\underline{\alpha}^{\mathcal{I}^+}$. For this, it suffices to recall that we proved $r\underline{\alpha}(F) \rightarrow \underline{\alpha}^{\mathcal{I}^+}$ in $L^{\infty}_{u,\omega}$.

\end{proof}
\begin{Rq}
In fact, assuming more decay on the initial data, we could prove using the equations $(M''_2)$, $(M_5'')$ and $(M_6'')$ of \cite{CK} that $|\alpha(F)| = O(\underline{u}^{-2-\delta})$ and that $r^2\rho(F)$ as well as $r^2 \sigma(F)$ converge as $\underline{u} \to + \infty$.
\end{Rq}
To conclude the proof of Theorem \ref{Thscat}, it remains us to show that $\mathscr{F}^+$ is surjective. For this, it suffices to prove Proposition \ref{ProscatMax}, which in particular implies that any smooth and compactly supported $\underline{\alpha}^{\mathcal{I}^+}$ has a preimage by $\mathscr{F}^+$. For this, we will make crucial use of \cite[Theorem~$1.1$]{SchlueLindblad}, which is a similar result for solutions to the homogeneous wave equation, and exploit that $\Box F_{\mu \nu}=0$ for any cartesian component $F_{\mu \nu}$.
\begin{Lem}\label{LemVolker}
Let $\Phi \in C(\R_u \times \mathbb{S}^2)$ be a sufficiently regular function, $0< a < \frac{1}{2}$ and $ N \in \mathbb{N}$. Then, there exists a unique solution to wave equation $\Box \phi =0$ on $\R_+ \times \R^3$ satisfying, for any $0< \delta \leq a$,
\begin{align*}
&\forall \; t \in \R_+, \quad \sum_{|\gamma| \leq N} \left\| \langle t-r \rangle^{a-\delta} Z^{\gamma} \phi(t,\cdot) \right\|^2_{L^2(\R^3_x)} \lesssim \sum_{|k|+|\beta| \leq N+3} \int_{u=-\infty}^{+\infty} \int_{\omega \in \mathbb{S}^2}\left| \left( \langle u \rangle \partial_u \right)^k \partial_{\omega}^{\beta} \Phi(u,\omega)\right|^2 \langle u \rangle^{2a} \dr \mu_{\mathbb{S}^2} \dr u 
\end{align*}
and such that $\Phi$ is the radiation field $\mathscr{R}(\phi)$ of $\phi$ along $\mathcal{I}^+$.
\end{Lem}
We will also require standard estimates for smooth solutions to the wave equation.
\begin{Lem}\label{Fried}
Let $\phi$ be a smooth solution to the wave equation $\Box \phi=0$ such that $\|Z^{\gamma}\phi(0,\cdot) \|_{L^2_x}<+\infty$ for any $|\gamma| \le 5$. Then, for any $|\beta| \leq 1$, the radiation field $\mathscr{R} (\partial_{t,x}^{\beta}\phi)$ of $\partial_{t,x}^{\beta}\phi$ is well-defined and 
$$ \forall \, \underline{u } \geq 1, \forall \,  (u,\omega) \in [-\underline{u},\underline{u}] \times \mathbb{S}^2, \qquad \left|r \partial^{\beta}_{t,x}\phi(u,\underline{u},\omega)-\mathscr{R}(\partial^{\beta}_{t,x}\phi)(u,\omega) \right| \lesssim \underline{u}^{-\frac{1}{2}}.$$
Moreover, $\mathscr{R}(\partial_t \phi)=\partial_u \mathscr{R}(\phi)$ and $\mathscr{R}(\partial_{x^i} \phi)=-\frac{x^i}{|x|}\partial_u \mathscr{R}(\phi)$ for all $i \in \llbracket 1 , 3 \rrbracket$.
\end{Lem}
\begin{proof}
The first part of the result is classical. Indeed, since $\Box Z^{\gamma} \phi =0$ for any $|\gamma| \leq 4$, we obtain by applying the standard Klainerman-Sobolev inequality and then an energy inequality (for a proof, see for instance Theorem $1.3$ and Lemma $3.5$ of \cite[Chapter~$II$]{Sogge}), that, for all $|\gamma| \leq 2$,
\begin{equation}\label{KSwave}
 \forall (t,x) \in \R_+ \times \R^3,    \quad (1+t+|x|)(1+|t-|x||)^{\frac{1}{2}}\left| Z^{\gamma} \phi \right|(t,x) \lesssim \sum_{|\beta| \leq |\gamma|+2} \left\| Z^{\beta} \phi (t,\cdot) \right\|_{L^2_x} \lesssim \sum_{|\beta| \leq 4} \left\| Z^{\beta} \phi (0,\cdot) \right\|_{L^2_x}\!.
 \end{equation}
 Now we claim that
 $$ \forall (t,x) \in \R_+ \times \R^3, \qquad |L(r\phi)|(t,x) \lesssim (1+t+|x|)^{-\frac{3}{2}}.$$
Indeed, if $|x|=r \leq \frac{1+t}{2}$, we have $1+t+r \lesssim 1+|t-r|$. Moreover, \eqref{eq:nullderiv} leads to $|L(r\phi)| \leq \sum_{|\beta| \leq 1}|Z^{\beta} \phi|$, so that the claim is implied by \eqref{KSwave}. Otherwise, $|x| \gtrsim 1+t+|x|=1+\underline{u}$ and, by writing the d'Alembertian in spherical coordinates, we obtain from $\Box \, \phi=0$ that
 \begin{equation}\label{eq:scatwave}
0= -L\underline{L} \phi+\frac{2}{r}\frac{L-\underline{L}}{2}\phi+\sum_{1 \leq i < j \leq 3}\frac{ \Omega_{ij}\Omega_{ij} \phi}{r^2}, \qquad \text{leading to} \qquad \underline{L}\left( L(r\phi) \right)  = \sum_{1 \leq i < j \leq 3}\frac{ \Omega_{ij}\Omega_{ij} \phi}{r}.
\end{equation}
In order to integrate along a null straight line $t+r=\underline{u}$, it will be convenient to work with the null coordinate system. We then write $x=|x|\omega$, with $\omega \in \mathbb{S}^2$. As $\underline{L}=2\partial_u$ and in view of \eqref{KSwave}-\eqref{eq:scatwave}, we have
\begin{align*}
 \left|L(r\phi)\right|(t,x)&=\left|L(r\phi)\right|(t-|x|,\underline{u},\omega)\leq\left|L(r\phi)\right|(-t-|x|,\underline{u},\omega)+\frac{1}{2}\int_{u=-t-|x|}^{t-|x|} |\underline{L}(L(r\phi)|(u,\underline{u},\omega) \dr u \\
 & \lesssim \left|L(r\phi)\right|(0,(t+|x|)\omega)+\int_{u=-t-|x|}^{t-|x|} \frac{ \dr u}{(1+\underline{u})^2(1+|u|)^{\frac{1}{2}}} \lesssim (1+\underline{u})^{-\frac{3}{2}},
 \end{align*}
 which concludes the proof of the claim. As $L=2\partial_{\underline{u}}$, we directly deduce from it that
 $$ \forall \, \underline{z} \geq \underline{u} \geq 0, \; \forall \, |u| \leq \underline{u}, \; \forall \, \omega \in \mathbb{S}^2, \qquad |r\phi(u,\underline{z},\omega)-r\phi(u,\underline{u},\omega)| \lesssim \int_{s=\underline{u}}^{\underline{z}}|L(r\phi)|(u,s,\omega) \dr s \lesssim (1+\underline{u})^{-\frac{1}{2}}.$$
 This implies the existence of the radiation field $\mathscr{R}(\phi)$ of $\phi$ as well as the rate of convergence given in the statement of the Lemma. Since $\Box \partial_{x^{\mu}} \phi =0$ and $\|Z^{\gamma}\partial_{x^{\mu}}\phi(0,\cdot) \|_{L^2_x}<+\infty$ for any $|\gamma| \le 4$, the same applies to $\partial_{x^{\mu}} \phi$. Now, note that
 $$ 2r\partial_t \phi = rL\phi+r\underline{L} \phi, \qquad 2 r\partial_{x^i} \phi = \frac{x^i}{|x|}rL\phi-\frac{x^i}{|x|}r\underline{L} \phi+2\langle \partial_{x^i}, e_{\theta} \rangle \, re_{\theta} \phi+2\langle \partial_{x^i}, e_{\varphi} \rangle \, re_{\varphi} \phi, \quad 1 \leq i \leq 3.$$
 Combining \eqref{KSwave} with \eqref{eq:nullderiv} yields $r|L\phi|+r|e_{\theta} \phi|+r|e_{\varphi}|+|\phi| \lesssim \underline{u}^{-1}$ so that $$ \exists \, \phi_{\infty}^{\underline{L}} \in L^\infty (\R_u \times \mathbb{S}^2_\omega), \qquad \underline{L}(r\phi) \xrightarrow[\underline{u} \to + \infty]{L^\infty_{u,\omega}} \phi_{\infty}^{\underline{L}}, \qquad \phi_{\infty}^{\underline{L}}=2\mathscr{R}(\partial_t \phi), \quad \frac{x^i}{|x|}\phi_{\infty}^{\underline{L}}=-2\mathscr{R}(\partial_{x^i} \phi).$$ It remains to use that $\underline{L}(r\phi) (\cdot, \underline{u}, \cdot) \rightharpoonup 2\partial_u \mathscr{R}(\phi)$ in $\mathcal{D}'(\R_u \times \mathbb{S}^2)$ since $r\phi (\cdot, \underline{u}, \cdot)$ converges to $\mathscr{R}(\phi)$ in $L^\infty_{u,\omega}$.
\end{proof}
We are now ready for the last part of this subsection.
\begin{refproof}[Proof of Proposition \ref{ProscatMax}.]
Fix $0 \leq q-\frac{1}{2} < a<1/2$, $N \in \mathbb{N}$ and $\underline{\alpha}^{\mathcal{I}^+} \in \mathcal{E}_{\mathcal{I}^+}$ such that the norm $C[\underline{\alpha}^{\mathcal{I}^+}]$ is finite. Recall that any sufficiently regular solution $F$ to the vacuum Maxwell equations \eqref{Maxvac} satisfies $\Box \, F_{\mu \nu}=0$ for any $0 \leq \mu , \nu \leq 3$. The first step consists in constructing each cartesian component $F_{\mu \nu}$ of the electromagnetic field by applying Lemma \ref{LemVolker} to well-chosen radiation fields. This will define a $2$-form $F$ which will verify the stated estimate. Then, we will prove that $F$ is indeed a solution to the Maxwell equations and, finally, we will derive the pointwise decay estimates.

Assume first that $N \geq 5$ and let us start by identifying the expected radiation field of $F_{\mu \nu}$. For this, assume that $F$ exists and recall the transfer matrix between the cartesian and the null frame
$$\partial_t=\frac{1}{2}L+\frac{1}{2}\underline{L}, \qquad \partial_{x^{i}} = \frac{\omega_i}{2} L-\frac{\omega_i}{2} \underline{L}+\omega_i^{e_\theta} e_{\theta}+\omega_i^{e_\varphi} e_{\varphi}, \qquad 1 \leq i \leq 3, $$
where $\omega_i$ and $\omega_i^{e_A}$ are bounded functions of the spherical variables and are given explicitly in Appendix \ref{SecB}. For convenience, we set $\omega_0:=-1$ and $\omega_0^{e_A}:=0$. Consequently, for any $0 \leq \mu , \nu \leq 3$, there exists smooth functions of $\omega \in \mathbb{S}^2$, $g^{\alpha,\theta}_{\mu \nu}$, $g^{\alpha,\varphi}_{\mu \nu}$, $g^{\rho}_{\mu \nu}$ and $g^{\sigma}_{\mu \nu}$, such that
$$ rF_{\mu \nu} = -\frac{1}{2}\left(\omega_{\mu}^{e_A} \omega_{\nu}-\omega_{\mu} \omega_{\nu}^{e_A}\right) r\underline{\alpha}(F)_{e_A}+g^{\alpha,A}_{\mu \nu} r \alpha(F)_{e_A}+g^{\rho}_{\mu \nu} r \rho(F)+g^{\sigma}_{\mu \nu} r\sigma(F).$$
We then obtain by \eqref{esti1good}-\eqref{esti2badnull} that
\begin{equation}\label{defradfield}
 \mathscr{R}(F_{\mu \nu})=-\frac{1}{2}\left(\omega_{\mu}^{e_A} \omega_{\nu}-\omega_{\mu} \omega_{\nu}^{e_A}\right) \underline{\alpha}^{\mathcal{I}^+}_{e_A}, \qquad 0 \leq \mu,\nu \leq 3.
 \end{equation}
According to Lemma \ref{LemVolker}, we can indeed define a $2$-form $F$ verifying \eqref{defradfield} as well as $\Box F_{\mu \nu}=0$ and
\begin{equation}\label{decayFsatisfied}
\forall \, t \in \R_+, \quad \sum_{|\gamma| \leq N} \| \langle t-r \rangle^{q-\frac{1}{2}} \left| \mathcal{L}_{Z^{\gamma}} F \right| (t,\cdot) \|_{L^2_x} \lesssim \sum_{|\gamma| \leq N} \, \sum_{0\leq \mu,\nu \leq 3} \| \langle t-r \rangle^{q-\frac{1}{2}} Z^{\gamma}( F_{\mu \nu})(t,\cdot)  \|_{L^2_x} \lesssim C[\underline{\alpha}^{\mathcal{I}^+}].
\end{equation}
The remainder of the proof of the case $N \geq 5$ essentially consists in performing linear algebra computations. In order to lighten the notations we temporarily denote $\partial_{x^{\lambda}}$ by $\partial_{\lambda}$. Our goal now is to prove that $F$ is solution to the vacuum Maxwell equations \eqref{Maxvac}, which read in cartesian coordinates 
\begin{equation}\label{Maxcart}
\partial^{\mu} F_{\mu \nu}=0, \qquad \partial^{\mu} {}^* \! F_{\mu \nu}=\partial_{[\lambda} F_{\mu\nu]}:=\partial_{\lambda} F_{\mu \nu}+\partial_{\mu} F_{\nu \lambda}+\partial_{\nu} F_{\lambda \mu}=0.
\end{equation}
For a proof of the second identity, see for instance \cite[Lemma~$2.2$]{massless}. Since $\Box \partial^{\mu} F_{\mu \nu}=0$ and $\Box \partial^{\mu}{}^*\! F_{\mu \nu}=0$, \eqref{Maxcart} would be implied, according to Lemma \ref{LemVolker}, by
$$ \mathscr{R}(\partial^{\mu} F_{\mu \nu})=0, \qquad \mathscr{R}(\partial^{\mu}{}^* \! F_{\mu \nu})=0, \qquad \qquad 0 \leq \nu \leq 3.$$
We compute, using Lemma \ref{Fried}, that for any $0 \leq \lambda \leq 3$,
$$ \mathscr{R}(\partial_{\lambda}F_{\mu \nu})=-\omega_\lambda \partial_u \mathscr{R}(F_{\mu \nu})=\frac{1}{2}\omega_\lambda\left(\omega_{\mu}^{e_A} \omega_{\nu}-\omega_{\mu} \omega_{\nu}^{e_A}\right)\partial_u \underline{\alpha}^{\mathcal{I}^+}_{e_A}, \qquad 0 \leq \mu,\nu \leq 3.$$
This implies in particular that $ \mathscr{R}(\partial_{[\lambda} F_{\mu\nu]})=0$. Furthermore, as $\partial^\mu=\eta^{\mu \lambda}\partial_{\lambda}$, we have
$$ \mathscr{R}(\partial^{\mu}F_{\mu \nu})=\frac{1}{2}\eta^{\mu \lambda} \omega_\lambda\left(\omega_{\mu}^{e_A} \omega_{\nu}-\omega_{\mu} \omega_{\nu}^{e_A}\right)\partial_u \underline{\alpha}^{\mathcal{I}^+}_{e_A}=\frac{1}{2}(\eta(e_A,L)\omega_\nu-\eta(L,L)\omega_\nu^{e_A})\partial_u \underline{\alpha}^{\mathcal{I}^+}_{e_A}=0.$$
We then deduce that $F$ is a smooth solution to the vacuum Maxwell equations. Finally, since the cartesian components of $\underline{L}={\eta_{\underline{L}}}^\mu \partial_\mu$ and $e_A={\eta_{e_A}}^\mu \partial_\mu$ are bounded functions of  $\omega \in \mathbb{S}^2$, we obtain from \eqref{defradfield} and Lemmas \ref{estiweighednorm}, \ref{Fried} that
$$ \mathscr{F}^+(F(0,\cdot))_{e_A}= \lim_{\underline{u} \to \infty} r \underline{\alpha}(F)_{e_A}(\cdot, \underline{u},\cdot)={\eta_{e_A}}^\mu {\eta_{\underline{L}}}^\nu \lim_{\underline{u} \to \infty} r F_{\mu \nu}(\cdot, \underline{u},\cdot)= {\eta_{e_A}}^\mu {\eta_{\underline{L}}}^\nu \mathscr{R}(F_{\mu \nu}) = \underline{\alpha}^{\mathcal{I}^+}_{e_A}, \qquad A \in \{ \theta , \varphi \}.$$
This concludes the proof of the first part of the Proposition for the case $N \geq 5$. Consider now the case $N=0$ and define similarly $F_{\mu\nu}$, through Lemma \ref{LemVolker}, as the unique solution to $\Box F_{\mu \nu}=0$ such that \eqref{defradfield} holds. This directly provides the estimate \eqref{decayFsatisfied} and let us prove that $F$ is a weak solution to \eqref{Maxvac}. For this, consider a sequence $(\underline{\alpha}^{I^+}_n) \in\mathcal{E}_{\mathcal{I}^+}^{\mathbb{N}}$ of smooth and compactly supported scattering states such that $C[\underline{\alpha}^{\mathcal{I}^+}-\underline{\alpha}^{I^+}_n] \to 0$ as $n \to +\infty$. Then, denote by $F_n$ the unique smooth solution to the vacuum Maxwell equations such that $\mathscr{F}^+(F_n (0,\cdot))=\underline{\alpha}^{\mathcal{I}^+}_n$. Applying once again Lemma \ref{LemVolker} to $\mathscr{R}(F_{\mu\nu}-F_{n,\mu \nu})$ yields
\begin{equation}\label{conditionconverge}
\sup_{t \in \R_+}  \|F(t,\cdot)-F_n(t,\cdot) \|_{L^2_x} \lesssim C\big[\underline{\alpha}^{I^+}-\underline{\alpha}^{I^+}_n \big].
\end{equation}
Fix $\psi \in C_c^{\infty}(\R_+\times \R^3_x)$ and $T_{\psi}$ such that $\psi(t,\cdot)=0$ for all $t \geq T_{\psi}$. Remark, since $F_n$ is a classical and then a weak solution to \eqref{Maxvac}, that for any $0 \leq \nu \leq 3$ and $n \in \mathbb{N}$,
\begin{align}
&\left|\int_{\R_+\times \R^3_x} F_{\mu\nu}(t,x) \partial^{\mu}\psi(t,x)\dr x \dr t+\int_{\R^3_x} F_{\mu \nu}(0,x) \psi(0,x)\dr x\right| \label{eq:tovanishtoprove} \\  & =\left| \int_{\R_+\times \R^3_x} (F-F_n)_{\mu\nu}(t,x) \partial^{\mu}\psi(t,x)\dr x \dr+\int_{\R^3_x} (F-F_n)_{\mu \nu}(0,x) \psi(0,x)\dr x \right| \nonumber  \lesssim (1+T_{\psi})\sup_{t \in \R_+}  \|(F-F_n)(t,\cdot) \|_{L^2_x}. 
\end{align}
By \eqref{conditionconverge}, the right hand side converges to $0$ as $n \to+\infty$ whereas the left hand side does not depend on $n$. This implies that \eqref{eq:tovanishtoprove} vanishes. The same applies applies to ${}^*\!F$, so that $F$ is a weak solution to the vacuum Maxwell equations \eqref{Maxvac}. Finally, by continuity of $\mathscr{F}^+$ and \eqref{conditionconverge}, $\mathscr{F}^+(F(0,\cdot))=\underline{\alpha}^{\mathcal{I}^+}$. 

We now focus on the second part of Proposition \ref{ProscatMax}, which merely concerns the cases $N \geq 4$. We apply \cite[Lemma~$3.3$]{SchlueLindblad}, a weighted version of the standard Klainerman-Sobolev inequality, to $Z^{\beta}(F_{\mu \nu})$. Using \eqref{equinorm}, we obtain for any $|\gamma| \leq N-2$ and all $(t,x) \in \R_+ \times \R^3$,
\begin{equation}\label{eq:refschlue}
 \left| \mathcal{L}_{Z^{\gamma}}(F)\right|(t,x) \lesssim \sum_{|\beta|\leq N-2} \sum_{0\leq \mu,\nu \leq 3} |Z^{\beta}(F_{\mu \nu})|(t,x)\lesssim \sum_{|\beta|\leq N} \frac{\big\| \langle t-r \rangle^{q-\frac{1}{2}} \left| \mathcal{L}_{Z^{\beta}}(F)\right|(t,\cdot) \big\|_{L^2_x}}{(1+t+|x|)(1+|t-|x||)^{q}}. 
\end{equation}
The numerator in the right hand side is bounded by $C[\underline{\alpha}^{\mathcal{I}^+}]$. Recall now that $\mathcal{L}_{Z^{\gamma}}(F)$ is a solution to the vacuum Maxwell equations as well. To conclude the proof, it then suffices to use the previous estimate and to apply Corollary \ref{Corgoodnull}, to $\mathcal{L}_{Z^{\gamma}}(F)$ for any $|\gamma| \leq N-3$, as well as Corollary \ref{Corgoodnull2}, to $\mathcal{L}_{Z^\xi}(F)$ for any $|\xi| \leq N-4$.
\end{refproof}

\begin{Rq}
A statement similar to Theorem \ref{Thscat} holds for scattering toward past null infinity $\mathcal{I}^-\cong \R_{\underline{u}} \times \mathbb{S}^2$. One can construct the past forward evolution bijective isometry $\mathscr{F}^- : \mathcal{E}_{\{t=0\}} \rightarrow \mathcal{E}_{\mathcal{I}^-}$, where, if $F(0,\cdot) \in \mathcal{E}_{\{t=0\}} \cap C_c^{\infty}$, $\mathscr{F}^-(F)(\underline{u},\omega):= \lim_{u \to -\infty} r \alpha(F)(u,\underline{u},\omega)$ and $\| \cdot\|_{\mathcal{I}^-}:=\|\cdot\|_{L^2(\R_{\underline{u}} \times \mathbb{S}^2)}$. The scattering map $\mathscr{S}= (\mathscr{F}^-)^{-1} \circ \mathscr{F}^+$ then defines a unitary isomorphism of Hilbert spaces.
\end{Rq}

Finally, we state a direct consequence of Theorem \ref{Thscat}, Proposition \ref{ProscatMax} and the commutation properties of the vacuum Maxwell equations with $\mathcal{L}_Z$, $Z \in \mathbb{K}$.

\begin{Def}
Let $N \geq 0$ and $\mathcal{E}^N_{\{t=0 \}} \subset \mathcal{E}_{\{t=0 \}}$ be the set of the $2$-forms on $\R^{1+3}$ independent of $t$ verifying
$$\|F_0\|^2_{\mathcal{E}^N_{\{t=0 \}}} := \sum_{|\gamma| \leq N}  \|\mathcal{L}_{Z^\gamma}(F_0)(0,\cdot) \|^2_{\mathcal{E}_{\{t=0 \}}} <+\infty.$$
Consider $\mathcal{E}^N_{\mathcal{I}^+} \subset\mathcal{E}_{\mathcal{I}^+}$, the set of the $1$-forms on $\R_u \times \mathbb{S}^2$ which are tangential to the $2$-spheres and such that
$$ \| \underline{\alpha}^{\mathcal{I}^+} \|^2_{N,\mathcal{I}^+} :=  \sum_{|\gamma| \leq N} \| \underline{\alpha}^{\mathcal{I}^+}_{Z^\gamma} \|^2_{\mathcal{I}^+}<+\infty ,$$
where $\underline{\alpha}^{\mathcal{I}^+}_{Z^\gamma}$ is defined recursively from $\underline{\alpha}^{\mathcal{I}^+}$ through Proposition \ref{derivscatMax}. Then, $(\mathcal{E}^N_{\{t=0 \}},\|\cdot \|^2_{\mathcal{E}^N_{\{t=0 \}}})$ and $(\mathcal{E}^N_{\mathcal{I}^+}, \| \cdot \|_{N,\mathcal{I}^+})$ are Hilbert spaces.
\end{Def}
\begin{Cor}\label{Corhigherscat}
For any $N \geq 0$, the restriction of $\mathscr{F}^+$ to $\mathcal{E}^N_{\{t=0 \}}$ is a bijective isometry from $\mathcal{E}^N_{\{t=0 \}}$ to $\mathcal{E}^N_{\mathcal{I}^+}$.
\end{Cor}

\subsection{Existence of an asymptotic state for $F$ and its derivatives}

In order to avoid any confusion, we precise that, as in Sections \ref{Sec3}-\ref{Sec6}, $F$ denotes the electromagnetic field of our solution to the Vlasov-Maxwell system $(f,F)$. The following statement can be easily deduced from previous results. 

\begin{Pro}\label{Proscatinpractice}
For any $|\gamma| \leq N-3$, $\underline{\alpha}(\mathcal{L}_{Z^{\gamma}}F)$ has a continuous radiation field $\underline{\alpha}^{\mathcal{I}^+}_{\gamma}$. Moreover, for any $0 \leq \delta < 1$, we have the rate of convergence
$$ \forall \, \underline{u} \in \R_+, \; |u| \leq \underline{u}, \; \omega \in \mathbb{S}^2, \qquad \left| \langle u \rangle^{\delta} \left( r\underline{\alpha}(\mathcal{L}_{Z^{\gamma}}F)(\underline{u},u,\omega)-\underline{\alpha}^{\mathcal{I}^+}_{\gamma}(u,\omega) \right) \right| \lesssim \Lambda \frac{\log (3+\underline{u})}{(1+\underline{u})^{1-\delta}}.$$
If $|\gamma|=0$, we simply denote the radiation field of $F$ by $\underline{\alpha}^{\mathcal{I}^+}$.
\end{Pro}
\begin{proof}
Recall from Proposition \ref{Com} the form of the source term in the commuted Maxwell equations. Hence, according to the estimates of Proposition \ref{Proproofnullcompo} and Corollary \ref{Corboo3}, $\mathcal{L}_{Z^{\gamma}}F$ satisfies the hypotheses of Proposition \ref{Corgoodnull2}.
\end{proof}

The last goal of this section consists in proving, if $N$ is large enough, that $F$ can be approached by a solution to the vacuum Maxwell equations through an application of Proposition \ref{ProscatMax}, which requires to control $\underline{\alpha}^{\mathcal{I}^+}$ and its derivatives up to order at least $3$. Note then that by iterating Proposition \ref{derivscatMax}, we get that $\underline{\alpha}^{\mathcal{I}^+}_{\gamma}$ can be computed in terms of derivatives of $\underline{\alpha}^{\mathcal{I}^+}$. Conversally, for any $0 \leq a <1/2$, we have
$$\sum_{n_u+n_\theta+n_\varphi \leq N-3}\int_{\R_u} \int_{\mathbb{S}^2} \langle u \rangle^{2a+2n_u} \left| \nabla^{n_u}_u \slashed{\nabla}^{n_\theta}_{e_\theta} \slashed{\nabla}^{n_\varphi}_{e_\varphi} \underline{\alpha}^{\mathcal{I}^+} \right|^2(u,\omega) \dr \mu_{\mathbb{S}^2} \dr u \lesssim \sum_{|\gamma| \leq N-3} \int_{\R_u} \int_{\mathbb{S}^2} \langle u \rangle^{2a} \left|  \underline{\alpha}^{\mathcal{I}^+}_{\gamma} \right|^2(u,\omega) \dr \mu_{\mathbb{S}^2} \dr u .$$
Applying Proposition \ref{Proscatinpractice} for $\delta = (3+2a)/4$ then yields
\begin{equation}\label{eq:defnormscatste}
\sum_{n_u+n_\theta+n_\varphi \leq N-3}\int_{\R_u} \int_{\mathbb{S}^2} \langle u \rangle^{2a+2n_u} \left| \nabla^{n_u}_u \slashed{\nabla}^{n_\theta}_{e_\theta} \slashed{\nabla}^{n_\varphi}_{e_\varphi} \underline{\alpha}^{\mathcal{I}^+} \right|^2(u,\omega) \dr \mu_{\mathbb{S}^2} \dr u \lesssim \Lambda \int_{\R_u} \langle u \rangle^{a-\frac{3}{2}} \dr u \lesssim \frac{\Lambda}{1-2a}.
\end{equation}
We are now ready to prove the following result.
\begin{Pro}
If $N \geq 10$, there exists a solution $F^{\mathrm{vac}}$ of class $C^{N-8}$ to the vacuum Maxwell equations \eqref{Maxvac} such that, for any $\frac{1}{2} \leq q < 1$ and $|\gamma| \leq N-10$,
$$ \forall \; (t,x) \in \R_+ \times \R^3, \qquad r|\mathcal{L}_{Z^{\gamma}}(F)-\mathcal{L}_{Z^{\gamma}}(F^{\mathrm{vac}})|(t,x) \leq \Lambda C_q (1+t+|x|)^{-q},$$
where the constant $C_q>0$ depends on $q$.
\end{Pro}
\begin{proof}
We fix $0 \leq q- \frac{1}{2} <a<1/2$. Since \eqref{eq:defnormscatste} holds, we get from Proposition \ref{ProscatMax} that there exists a solution $F^{\mathrm{vac}}$ of class $C^{N-8}$ to the vacuum Maxwell equations satisfying, for any $|\gamma| \leq N-9$ and $|\xi| \leq N-10$,
\begin{align}
\forall \; (t,x) \in \R_+ \times \R^3, \qquad \quad \left( |\alpha(\mathcal{L}_{Z^{\gamma}}F^{\mathrm{vac}})|+|\rho(\mathcal{L}_{Z^{\gamma}}F^{\mathrm{vac}})|+|\sigma(\mathcal{L}_{Z^{\gamma}}F^{\mathrm{vac}})|\right)(t,x) & \lesssim \frac{\Lambda}{(1+t+|x|)^{1+q}} , \label{eq:1tofinishpaper}\\
\left|r\underline{\alpha}(\mathcal{L}_{Z^{\xi}}F^{\mathrm{vac}})(t,x) -\mathscr{F}^+(\mathcal{L}_{Z^{\xi}}F^{\mathrm{vac}}(0,\cdot))\left( t-|x|,\frac{x}{|x|} \right) \right| & \lesssim \frac{\Lambda}{(1+t+|x|)^q} \label{eq:2tofinishpaper}
\end{align} 
and $\mathscr{F}^+(F^{\mathrm{vac}}(0,\cdot))=\underline{\alpha}^{\mathcal{I}^+}$. Together with Proposition \ref{Proproofnullcompo} and Corollary \ref{Corboo3}, these estimates imply that $\mathcal{L}_{Z^{\gamma}}(F-F^{\mathrm{vac}})$ verifies the assumptions of Proposition \ref{derivscatMax} for any $|\gamma| \leq N-10$. We then deduce, by a straightforward induction, that $\underline{\alpha}^{\mathcal{I}^+}_\gamma=\mathscr{F}^+(\mathcal{L}_{Z^{\gamma}}F^{\mathrm{vac}}(0,\cdot))$. Combining \eqref{eq:2tofinishpaper} with Proposition \ref{Proscatinpractice} then yields
$$ \forall \; (t,x) \in \R_+ \times \R^3, \qquad r\left|\underline{\alpha}(\mathcal{L}_{Z^{\gamma}}F)-\underline{\alpha}(\mathcal{L}_{Z^{\gamma}}F^{\mathrm{vac}})\right|(t,x) \lesssim \Lambda (1+t+|x|)^{-q}, \qquad \quad |\gamma| \leq N-10.$$ 
On the other hand, Proposition \ref{Proproofnullcompo} and \eqref{eq:1tofinishpaper} give, for any null component $\zeta \in \{ \alpha , \rho, \sigma \}$,
$$ \forall \; (t,x) \in \R_+ \times \R^3, \qquad r\left|\zeta(\mathcal{L}_{Z^{\gamma}}F)-\zeta(\mathcal{L}_{Z^{\gamma}}F^{\mathrm{vac}})\right|(t,x) \lesssim \Lambda (1+t+|x|)^{-q}, \qquad \quad |\gamma| \leq N-9,$$
which concludes the proof. 
\end{proof}
\begin{Rq}
According to Corollary \ref{Corhigherscat} and Lemma \ref{estiweighednorm}, $F^{\mathrm{vac}}$ is in fact of class $C^{N-5}$. Moreover, if $N \geq 7$, then the statement of Proposition \ref{Proscatinpractice} still holds for any $|\gamma| \leq N-7$ and the particular value $q=1/2$.
\end{Rq}

\section{Conservation of the total energy of the system}\label{Secenergy}

Since $(f,F)$ is a solution to the Vlasov-Maxwell system, the energy momentum tensor $\mathbb{T}[f,F]$, defined as
$$ \mathbb{T}[f,F]_{\mu \nu} := \mathbb{T}[f]_{\mu \nu}+\mathbb{T}[F]_{\mu \nu}, \qquad \mathbb{T}[f]_{\mu \nu} := \int_{\R^3_v} f v_\mu v_\nu \frac{\dr v}{v^0}, \quad \mathbb{T}[F]_{\mu \nu}:=F_{\mu \beta} {F_{\nu}}^{\beta}-\frac{1}{4}\eta_{\mu \nu} F_{\xi \lambda} F^{\xi \lambda},$$
is divergence free. It provides the conservation of the total energy of the system
$$ \mathbb{E}_t := \int_{\R^3_x} \int_{\R^3_v} v^0 f(t,x,v) \, \dr v \dr x+\frac{1}{2} \int_{\R^3_x} |F|^2(t,x) \, \dr x = \mathbb{E}_0, \qquad |F|^2=\sum_{0 \leq \mu < \nu \leq 3}|F_{\mu \nu}|^2=|E|^2+|B|^2.$$
We would like to relate $\mathbb{E}_0$ to the energy of the scattering states $f_\infty$ and $\underline{\alpha}^{\mathcal{I}^+}$. More precisely, the goal of this section is to prove
\begin{equation}\label{eq:energycons}
 \mathbb{E}_\infty := \int_{\R^3_x} \int_{\R^3_v} v^0 f_\infty(x,v) \, \dr v \dr x+\frac{1}{4} \int_{\R_u} \int_{\mathbb{S}^2_\omega} |\underline{\alpha}^{\mathcal{I}^+}|^2(u,\omega) \, \dr \mu_{\mathbb{S}^2} \dr u=  \mathbb{E}_0.
\end{equation}
Note that $\mathbb{E}_\infty < +\infty$ according to Remark \ref{Rqfinfty} and Proposition \ref{Proscatinpractice}. The statement \eqref{eq:energycons} is a consequence of $\mathbb{E}_t=\mathbb{E}_0$ and the following two Propositions. 
\begin{Pro}\label{Profinfty}
There holds
$$ \lim_{t \to +\infty} \int_{\R^3_x} \int_{\R^3_v} v^0 f(t,x,v) \, \dr v \dr x = \int_{\R^3_x} \int_{\R^3_v} v^0 f_\infty(x,v) \, \dr v \dr x.$$
\end{Pro}
\begin{proof}
Let $t \geq 3$ and perform the change of variables $x^j=y^j+\widehat{v}^j t-\log(t)\widehat{v}^{\mu} (F_{\mu j}^{\infty}(v)+\widehat{v}^j F_{\mu 0}^{\infty}(v))$ to get
\begin{align*}
 \int_{\R^3_x} \int_{\R^3_v} v^0 f(t,x,v) \, \dr v \dr x &= \int_{\R^3_y} \int_{\R^3_v} v^0 f\big(t,X_\C(t,y,v),v \big) \, \dr v \dr y.
 \end{align*}
We then deduce that
 $$ \left|  \int_{\R^3_x} \int_{\R^3_v} v^0 f(t,x,v) \, \dr v \dr x- \int_{\R^3_x} \int_{\R^3_v} v^0 f_\infty(x,v) \, \dr v \dr x\right| \leq \sup_{(x,v) \in \R^6} \langle x \rangle^{\frac{7}{2}} |v^0|^5 \big|f\big(t,X_\C(t,x,v),v\big)-f_\infty(x,v)\big|,$$
which, in view of $N_v \geq 12$, $N_x \geq 11/2$ and Remark \ref{Rqfinfty}, implies the result.
\end{proof}

\begin{Pro}\label{Prounderalphainfty}
We have
$$ \lim_{t \to +\infty} \frac{1}{2} \int_{\R^3_x} |F|^2(t,x) \, \dr x = \frac{1}{4} \int_{\R_u} \int_{\mathbb{S}^2_\omega} |\underline{\alpha}^{\mathcal{I}^+}|^2(u,\omega) \, \dr \mu_{\mathbb{S}^2} \dr u.$$
\end{Pro}
\begin{proof}
Consider $\underline{u} \geq \tau \geq 3$ and introduce the domain $\mathcal{D}_{\underline{u}}^\tau =\{ t+|x| \leq \underline{u}, \, t \geq \tau \}$, which is bounded by the truncated backward light cone $\underline{C}_{\underline{u}}^\tau := \{ t+|x|=\underline{u}, \, t \geq \tau \}$ and $\{t=\tau \} \cap \{ |x| \leq \underline{u}-\tau \}$. In the same spirit as \eqref{energyeq0}, the divergence theorem, applied to $\mathbb{T}[F]_{\mu 0}$ in $\mathcal{D}_{\underline{u}}^\tau$, yields
\begin{equation}\label{eq:tooberagain}
 \int_{\underline{C}_{\underline{u}}^\tau} \mathbb{T}[F]_{\underline{L}0} \dr \mu_{\underline{C}_{\underline{u}}}=\int_{|x| \leq \underline{u}-\tau} \mathbb{T}[F]_{00}(\tau,x) \dr x+\int_{(t,x) \in \mathcal{D}_{\underline{u}}^\tau} F_{0 \lambda} J(f)^\lambda \dr x \dr t.
 \end{equation}
First, we have
\begin{equation*}
\lim_{\underline{u} \to + \infty}\int_{|x| \leq \underline{u}-\tau} \mathbb{T}[F]_{00}(\tau,x) \dr x = \lim_{\underline{u} \to + \infty} \frac{1}{2} \int_{|x| \leq \underline{u}-\tau} |F|^2(\tau,x) \dr x = \frac{1}{2} \int_{\R^3_x} |F|^2(\tau,x) \dr x.
\end{equation*}
Next, since $|F|(t,x) \lesssim (1+t+|x|)^{-1}(1+|t-|x||)^{-1}$ by \eqref{boot1} and $|J(f)|\lesssim  (1+t+|x|)^{-3}$ by Corollary \ref{Corboo3},
$$
\int_{(t,x) \in \mathcal{D}_{\underline{u}}^\tau} F_{0 \lambda} J(f)^\lambda \dr x \dr t \lesssim \int_{t=\tau}^{+\infty}\int_{r=0}^{+\infty} \frac{r^2 \dr r \dr t}{(1+t+r)^4(1+|t-r|)} \lesssim \int_{t=\tau}^{+\infty} \int_{r=0}^{+\infty} \frac{ \dr r \dr t }{(1+t)^{\frac{3}{2}}(1+|t-r|)^{\frac{3}{2}}}  \lesssim \tau^{-\frac{1}{2}}.
$$
Recall from Definition \ref{Defenergytensor} the value of the null components of $\mathbb{T}[F]$. As $|\rho|(t,x)+|\sigma|(t,x) \lesssim (1+t+|x|)^{-\frac{7}{4}}$ by Proposition \ref{Proproofnullcompo} and in view of Proposition \ref{Proscatinpractice}, applied for $\delta > 1/2$,
\begin{align*}
 \int_{\underline{C}_{\underline{u}}^\tau} \mathbb{T}[F]_{\underline{L}0} \dr \mu_{\underline{C}_{\underline{u}}} &=\frac{1}{4} \! \int_{2\tau-\underline{u} \leq u \leq \underline{u}} \int_{\mathbb{S}^2_{\omega}} \left( |\underline{\alpha}(F)|^2+|\rho(F)|^2+|\sigma(F)|^2 \right)(u,\underline{u},\omega) r^2 \dr \mu_{\mathbb{S}^2} \dr u \\
 & = \frac{1}{4} \! \int_{2\tau-\underline{u} \leq u \leq \underline{u}} \! \int_{\mathbb{S}^2_{\omega}}  r^2|\underline{\alpha}(F)|^2(u,\underline{u},\omega)  \dr \mu_{\mathbb{S}^2} \dr u +O(\underline{u}^{-\frac{1}{2}}) \xrightarrow[\underline{u} \to + \infty]{}  \frac{1}{4} \! \int_{\R_u} \! \int_{\mathbb{S}^2_{\omega}} \! \big|\underline{\alpha}^{\mathcal{I}^+}\big|^2(u,\omega)  \dr \mu_{\mathbb{S}^2} \dr u.
 \end{align*}
Letting $\underline{u} \to + \infty$ and then $\tau \to + \infty$ in \eqref{eq:tooberagain} yields the result.
\end{proof}

\appendix

\section{Estimates for the gradients of the kernels}\label{SecA}

In order to estimate the kernels and their derivatives in the integrals of Propositions \ref{GSdecompo} and \ref{GSdecomoderiv}, we introduce the following class of terms.
\begin{Def}
Let $(p,q,d,d_{\mathbf{w}}) \in \mathbb{N}^4$. We define $\mathbf{S}^{d,d_{\mathbf{w}}}_{p,q}$ as the set of the functions $\mathcal{G}:\mathbb{S}^2 \times \R^3 \rightarrow \R$ of the form
\begin{equation}\label{deftermfor}
\mathcal{G}(\omega,v)=\frac{P(\widehat{v},\omega)Q(\mathbf{w}(\omega,v))}{|v^0|^p(1+\omega \cdot \widehat{v})^q},
\end{equation}
where $P$ is a monomial of degree $d$ in $(\widehat{v}^1,\widehat{v}^2, \widehat{v}^3,\omega_1,\omega_2,\omega_3)$ and $Q$ is a monomial of degree $d_{\mathbf{w}}$ in $\mathbf{w}_{\mu \nu}(\omega,v)$, where $0 \leq \mu < \nu \leq 3$.
\end{Def}
All the kernels considered in this paper can be written as linear combination of such terms, with $d_{\mathbf{w}}\in \llbracket 0,3 \rrbracket$. Moreover, if $2q \geq d_{\mathbf{w}}$, by a direct application of Lemma \ref{LemtechestiW}, one can bound $\mathcal{G}(\omega,v)$ in \eqref{deftermfor} by $|v^0|^{2q-d_{\mathbf{w}}-p}$. The estimates of Corollaries \ref{estiW} and \ref{estikernels} of the derivatives of the kernels then follows from the next result.
\begin{Lem}\label{Lemapp}
Let $(p,q,d,d_{\mathbf{w}}) \in \mathbb{N}^4$ and consider $\mathcal{G} \in \mathbf{S}^{d,d_{\mathbf{w}}}_{p,q}$. Then, for any multi-index $\gamma$, $\partial^{\gamma}_v\mathcal{G}(\omega,v)$ can be written as linear combination of terms belonging to certain $\mathbf{S}^{d_0,d_{\mathbf{w},0}}_{p_0,q_0}$, where 
$$(p_0,q_0,d_0,d_{\mathbf{w},0}) \in \mathbb{N}^4, \qquad 2q_0-d_{\mathbf{w},0}-p_0\leq 2q-d_{\mathbf{w}}-p, \qquad q-d_{\mathbf{w}} \leq q_0-d_{\mathbf{w},0} .$$
This implies $|\partial^{\gamma}_v\mathcal{G}|(\cdot,v) \lesssim |v^0|^{2q-d_{\mathbf{w}}-p}$ if $2q \geq d_{\mathbf{w}}$.
\end{Lem}
\begin{proof} This follows from a straightforward induction and the following relations. For any $(i,j,k) \in \llbracket 1 , 3 \rrbracket^3$, 
\begin{align*}
& \partial_{v^j}\widehat{v}^i=\frac{\delta_i^j-\widehat{v}^i \widehat{v}^j}{v^0}, \qquad \partial_{v^j} \omega^i=0, \qquad \partial_{v^j}|v^0|^{-p}=-p\frac{\widehat{v}^j}{|v^0|^{p+1}},  \qquad \partial_{v^j}\mathbf{w}_{0i}(\omega,v)=\frac{\delta_i^j-\widehat{v}^i \widehat{v}^j}{v^0}, \\
&\partial_{v^j}\mathbf{w}_{ik}(\omega,v)=\omega^i\frac{\delta_k^j-\widehat{v}^k \widehat{v}^j}{v^0}-\omega^k\frac{\delta_i^j-\widehat{v}^i \widehat{v}^j}{v^0}, \qquad \partial_{v^j}\left( \frac{1}{1+\omega \cdot \widehat{v}} \right) =\frac{\widehat{v}^j}{v^0(1+\omega \cdot \widehat{v})}-\frac{\mathbf{w}_{0j}(\omega,v)}{v^0(1+\omega \cdot \widehat{v})^2}.
\end{align*}
\end{proof}

\section{The radiation field of the derivatives of the Maxwell field}\label{SecB}

We fix, for all this section, a $C^1$ solution $G$ to the Maxwell equations \eqref{Maxwithsource} with a continuous source term $J$. We assume that there exist $C[G] >0$ and $q >0$ such that, for all $(t,x) \in \R_+ \times \R^3$,
\begin{equation*}
|rG|(t,x)\leq C[G], \qquad r|J|(t,x)+\sum_{|\gamma| \leq 1} |\rho(\mathcal{L}_{Z^{\gamma}}G)|(t,x)+|\sigma(\mathcal{L}_{Z^{\gamma}}G)|(t,x) \leq \frac{C[G]}{(1+t+|x|)^{1+q}}.
 \end{equation*}
As a consequence, $G$ verifies the hypotheses \eqref{eq:assump0} of Proposition \ref{Corgoodnull2} and then has a radiation field $\underline{\alpha}^{\mathcal{I}^+}$. The purpose of this section is to prove that, for any $Z \in \mathbb{K}$, $\mathcal{L}_Z G$ has a radiation field $\underline{\alpha}^{\mathcal{I}^+}_Z$ which can be expressed in terms of the derivatives of $\underline{\alpha}^{\mathcal{I}^+}$. For this, we will use the following bounded functions depending only on the spherical variables,
\begin{align*}
&\omega_i:= \langle \partial_{x^i}, \partial_r \rangle= \frac{x^i}{|x|} , \qquad \omega^{e_A}_i:=\langle \partial_{x^i}, e_A \rangle, \qquad 1 \leq i \leq 3, \quad A \in \{\theta, \varphi \},\\ &\omega^{e_\theta}_1=\cos (\varphi ) \cos (\theta), \quad \omega^{e_\theta}_2=\sin (\varphi) \cos (\theta), \quad \omega^{e_\theta}_3= -\sin (\theta), \quad  \omega^{e_\varphi}_1=-\sin(\varphi), \quad \omega^{e_\varphi}_2=\cos (\varphi), \quad \omega^{e_\varphi}_3=0,
\end{align*}
and we will work in the space of distributions $\mathcal{D}'(\R_u \times \mathbb{S}^2)$. For simplicity, we will simply write $\psi \rightharpoonup \psi^{\mathcal{I}^+}$ if the weak convergence
$$\psi(u,\underline{u},\omega)   \xrightharpoonup[\underline{u} \to + \infty]{} \psi^{\mathcal{I}^+}(u,\omega) \qquad \text{in $\mathcal{D}'(\R_u \times \mathbb{S}^2)$}$$
holds. In particular, the following convergences will be crucial for us.
\begin{Lem}\label{Lemweakconv}
For any $1 \leq  i \leq 3$ and $B \in \{\theta, \varphi \}$,
\begin{align*}
&|G| \rightharpoonup 0, \qquad \frac{r}{2}\underline{L} \big( \underline{\alpha}(G)_{e_B} \big) \rightharpoonup \partial_u \left(\underline{\alpha}^{\mathcal{I}^+}_{e_B}\right), \qquad r^2 L\big( \underline{\alpha}(G)_{e_B}\big) \rightharpoonup -\underline{\alpha}^{\mathcal{I}^+}_{e_B}, \\ & r^2\omega_i^A e_A\big( \underline{\alpha}(G)_{e_B} \big) \rightharpoonup \omega_i^A e_A \left( \underline{\alpha}^{\mathcal{I}^+}_{e_B} \right), \qquad r \rho (G) \rightharpoonup 0, \qquad  r \sigma (G) \rightharpoonup 0.
\end{align*}
Since $2r=\underline{u}-u$, we also have $r L\big( \underline{\alpha}(G)_{e_B}\big) \rightharpoonup 0$, $r\omega_i^A e_A\big( \underline{\alpha}(G)_{e_B} \big) \rightharpoonup 0$, $\rho (G) \rightharpoonup 0$ and $ \sigma (G) \rightharpoonup 0$.
\end{Lem}
\begin{proof}
The first weak convergence follows from $2|G|(u,\underline{u},\omega) \leq C[G](\underline{u}-u)^{-1}$, so that $|G|(\cdot, \underline{u},\cdot) \to 0$ uniformly on any compact subset of $\R_u \times \mathbb{S}^2$. The other ones are a direct consequence of the strong uniform convergence $r\underline{\alpha}(G)(u,\underline{u},\omega) \to \underline{\alpha}^{\mathcal{I}^+}(u,\omega)$, as $\underline{u} \to +\infty$, which is given by Proposition \ref{Corgoodnull2} since $G$ verifies \eqref{eq:assump0}. 
\begin{itemize}
\item For the second one, use $r \underline{L}=\underline{L}r +1$, $\underline{L}=2\partial_u$ and that $\underline{\alpha}(F)_{e_B}(\cdot,\underline{u},\cdot) \to 0$ uniformly on compact subsets of $\R_u \times \mathbb{S}^2$. 
\item The third one is in fact a strong and uniform convergence. Indeed, $r^2 L (\underline{\alpha}(G)_{e_B})=rL(r\underline{\alpha}(G)_{e_B})-r\underline{\alpha}(G)_{e_B}$ and according to \eqref{eq:nablaLralaphabarre}, $r|L(r\underline{\alpha}(G)_{e_B})|\lesssim \underline{u}^{-q}$. 
\item Next, fix $(t,r) \in \R_+^2$, $\psi \in C^{\infty}_c(\R_u \times \mathbb{S}^2)$ and denote by $\vec{v}_i$ the vector field $\omega_i^{e_A}e_A$, which is the projection on the $2$-spheres of $\partial_{x^i}$. Since $(r e_\theta, r e_\varphi)= (\partial_\theta, \partial_{\varphi}/\sin(\theta))$, we have
\begin{align*}
\omega_i^A r^2\big( e_A (\underline{\alpha}(G)_{e_B}) \big)(t,r\omega) \psi(u,\omega)&=  r \psi(u,\omega)  \vec{v}_i \cdot \slashed{\nabla}\big( \underline{\alpha}(G)_{e_B} (t,r\omega)\big), \\
 \omega_i^A  e_A (\underline{\alpha}_{e_B}^{\mathcal{I}^+} \big)(u,\omega) \psi(u,\omega)&= \psi(u,\omega)  \vec{v}_i \cdot \slashed{\nabla}\big( \underline{\alpha}_{e_B}^{\mathcal{I}^+}\big) (u,\omega),
\end{align*}
so that it suffices to apply the divergence theorem on $\mathbb{S}^2$ and to use $r \underline{\alpha}(G)_{e_B} \rightharpoonup \underline{\alpha}_{e_B}^{\mathcal{I}^+}$.
\item Finally, the last two ones ensue from $r|\rho (G)|+r|\sigma (G) | \lesssim \underline{u}^{-q}$.
\end{itemize}
\end{proof}
We now prove a result which directly implies Proposition \ref{derivscatMax}. We consider a more general setting since it does not complicate the proof and in order to be able to apply these properties in different contexts. For this, given a strictly increasing and unbounded sequence of advanced times $\mathbf{s}=(\underline{u}_n)_{n \geq 0}$, we will write $\psi \rightharpoonup_{\mathbf{s}} \psi^{\mathcal{I}^+}$ if the following weak convergence holds,
$$\psi(u,\underline{u}_n,\omega)   \xrightharpoonup[n \to + \infty]{} \psi^{\mathcal{I}^+}(u,\omega) \qquad \text{in $\mathcal{D}'(\R_u \times \mathbb{S}^2)$}.$$
\begin{Pro}
Consider $G$ an $H^1_{\mathrm{loc}}(\R_+ \times \R^3)$ $2$-form and $\underline{\alpha}^{\mathcal{I}^+}$ an $L^2_{\mathrm{loc}}(\R_u \times \mathbb{S}^2_\omega)$ $2$-form tangential to the spheres. Assume that there exists a strictly increasing and unbounded sequence of advanced times $\mathbf{s}=(\underline{u}_n)_{n \geq 0}$ such that
\begin{itemize}
\item $r \underline{\alpha}(G) \rightharpoonup_{\mathbf{s}} \underline{\alpha}^{\mathcal{I}^+}$,
\item all the weak convergences of Lemma \ref{Lemweakconv} holds, at least for the sequence $\mathbf{s}\subset \R_{+,\underline{u}}$.
\end{itemize}
Then, for any $Z \in \mathbb{K}$, there exists $\underline{\alpha}^{\mathcal{I}^+}_Z \in L^2_{\mathrm{loc}}(\R_u \times \mathbb{S}^2_\omega)$, a $2$-form tangential to the spheres, which satisfies $r \underline{\alpha}(\mathcal{L}_Z G)  \rightharpoonup_{\mathbf{s}} \underline{\alpha}^{\mathcal{I}^+}_Z $. Moreover, for any $1 \leq i \leq 3$ and $1 \leq j < k \leq 3$,
\begin{alignat*}{2}
&  \underline{\alpha}^{\mathcal{I}^+}_{\partial_t}= \nabla_u \underline{\alpha}^{\mathcal{I}^+}, \qquad &&\underline{\alpha}^{\mathcal{I}^+}_{\partial_{x^i}}=-\omega_i \nabla_u \underline{\alpha}^{\mathcal{I}^+}, \qquad \underline{\alpha}^{\mathcal{I}^+}_{S}=u\nabla_u \underline{\alpha}^{\mathcal{I}^+}+\underline{\alpha}^{\mathcal{I}^+} \\
&\underline{\alpha}^{\mathcal{I}^+}_{\Omega_{jk}} = \mathcal{L}_{\Omega_{jk}} \big( \underline{\alpha}^{\mathcal{I}^+} \big), \qquad && \underline{\alpha}^{\mathcal{I}^+}_{\Omega_{0i}}=- \omega_i u\nabla_u  \underline{\alpha}^{\mathcal{I}^+}-2\omega_i \underline{\alpha}^{\mathcal{I}^+}+ \omega^{e_A}_i \slashed{\nabla}_{e_A}  \underline{\alpha}^{\mathcal{I}^+} .
\end{alignat*}
\end{Pro}
\begin{proof}
In order to avoid technical difficulties related to the degeneracies of the spherical coordinate system, we will in fact prove weak convergences in 
$$\mathcal{D}'(\R_u \times K), \qquad K := \{ \omega \in \mathbb{S}^2 \, | \, \sin \theta > 1/8 \}.$$
 The convergences in the full space $\mathcal{D}'(\R_u \times \mathbb{S}^2)$ can then be obtained by applying the upcoming results to an other well-chosen spherical coordinate system. 

We fix, for all this proof, $B \in \{\theta, \varphi \}$, $i \in \llbracket 1,3 \rrbracket$ and we recall that for any $Z \in \mathbb{K}$,
$$ r\underline{\alpha}( \mathcal{L}_Z G)_{e_B}=rZ(\underline{\alpha}(G)_{e_B})-rG([Z,e_B],\underline{L})-rG(e_B,[Z,\underline{L}]).$$
Then, we have
$$ r\underline{\alpha}( \mathcal{L}_{\partial_t} G)_{e_B}=\frac{r}{2}\underline{L}(\underline{\alpha}(G)_{e_B})+\frac{r}{2}L(\underline{\alpha}(G)_{e_B}) \rightharpoonup_{\mathbf{s}} \partial_u \left(\underline{\alpha}^{\mathcal{I}^+}_{e_B}\right).$$
For the spatial translation $\partial_{x^i}=-\frac{\omega_i}{2}\underline{L}+\frac{\omega_i}{2}L+\omega_i^Ae_A$, we use that $[\partial_{x^i} , \underline{L}]=-\frac{\omega_i^{e_A}}{r} e_A$ and $[\partial_{x^i}, e_A]=\partial_{x^i}(\omega^{e_A}_j) \partial_{x^j}$, with $\partial_{x^i}(\omega^{e_A}_j)\lesssim r^{-1}$ on $K$. We get
\begin{align*}
r\underline{\alpha}( \mathcal{L}_{\partial_{x^i}} G)_{e_B}&=-\frac{\omega_i r}{2}\underline{L}(\underline{\alpha}(G)_{e_B})+\frac{\omega_i r}{2}L(\underline{\alpha}(G)_{e_B})+r\omega_i^Ae_A(\underline{\alpha}(G)_{e_B})-r\partial_{x^i}(\omega^{e_A}_j)G(e_B,\partial_{x^j})+\omega_i^{e_A}G(e_B, e_A)\\ 
& \rightharpoonup_{\mathbf{s}} -\omega_i \partial_u \left(\underline{\alpha}^{\mathcal{I}^+}_{e_B}\right).
 \end{align*}
For the scaling, recall that $[S,\underline{L}]=-\underline{L}$ and $[S,e_B]=-e_B$. As $2S=u\underline{L}+(u+2r)L$, we have
$$ r\underline{\alpha}( \mathcal{L}_S G)_{e_B}=\frac{u}{2}r\underline{L}(\underline{\alpha}(G)_{e_A})+\frac{u+2r}{2}rL(\underline{\alpha}(G)_{e_A})+2rG(e_B,\underline{L}) \rightharpoonup_{\mathbf{s}} u \partial_u \left(\underline{\alpha}^{\mathcal{I}^+}_{e_B}\right)+\underline{\alpha}^{\mathcal{I}^+}_{e_B}.$$
Next, for the Lorentz boost $\Omega_{0i}$, we use
$$ \Omega_{0i}=\frac{\omega_i}{2}(\underline{u} L-u \underline{L})+t \omega_i^{e_A} e_A, \qquad [\Omega_{0i},e_B]=\frac{\omega_i^{e_B}}{2r}(u \underline{L}-\underline{u}L)+\frac{t}{r}\omega_i^{e_A} \slashed{\Gamma}^D_{AB}e_D , \qquad  [\Omega_{0i}, \underline{L}]=\omega_i \underline{L}-\frac{\underline{u}}{r} \omega_i^{e_A}e_A,$$
where $\slashed{\Gamma}^D_{AB}$ are the Cristofell symbols of $\mathbb{S}^2$ in the nonholonomic basis $(e_\theta,e_\varphi)$. In particular, $\slashed{\Gamma}^D_{AB}$ is bounded on $K$. As $\underline{u}=u+2r$ and $t=u+r$, we obtain
\begin{align*}
 r\underline{\alpha}( \mathcal{L}_{\Omega_{0i}} G)_{e_B}=&-\frac{\omega_i u}{2}r\underline{L}(\underline{\alpha}(G)_{e_A})+\frac{\omega_i (u+2r)}{2}rL(\underline{\alpha}(G)_{e_B})+ \omega^{e_A}_i (u+r)r e_A(\underline{\alpha}(G)_{e_B}) -\frac{\omega_i^{e_B}}{2}uG(\underline{L},\underline{L})\\
 & +\frac{\omega_i^{e_B}}{2}\underline{u}G(L,\underline{L})-(u+r)\omega_i^{e_A} \slashed{\Gamma}^D_{AB}G(e_D,\underline{L}) -\omega_i rG(e_B,\underline{L})+\underline{u}\omega_i^{e_A} G(e_B,e_A) .
 \end{align*}
Since $G(\underline{L},\underline{L})=0$ and $\underline{u}(|G(L,\underline{L})|+|G(e_A,e_B)|)= (u+2r) (2|\rho(G)|+|\sigma(G)|)\rightharpoonup_{\mathbf{s}} 0$, we get
\begin{align*}
 r\underline{\alpha}( \mathcal{L}_{\Omega_{0i}} G)_{e_B} &\rightharpoonup_{\mathbf{s}} - \omega_i u \partial_u \left(\underline{\alpha}^{\mathcal{I}^+}_{e_B}\right)-\omega_i \underline{\alpha}^{\mathcal{I}^+}_{e_B}+\omega^{e_A}_i e_A \left( \underline{\alpha}^{\mathcal{I}^+}_{e_B} \right)-0+0- \omega_i^{e_A} \slashed{\Gamma}^D_{AB}\underline{\alpha}^{\mathcal{I}^+}_{e_D}-\omega_i \underline{\alpha}^{\mathcal{I}^+}_{e_B} +0 \\
 &= - \omega_i u\partial_u\left(  \underline{\alpha}^{\mathcal{I}^+}_{e_B}\right)-2\omega_i \underline{\alpha}^{\mathcal{I}^+}_{e_B}+ \omega^{e_A}_i \slashed{\nabla}_{e_A} \big( \underline{\alpha}^{\mathcal{I}^+} \big)_{e_B}.
 \end{align*}
Finally, we recall the expression of the rotations in the spherical coordinate system $(t,r,\theta,\varphi)$,
$$ \Omega_{12}=\partial_{\varphi}, \qquad \Omega_{13}=\cos(\varphi ) \partial_{\theta }- \cot(\theta) \sin (\varphi) \partial_{\varphi}, \qquad \Omega_{23}=-\sin(\varphi ) \partial_{\theta }- \cot(\theta) \cos (\varphi) \partial_{\varphi}.$$
In particular, these vector fields, tangential to the spheres, are well-defined on $\mathcal{I}^+ \simeq \R_u \times \mathbb{S}^2$. Fix now $(j,k) \in \llbracket 1 , 3 \rrbracket^2$ and write $\Omega_{jk}=\Omega_{jk}^\theta \partial_\theta+\Omega_{jk}^{\varphi} \partial_{\varphi}$. Remark, using first $[\Omega_{jk},\underline{L}]=0$ and then the expression of the Lie derivative in the spherical coordinate system, that
$$\underline{\alpha}(\mathcal{L}_{\Omega_{jk}} G)_{\partial_B}=\mathcal{L}_{\Omega_{jk}} \left(\underline{\alpha}( G)\right)_{\partial_B}= \Omega_{jk} \left( \underline{\alpha}(G)_{\partial_B} \right)+ \partial_B\left( \Omega_{jk}^A \right)\underline{\alpha}(G)_{\partial_A}.$$
Recall now that $(re_\theta,re_\varphi)=(\partial_\theta,\partial_{\varphi}/\sin(\theta))$ on $\R_+ \times \R^3$ and $(e_\theta,e_\varphi)=(\partial_\theta,\partial_{\varphi}/\sin(\theta))$ on $\R_u\times \mathbb{S}^2$. Hence, using $r \underline{\alpha}(G)_{e_A} \rightharpoonup_{\mathbf{s}} \underline{\alpha}^{\mathcal{I}^+}_{e_A}$ and since any of the quantities considered is smooth and bounded on $K$,
\begin{align*}
r\underline{\alpha}(\mathcal{L}_{\Omega_{jk}} G)_{e_\theta}&= \Omega_{jk}^{\theta} \partial_{\theta}\left( r\underline{\alpha}(G)_{e_\theta} \right)+\Omega_{jk}^{\varphi} \partial_{\varphi}\left(r \underline{\alpha}(G)_{e_\theta} \right)+ \partial_\theta \left( \Omega_{jk}^\theta \right)r\underline{\alpha}(G)_{e_\theta}+\sin(\theta)\partial_\theta \left( \Omega_{jk}^\varphi \right)r\underline{\alpha}(G)_{e_\varphi} \\
& \rightharpoonup_{\mathbf{s}} \Omega_{jk} \left( \underline{\alpha}^{\mathcal{I}^+}_{e_\theta} \right)+ \partial_\theta \left( \Omega_{jk}^\theta \right)\underline{\alpha}^{\mathcal{I}^+}_{e_\theta}+\sin(\theta)\partial_\theta \left( \Omega_{jk}^\varphi \right)\underline{\alpha}^{\mathcal{I}^+}_{e_\varphi} \\
& =\Omega_{jk} \left( \underline{\alpha}^{\mathcal{I}^+}_{\partial_\theta} \right)+ \partial_\theta\left( \Omega_{jk}^A \right)\underline{\alpha}^{\mathcal{I}^+}_{\partial_A} = \mathcal{L}_{\Omega_{kl}}\big(\underline{\alpha}^{\mathcal{I}^+} \big)_{e_\theta}.
\end{align*}
Similarly, we get
\begin{align*}
r\underline{\alpha}(\mathcal{L}_{\Omega_{jk}} G)_{e_\varphi}&= \Omega_{jk} \left(r \underline{\alpha}(G)_{e_\varphi} \right)-\Omega_{kl}\left( \frac{1}{\sin \theta} \right)r \underline{\alpha}(G)_{e_\varphi} + \frac{1}{\sin \theta}\partial_\varphi \left( \Omega_{jk}^\theta \right)r\underline{\alpha}(G)_{e_\theta}+\partial_\varphi \left( \Omega_{jk}^\varphi \right)r\underline{\alpha}(G)_{e_\varphi} \\
& \rightharpoonup_{\mathbf{s}} \Omega_{jk} \left( \underline{\alpha}^{\mathcal{I}^+}_{e_\varphi} \right)-\Omega_{kl}\left( \frac{1}{\sin \theta} \right)\underline{\alpha}^{\mathcal{I}^+}_{e_\varphi}+ \frac{1}{\sin \theta}\partial_\varphi \left( \Omega_{jk}^\theta \right)\underline{\alpha}^{\mathcal{I}^+}_{e_\theta}+\partial_\varphi \left( \Omega_{jk}^\varphi \right)\underline{\alpha}^{\mathcal{I}^+}_{e_\varphi}  = \mathcal{L}_{\Omega_{kl}}\big(\underline{\alpha}^{\mathcal{I}^+} \big)_{e_\varphi}.
\end{align*}
\end{proof}

\section{Remarks on $F^\infty$ and the modified characteristics}\label{SecAnnexeC}

\subsection{Alternative expression for $F^\infty$}\label{AnnexeC1} We could define $F^\infty$ in a slightly different way. However, contrary to what we did in Section \ref{Subsecmachin}, we could not define in such a way $\mathcal{L}_{Z^\gamma}(F)^\infty$ for the derivatives of order $|\gamma|=N-1$. Using the representation formula for the wave equation satisfied by $F_{\mu \nu}$,
\begin{align*}
 F_{\mu \nu} = F^{\mathrm{hom}}_{\mu \nu}+ \big[ f \big]^{\mathrm{inh}}_{\mu \nu}, \qquad \left[ f \right]^{\mathrm{inh}}_{\mu \nu}(t,x) := \frac{1}{4 \pi}\int_{|y-x|\leq t} \int_{\R^3_v}\left(\widehat{v}_{\mu} \partial_{x^{\nu}}  f-\widehat{v}_{\nu} \partial_{x^\mu}  f\right)(t-|y-x|,y,v) \dr v \frac{\dr y}{|y-x|}.
\end{align*}  
In order to investigate the asymptotic behavior of $ [  f ]^{\mathrm{inh}}$, it is important to determine the asymptotic profile of the source term of the wave equation. In particular, we need to obtain a better estimate than the one given by Proposition \ref{gainderivdecaymoyv} which does not provide the expected time decay $t^{-4}$. The starting point consists in observing that a kind of null condition holds,
\begin{alignat*}{2}
& t\left(\partial_{x^i}+\widehat{v}_i \partial_t \right)=\Omega_{0i}+z_{0i} \partial_t=\widehat{\Omega}_{0i}-v^0 \partial_{v^i}+ \partial_t z_{0i}-\widehat{v}_i=\widehat{\Omega}_{0i}-\partial_{v^i} v^0+\partial_t z_{0i}, \qquad \quad && 1 \leq i \leq 3, \\
& t\left(\widehat{v}_j\partial_{x^k}-\widehat{v}_k \partial_{x^j} \right) = \widehat{v}_j \widehat{\Omega}_{0k}-\widehat{v}_k \widehat{\Omega}_{0j}-\partial_t z_{jk}-\widehat{v}_j \partial_{v^k} v^0+\widehat{v}_k \partial_{v^j} v^0, \qquad && 1 \leq j < k \leq 3.
\end{alignat*}
Hence, using the convention $\widehat{\Omega}_{00}=0$ and peforming integration by parts, we obtain, for any $0 \leq \mu < \nu \leq 3$,
\begin{equation*}
 \int_{\R^3_v} \widehat{v}_{\mu} \partial_{x^{\nu}}  f-\widehat{v}_{\nu} \partial_{x^\mu}  f \dr v= \frac{1}{t}\int_{\R^3_v}  \left(\widehat{v}_\mu \widehat{\Omega}_{0\nu}  f- \widehat{v}_{\nu}  \widehat{\Omega}_{0\mu}  f\right)\dr v-\frac{1}{t}\int_{\R^3_v} \partial_t \left( z_{\mu \nu}  f \right)\dr v.
 \end{equation*}
The leading order term of its asymptotic expansion is the first term on the right hand side. Its behavior can be easily obtained from Proposition \ref{ProprifilJ}. Following the proof of Proposition \ref{gainderivdecaymoyv}, one could prove that last term almost decay as $t^{-5}$. It will then be convenient to use the notation $Q^{\widehat{\Omega}_{0\ell}  }_{\infty}$ in order to denote $Q^{\kappa}_{\infty}$, where $\widehat{Z}^\kappa=\widehat{\Omega}_{0 \ell} $, and to set $Q^{\widehat{\Omega}_{00}}_\infty:=0$. Following the proof Proposition \ref{Proinducedfield}, we could obtain
$$
\lim_{t \to + \infty} t^2\big[ f \big]^{\mathrm{inh}}_{\mu \nu}(t,x+t\widehat{v}) := \frac{1}{4 \pi}\int_{\substack{|z| \leq 1 \\  |z+\widehat{v}| < 1-|z|  }}  \left( |v^0|^5 \left( \widehat{v}_\mu Q^{\widehat{\Omega}_{0\nu}}_{\infty}-\widehat{v}_\nu Q^{\widehat{\Omega}_{0\mu}}_{\infty} \right) \right)\left(\frac{\widecheck{z+\widehat{v}}}{1-|z|} \right) \frac{\dr z}{|z|(1-|z|)^4},
$$
which is necessarily equal to $F^{\infty}$.
\subsection{The support of the corrections of the linear characteristics and commutators}\label{AnnexeC2} We could obtain similar results by modifying the trajectories and the homogeneous vector fields only inside the light cone. More precisely, we could consider, for $\widehat{Z} \in \widehat{\mathbb{P}}_0 \setminus \{ \partial_t, \partial_{x^1} , \partial_{x^2},  \partial_{x^3} \} $,
$$ \widetilde{X}_{\C}(t,x,v):=x+t\widehat{v}+\C(t,v) \, \chi (t-|x-t\widehat{v}|), \qquad \widetilde{Z}^{\mathrm{mod}}:=\widehat{Z}+\C_{\widehat{Z}}^i \, \chi (t-r) \partial_{x^i},$$
where $\chi$ is a cutoff function satisfying $\chi(s)=0$ for $s \leq 1$ and $\chi(s)=1$ for $s \geq 2$. It is not surprising that all the results proved for $X_\C$ and $\widehat{Z}^{\mathrm{mod}}$ hold as well with these corrections since the Vlasov field enjoy strong decay properties in the exterior of the light cone (see Lemma \ref{gainv}). We could even avoid the loss of the weight $\mathbf{z}^{\beta_H}$ in Proposition \ref{ProestCommod} and Corollary \ref{Proimproderiv}. Indeed, these weights come from the identity $x^i/t=(x^i-t\widehat{v}^i)/t+\widehat{v}^i$ that we performed during the proof of Proposition \ref{improvedCom}. On the support of $\chi$, we can simply use that $|x|/t \leq 1$. However, we could not save any $\langle x \rangle$ weight in the analogue version of the scattering statement of Proposition \ref{mainresultPro} since we would have to lose a power of $\mathbf{z}^{\beta_H}$ in order to estimate $|v^0|^{|\xi}\partial_v^{\xi} (\chi (t-|x+t\widehat{v}|))$.

\renewcommand{\refname}{References}
\bibliographystyle{abbrv}
\bibliography{biblio}

\end{document}